\documentclass[12pt,A4paper]{article}
\usepackage{hyperref}
\usepackage{amsmath, amsthm, amsfonts, amssymb}
\usepackage{graphicx}
\usepackage{stackrel}

\input xy
\xyoption{all}

 \swapnumbers
\usepackage[active]{srcltx}

\usepackage[english]{babel}
\usepackage{color}
\usepackage{verbatim}
\usepackage[all]{xy}
\usepackage{graphicx}
\usepackage[numbers]{natbib}

\usepackage{indentfirst}

\numberwithin{equation}{subsection}
\theoremstyle{plain}
  \begingroup
        \newtheorem{theorem}[equation]{Theorem}
        \newtheorem{lemma}[equation]{Lemma}
        \newtheorem{proposition}[equation]{Proposition}
        \newtheorem{corollary}[equation]{Corollary}

	    \newtheorem{definition}[equation]{Definition}

        \newtheorem{sinnadaitalica}[equation]{}

\endgroup

\theoremstyle{definition}
  \begingroup
        \newtheorem{remark}[equation]{Remark}
        
        \newtheorem{sinnadastandard}[equation]{}

\endgroup

\newcommand{\eps}{\varepsilon}

\newcommand{\cqd}{\hfill$\Box$}  

\newcommand{\comw}{\textcolor{white}}

\newcommand{\cc}{\mathcal}

\newcommand{\ff}{\mathsf}

\newcommand{\mr}[1]{\overset {#1} {\longrightarrow}}

\newcommand{\smr}[1]{\overset {#1} {\rightarrow}}

\newcommand{\ml}[1]{\overset {#1} {\longleftarrow}}

\newcommand{\Mr}[1]{\overset {#1} {\Longrightarrow}}

\newcommand{\Mrsimeq}[1]{\stackrel[\cong] {#1} {\Longrightarrow}}

\newcommand{\mrpairviejo}[2]
   {
    \xymatrix@C=5ex@R=2.4ex
            {
             {} \ar@<1.6ex>[r]^{#1}
	            \ar@<-1.1ex>[r]^{#2}
	         & {}
            }
   }

\newcommand{\mrpair}[2]
   {
    \xymatrix@C=5ex@R=2.4ex
            {
             {} \ar@<1ex>[r]^{#1}
	            \ar@<-1ex>[r]_{#2}
	         & {}
            }
   }

\newcommand{\mlpair}[2]
   {
    \xymatrix@C=5ex@R=2.4ex
            {
             {}
              & {} \ar@<1.0ex>[l]_{#2}
	          \ar@<-1.7ex>[l]_{#1}
            }
    }

\newcommand{\cellrd}[3] 
 {
  \xymatrix@C=7ex@R=2.4ex
         {
          {} \ar@<1.6ex>[r]^{#1}
             \ar@{}@<-1.3ex>[r]^{\Downarrow \; {#2}}
             \ar@<-1.1ex>[r]_{#3}
          & {}
         }
}
\newcommand{\cellrdb}[3] 
 {
  \xymatrix@C=7ex@R=2.4ex
         {
          {} \ar@<1.9ex>[r]^{#1}
             \ar@{}@<-1.3ex>[r]^{\Downarrow \; {#2}}
             \ar@<-1.1ex>[r]_{#3}
          & {}
         }         
 }
 \newcommand{\scellrd}[3] 
 {
  \xymatrix@C=4.5ex@R=2.4ex
         {
          {} \ar@<1.6ex>[r]^{#1}
             \ar@{}@<-1.3ex>[r]^{\!\! \Downarrow \, {#2}}
             \ar@<-1.1ex>[r]_{#3}
          & {}
         }
}
 \newcommand{\modif}[3] 
 {
  \xymatrix@C=7ex@R=2.4ex
         {
          {} \ar@<1.6ex>@{=>}[r]^{#1}
             \ar@{}@<-1.3ex>@{=>}[r]^{\!\! {#2} \, \!\downarrow}
             \ar@{}@<-1.1ex>[r]_{#3}
          & {}
         }
 }

\newcommand{\cellld}[3] 
 {
  \xymatrix@C=6ex@R=2.4ex
         {
            {}
          & {} \ar@<1.0ex>[l]^{#3}
          \ar@{}@<-1.7ex>[l]^{\!\! {#2} \, \!\Downarrow}
	                                 \ar@<-1.7ex>[l]_{#1}
         }
 }

\newcommand{\cellpairrd}[4] 
 {
  \xymatrix@C=8ex@R=2.2ex
         {
          {} \ar@<1.6ex>[r]^{#1}
             \ar@{}@<-1.3ex>[r]^{\!\! \Downarrow \, {#2} 
                                 \;\;\; \Downarrow \, {#3} }
             \ar@<-1.1ex>[r]_{#4}
          & {}
         }
 }

\newcommand{\coLim}[2]
   {
    \underset{#1}{\underrightarrow{\ff{Lim}}}
    \; {#2}
   }

\newcommand{\Lim}[2]
   {
    \underset{#1}{\underleftarrow{\ff{Lim}}}
    \; {#2}
   }

\newcommand{\biLim}[2]
   {
    \underset{#1}{\underleftarrow{\ff{biLim}}}
    \; {#2}
   }

\newcommand{\bicoLim}[2]
   {
    \underset{#1}{\underrightarrow{\ff{biLim}}}
    \; {#2}
   }

\newcommand{\dcell}[1]  
          {
					 \ar@<8pt>@{-}[d]+<-4pt,8pt>
           \ar@<-8pt>@{-}[d]+<4pt,8pt>
           \ar@{}[d]|{#1}
          }

\newcommand{\dcellb}[1]   
          {
           \ar@<10pt>@{-}[d]+<-5pt,8pt>
           \ar@<-10pt>@{-}[d]+<5pt,8pt>
           \ar@{}[d]|{#1}
          }

\newcommand{\dcellbymedio}[1]   
          {
           \ar@<15pt>@{-}[d]+<-7.5pt,10pt>
           \ar@<-15pt>@{-}[d]+<7.5pt,10pt>
           \ar@{}[d]|{#1}          
          }
\newcommand{\dcellbymediobis}[1]   
          {
           \ar@<15pt>@{-}[d]+<-7.5pt,-5pt>
           \ar@<-15pt>@{-}[d]+<7.5pt,0pt>
           \ar@{}[d]|{#1}          
          }
\newcommand{\deq}        
         {
          \ar@{=}[d]
         }

\newcommand{\ddeq}{\ar@{=}[dd]}
\newcommand{\dddeq}{\ar@{=}[ddd]}         
\newcommand{\ddddeq}{\ar@{=}[dddd]}         
\newcommand{\dddddeq}{\ar@{=}[ddddd]}                  
\newcommand{\ddddddeq}{\ar@{=}[dddddd]}                           
\newcommand{\dddddddeq}{\ar@{=}[ddddddd]}         
\newcommand{\ddddddddeq}{\ar@{=}[dddddddd]}         
\newcommand{\dddddddddeq}{\ar@{=}[ddddddddd]}                  
\newcommand{\ddddddddddeq}{\ar@{=}[dddddddddd]}                  
\newcommand{\dddddddddddeq}{\ar@{=}[ddddddddddd]}                  

\newcommand{\dreq}       
         {
          \ar@{=}[dr]
         }

\newcommand{\dleq}       
         {
          \ar@{=}[dl]
         }

\newcommand{\dccell}[1]    
          {
           \ar@{-}[ld]
           \ar@{-}[rd]
           \ar@{}[d]|{#1}
          }

\newcommand{\dcellbb}[1]   
          {
           \ar@<20pt>@{-}[d]+<-10pt,12pt>
           \ar@<-20pt>@{-}[d]+<10pt,12pt>
           \ar@{}[d]|{#1}
          }

\newcommand{\dl}    
          {
           \ar@<-2pt>@{-}[d]+<4pt,8pt>
          }

\newcommand{\dr}    
          {
           \ar@<2pt>@{-}[d]+<-4pt,8pt>
          }
\newcommand{\drbis}    
          {
           \ar@<-2pt>@{-}[d]+<-4pt,8pt>
          }
\newcommand{\drmediobis}    
          {
           \ar@<-1pt>@{-}[d]+<-4pt,8pt>
          }
\newcommand{\dc}[1]    
          {
           \ar@{}[d]|{#1}
          }
\newcommand{\dcr}[1]    
          {
           \ar@{}[dr]|{#1}
          }
\newcommand{\dcl}[1]    
          {
           \ar@{}[dl]|{#1}
          }

\newcommand{\uccell}[1]      
          {
           \ar@{-}[ur]
           \ar@{}[u]|{#1}
           \ar@{-}[ul]
          }

\newcommand{\uccellb}[1]     
          {
           \ar@<-1ex>@{-}[ur]
           \ar@{}[u]|{#1}
           \ar@<1ex>@{-}[ul]
          }

\newcommand{\dcellop}[1]  
          {
					 \ar@<6pt>@{-}[d]+<6pt,8pt>
           \ar@<-6pt>@{-}[d]+<-6pt,8pt>
           \ar@{}[d]|{#1}
          }

\newcommand{\dcellopb}[1]  
          {
					 \ar@<7pt>@{-}[d]+<7pt,8pt>
           \ar@<-7pt>@{-}[d]+<-7pt,8pt>
           \ar@{}[d]|{#1}
          }

\newcommand{\dcellopbb}[1]  
          {
					 \ar@<8pt>@{-}[d]+<8pt,8pt>
           \ar@<-8pt>@{-}[d]+<-8pt,8pt>
           \ar@{}[d]|{#1}
          }
          
\newcommand{\did}{\ar@2{-}[d]}

\newcommand{\dig}{ \ar@2{-}[d] & & }

\newcommand{\op}[1]
          {
           \ar@{-}[ld]
           \ar@{-}[rd]
           \ar@{}[d]|{#1}
          }

\newcommand{\opb}[1]
          {
           \ar@<-2pt>@{-}[ld]
           \ar@<2pt>@{-}[rd]
           \ar@{}[d]|{#1}
          }        
\newcommand{\opmediob}[1]
          {
           \ar@<-1pt>@{-}[ld]
           \ar@<1pt>@{-}[rd]
           \ar@{}[d]|{#1}
          }  
          
\newcommand{\opbymedio}[1]
          {
           \ar@<-3pt>@{-}[ld]
           \ar@<3pt>@{-}[rd]
           \ar@{}[d]|{#1}
          }     
\newcommand{\opbb}[1]
          {
           \ar@<-4pt>@{-}[ld]
           \ar@<4pt>@{-}[rd]
           \ar@{}[d]|{#1}
          }               

\newcommand{\opunodos}[1]
          {
           \ar@{-}[ld]
           \ar@{-}[rrd]
           \ar@{}[dr]|{#1}
          }
          
\newcommand{\opunodosb}[1]
          {
           \ar@<-2pt>@{-}[ld]
           \ar@<2pt>@{-}[rrd]
           \ar@{}[dr]|{#1}
          }

\newcommand{\opdosuno}[1]
          {
           \ar@{-}[lld]
           \ar@{-}[rd]
           \ar@{}[d]|{#1}
          }          
          
\newcommand{\opdosdos}[1]
          {
           \ar@{-}[lld]
           \ar@{-}[rrd]
           \ar@{}[d]|{#1}
          }    
\newcommand{\opdostres}[1]
          {
           \ar@{-}[lld]
           \ar@{-}[rrrd]
           \ar@{}[d]|{#1}
          }    
\newcommand{\optresuno}[1]
          {
           \ar@{-}[llld]
           \ar@{-}[rd]
           \ar@{}[d]|{#1}
          }   
\newcommand{\optresdos}[1]
          {
           \ar@{-}[llld]
           \ar@{-}[rrd]
           \ar@{}[d]|{#1}
          }    

\newcommand{\optrestres}[1]
          {
           \ar@{-}[llld]
           \ar@{-}[rrrd]
           \ar@{}[d]|{#1}
          }            
\newcommand{\opcincocinco}[1]
          {
           \ar@{-}[llllld]
           \ar@{-}[rrrrrd]
           \ar@{}[d]|{#1}
          }

\newcommand{\cl}[1]
          {
           \ar@{-}[ur]
           \ar@{}[u]|{#1}
           \ar@{-}[ul]
          }
\newcommand{\clb}[1]
          {
           \ar@<-1ex>@{-}[ur]
           \ar@{}[u]|{#1}
           \ar@<1ex>@{-}[ul]
          }
\newcommand{\clmediob}[1]
          {
           \ar@<-.5ex>@{-}[ur]
           \ar@{}[u]|{#1}
           \ar@<.5ex>@{-}[ul]
          } 
\newcommand{\clrightb}[1]
          {
           \ar@<-1ex>@{-}[ur]
           \ar@{}[u]|{#1}
           \ar@{-}[ul]
          }
\newcommand{\clunodos}[1]
          {
           \ar@{-}[urr]
           \ar@{}[u]|{#1}
           \ar@{-}[ul]
          }
          
\newcommand{\cldosuno}[1]
          {
           \ar@{-}[ur]
           \ar@{}[u]|{#1}
           \ar@{-}[ull]
          }
\newcommand{\cldosdos}[1]
          {
           \ar@{-}[urr]
           \ar@{}[u]|{#1}
           \ar@{-}[ull]
          }         
\newcommand{\cltresdos}[1]
          {
           \ar@{-}[urr]
           \ar@{}[u]|{#1}
           \ar@{-}[ulll]
          }           
\newcommand{\cldostres}[1]
          {
           \ar@{-}[urrr]
           \ar@{}[u]|{#1}
           \ar@{-}[ull]
          }     
\newcommand{\cltrestres}[1]
          {
           \ar@{-}[urrr]
           \ar@{}[u]|{#1}
           \ar@{-}[ulll]
          }            
\newcommand{\clcincocinco}[1]
          {
           \ar@{-}[urrrrr]
           \ar@{}[u]|{#1}
           \ar@{-}[ulllll]
          }            

\newcommand{\Pro}[1]{2\hbox{-}\cc{P}ro(\cc{#1})}

\newcommand{\Prop}[1]{2\hbox{-}\cc{P}ro_p(\cc{#1})}
\newcommand{\CJ}{\cc{H}om_p(\ff{J}^{op},\cc{C})}
\newcommand{\pCJ}{p\cc{H}om_p(\ff{J}^{op},\cc{C})}

\newcommand{\C}{\ff{C}}
\newcommand{\fv}{\ff{v}}
\newcommand{\fu}{\ff{u}}
\newcommand{\f}{\ff{f}}
\newcommand{\X}{\ff{X}}
\newcommand{\Y}{\ff{Y}}
\newcommand{\Z}{\ff{Z}}
\newcommand{\D}{\ff{D}}
\newcommand{\F}{\ff{F}}
\newcommand{\G}{\ff{G}}
\newcommand{\A}{\ff{A}}
\newcommand{\B}{\ff{B}}
\newcommand{\E}{\ff{E}}
\newcommand{\fP}{\ff{P}}
\newcommand{\Q}{\ff{Q}}
\newcommand{\fa}{\ff{a}}
\newcommand{\ii}{\ff{i}}
\newcommand{\fb}{\ff{b}}
\newcommand{\fc}{\ff{c}}
\newcommand{\q}{\ff{q}}
\newcommand{\p}{\ff{p}}
\newcommand{\fr}{\ff{r}}
\newcommand{\s}{\ff{s}}
\newcommand{\ft}{\ff{t}}
\newcommand{\g}{\ff{g}}
\newcommand{\h}{\ff{h}}
\newcommand{\fk}{\ff{k}}
\newcommand{\m}{\ff{m}}
\newcommand{\e}{\ff{e}}

\newcommand{\Yal}{\ff{Y}_a^l}
\newcommand{\Ybl}{\ff{Y}_b^l}
\newcommand{\Ycl}{\ff{Y}_c^l}

\newcommand{\pkj}{(k \! < \! j)}
\newcommand{\pkl}{(k \! < \! l)}
\newcommand{\plj}{(l \! < \! j)}
\newcommand{\pkjj}{(k \! < \! j')}
\newcommand{\pjjj}{(j \! < \! j')}

\newcommand{\oj}{_{0<j}}
\newcommand{\ok}{_{0<k}}

\newcommand{\lj}{_{l<j}}
\newcommand{\kj}{_{k<j}}
\newcommand{\kl}{_{k<l}}

\newcommand{\kjj}{_{k<j'}}

\newcommand{\Xuc}{\X_{u_0}}
\newcommand{\Xuu}{\X_{u_1}}
\newcommand{\Xud}{\X_{u_2}}

\newcommand{\Xvc}{\X_{v_0}}
\newcommand{\Xvu}{\X_{v_1}}
\newcommand{\Xvd}{\X_{v_2}}
\newcommand{\Xvt}{\X_{v_3}}

\newcommand{\Xwc}{\X_{w_0}}
\newcommand{\Xwu}{\X_{w_1}}
\newcommand{\Xwd}{\X_{w_2}}
\newcommand{\Xwt}{\X_{w_3}}

\newcommand{\piuc}{\pi_{u_0}}
\newcommand{\piuu}{\pi_{u_1}}

\newcommand{\pivc}{\pi_{v_0}}
\newcommand{\pivu}{\pi_{v_1}}

\newcommand{\piwc}{\pi_{w_0}}

\newcommand{\piic}{\pi_{i_0}}
\newcommand{\piiu}{\pi_{i_1}}
\newcommand{\piid}{\pi_{i_2}}

\newcommand{\kk}{_{k+1}}
\newcommand{\kko}{_{\tilde{k}_0\leq k_0}}
\newcommand{\kkj}{_{\tilde{k}_j\leq \tilde{\tilde{k}}_j}}
\newcommand{\vokj}{_{\varphi(0) \leq \tilde{\tilde{k}}_j}}
\newcommand{\kvj}{_{\tilde{k}_j\leq \varphi(j)}}

\newcommand{\skto}{_{\tilde{k}_0}}
\newcommand{\sko}{_{k_0}}
\newcommand{\svo}{_{\varphi(0)}}
\newcommand{\svj}{_{\varphi(j)}}
\newcommand{\svjj}{_{\varphi(j')}}
\newcommand{\sktj}{_{\tilde{k}_j}}
\newcommand{\skttj}{_{\tilde{\tilde{k}}_j}}

\newcommand{\tjj}{^{j''}}
\newcommand{\tj}{^{j'}}
\newcommand{\sjj}{_{j''}}
\newcommand{\sj}{_{j'}}
\newcommand{\sk}{_{k'}}

\newcommand{\tjjj}{^{j,j'}}

\newcommand{\tkj}{^{k,j}}
\newcommand{\tkl}{^{k,l}}
\newcommand{\tkjj}{^{k,j'}}
\newcommand{\tB}{^{\B}}
\newcommand{\tA}{^{\A}}
\newcommand{\tE}{^{\E}}
\newcommand{\tY}{^{\Y}}
\newcommand{\tZ}{^{\Z}}
\newcommand{\tX}{^{\X}}

\newcommand{\sY}{_{\Y}}

\newcommand{\sii}{_{i''}}
\newcommand{\si}{_{i'}}
\newcommand{\sit}{_{\tilde{i}}}
\newcommand{\swt}{_{\tilde{w}}}
\newcommand{\smut}{_{\tilde{\mu}}}

\newcommand{\sXY}{_{\X,\Y}}
\newcommand{\sEY}{_{\E,\Y}}

\newcommand{\sB}{_{\B}}

\newcommand{\ta}{^{a'}}

\newcommand{\tba}{^{ba}}

\newcommand{\tca}{^{ca}}
\newcommand{\tcb}{^{cb}}

\newcommand{\sba}{_{ba}}

\newcommand{\sd}[2]{_{({#1},{#2})}}
\newcommand{\st}[3]{_{({#1},{#2},{#3})}}

\newcommand{\ardr}{\ar@{-}[dr]}
\newcommand{\ardrr}{\ar@{-}[drr]}
\newcommand{\ardrrr}{\ar@{-}[drrr]}
\newcommand{\ardrrrr}{\ar@{-}[drrrr]}
\newcommand{\ardl}{\ar@{-}[dl]}
\newcommand{\ardll}{\ar@{-}[dll]}
\newcommand{\ardlll}{\ar@{-}[dlll]}
\newcommand{\ardllll}{\ar@{-}[dllll]}
\newcommand{\ardlllll}{\ar@{-}[dlllll]}

\newcommand{\inv}{^{-1}}

\begin{document}
\thispagestyle{empty}

\begin {center}
\includegraphics[scale=.3]{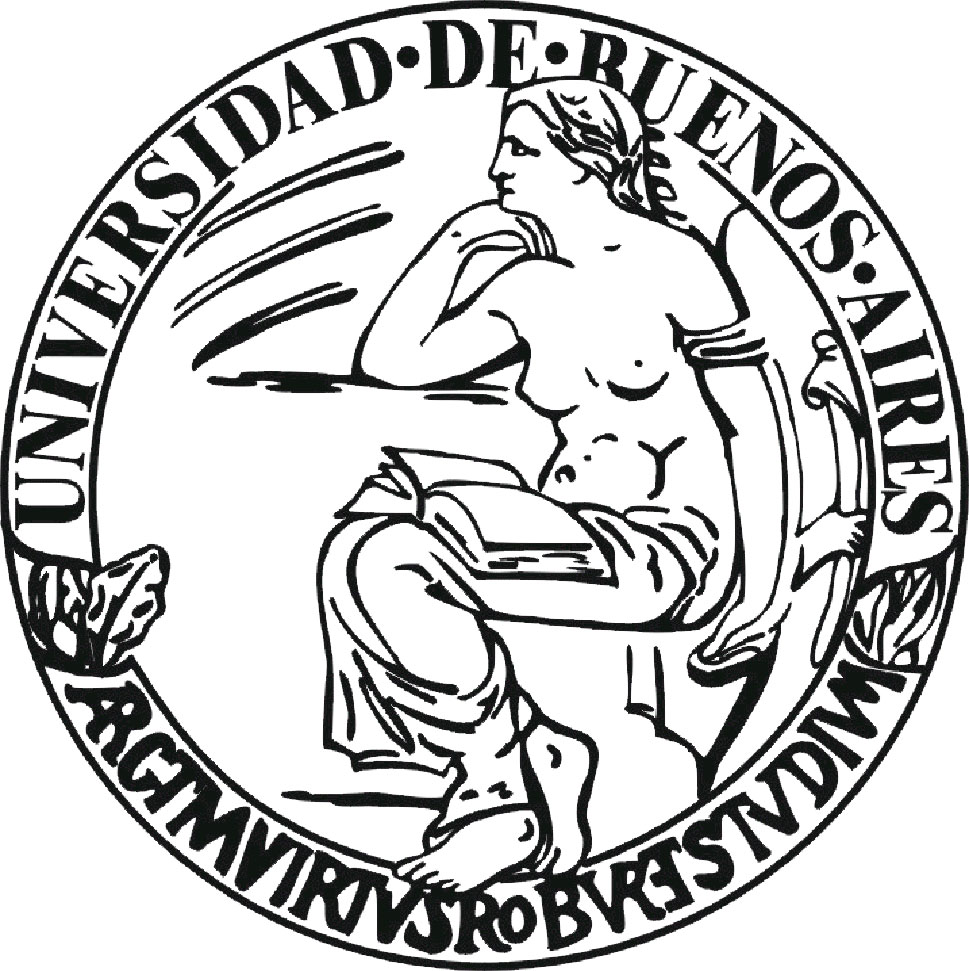}

\medskip
UNIVERSIDAD DE BUENOS AIRES

Facultad de Ciencias Exactas y Naturales

Departamento de Matem\'atica

\vspace{3cm}

\textbf{\large  Una teor\'ia de 2-pro-objetos, una teor\'ia de 2-categor\'ias de 2-modelos y la estructura de 2-modelos para $\Pro{C}$}

\vspace{2cm}

Tesis presentada para optar al t\'itulo de Doctor de la Universidad de Buenos Aires en el \'area Ciencias Matem\'aticas

\vspace{2cm}

\textbf{Mar\'ia Emilia Descotte}

\end {center}

\vspace{1.5cm}

\noindent Director de tesis: Eduardo J. Dubuc
 
\noindent Consejero de estudios: Eduardo J. Dubuc

\vspace{1cm}

\noindent Lugar de trabajo: Instituto de investigaciones matem\'aticas ``Luis A. Santal\'o''

\noindent Lugar y fecha de defensa: Buenos Aires, 7 de julio de 2015

\pagebreak

\begin{center}\large
\bf Una teor\'ia de 2-pro-objetos, una teor\'ia de 2-categor\'ias de 2-modelos y la estructura de 2-modelos para $\Pro{C}$
\end{center}

\noindent \small {\bf Resumen.} En los 60, Grothendieck desarrolla la teor\'ia de \mbox{pro-objetos} de una categor\'ia.  La propiedad fundamental de $\ff{Pro}(\ff{C})$ es que se tiene un embedding $\ff{C} \mr{c} \ff{Pro}(\ff{C})$, $\ff{Pro}(\ff{C})$ tiene l\'imites cofiltrantes peque\~nos, y estos son libres en el sentido de que para cualquier otra categor\'ia $\ff{E}$ con l\'imites cofiltrantes peque\~nos, la precomposici\'on con $c$ determina una equivalencia de categor\'ias
$\cc{C}at(\ff{Pro}(\ff{C}),\,\ff{E})_+ \simeq \cc{C}at(\ff{C},\, \ff{E})$, (el
``$+$'' indica la subcategor\'ia plena formada por los funtores que preservan l\'imites cofiltrantes). 

En este trabajo, desarrollamos la teor\'ia de pro-objetos ``2-dimensional''. Dada una 2-categor\'ia 
$\cc{C}$, definimos la 2-categor\'ia $\Pro{C}$ cuyos objetos llamamos 2-pro-objetos. Probamos que $\Pro{C}$ tiene todas las propiedades b\'asicas esperadas relativizadas adecuadamente al caso 2-categ\'orico, incluyendo la propiedad universal correspondiente. Damos una definici\'on de ``closed 2-model 2-category'' adecuada y demostraciones de sus propiedades b\'asicas. Dejamos para un trabajo futuro la construcci\'on de su categor\'ia homot\'opica. Finalmente, probamos que nuestra 2-categor\'ia $\Pro{C}$ tiene una estructura de ``closed 2-model 2-category'' si $\cc{C}$ la tiene.

Parte de la motivaci\'on de este trabajo fue desarrollar un contexto te\'orico para manipular el nervio de $\check{C}$ech en teor\'ia de homotop\'ia, \cite{AM}, en particular en teor\'ia de la forma fuerte, \cite{M1}. El nervio de $\check{C}$ech est\'a indexado por las categor\'ias de cubrimientos e hipercubrimientos con morfismos dados por los refinamientos, que no son categor\'ias filtrantes pero s\'i determinan 2-categor\'ias 2-filtrantes en las cuales el nervio de $\check{C}$ech tambi\'en est\'a definido, manda las 2-celdas en homotop\'ias, y determina un 2-pro-objeto sobre los conjuntos simpliciales. Usualmente, el nervio de $\check{C}$ech debe ser considerado como un 2-pro-objeto en la categor\'ia homot\'opica, perdiendo la informaci\'on codificada en las homotop\'ias expl\'icitas.



\vspace{20ex}

\noindent {\bf Palabras claves.} 2-pro-objeto, 2-filtrante, pseudo-l\'imite, bi-l\'imite, 2-cofinal, 2-categor\'ia de \mbox{2-modelos}.

\pagebreak

\begin{center}\large
\bf A theory of 2-pro-objects, a theory of 2-model 2-categories and the 2-model structure for $\Pro{C}$
\end{center}

\noindent \small {\bf Abstract.} In the sixties, Grothendieck developed the theory of pro-objects over a \mbox{category}.  The fundamental property of the category $\ff{Pro}(\ff{C})$ is that there is an \mbox{embedding} $\ff{C} \mr{c} \ff{Pro}(\ff{C})$, $\ff{Pro}(\ff{C})$ is closed under small cofiltered limits, and these are free in the sense that for any \mbox{category} $\ff{E}$ closed under small cofiltered \mbox{limits}, pre-composition with $c$ determines an equivalence of categories
$\cc{C}at(\ff{Pro}(\ff{C}),\,\ff{E})_+ \simeq \cc{C}at(\ff{C},\, \ff{E})$, (the
``$+$'' indicates the full subcategory of the functors that preserve cofiltered limits).  

In this work we develop a ``2-dimensional'' pro-object theory. Given a 2-category
$\cc{C}$, we define the \mbox{2-category} $\Pro{C}$ whose objects we call 2-pro-objects. We prove that $\Pro{C}$ has all the expected basic properties adequately relativized to the \mbox{2-categorical} \mbox{setting}, including the corresponding universal property. We give an adecuate definition of closed 2-model 2-category and demonstrations of its basic properties. We leave for a future work the construction of its homotpy 2-category. Finally, we
prove that our \mbox{2-category} $\Pro{C}$ has a closed 2-model 2-category structure provided that $\cc{C}$ has one.


Part of the motivation of this work was to develop a conceptual framework to handle the $\check{C}$ech nerve in homotopy theory, \cite{AM}, in particular in strong shape theory, \cite{M1}. The $\check{C}$ech nerve is indexed by the categories of covers and of hypercovers, with cover refinments as morphisms, which are not filtered categories, but determine 2-filtered \mbox{2-categories} on which the $\check{C}$ech nerve is also defined, sends 2-cells into homotopies, and determines a 2-pro-object of simplicial sets. Usually, the $\check{C}$ech nerve has to be considered as a pro-object in the homotopy category, loosing the information encoded in the explicit homotopies.



\vspace{20ex}


\pagebreak

\section*{Agradecimientos}

En primer indiscutible lugar a Martu. Necesitar\'ia otras 226 p\'aginas para enumerar todas las maneras en las que me ayudaste pero se puede resumir en que sos el mejor compa\~nero del mundo. !`Gracias por estar ah\'i tan incondicional y tan dulce siempre! !`Te amo!

A nuestra bebota hermosa Sofi. Simplemente por haber llegado a hacernos la vida m\'as feliz cada d\'ia con sus sonrisas y su dulzura. Y por portarse tan bien ayudando a que pueda terminar a tiempo. !`Sos todo, hermosa!

A Eduardo por ser un director de primera tanto en lo acad\'emico como en lo humano, por seguir ense\~n\'andome una forma de ver la matem\'atica incre\'iblemente genial. Te admiro much\'isimo y no podr\'ia estar m\'as contenta de haberte elegido todos estos a\~nos. !`Que el adi\'os sea s\'olo en los papeles! ;)

Al CONICET, por el apoyo econ\'omico que me permiti\'o dedicarme a esta tesis. 

A los jurados por aceptar esta tarea y por dedicar su tiempo y esfuerzo a evaluar mi tesis, por las correcciones y observaciones tan valiosas y por qu\'e no, tambi\'en por los elogios que tan bien vienen despu\'es de tanto esfuerzo.

A Santi por ser la persona correcta en el lugar correcto en el momento correcto. Adem\'as de aprender much\'isimo de vos como docente, como cient\'ifico y como persona, gracias a tu confianza y tu ayuda me anim\'e y logr\'e muchas cosas que no hubiese hecho sin vos. Es un placer trabajar con un amigo as\'i que !`espero la postdoc ansiosa! :)

A Deby que, como ella dice, nos acompa\~na a un costadito en este proceso tan importante para nosotros y, agrego yo, lo hace siempre con la mejor onda y con una eficiencia incre\'ible. Gracias a ella odiamos menos la burocracia. Una genia y dulce total.

A mi mam\'a que estuvo siempre que necesit\'e algo desde que tengo memoria y en particular en estos \'ultimos meses ayud\'andonos con todo para que pueda terminar la tesis. A mi hermanita querida que tambi\'en me ayud\'o mucho. A Edi, Gra, Lauri y Mati que tambi\'en estuvieron pendientes de nosotros para aliviarnos en la crisis del final de tesis. !`Muchas gracias a los seis por toda la invalorable ayuda y por cuidar a Sofi con tanto amor!

A mis amigos y al resto de la familia por estar y por bancarse que yo no lo haya estado tanto estos \'ultimos meses. Como me dijo Manu el primer d\'ia que me vio post entrega: ``!`Volvi\'o Emi!'', !`Volv\'i! !`Gracias por ser parte de mi vida y dejarme ser parte de las suyas! !`Los quiero!

\begin{flushright}
 Emi
\end{flushright}

\pagebreak
\begin{flushright}
 A Martu y Sofi que me hacen tan feliz.
\end{flushright}

\pagebreak

\section*{Introducci\'on}

La teor\'ia de pro-objetos comenz\'o en Francia en los a\~nos 60 en el \emph{Seminaire de \mbox{Geometrie} Algebrique du Bois-Marie} llevado a cabo por Alexander Grothendieck y otros matem\'aticos. Este seminario fue un fen\'omeno \'unico de investigaci\'on y tuvo lugar entre los a\~nos 1960 y 1969 en el IHÉS cerca de Par\'is. En \cite{G2} se re\'unen parte de las notas de estos seminarios. La categor\'ia $\ff{Pro}(\ff{C})$ de pro-objetos de una categor\'ia $\ff{C}$ se define all\'i. Aqu\'i tambi\'en se demuestran sus propiedades b\'asicas y se da una caracterizaci\'on de la misma por propiedad universal:

\begin{quote}
\sl El funtor can\'onico $\ff{C}\mr{c}\ff{Pro}(\ff{C})$ es 2-universal respecto de los funtores de $\ff{C}$ en una categor\'ia con l\'imites cofiltrantes, m\'as expl\'icitamente:
 Dada $\ff{E}$ una categor\'ia con l\'imites cofiltrantes
 
 $$\ff{Hom}(\ff{Pro}(\ff{C}),\ff{E})_+\mr{c^*}\ff{Hom}(\ff{C},\ff{E})$$ 
 
 \noindent es una equivalencia de categor\'ias (aqu\'i el ``+'' indica la subcategor\'ia plena formada por aquellos funtores que preservan l\'imites cofiltrantes).
\end{quote}

En esa misma \'epoca, Daniel Quillen desarrollaba la teor\'ia de categor\'ias de modelos \cite{Q1}, vastamente utilizada en teor\'ia de homotop\'ia. Las categor\'ias de modelos de Quillen permiten construir la categor\'ia homot\'opica $\ff{Ho}(\ff{C})$ asociada a una categor\'ia $\ff{C}$. Esta categor\'ia se obtiene invirtiendo formalmente la clase de morfismos formada por las equivalencias d\'ebiles de una estructura de modelos de Quillen de $\ff{C}$. Las categor\'ias homot\'opicas as\'i obtenidas tienen la ventaja de tener muchas buenas propiedades que las hacen muy \'utiles en la pr\'actica.

El nervio de $\check{C}ech$ de un cubrimiento es una herramienta de base en ciertos \mbox{desarrollos} de la teor\'ia de homotop\'ia, y en teor\'ia de la forma (\cite{AM}, \cite{MS}). Dado un sitio $\cc{C}$ (por \linebreak ejemplo, el reticulado $\cc{O}(X)$ de los abiertos de un espacio topol\'ogico $X$), los \linebreak cubrimientos (tomando como morfismos los refinamientos) forman una categor\'ia $\ff{COV}(\cc{C})$ que no es cofiltrante, por lo cual el nervio de $\check{C}ech$, que es un funtor $\ff{COV}(\cc{C}) \mr{\check{C}} \ff{SS}$, no determina un pro-objeto en la categor\'ia de los conjuntos simpliciales y no se pueden utilizar las herramientas de la teor\'ia de pro-objetos. Este problema se resuelve pasando a la categor\'ia homot\'opica $\ff{Ho}(\ff{SS})$.  Los cubrimientos ordenados bajo refinamiento s\'i forman una categor\'ia cofiltrante $\ff{cov}(\cc{C})$, y dados dos refinamientos, los morfismos inducidos entre los nervios son homot\'opicos, por lo que se tiene un pro-objeto 
$\ff{cov}(\cc{C}) \mr{\check{C}} \ff{Ho}(\ff{SS})$. Este pasaje m\'odulo homotop\'ia pierde la informaci\'on dada por las homotop\'ias expl\'icitas asociadas a los refinamientos, haciendo que la teor\'ia no sea suficientemente fina en muchas aplicaciones. En la teor\'ia de la forma fuerte, las homotop\'ias expl\'icitas no se descartan pero el contexto conceptual de la teor\'ia de pro-objetos se pierde para el nervio de $\check{C}ech$.

La teor\'ia de 2-categor\'ias se remonta a los a\~nos 70 pero viene teniendo un gran auge en los \'ultimos tiempos. La mayor\'ia de los resultados y construcciones b\'asicos de la teor\'ia de categor\'ias han sido generalizados al contexto 2-categ\'orico (ver por ejemplo \cite{KS} or \cite{L}). Es esencial para nuestro trabajo la definici\'on de 2-categor\'ia 2-filtrante \cite{K}, reformulada en \cite{DS}, as\'i como tambi\'en es fundamental la noci\'on de pseudo-l\'imites y, en particular, la construcci\'on expl\'icita de pseudo-col\'imites 2-filtrantes de categor\'ias dada por Dubuc y Street en ese trabajo. 

La teor\'ia de pro-objetos 2-categ\'orica ha demostrado ser de gran inter\'es en s\'i misma y ha planteado muchos problemas interesantes de la teor\'ia de 2-categor\'ias. Nuestra motivaci\'on original para estudiar estos temas fue la de dar las herramientas necesarias para poder trabajar en teor\'ia de la forma fuerte con el nervio de $\check{C}$ech y no perder informaci\'on pasando m\'odulo homotop\'ia como sucede en teor\'ia de la forma, ni tampoco tener que reemplazar el nervio de $\check{C}$ech por el menos conveniente nervio de Vietoris como se hace actualmente en teor\'ia de la forma fuerte. Esta motivaci\'on provino de observar que si bien, como ya mencionamos, la categor\'ia $\ff{COV}(\cc{C})$ no es cofiltrante, s\'i determina una 2-categor\'ia 2-cofiltrante sobre la cual el nervio de $\check{C}ech$ est\'a definido y determina un 2-funtor mandando las 2-celdas en homotop\'ias. Esto motiv\'o nuestra definici\'on de 2-pro-objeto que har\'a que el nervio de $\check{C}$ech, que no era un pro-objeto 
simplicial, s\'i resulte un 2-pro-objeto simplicial. La teor\'ia de 2-pro-objetos fue de hecho una teor\'ia muy interesante en s\'i misma y requiri\'o mucho trabajo en teor\'ia de 2-categor\'ias, posponiendo las aplicaciones previstas para un trabajo futuro.

\paragraph{Estructuraci\'on del trabajo}

En esta tesis desarrollamos una teor\'ia de pro-objetos 2-dimensional. Tambi\'en damos una noci\'on de 2-funtor 2-cofinal que nos permite probar la versi\'on 2-categ\'orica de los teoremas de reindexaci\'on de pro-objetos. Por \'ultimo damos una noci\'on de ``closed 2-bmodel 2-category'' y demostramos que nuestra 2-categor\'ia $\Pro{C}$ satisface esta definici\'on.

La secci\'on \ref{prelims} est\'a dedicada a fijar notaci\'on y dejar en claro los resultados b\'asicos de la teor\'ia de 2-categor\'ias que usaremos a lo largo de la tesis. La mayor\'ia de estos resultados son conocidos, sin embargo hay algunos (para los cuales damos demostraciones expl\'icitas) que no parecen encontrarse en la literatura. En \ref{weak limits and colimits} probamos que los pseudo-l\'imites (c\'onicos) en las 2-categor\'ias de 2-funtores $\cc{H}om(\cc{C},\cc{D})$, $\cc{H}om_p(\cc{C},\cc{D})$ y $p\cc{H}om_p(\cc{C},\cc{D})$ (definici\'on \ref{ccHom}) y los bi-l\'imites en $p\cc{H}om_p(\cc{C},\cc{D})$ se calculan punto a punto. Este resultado, si bien era esperable, necesita indefectiblemente una demostraci\'on. En \ref{2-cofinal 2-functors} 
definimos la noci\'on de pseudo-funtor 2-cofinal entre 2-categor\'ias y probamos ciertas propiedades que usaremos en la secci\'on \ref{Mtrick} para demostrar las propiedades de 
reindexaci\'on de 2-pro-objetos.  
En \ref{A sombrero} construimos un 2-funtor asociado via un pseudo-funtor 2-cofinal a un pseudo-funtor dado. Este resultado tiene inter\'es independiente y ser\'a usado en la secci\'on \ref{2-modelos}. 
Finalmente, en \ref{further results} consideramos la noci\'on de funtores flexibles dada en \cite{BKP} y enunciamos una caracterizaci\'on de los mismos muy \'util e independiente del adjunto a izquierda de la inclusi\'on $\cc{H}om(\cc{C},\cc{D}) \rightarrow \cc{H}om_p(\cc{C},\cc{D})$ (Proposici\'on
\ref{flexiblechar}). Usando esta caracterizaci\'on, el pseudo lema de Yoneda dice directamente que los 2-funtores representables son flexibles. Se sigue tambi\'en que el 2-funtor asociado a cualquier 2-pro-objeto es flexible, lo cual tiene consecuencias importantes en la teor\'ia de 2-pro-objetos.

En la secci\'on \ref{2-Pro-objects} se encuentran algunos de los resultados claves de este trabajo. En \ref{def2pro}, dada una 2-categor\'ia $\cc{C}$ definimos la 2-categor\'ia $2$-$\cc{P}ro(\cc{C})$ cuyos objetos llamamos 2-pro-objetos. Un 2-pro-objeto de $\cc{C}$ es un 2-funtor a valores en $\cc{C}$ (o diagrama en $\cc{C}$)
indexado por una 2-categor\'ia 2-cofiltrante. Nuestra teor\'ia va m\'as all\'a de la teor\'ia de categor\'ias enriquecidas
porque en la definici\'on de morfismos, en lugar de usar 2-l\'imites estrictos, usamos la noci\'on no estricta de pseudo-l\'imites, que es usualmente la de inter\'es pr\'actico. Tambi\'en en \ref{def2pro}, establecemos la f\'ormula b\'asica que describe los morfismos y las 2-celdas entre \mbox{2-pro-objetos} en t\'erminos de pseudo-l\'imites y pseudo-col\'imites de las categor\'ias de morfismos de $\cc{C}$.
Inspirados en la definici\'on hallada en \cite{AM} de que un morfismo en la categor\'ia \mbox{original} represente a un morfismo de pro-objetos, introducimos en \ref{lemas2pro} la noci\'on de que un morfismo y una 2-celda en $\cc{C}$ representen un morfismo y una \mbox{2-celda} en $\Pro{C}$ respectivamente. Tambi\'en demostramos propiedades t\'ecnicas de los 2-pro-objetos que permiten hacer c\'alculos con ellos y, en particular, son necesarias en la demostraci\'on del teorema que establece que la 2-categor\'ia \mbox{$2$-$\cc{P}ro(\cc{C})$} tiene pseudo-l\'imites 2-cofiltrantes. En \ref{pseudo-limitsen2pro}, construimos una 2-categor\'ia \mbox{2-filtrante} que sirve como 2-categor\'ia de \'indices para el pseudo-l\'imite 2-cofiltrante de 2-pro-objetos (Definici\'on \ref{kequis} y proposici\'on \ref{teo}). Esto tambi\'en fue inspirado por una construcci\'on con el mismo prop\'osito hallada en \cite{AM} para el caso 1-dimensional, pero que en el caso 2-dimensional resulta ser mucho m\'as compleja. Nos vimos 
forzados a recurrir a esta complicada construcci\'on debido a que el tratamiento conceptual hecho en \cite{G2} no puede ser aplicado al caso 2-dimensional. Esto se debe a que un 2-funtor a valores en la 2-categor\'ia de categor\'ias $\cc{C}at$ no es el pseudo-col\'imite (c\'onico) de 2-funtores 2-representables indexado por su 2-diagrama, como s\'i pasa en el caso \mbox{1-dimensional}. Finalmente, en \ref{pu2pro}, enunciamos y demostramos la propiedad universal de $2$-$\cc{P}ro(\cc{C})$ (Teorema \ref{pseudouniversal}), de una manera in\'edita incluso si se aplica al caso cl\'asico de la teor\'ia de pro-objetos.

Tambi\'en consideramos en esta secci\'on la 2-categor\'ia $\Prop{C}$ que es ``retract pseudo-equivalent'' a $\Pro{C}$, \ref{proppseudoeqapro}, hecho que se sigue de que los 2-funtores a valores en $\cc{C}at$ asociados a 2-pro-objetos son flexibles. Esta 2-categor\'ia ser\'a esencial en la secci\'on \ref{2-modelos} y probar\'a ser interesante en s\'i misma.

La mayor parte de los resultados de las secciones \ref{prelims} y \ref{2-Pro-objects} fueron publicados en \cite{DD2}.

En la secci\'on \ref{Mtrick} probamos los teoremas de reindexaci\'on de pro-objetos para el caso \mbox{2-categ\'orico}. Esta secci\'on est\'a inspirada en los resultados an\'alogos en el caso \mbox{1-dimensional} dados en \cite{AM} pero, como pasaba con los resultados de la secci\'on \ref{2-Pro-objects}, su versi\'on 2-categ\'orica supone un desaf\'io mayor. El primer resultado es una versi\'on \mbox{2-categ\'orica} de un resultado debido a Deligne \cite[Expose I, 8.1.6]{G2} que es clave en el caso 1-dimensional en el desarrollo de la estructura de modelos de la categor\'ia 
$\mathsf{Pro(C)}$ \cite{EH}. 
El enunciado 1-dimensional establece que todo pro-objeto es isomorfo a uno indexado por un poset cofinito y filtrante. Nuestra versi\'on establece que todo 2-pro-objeto es equivalente a uno indexado por un poset cofinito y filtrante. El segundo resultado establece que todo morfismo de 2-pro-objetos puede ser levantado salvo equivalencia a un morfismo entre 2-pro-objetos indexados por un poset cofinito y filtrante. Esto es un caso particular de un tercer resultado que establece que todo diagrama finito en $2$-$\cc{P}ro(\cc{C})$ puede ser levantado salvo equivalencia a un diagrama finito de 2-pro-objetos indexados por un poset cofinito y filtrante. Es clave para estos resultados la noci\'on de pseudo-funtor 2-cofinal dada en la secci\'on \ref{prelims}. Toda esta secci\'on ser\'a usada para probar el teorema central de la secci\'on \ref{2-modelos}.

En la secci\'on \ref{Definitions and basic lemmas} introducimos las nociones in\'editas de ``closed 2-model 2-category'' y ``closed 2-bmodel 2-category'' y enunciamos y demostramos algunos lemas y proposiciones que usaremos m\'as adelante. Nuestra noci\'on es m\'as fuerte que las ``fibration structures'' de Pronk (\cite{Pronk}) pues es una versi\'on 2-dimensional de los axiomas de Quillen completos para ``closed model categories''. Tambi\'en difiere en el hecho importante de que no asumimos la elecci\'on de una factorizaci\'on global privilegiada dada de forma pseudo-funtorial sino que estipulamos, como Quillen, solo la existencia de factorizaciones para cada flecha. La mayor\'ia de los resultados de esta secci\'on son generalizaciones al contexto de 2-categor\'ias de enunciados bien conocidos de la teor\'ia de ``closed model categories''.  

Para terminar, en la secci\'on \ref{2-modelos} 
probamos uno de los teoremas centrales de esta tesis (\ref{Proesde2-bmodelos}) que establece que si $\cc{C}$ es una ``closed 2-bmodel 2-category'', entonces $\Pro{C}$ tambi\'en lo es. Para lograrlo, fue necesario demostrar primero los teoremas \ref{CJdeclosed 2-modelos} and \ref{Propde2modelos} que establecen respectivamente que la 2-categor\'ia $p\cc{H}om_p(\ff{J}^{op},\cc{C})$ (definici\'on \ref{ccHom}) y la 2-categor\'ia $\Prop{C}$ son ``closed 2-bmodel 2-categories'' si $\ff{J}$ es un poset cofinito y filtrante y $\cc{C}$ es de una ``closed 2-bmodel 2-category''. Las propiedades de reindexaci\'on probadas en la secci\'on \ref{Mtrick} son claves para obtener \ref{Propde2modelos} a partir de \ref{CJdeclosed 2-modelos}.

\vspace{2ex}

\paragraph{Notaci\'on}

Adem\'as del usual ``pegado'' de diagramas, usaremos el \emph{C\'alculo de ascensores} para expresiones que denotan 2-celdas (comparar con la notaci\'on usada en \mbox{\cite[3.10,  3.17]{F}}). 
Esta es una notaci\'on muy gr\'afica inventada por Eduardo Dubuc en 1969  
para escribir ecuaciones con transformaciones naturales entre funtores. En este trabajo usamos los ascensores para escribir ecuaciones con 2-celdas en 2-categor\'ias. Los objetos se omiten, las 2-celdas se escriben con celdas, y las 2-celdas identidades como una doble l\'inea. Es importante remarcar que cuando una 2-celda entre flechas distintas es la identidad, de todas formas se escribe como una 2-celda etiquetada por ``$=$''. Por ejemplo, la 2-celda estructural de un 2-funtor visto como caso particular de un pseudo-funtor. Las composiciones se leen de arriba para abajo y de derecha a izquierda. La ecuaci\'on \ref{basicelevator} es la igualdad b\'asica para el c\'alculo de ascensores:
$$ \label{ascensor}
\xymatrix@C=0ex
         {
            \,\ff{f}'\, \dcell{\alpha'} & \,\ff{f}\, \did
          \\
            \,\ff{g}'\, \did & \,\ff{f}\, \dcell{\alpha}
          \\
             \,\ff{g}'\,  &  \,\ff{g}\, 
         }
\xymatrix@R=6ex{\\ \;\;\;=\;\;\; \\}
\xymatrix@C=0ex
         {
             \,\ff{f}'\, \did & \,\ff{f}\, \dcell{\alpha}
          \\
             \,\ff{f}'\,\dcell{\alpha'} & \,\ff{g}\, \did
          \\
             \,\ff{g}'\, & \,\ff{g}\, 
         }
\xymatrix@R=6ex{ \\ \;\;\;=\;\;\; \\}
\xymatrix@C=0ex@R=0.9ex
         {
             {} & {}
          \\
                \,\ff{f}'\, \ar@<4pt>@{-}'+<0pt,-6pt>[ddd] 
                   \ar@<-4pt>@{-}'+<0pt,-6pt>[ddd]^{\alpha'}
             &  \,\ff{f}\,  \ar@<4pt>@{-}'+<0pt,-6pt>[ddd] 
                   \ar@<-4pt>@{-}'+<0pt,-6pt>[ddd]^{\alpha}
          \\ 
             {} & {}
          \\ 
             {} & {}
          \\
             \,\ff{g}'\, & \,\ff{g}\,.
         }
$$
Esto permite mover celdas de arriba hacia abajo y viceversa cuando no hay obst\'aculos, como si fueran ascensores. 
Con esto movemos celdas para formar configuraciones que den nuevas ecuaciones a partir de ecuaciones v\'alidas.

\pagebreak

\section*{Introduction}

Pro-object theory started in the sixties in France with the \emph{Seminaire de \mbox{Geometrie} Algebrique du Bois-Marie} conducted by Alexander Grothendieck and other mathematicians. This seminaire was a unique research phenomenon and took place between years 1960 and 1969 in the IHÉS near Paris. \cite{G2} consists on some of the notes of this seminaires. The category $\ff{Pro}(\ff{C})$ of pro-objects of a category $\ff{C}$ is defined there. The authors also prove the basic properties of this category and give a characterization by universal property:

\begin{quote}
\sl The canonical functor $\ff{C}\mr{c}\ff{Pro}(\ff{C})$ is 2-universal over the functors from $\ff{C}$ into a category closed under cofiltered limits, more explicitly:
 Given a category $\ff{E}$ closed under cofiltered limits
 
 $$\ff{Hom}(\ff{Pro}(\ff{C}),\ff{E})_+\mr{c^*}\ff{Hom}(\ff{C},\ff{E})$$ 
 
 \noindent is an equivalence of categories (here the ``+'' indicates the full subcategory of those functors that preserve cofiltered limits).
\end{quote}

By the same time, Daniel Quillen developed model category theory \cite{Q1} which was widely applied in homotopy theory. Quillen's model categories are useful to construct the homotopy category $\ff{Ho}(\ff{C})$ associated to a category $\ff{C}$. This category is obtained by formally turning the class of weak equivalences of the model structure into isomorphisms. Homotopy categories associated to a model category have many good properties that make them very useful in practice.

The $\check{C}ech$ nerve associated to a covering is a fundamental tool in some developments in homotopy theory and shape theory (\cite{AM}, \cite{MS}). Given a site $\cc{C}$ (for example, the lattice $\cc{O}(X)$ formed by the opened sets of a topological space $X$), coverings (taking refinements as morphisms) form a category $\ff{COV}(\cc{C})$ that fails to be cofiltered and so the $\check{C}ech$ nerve, that is a functor $\ff{COV}(\cc{C}) \mr{\check{C}} \ff{SS}$, does not determine a pro-object over simplicial sets, and pro-object theory can't be applied to this setting. This problem is solved by working in the homotopy category $\ff{Ho}(\ff{SS})$. Coverings under refinement does form a cofiltered category (poset) $\ff{cov}(\cc{C})$, and given two refinements, the induced morphisms between the nerves are homotopic, so there is a pro-object
$\ff{cov}(\cc{C}) \mr{\check{C}} \ff{Ho}(\ff{SS})$. Working in the homotopy category has the disadvantage that information given by the explicit homotopies associated to the refinements gets lost, making the theory not enough refined for many applications. In strong shape theory, the explicit homotopies are not discarded, but the conceptual framework of the theory of pro-objects is lost for the 
$\check{C}ech$ nerve. 

2-category theory goes back to the seventies but it's been having a heyday lately. Most of the results and basic constructions of category theory had been generalized to the 2-categorical context (see for example \cite{KS}, \cite{L}). It is essential to our work the definition of 2-filtered 2-category \cite{K}, reformulated in \cite{DS}. It is also key to our work the notion of pseudo-limit and, in particular, the explicit construction of 2-filtered pseudo-colimits of categories given by Dubuc and Street in that paper. 

2-categorical pro-object theory had proved to be very interesting itself and had raised many interesting problems in 2-category theory. Our original motivation to begin with this work was to give the needed tools to be able to work with the $\check{C}$ech nerve in strong shape theory so no information is lost by working modulo homotopy as it happens in shape theory.
This motivation came from observing that although, as we mentioned before, the category $\ff{COV}(\cc{C})$ is not cofiltered, it determines a 2-cofiltered 2-category over which the $\check{C}$ech nerve is defined and determines a 2-functor sending 2-cells into homotopies. This encouraged our definition of 2-pro-object that would make the $\check{C}$ech nerve, that was not a simplicial pro-object, a simplicial 2-pro-object. The tl llamado \emph{Marde$\check{s}$i\'c trick} debido a heory of 2-pro-objects was in fact a very interesting theory itself and it required much work in 2-category theory, postponing its intended applications to future work.

\paragraph{Work structure}

In this thesis, we develop a 2-dimensional pro-object theory. We also give a 2-cofinal pseudo-functor notion that allows as to prove the 2-categorical version of pro-objects reindexing properties. Finally, we give a notion of closed 2-bmodel 2-category and we prove that our 2-category $\Pro{C}$ has a closed 2-bmodel structure.

Section \ref{prelims} is intended to fix notation and set down some basic results from 2-category theory that we will use all along this thesis. Most of this results are well known, although there are some of them (for which we give explicit proofs) that seem not to be in the literature. In \ref{weak limits and colimits} we prove that (conical) pseudo-limits in the \mbox{2-categories} $\cc{H}om(\cc{C},\cc{D})$, $\cc{H}om_p(\cc{C},\cc{D})$ and $p\cc{H}om_p(\cc{C},\cc{D})$ (definition \ref{ccHom}) and bi-limits in $p\cc{H}om_p(\cc{C},\cc{D})$ are computed pointwise. These result, though expected, necessarily requires demonstration. 
In \ref{2-cofinal 2-functors} we define the notion of 2-cofinal pseudo-functor between 2-categories and prove some properties that we will use in section \ref{Mtrick} to prove 2-pro-objects reindexing properties. In \ref{A sombrero} we construct a \mbox{2-functor} associated to a given pseudo-functor via a 2-cofinal pseudo-functor. This result has independent interest and we will use it in section \ref{2-modelos}.
Finally, in \ref{further results} we consider the notion of flexible functor given in \cite{BKP} and we state a characterization of them that is very useful and independent of the left adjoint of the inclusion $\cc{H}om(\cc{C},\cc{D}) \rightarrow \cc{H}om_p(\cc{C},\cc{D})$ (Proposition \ref{flexiblechar}). Using this characterization, the pseudo Yoneda lemma says that representable 2-functors are flexible. It also follows that the \mbox{2-functor} associated to any 2-pro-object is flexible, fact which has important consequences in the theory of 2-pro-objects.

In section \ref{2-Pro-objects} are some of the most important results of this thesis. In \ref{def2pro}, given a \mbox{2-category} $\cc{C}$ we define the 2-category $2$-$\cc{P}ro(\cc{C})$ whose objects we call 2-pro-objects. A \mbox{2-pro-object} over $\cc{C}$ is a 2-functor landing on $\cc{C}$ (or a diagram in $\cc{C}$)
indexed by a \mbox{2-cofiltered} 2-category. Our theory goes beyond enriched category theory because in the definition of morphisms, instead of using strict 2-limits, we use the non-strict notion of pseudo-limits, which is usually the one of practical interest. In \ref{def2pro}, we establish the basic formula describing morphisms and 2-cells between 2-pro-objects in terms of a pseudo-limit of pseudo-colimits of categories.
Inspired on the definition found in \cite{AM} of a morphism of the original category representing a pro-objects morphism, in \ref{lemas2pro} we introduce the notion of a morphism or a 2-cell in $\cc{C}$ being a representative of a morphism or a 2-cell in $2$-$\cc{P}ro(\cc{C})$ respectively. We also prove some technical properties of 2-pro-objects that allow us to make calculations with them, and, in particular, are needed in the proof of the theorem that states the 2-category \mbox{$2$-$\cc{P}ro(\cc{C})$} is closed under 2-cofiltered pseudo-limits. In \ref{pseudo-limitsen2pro}, we construct a 2-filtered 2-category that will be the index 2-category of the 2-cofiltered pseudo-limit of 2-pro-objects (Definition \ref{kequis} and Theorem \ref{teo}). This was also inspired by a construction with the same purpouse in the 1-dimensional case found in \cite{AM}, but the 2-categorical case turned out to be significantly  more complicated. We were forced to make this complicated construction because the 
conceptual treatment made in \cite{G2} can't be applied to the 2-categorical setting. This is due to the fact that a 2-functor landing in the 2-category of categories $\cc{C}at$ is not the (conical) pseudo-colimit of representable 2-functors indexed by its 2-diagram, as it is in the 1-dimensional case. Finally, in \ref{pu2pro}, we state and prove the universal property of $2$-$\cc{P}ro(\cc{C})$ (Theorem \ref{pseudouniversal}), in an original way even applied to the classical pro-object theory. 

We also consider in this section a 2-category $\Prop{C}$ which is retract pseudo-equivalent to $\Pro{C}$, 
\ref{proppseudoeqapro}, fact that follows from the flexible nature of the category-valued 2-functor associated to a 2-pro-object. This 2-category will be essential in section \ref{2-modelos}, and may prove to be interesting in itself.

Most of the results of sections \ref{prelims} and \ref{2-Pro-objects} have been published \cite{DD2}.

In section \ref{Mtrick} we prove reindexing properties of pro-objects in the 2-categorical case. This section is inspired in the 1-dimensional analogous results given in \cite{AM}, but, as happened with results of section \ref{2-Pro-objects}, its 2-categorical version suppose a greater challenge.
The first result is a 2-categorical version of a result due to Deligne  \cite[Expose I, 8.1.6]{G2} and that is key to develop the closed 2-bmodel structure for $\mathsf{Pro(C)}$ in the 1-dimensional case treated in \cite{EH}. 
The 1-dimensional statement establishes that every pro-object is isomorphic to a pro-object indexed by a cofinite and filtered poset. Our version establishes that every \mbox{2-pro-object} is equivalent to a 2-pro-object indexed by a cofinite and filtered poset. The second result establishes that every morphism of \mbox{2-pro-objects} can be lifted up to equivalence to a morphism between 2-pro-objects indexed by a cofinite and filtered poset.
This is a particular case of the third result that establishes that every finite diagram in $2$-$\cc{P}ro(\cc{C})$ can be lifted up to equivalence to a diagram of 2-pro-objects indexed by a cofinite and filtered poset.
It is key for these results the notion of 2-cofinal pseudo-functor given in section \ref{prelims}. All this section will be used to prove the central theorems of section \ref{2-modelos}. 

In section \ref{Definitions and basic lemmas} we introduce original notions of closed 2-model and closed 2-bmodel \mbox{2-category} and state some lemmas and propositions that we are going to use later. Our notion is stronger than Pronk's ``fibration structures'' (\cite{Pronk}) since it is a 2-dimensional version of the full Quillen's axioms for closed model structures. It also differs in the important fact that we do not assume the choice of a privileged global factorization given in a pseudo-functorial way, but stipulates, as Quillen does, only the existence of factorizations for each arrow. Most of the results of this section are generalizations to the context of 2-categories of well known statements about closed model categories.


To conclude, in section \ref{2-modelos}, we prove one of the central theorems of this thesis (\ref{Proesde2-bmodelos}) which establishes that if $\cc{C}$ is a closed 2-bmodel 2-category, then so is $\Pro{C}$. For this result, it was necessary to prove first theorems \ref{CJdeclosed 2-modelos} and \ref{Propde2modelos} which establish that the 2-category $p\cc{H}om_p(\ff{J}^{op},\cc{C})$ (definition \ref{ccHom}) and the 2-category $\Prop{C}$ are closed 2-bmodel 2-categories respectively if $\cc{C}$ is ($\ff{J}$ will be a cofinite and filtered poset with a unique initial object). 
Reindexing properties proved in section \ref{Mtrick} were key to obtain \ref{Propde2modelos} from \ref{CJdeclosed 2-modelos}.

\vspace{2ex}

\paragraph{Notation}

In addition to the usual ``pasting'' of diagrams, we will use the \emph{Elevators calculus} for expressions denoting 2-cells (compare with the notation used in \mbox{\cite[3.10,  3.17]{F}}). 
This is a very graphic notation created by Eduardo Dubuc in 1969  
to write down equations with natural transformations between functors. In this thesis, we use elevators to write down equations with 2-cells in 2-categories. Objects are omitted, 
2-cells are denoted by cells and identity 2-cells as a double line. It is important to remark that when a 2-cell between different arrows is the identity, it is still written as a 2-cell with ``$=$'' as label. For example, the structural 2-cell of a 2-functor viewed as a particular case of a pseudo-functor. Compositions must be read from top to bottom and from right to left. Equation \ref{basicelevator} is the basic equality for elevators calculus:
$$ \label{ascensor}
\xymatrix@C=0ex
         {
            \,\ff{f}'\, \dcell{\alpha'} & \,\ff{f}\, \did
          \\
            \,\ff{g}'\, \did & \,\ff{f}\, \dcell{\alpha}
          \\
             \,\ff{g}'\,  &  \,\ff{g}\, 
         }
\xymatrix@R=6ex{\\ \;\;\;=\;\;\; \\}
\xymatrix@C=0ex
         {
             \,\ff{f}'\, \did & \,\ff{f}\, \dcell{\alpha}
          \\
             \,\ff{f}'\,\dcell{\alpha'} & \,\ff{g}\, \did
          \\
             \,\ff{g}'\, & \,\ff{g}\, 
         }
\xymatrix@R=6ex{ \\ \;\;\;=\;\;\; \\}
\xymatrix@C=0ex@R=0.9ex
         {
             {} & {}
          \\
                \,\ff{f}'\, \ar@<4pt>@{-}'+<0pt,-6pt>[ddd] 
                   \ar@<-4pt>@{-}'+<0pt,-6pt>[ddd]^{\alpha'}
             &  \,\ff{f}\,  \ar@<4pt>@{-}'+<0pt,-6pt>[ddd] 
                   \ar@<-4pt>@{-}'+<0pt,-6pt>[ddd]^{\alpha}
          \\ 
             {} & {}
          \\ 
             {} & {}
          \\
             \,\ff{g}'\, & \,\ff{g}\,.
         }
$$
This allows to move cells up and down when there are no obstacles, as if they were elevators.
In this way, we move cells to form configurations that fit valid equations in order to prove a new equation out of known ones.

\pagebreak

\tableofcontents

\pagebreak

\section{Preliminaries on 2-categories}\label{prelims}

We distinguish between \emph{small} and \emph{large} sets. For us \emph{legitimate} \mbox{categories} are categories with small hom sets, also called \emph{locally small}. We freely consider without previous warning illegitimate categories with large hom sets, for example the category of all (legitimate) categories, or functor categories with large (legitimate) exponent. They are \mbox{legitimate} as categories in some higher universe, or they can be considered as \mbox{convenient} notational \mbox{abbreviations} for extended collections of data. In fact, questions of size play no overt role in this work, except that we elect for simplicity to consider only small \mbox{2-pro-objects}. We will \mbox{explicitly} mention whether the categories are legitimate or  small when necessary. We reserve the notation $\cc{C}at$ for the legitimate 2-category of small categories, and we will denote $\cc{CAT}$ the illegitimate category (or 2-category) of all legitimate categories.

\begin{sinnadastandard}
 {\bf Notation.} 2-Categories will be denoted with the ``mathcal'' font 
 $\cc{C},\;\cc{D},\; \ldots \;$, \mbox{pseudo-functors} (in particular \mbox{2-functors}) with the capital ``mathff'' font, $\ff{F}$, $\ff{G}$, ... and pseudo-natural transformations (in particular \mbox{2-natural transformations}) and modifications with the Greek alphabet.
For objects in a 2-category, we will use capital ``mathff'' font $\ff{C},\;\ff{D}, \ldots \;$, for arrows in a 2-category, small case letters in ``mathff'' font $\ff{f},\;\ff{g},\;\ldots \;$, and we will use the Greek \mbox{alphabet} for 2-cells. \mbox{However}, when a \mbox{2-category} is intended to be used as the index 2-category of a \mbox{2-diagram}, we will use small case letters $i,\;j,\;\ldots \;$ to denote its objects, and small case \mbox{letters} $u,\;v,\; \dots \;$ to denote its arrows. Categories will be denoted with capital ``mathff'' font $\ff{C},\;\ff{D}, \ldots \;$, objects in a category with capital letters $C,\;D, \ldots \;$ and arrows in a category with small case letters $f,\;g,\;\ldots \;$.
\end{sinnadastandard}

 We begin with some background material on 2-categories. Most of
this is standard, but some results (for which we provide proofs) do not appear to be in the literature. We also set notation and terminology as we will explicitly use in this thesis.

\subsection{Basic theory}\label{seccion2-categorias}

Let $\cc{C}at$ be the category of small categories. By a \mbox{2-category}, we mean a $\cc{C}at$ enriched category. A 2-functor, a 2-fully-faithful \mbox{2-functor}, a 2-natural transformation and a \mbox{2-equivalence} of 2-categories, are a \mbox{$\cc{C}at$-functor}, a \mbox{$\cc{C}at$-fully-faithful} functor, a $\cc{C}at$-natural transformation and a $\cc{C}at$-equivalence respectively. For an extended treatment on enriched category theory see \cite{K1}.

 In the sequel we will call \emph{2-category} a structure satisfying the \mbox{following} descriptive definition free of the size restrictions implicit above. As usual, given a \mbox{2-category,} we denote horizontal composition by \mbox{juxtaposition}, and vertical
composition by ``$\circ$''.
\begin{sinnadaitalica} {\bf 2-Category.}
A \emph{2-category} $\cc{C}$ consists on objects or 0-cells $\ff{C}$, $\ff{D}$ ... , arrows or 1-cells $\ff{f}$, $\ff{g}$ ... ,  and 2-cells $\alpha$, $\beta$, ... .

$$\ff{C} \cellrd{\ff{f}}{\alpha}{\ff{g}} \ff{D}$$

The objects and the arrows form a category (called the \mbox{underlying} \mbox{category} of $\cc{C}$), with composition (called ``horizontal'') denoted by \mbox{juxtaposition}. For a fixed $\ff{C}$ and $\ff{D}$, the arrows between them and the 2-cells between these arrows form a category $\cc{C}(\ff{C},\ff{D})$ under ``vertical'' \mbox{composition}, denoted by ``$\circ$''. There is also an associative horizontal \mbox{composition} between 2-cells denoted by \mbox{juxtaposition}, with units
$id_{id_{\ff{C}}}$. The following is the basic \mbox{2-category} diagram: 
\begin{equation}\label{2cat}
\xymatrix@R=1ex
        {
           {\;\;} \ar[rr]^{\ff{f}}
         & {\;\;} \ar@{}[dd]|<<<{{\Downarrow} \; \alpha}
         & {\;\;} \ar[rr]^{\ff{f}'}
         & {\;\;} \ar@{}[dd]|<<<{{\Downarrow} \; {\alpha'}}
         & {\;\;}
         \\
           {\;\;}
         & 
         & {\;\;}
         & 
         & {\;\;}
         & {\;\;}
         \\
           {\ff{C}} \ar[rr]^{\ff{g}}
         & {\;\;} \ar@{}[dd]|<<<<{{\Downarrow} \; \beta}
         & {\ff{D}} \ar[rr]^{\ff{g}'}
         & {\;\;} \ar@{}[dd]|<<<<{{\Downarrow} \; {\beta'}}
         & {\ff{E}}
         \\
           {\;\;}
         & 
         & {\;\;}
         & 
         & {\;\;}
         & {\;\;}
         \\
           {\comw{\C}} \ar[rr]^{\ff{h}}
         & {\;\;}
         & {\;\;} \ar[rr]^{\ff{h}'}
         & {\;\;}
         & {\;\;}}
\end{equation}
with the equations
$
(\beta'\beta)\circ(\alpha' \alpha) =
           (\beta'\circ\alpha')(\beta\circ\alpha),
$
$id_{\ff{f}'}id_\ff{f}=id_{\ff{f}'\ff{f}}$.
\end{sinnadaitalica}

In particular it follows that given 
$\ff{C} \cellrd{\ff{f}}{\alpha}{\ff{g}} \ff{D}
\cellrd{\ff{f}'}{\alpha'}{\ff{g}'} \ff{E}$, we have:
\begin{equation} \label{basicelevator}
 (\alpha'\,id_{\ff{g}}) \circ (id_{\ff{f}'}\,\alpha) =
(id_{\ff{g}'}\,\alpha) \circ (\alpha'\,id_{\ff{f}}) =
(\alpha'\alpha).
\end{equation}

\vspace{1ex}

We consider juxtaposition more binding than ``$\circ$'',
thus $\alpha \beta \circ \gamma $ means $(\alpha \beta ) \circ \gamma$. We will abuse notation by writing $\ff{f}$ instead of $id_\ff{f}$ for arrows $\ff{f}$ when there is no risk of confusion.

\begin{sinnadaitalica} {\bf Dual 2-Category.} \label{oposite}
If $\cc{C}$ is a 2-category, we denote by $\cc{C}^{op}$ the \mbox{2-category} with the same objects as $\cc{C}$ but with $\cc{C}^{op}(\ff{C},\ff{D})=\cc{C}(\ff{D},\ff{C})$, i.e. we reverse the 1-cells but not the 2-cells.
\end{sinnadaitalica}

\begin{remark}
The category of all categories $\cc{C}at$ has a 2-category structure given by the following:

- $\:$ Its objects are the categories.

- $\:$ Its arrows are the functors.

- $\:$ Its 2-cells are the natural transformations.

With the notation of $\eqref{2cat}$, the composition between functors and the vertical composition between natural transformations are the usual ones. And the horizontal composition between natural transformations is given by        
$(\alpha'\alpha)_C=\alpha'_{\ff{g}C}\circ \ff{f}'(\alpha_C)$ for $C\in \ff{C}$.

One can easily check that this gives a 2-category structure.\cqd        

\end{remark}

\begin{sinnadaitalica} {\bf Equivalence.}\label{equivalencia}
An arrow $\ff{C}\mr{\ff{f}}\ff{D}$ in a 2-category $\cc{C}$ is said to be an \emph{equivalence} in $\cc{C}$ if there exist another arrow $\ff{D}\mr{\ff{g}}\ff{C} \in \cc{C}$ and invertible 2-cells $\ff{f}\ff{g}\Mr{\alpha}id_{\ff{D}}$, $\ff{g}\ff{f}\Mr{\beta}id_{\ff{C}}$.
\end{sinnadaitalica}

\begin{sinnadastandard} {\bf Notation.}
 We will denote equivalences by $\simeq$ and isomorphisms by $\cong$.
\end{sinnadastandard}

\begin{remark}
Equivalences in $\cc{C}at$ are usual equivalences of categories. \cqd
\end{remark}

\begin{sinnadaitalica} {\bf 2-functor.}\label{2functor}
A \emph{2-functor} $\ff{F}:\cc{C} \mr{} \cc{D}$ between 2-categories is an enriched functor over $\cc{C}at$. As such, sends objects to objects, arrows to arrows and 2-cells to 2-cells, strictly preserving all the structure. \end{sinnadaitalica}

\begin{sinnadaitalica} {\bf Pseudo-functor.}\label{pseudo-functor}
 A \emph{pseudo-functor} $\ff{F}:\cc{C} \mr{} \cc{D}$ between 2-categories is a correspondence that sends objects to objects, arrows to arrows and 2-cells to 2-cells, preserving all the structure up to invertible 2-cells $\ff{F}\ff{g}\ff{F}\ff{f}\Mr{}\ff{F}(\ff{g}\ff{f})$ and $id_{\ff{F}\ff{C}} \Mr{}\ff{F}(id_{\ff{C}})$ instead of equalities. More explicitly, it is given by the following data:


\begin{itemize}
 \item[-] For each object $\ff{C}\in \cc{C}$, an object $\ff{F}(\ff{C})\in \cc{D}$. We will abuse notation and write $\ff{F}\ff{C}$ when there is no risk of confusion.
 \item[-] For each hom-category $\cc{C}(\ff{C},\ff{D})$, a functor $\ff{F}_{\ff{C},\ff{D}}:\cc{C}(\ff{C},\ff{D})\mr{} \cc{D}(\ff{F}\ff{C},\ff{F}\ff{D})$.
\end{itemize}

\begin{center}
  \begin{minipage}{10.5cm}
We will abuse notation and write $\ff{F}\ff{f}$ instead of $\ff{F}_{\ff{C},\ff{D}}(\ff{f})$ and $\ff{F}\alpha$ instead of $\ff{F}_{\ff{C},\ff{D}}(\alpha)$ for $\ff{C}\cellrd{\ff{f}}{\alpha}{\ff{g}}\ff{D} \in \cc{C}$ when there is no risk of confusion.
  \end{minipage}
\end{center}

\begin{itemize}
 \item[-] For each object $\ff{C}\in \cc{C}$, an invertible 2-cell $\alpha^{\ff{F}}_{\ff{C}}:id_{\ff{F}\ff{C}}\Mr{} \ff{F}(id_{\ff{C}}) \in \cc{D}$.
 \item[-] For each triplet $\ff{C}, \ff{D}, \ff{E}$ of objects of $\cc{C}$, a natural isomorphism:

$$\xymatrix{\cc{C}(\ff{C},\ff{D})\times \cc{C}(\ff{D},\ff{E}) \ar[rr]^{\ff{F}\times \ff{F}} \ar[d]_c 
            & \ar@{}[d]|{\cong \; \Downarrow \; \alpha^{\ff{F}}} 
            & \cc{D}(\ff{F}\ff{C},\ff{F}\ff{D})\times \cc{D}(\ff{F}\ff{D},\ff{F}\ff{E}) \ar[d]^c 
            \\
            \cc{C}(\ff{C},\ff{E}) \ar[rr]_{\ff{F}} 
            &
            & \cc{D}(\ff{F}\ff{C},\ff{F}\ff{E})}$$

where $c$ denotes the composition functors.

More explicitly, $\alpha^{\ff{F}}$ consists on an invertible 2-cell $\ff{F}\ff{C}\cellrdb{\ff{F}\ff{g}\ff{F}\ff{f}}{\alpha^{\ff{F}}_{\ff{f},\ff{g}}}{\ff{F}(\ff{g}\ff{f})} \ff{F}\ff{E}$ for each \mbox{configuration} $\ff{C}\mr{\ff{f}} \ff{D} \mr{\ff{g}} \ff{E}\in \cc{C}$ such that $\forall \ \ff{C}\cellrd{\ff{f}}{\theta}{\ff{f}'} \ff{D} \cellrd{\ff{g}}{\rho}{\ff{g}'} \ff{E}$, $\ff{F}(\rho \theta)\circ \alpha^{\ff{F}}_{\ff{f},\ff{g}}=\alpha^{\ff{F}}_{\ff{f}',\ff{g}'} \circ \ff{F}\rho \ff{F}\theta$

$$\vcenter{\xymatrix@C=0pc{\F\g & & \F\f \\
                         & \F(\g\f) \cl{\alpha^{\ff{F}}_{\f,\g}} \dcellbb{\F(\rho \theta)} \\
                         & \F(\g'\f')}}
\vcenter{\xymatrix@C=-0pc{\; = \;}}
\vcenter{\xymatrix@C=-0pc{\F\g \dcellb{\F\rho} & & \F\f \dcellb{\F\theta} \\
            \F\g' &  & \F\f' \\
            & \F(\g'\f') \cl{\alpha^{\ff{F}}_{\f',\g'}} }}
$$
\end{itemize}

\vspace{1ex}

\noindent All this data must satisfy the following equalities:

\begin{itemize}
 \item[-]For each $\ff{C}\mr{\ff{f}}\ff{D} \in \cc{C}$,

$$\vcenter{\xymatrix@C=-0pc
       {
        id_{\ff{F}\ff{D}} \dcellb{\alpha^{\ff{F}}_\ff{D}} 
	& &
	\ff{F}\ff{f} \did
	\\
	\ff{F}(id_\ff{D}) 
	& 
	& \ff{F}\ff{f} 
	\\
	&
        \ff{F}\ff{f} \clb{\alpha^{\ff{F}}_{\ff{f},id_\ff{D}}}
        &
        }}
\vcenter{\xymatrix@C=-0pc{\; = \;}}
\vcenter{\xymatrix@C=-0pc
       {
        id_{\ff{F}\ff{D}} 
        &  
        & \ff{F}\ff{f} 
        \\
        &
        \ff{F}\ff{f} \cl{=}
        &
        }}
        \quad \hbox{ and } \quad
\vcenter{\xymatrix@C=-0pc
       {
        \ff{F}\ff{f} \did
	& &
	id_{\ff{F}\ff{C}} \dcellb{\alpha^{\ff{F}}_\ff{C}} 
	\\
	\ff{F}\ff{f} 
	& 
	& \ff{F}(id_\ff{C}) 
	\\
	&
        \ff{F}\ff{f} \clb{\alpha^{\ff{F}}_{id_\ff{C},\ff{f}}}
        &
        }}        
\vcenter{\xymatrix@C=-0pc{\; = \; }}
\vcenter{\xymatrix@C=-0pc
       {
        \ff{F}\ff{f} 
        &  
        & id_{\ff{F}\ff{C}} 
        \\
        &
        \ff{F}\ff{f} \cl{=}  
        &
        }}
$$

\item[-] For each configuration $\ff{A} \mr{\ff{f}} \ff{B} \mr{\ff{g}} \ff{C} \mr{\ff{h}} \ff{D} \in \cc{C}$,

$$\vcenter{\xymatrix@C=-0pc
       {
        \ff{F}\ff{h} \did
	& \ff{F}\ff{g} 
	& 
	& \ff{F}\ff{f} 
	\\
	\ff{F}\ff{h} 
	& 
	& \ff{F}(\ff{gf}) \cl{\alpha^{\ff{F}}_{\ff{f},\ff{g}}}
	&
	\\
	&
        \ff{F}(\ff{hgf}) \cl{\alpha^{\ff{F}}_{\ff{gf},\ff{h}}}
        &
        }}        
\vcenter{\xymatrix@C=-.6pc{\quad = \quad}}
\vcenter{\xymatrix@C=-0pc
       {
        \ff{F}\ff{h} 
	& 
	& \ff{F}\ff{g} 
	& \ff{F}\ff{f} \did
	\\
	& \ff{F}(\ff{hg}) \cl{\alpha^{\ff{F}}_{\ff{g},\ff{h}}} 
	& 
	& \ff{F}\ff{f} 
	\\
	& &         \ff{F}(\ff{hgf}) \cl{\alpha^{\ff{F}}_{\ff{f},\ff{hg}}}
        &
        }}
$$
\end{itemize}
\end{sinnadaitalica}

\begin{remark}
 A 2-functor is a pseudo-functor such that $\alpha_{\ff{C}}^{\ff{F}}$ is the equality for each $\ff{C} \in \cc{C}$ and $\alpha_{\ff{f},\ff{g}}^{\ff{F}}$ is the equality for each $\ff{C} \mr{\ff{f}} \ff{D} \mr{\ff{g}} \ff{E} \in \cc{C}$. \cqd
\end{remark}

\begin{sinnadaitalica}{\bf Pseudo-essentially surjective on objects.} 
A pseudo functor $\ff{F}:\cc{C}\mr{} \cc{D}$ is said to be \mbox{\emph{pseudo-essentially surjective on objects}} if for each $\ff{D} \in \cc{D}$, there exist $\ff{C}\in \cc{C}$ and an equivalence $\ff{F}\ff{C}\mr{}\ff{D} \in \cc{D}$.
\end{sinnadaitalica}

\begin{sinnadaitalica} {\bf 2-fully-faithful.}\label{2f&f}
A 2-functor $\ff{F}:\cc{C}\mr{} \cc{D}$ is said to be \mbox{\emph{2-fully-faithful}} if for each $\ff{C},\ \ff{D}\in \cc{C}$, $\ff{F}:\cc{C}(\ff{C},\ff{D})\mr{} \cc{D}(\ff{F}\ff{C},\ff{F}\ff{D})$ is an isomorphism of categories.
\end{sinnadaitalica}

\begin{sinnadaitalica} {\bf Pseudo-fully-faithful.}\label{pf&f}
A pseudo-functor $\ff{F}:\cc{C}\mr{} \cc{D}$ is said to be \mbox{\emph{pseudo-fully-faithful}} if for each $\ff{C},\ \ff{D}\in \cc{C}$, $\ff{F}:\cc{C}(\ff{C},\ff{D})\mr{} \cc{D}(\ff{F}\ff{C},\ff{F}\ff{D})$ is an equivalence of categories.
\end{sinnadaitalica}

\begin{sinnadaitalica} {\bf Pseudo-natural transformation.} \label{pseudo-naturalentrepseudo-functors}
A \emph{pseudo-natural transformation} \mbox{$\theta:\ff{F}\Rightarrow \ff{G}:\cc{C} \rightarrow \cc{D}$} between pseudo-functors consists of a family of arrows $\left\{ \ff{F}\ff{C}\stackrel{\theta_{\ff{C}}}\rightarrow \ff{G}\ff{C} \right\}_{\ff{C}\in \cc{C}}$ and a family of invertible 2-cells $\left\{ \ff{G}\ff{f}\theta_{\ff{C}}\Mr{\theta_\ff{f}} \theta_{\ff{D}}\ff{F}\ff{f} \right\}_{\ff{C}\stackrel{\ff{f}}\rightarrow \ff{D}\in \cc{C}}$ 

$$\xymatrix@R=1.5pc@C=1.5pc{\ff{F}\ff{C}\ar[rr]^{\theta_{\ff{C}}} \ar[dd]_{\ff{F}\ff{f}} \ar@{}[ddrr]|{\cong \; \Downarrow \; \theta_\ff{f}}  & & \ff{G}\ff{C} \ar[dd]^{\ff{G}\ff{f}} \\ & & \\ \ff{F}\ff{D} \ar[rr]_{\theta_{\ff{D}}} && \ff{G}\ff{D}}$$

\noindent satisfying the following conditions:

\begin{itemize}
\item[PN0.] For each $\ff{C}\in \cc{C}$, $\; \theta_{\ff{C}} \alpha^\ff{F}_\ff{C}=\theta_{id_{\ff{C}}} \circ \alpha^\ff{G}_\ff{C} \theta_{\ff{C}}$, $\;$ i.e.
                                  
$$\vcenter{\xymatrix@C=.9pc@R=3pc
	   {
		   \theta_{\ff{C}} \did 
		& id_{\ff{F}\ff{C}} \dcellb{\alpha^\ff{F}_{\ff{C}}}  
		\\
		\theta_{\ff{C}} 
	        &
	        \ff{F}id_\ff{C}
		}}
\vcenter{\xymatrix@C=-.4pc{\quad = \quad \quad }}
\vcenter{\xymatrix@C=.9pc
       {
        id_{\ff{G}\ff{C}} \dcellb{\alpha^\ff{G}_\ff{C}} 
	& 
	\theta_\ff{C} \did
	\\
	\ff{G}id_\ff{C} \dl
	& \theta_\ff{C} \ar@{}[dl]|{\theta_{id_\ff{C}}}  \dr
	\\
        \theta_\ff{C}
	&
	\ff{F}id_\ff{C}
	}}$$
$$ i.e. \ \ \ \ \vcenter{\xymatrix{\ff{F}\ff{C} \ar@/^2ex/[r]^{id_{\ff{F}\ff{C}}} \ar@{}[r]|{\cong \; \Downarrow \; \alpha^\ff{F}_\ff{C}} \ar@/_2ex/[r]_{\ff{F}id_{\ff{C}}} 
                                   &
                                   \ff{F}\ff{C} \ar[r]^{\theta_\ff{C}}
                                   &
                                   \ff{G}\ff{C}
                                   }}
\vcenter{\xymatrix@C=-.4pc{\quad = \quad \quad }}
\vcenter{\xymatrix@R=1.2pc{\ff{F}\ff{C} \ar[r]^{\theta_\ff{C}} \ar@/_2ex/[rdd]_{\ff{F}id_\ff{C}} 
                                   &
                                   \ff{G}\ff{C} \ar@{}[dd]|{\cong \; \Downarrow \; \theta_{id_\ff{C}} } \ar@/^2ex/[r]^{id_{\ff{G}\ff{C}}} \ar@{}[r]|{\cong \; \Downarrow \; \alpha^\ff{G}_\ff{C}} \ar@/_2ex/[r]_{\ff{G}id_{\ff{C}}}
                                   &
                                   \ff{G}\ff{C} 
                                   \\
                                   &
                                   \\
                                   &
                                   \ff{F}\ff{C} \ar@/_2ex/[ruu]_{\theta_\ff{C}}
                                   }}$$ 
                                   
\item[PN1.] For each $\ff{C}\stackrel{\ff{f}}\rightarrow \ff{D}\stackrel{\ff{g}}\rightarrow \ff{E}$, $\; \theta_{\ff{E}} \alpha^{\F}_{\ff{f},\ff{g}} \circ \theta_{\ff{g}} \ff{F}\ff{f} \circ \ff{G}\ff{g} \theta_{\ff{f}}=\theta_{\ff{g}\ff{f}}\circ \alpha^{\G}_{\ff{f},\ff{g}} \theta_{\ff{C}}$, $\;$ i.e. 
   
 $$\vcenter{\xymatrix@C=-0pc{
 	   \ff{G}\ff{g} \did	 
 	   & & 
 	   \ff{G}\ff{f} \dl 
 	   &&
 	   \theta_\ff{C} \ar@{}[dll]|{\theta_\ff{f}} \dr
 	   \\
 	   \ff{G}\ff{g} \dl
 	   & & 
 	   \theta_\ff{D} \ar@{}[dll]|{\theta_\ff{g}} \dr
 	   &&
 	    \ff{F}\ff{f} \did
 	   \\ 
 	   \theta_{\ff{E}} \did
            & &  \ff{F}\ff{g} \dl
            &&
            \ff{F}\ff{f} \ar@{}[dll]|{\alpha^{\F}_{\ff{f},\ff{g}}} \dr
            \\
            \theta_{\ff{E}} 
            &&&
            \!\!\!\!\! \ff{F}\ff{g}\ff{f} \!\!\!\!\! &
            }}
 \vcenter{\xymatrix@C=-.4pc{\quad = \quad }}
 \vcenter
   {
    \xymatrix@C=-0pc
        {\ff{G}\ff{g} 
 		 &&
 		 \ff{G}\ff{f} 
 		 & & \theta_\ff{C} \did 
 	   \\
 	   &
 		 \ff{G}\ff{g}\ff{f} \cl{\alpha^{\G}_{\ff{f},\ff{g}}}   \dl
 		&& &
 		\theta_\ff{C} \ar@{}[dll]_{\theta_{\ff{g}\ff{f}}}  \dr
 		\\
 		&
 		 \theta_\ff{E} 
 		& &&
 		 \ff{F}\ff{g}\ff{f} 
 		}
   }
 $$
$$ i.e. \ \ \ \ \vcenter{
\xymatrix@C=1.5pc@R=.8pc
   {
&      \ff{F}\ff{C} \ar@{}[ddrr]|{ \cong \;\Downarrow \; \theta_\ff{f}}  \ar@/_8ex/[dddd]_{\ff{F}\ff{g}\ff{f}} \ar[dd]_{\ff{F}\ff{f}} \ar[rr]^{\theta_{\ff{C}}} 
   && \ff{G}\ff{C} \ar[dd]^{\ff{G}\ff{f}} 
   && \ff{F}\ff{C} \ar@{}[ddddrr]|{\cong \; \Downarrow \; \theta_{\ff{g}\ff{f}}} \ar[dddd]_{\ff{F}\ff{g}\ff{f}} \ar[rr]^{\theta_{\ff{C}}}  
   && \ff{G}\ff{C} \ar[dddd]_{\ff{G}\ff{g}\ff{f}} \ar@/^2ex/[rdd]^{\ff{G}\ff{f}}
   \\ 
&   & 
   &&&&
   \\
\ar@{}[r]|<<<{\stackrel[\cong]{\alpha^{\F}_{\ff{f},\ff{g}}}{\Leftarrow}} &       \ff{F}\ff{D} \ar@{}[ddrr]|{ \cong \; \Downarrow \; \theta_\ff{g}}   \ar[dd]_{\ff{F}\ff{g}} \ar[rr]^{\theta_{\ff{D}}} 
   && \ff{G}\ff{D} \ar[dd]^{\ff{G}\ff{g}}
   & = \;\;\;
   & & 
   & \ar@{}[r]|{\stackrel[\cong]{\alpha^{\G}_{\ff{f},\ff{g}}}{\Leftarrow}} &
   \ff{G}\ff{D} \ar@/^2ex/[ldd]^{\ff{G}\ff{g}}
   \\
&   & 
   \\
&      \ff{F}\ff{E} \ar[rr]_{\theta_{\ff{E}}} 
   && \ff{G}\ff{E} 
   && \ff{F}\ff{E}\ar[rr]_{\theta_{\ff{E}}}  
   && \ff{G}\ff{E} }}$$

\item[PN2.] For each $\ff{C}\cellrd{\ff{f}}{\alpha}{\ff{g}}\ff{D} \in \cc{C}$, $\; \theta_{\ff{g}} \circ \ff{G}\alpha \theta_{\ff{C}}=\theta_{\ff{D}}\ff{F}\alpha \circ \theta_{\ff{f}}$, $\;$ i.e.

$$\vcenter{\xymatrix@C=-0pc{
                      \ff{G}\ff{f} \dcellb{\ff{G}\alpha}   
		      &&
		      \theta_\ff{C}  \did 
		      \\
		       \ff{G}\ff{g} \dl
		       &&
		      \theta_\ff{C} \dr \ar@{}[dll]|{\theta_\ff{g}}
		      \\
		       \theta_{\ff{D}} 
		       && 
		      \ff{F}\ff{g} 
		      }}
      \vcenter{\xymatrix@C=-.4pc{\quad = \quad }}
      \vcenter{\xymatrix@C=-0pc{
		      \ff{G}\ff{f} \dl
		      &&
		      \theta_\ff{C} \dr \ar@{}[dll]|{\theta_\ff{f}} 
		      \\
		      \theta_{\ff{D}} \did
		      &&
		      \ff{F}\ff{f} \dcellb{\ff{F}\alpha} 
		      \\ 
		      \theta_\ff{D}
		      &&
		      \ff{F}\ff{g} 
		      }}
		      $$
$$ i.e. \ \ \ \ \vcenter{\xymatrix@C=.8pc@R=.8pc
  {
    \ff{F}\ff{C} \ar@{}[ddrrr]|{ \cong \; \Downarrow \; \theta_\ff{g} \quad } \ar[dd]_{\ff{F}\ff{g}} \ar[rrr]^{\theta_{\ff{C}}} 
    &&& \ff{G}\ff{C} \ar@<-1.5ex>[dd]_{\ff{G}\ff{g}} \ar@<1.5ex>[dd]^{\ff{G}\ff{f}} \ar@{}[dd]|{\stackrel[\cong]{\ff{G}\alpha}{\Leftarrow}}
    &&& & \ff{F}\ff{C} \ar@<-1.5ex>[dd]_{\ff{F}\ff{g}} \ar@<1.5ex>[dd]^{\ff{F}\ff{f}} \ar@{}[dd]|{\stackrel[\cong]{\ff{F}\alpha}{\Leftarrow}} \ar[rrr]^{\theta_{\ff{C}}} 
    &&& \ff{G}\ff{C} \ar[dd]^{\ff{G}\ff{f}}
    \\
    & 
    && 
    && = 
    && 
    &&  
    & 
    \\
    \ff{F}\ff{D}\ar[rrr]_{\theta_{\ff{D}}}  
    &&& \ff{G}\ff{D} 
    &&&& \ff{F}\ff{D} \ar[rrr]_{\theta_{\ff{D}}} 
    &&& \ff{G}\ff{D} \ar@{}[uulll]|{\cong \;\Downarrow \; \theta_\ff{f}}
    }}$$	
\end{itemize}
\end{sinnadaitalica}

As a particular case, we have the notion of pseudo-natural transformation between 2-functors.

\begin{sinnadaitalica} {\bf 2-Natural transformation. }\label{2espseudo}
A 2-natural transformation $\theta$ between 2-functors is a pseudo-natural transformation such that $\theta_\ff{f}$ is the equality for each $\ff{f}\in \cc{C}$. Equivalently, it is a $\cc{C}at$-enriched natural transformation, that is, a natural transformation  between the functors determined by $\ff{F}$ and $\ff{G}$, such that for each 2-cell
$\ff{C}\cellrd{\ff{f}}{\alpha}{\ff{g}}\ff{D}$, the equation $\ff{G}\alpha\theta_{\ff{C}}=\theta_{\ff{D}}\ff{F}\alpha$ holds.  
\end{sinnadaitalica}

\begin{sinnadaitalica} {\bf Modification.}
Given pseudo-functors $\ff{F}$ and $\ff{G}$ from $\cc{C}$ to
$\cc{D}$ (as a particular case $\ff{F}$ and $\ff{G}$ might be 2-functors), a \mbox{\emph{modification}}
$\ff{F} \cellrd{\theta}{\rho}{\eta} \ff{G}$ between pseudo-natural transformations is a family
\mbox{$\left\{\theta_{\ff{C}}\Mr{\rho_{\ff{C}}}\eta_{\ff{C}}\right\}_{\ff{C}\in \cc{C}}$} of 2-cells of
$\cc{D}$ such that:

\begin{itemize}
 \item[PM.] For each $ \ff{C}\mr{\ff{f}} \ff{D} \in \cc{C},
\; \rho_{\ff{D}} \ff{Ff} \circ \theta_\ff{f} =
 \eta_\ff{f} \circ \ff{Gf} \rho_\ff{C}
$, $\;$ i.e.

$$\vcenter{\xymatrix@C=-0pc{
		      \ff{G}\ff{f} \dl 
		      && 
		      \theta_{\ff{C}} \dr \ar@{}[dll]|{\theta_\ff{f}} 
		      \\
		      \theta_\ff{D} \dcellb{\rho_\ff{D}}
		      && 
		      \ff{F}\ff{f} \did  
		      \\
		      \eta_\ff{D}
		      &&
		      \ff{F}\ff{f}
		      }}
      \vcenter{\xymatrix@C=-.4pc{\quad = \quad }}
      \vcenter{\xymatrix@C=-0pc{
		      \ff{G}\ff{f} \did 
		      && 
		      \theta_\ff{C} \dcellb{\rho_\ff{C}} 
		      \\
		      \ff{G}\ff{f} \dl 
		      && 
		      \eta_\ff{C} \dr \ar@{}[dll]|{\eta_\ff{f}} 
		      \\
		      \eta_\ff{D}
		      && 
		      \ff{F}\ff{f}
		      }}$$
$$ i.e. \ \ \ \ \vcenter{\xymatrix@R=1pc@C=1pc
  {
  \ff{F}\ff{C} \ar[rr]^{\theta_{\ff{C}}}   \ar[dd]_{\ff{F}\ff{f}} \ar@{}[ddrr]|{\cong \; \Downarrow \; \theta_\ff{f} }
  &
  & \ff{G}\ff{C} \ar[dd]^{\ff{G}\ff{f}}
  && \ff{F}\ff{C} \ar@{}[rr]|{\Downarrow \;\rho_{\ff{C}}} \ar[dd]_{\ff{F}\ff{f}} \ar@<1ex>[rr]^{\theta_{\ff{C}}} \ar@<-1ex>[rr]_{\eta_{\ff{C}}} \ar@{}[rrdd]|{\cong \;\Downarrow \; \eta_\ff{f}}
  &
  & \ff{G}\ff{C} \ar[dd]^{\ff{G}\ff{f}} 
  \\
  & 
  &&
  = 
  && 
  & 
  \\
  \ff{F}\ff{D} \ar@<1ex>[rr]^{\theta_{\ff{D}}} \ar@<-1ex>[rr]_{\eta_{\ff{D}}} \ar@{}[rr]|{ \Downarrow \; \rho_{\ff{D}}}
  &
  & \ff{G}\ff{D} 
  && 
  \ff{F}\ff{D} \ar[rr]_{\eta_\ff{D}}
  & & 
  \ff{G}\ff{D}
  }}$$
\end{itemize}

As a particular case, we have \mbox{\emph{modifications}} between 2-natural transformations, which are families of 2-cells as above satisfying $\rho_{\ff{D}}\ff{F}\ff{f}=\ff{G}\ff{f}\rho_{\ff{C}}$.
\end{sinnadaitalica}

\begin{sinnadastandard} \label{2CAT}
By the theory of enriched categories, it is well known that 2-categories, 2-functors and \mbox{2-natural} transformations form a 2-category
(which actually underlies a \mbox{3-category)} that we denote $2\hbox{-}\cc{CAT}$. Horizontal composition of \mbox{2-functors} and vertical composition of 2-natural transformations are the usual ones, and the horizontal composition of 2-natural transformations is defined by:

\noindent Given
$\ff{C} \cellrd{\ff{F}}{\alpha}{\ff{G}} \ff{D}
\cellrd{\ff{F}'}{\alpha'}{\ff{G}'} \ff{E}$,
$\;(\alpha'\alpha)_\ff{C} = \alpha'_{\ff{GC}}\circ \ff{F}'\alpha_\ff{C} \quad$ ($ = \ff{G}'\alpha_\ff{C} \circ \alpha'_{\ff{FC}}$).
\end{sinnadastandard}

\begin{definition}\label{ccHom}
Given two 2-categories $\cc{C}$ and $\cc{D}$, we consider three \mbox{2-categories} defined as follows:

\begin{itemize}
 \item[-] $\cc{H}om(\cc{C},\cc{D})$: 2-functors and 2-natural transformations.

 \item[-] $\cc{H}om_p(\cc{C},\cc{D})$: 2-functors and pseudo-natural transformations.

 \item[-] $p\cc{H}om_p(\cc{C},\cc{D})$: pseudo-functors and pseudo-natural transformations.
\end{itemize}

\noindent In all cases the 2-cells are the modifications. To define compositions we draw the basic 2-category diagram:
$$
\vcenter{\xymatrix@R=1ex
        {
           {\;\;} \ar[rr]^{\theta}
         & {\;\;} \ar@{}[dd]|<<<{{\Downarrow} \; {\rho}}
         & {\;\;} \ar[rr]^{\theta'}
         & {\;\;} \ar@{}[dd]|<<<{{\Downarrow} \; {\rho'}}
         & {\;\;}
         \\
           {\;\;}
         & 
         & {\;\;}
         & 
         & {\;\;}
         & {\;\;}
         \\
           {\ff{F}} \ar[rr]^{\eta}
         & {\;\;} \ar@{}[dd]|<<<<{{\Downarrow}{\;\varepsilon}}
         & {\ff{G}} \ar[rr]^{\eta'}
         & {\;\;} \ar@{}[dd]|<<<<{{\Downarrow}{\;\varepsilon'}}
         & {\ff{H}}
         \\
           {\;\;}
         & & {\;\;}
         & 
         & {\;\;}
         & {\;\;}
         \\
           {\comw{\F}} \ar[rr]^{\mu}
         & {\;\;}
         & {\;\;} \ar[rr]^{\mu'}
         & {\;\;}
         & {\;\;}
        }}
\hspace{3ex}
\vcenter{\xymatrix@R=1ex
     {
      (\theta'\theta)_{\ff{C}} =
                        \theta'_{\ff{C}}\theta_{\ff{C}}
     \\
     (\theta'\theta)_{\ff{f}} =
                          \theta'_{\ff{D}}\theta_{\ff{f}} \circ \theta'_{\ff{f}}\theta_{\ff{C}}
     \\
      (\rho'\rho)_{\ff{C}}=\rho'_{\ff{C}} \rho_{\ff{C}}
     \\
      (\epsilon\circ \rho)_{\ff{C}} =
                  \epsilon_{\ff{C}} \circ \rho_{\ff{C}}
     }}
$$
It is straightforward to check that these definitions determine \mbox{2-category} structures. \cqd
\end{definition}

\begin{remark} \cite[3.17]{F} \label{equivalencias pointwise}
A pseudo-natural transformation $\ff{F}\Mr{\theta}\ff{G}\in \cc{H}om_p(\cc{C},\cc{D})$ (respectively $p\cc{H}om_p(\cc{C},\cc{D})$) is an equivalence iff for each $\ff{C}\in \cc{C}$, $\theta_{\ff{C}}$ is an equivalence in $\cc{D}$. The same assertion does not hold in $\cc{H}om(\cc{C},\cc{D})$ (c.f. \ref{eqencadaCeseq}).\cqd
\end{remark}


\begin{remark}\label{mapsp}
Since we are going to make manipulations with the 2-category $\cc{H}om_p(\ff{2},\cc{C})$ (where $\ff{2}$ stands for the trivial 2-category with two objects, one morphism between them and no 2-cells other than identities, i.e. $\ff{2}=\{0\mr{}1\}$),
we will give a more explicit description of it:

\begin{itemize}
 \item[-] An object is a morphism $\ff{C}\mr{\f}\ff{D} \in \cc{C}$.

%
            
 \item[-] A morphism $\theta$ between $\ff{C}\mr{\f}\ff{D}$ and $\ff{C}'\mr{\g}\ff{D}'$ in $\cc{H}om_p(\ff{2},\cc{C})$ is given by two morphisms $\ff{C}\mr{\theta_0}\ff{C}'$, $\ff{D}\mr{\theta_1}\ff{D}' \in \cc{C} $ and an invertible 2-cell $\g \theta_0 \Mr{\theta_m} \theta_1 \f$ as in the following diagram:

$$\xymatrix@R=.8pc@C=.8pc{\ff{C} \ar@{}[rrdd]|{\cong \; \Downarrow \; \theta_m} \ar[rr]^{\theta_0} \ar[dd]_{\f} 
            && \ff{C}' \ar[dd]^{\g} \\
            &  \\
            \ff{D} \ar[rr]_{\theta_1} 
            &&\ff{D}'}$$     
            
%

 \item[-] A 2-cell $\mu$ in $\cc{H}om_p(\ff{2},\cc{C})$ between $\theta$ and $\eta$ from $\ff{C}\mr{\f}\ff{D}$ to $\ff{C}'\mr{\g}\ff{D}'$ is given by two 2-cells $\theta_0 \Mr{\mu_0}\eta_0$, $\theta_1 \Mr{\mu_1}\eta_1 \in \cc{C} $ such that $\mu_1 \f \circ \theta_m = \eta_m \circ \g \mu_0$.   

$$\vcenter{\xymatrix@C=-0pc{\ff{g} \dl 
                    & \dc{\theta_m} 
                    & \theta_0 \dr 
                    \\
                    \theta_1 \dcell{\mu_1} 
                    && \ff{f} \deq
                    \\
                    \eta_1 
                    && \ff{f}}}
 \vcenter{\xymatrix@C=-.4pc{\quad = \quad }}
 \vcenter{\xymatrix@C=-0pc{\ff{g} \deq
                            && \theta_0 \dcell{\mu_0} 
                            \\
                            \ff{g} \dl
                            & \dc{\eta_m}
                            & \eta_0 \dr
                            \\
                            \eta_1 
                            && \ff{f}}}$$
\end{itemize}
\end{remark}

\begin{definition}\label{retractos}


 Let $\cc{C}$ be a 2-category and $\ff{C}\mr{\ff{f}}\ff{D}$, $\ff{C}'\mr{\ff{g}}\ff{D}'$ two morphisms in $\cc{C}$. We say that $\ff{f}$ is a retract of $\ff{g}$ in $\cc{H}om_p(\ff{2},\cc{C})$ if there are morphisms $\ff{f} \mr{\theta} \ff{g}$, $\ff{g} \mr{\eta} \ff{f}$ and an invertible 2-cell $\eta \theta \Mr{\mu} id_{\ff{f}}$ in $\cc{H}om_p(\ff{2},\cc{C})$. More explicitly, the retraction consists in a tuple $(\theta_0,\theta_1,\theta_m,\eta_0,\eta_1,\eta_m,\mu_0,\mu_1)$ such that $\ff{g}\theta_0 \Mrsimeq{\theta_m} \theta_1\ff{f}$, $\ff{f}\eta_0 \Mrsimeq{\eta_m} \eta_1\ff{g}$, $\eta_0\theta_0 \Mrsimeq{\mu_0} id_{\ff{C}}$, $\eta_1 \theta_1 \Mrsimeq{\mu_1} id_{\ff{D}}$ and the following equality holds:
 
$$\vcenter{\xymatrix@C=-0pc{\ff{f} \dl 
                    & \dc{\eta_m} 
                    & \eta_0 \dr 
                    & & \theta_0 \deq
                    \\
                    \eta_1 \deq
                    && \ff{g} \dl
                    & \dc{\theta_m}
                    & \theta_0 \dr
                    \\
                    \eta_1 
                    & 
                    & \theta_1 
                    && \ff{f} \deq
                    \\
                    & id_{\ff{D}} \cl{\mu_1}
                    &&& \ff{f}}}
 \vcenter{\xymatrix@C=-.4pc{\quad = \quad }}
 \vcenter{\xymatrix@C=-0pc{\ff{f} \deq
                            & \eta_0 
                            &  
                            & \theta_0 
                            \\
                            \ff{f} \dl
                            & \dc{=}
                            & id_{\ff{C}} \cl{\mu_0} \dr
                            & \\
                            id_{\ff{D}} 
                            && \ff{f}}}$$

\end{definition}

\begin{sinnadaitalica}{\bf Bi-universal arrows.} \emph{\cite[9.4]{F}} \label{biuniversal arrows}
 Let $\cc{D} \mr{\ff{G}} \cc{C}$ be a pseudo-functor, $\ff{C}\in \cc{C}$ and $\ff{D}\in \cc{D}$. A morphism $\ff{C}\mr{\ff{f}}\ff{G}\ff{D} \in \cc{C}$ is a bi-universal arrow from $\ff{C}$ to $\ff{G}$ if for each $\ff{D}'\in \cc{D}$, the following functor is an equivalence of categories 

 $$\xymatrix@R=.5pc{\cc{D}(\ff{D},\ff{D}') & \ar[rr] && & \cc{C}(\ff{C},\ff{G}\ff{D}') \\
		      \ff{g} \Mr{\alpha} \ff{g}' & \ar@{|->}[rr] &&& \ff{G}(\ff{g}) \ff{f} \Mr{\ff{G}(\alpha) \ff{f}} \ff{G}(\ff{g}') \ff{f}
}$$
 \end{sinnadaitalica}

\begin{sinnadaitalica}{\bf Bi-adjoint pseudo-functors.} \emph{\cite[9.8]{F}} \label{2-adjuntos}
 Let $\ff{F}: \vcenter{\xymatrix{\cc{C} \ar@<.75ex>[r] \ar@<-.75ex>@{<-}[r] & \cc{D}}}:\ff{G}$ be pseudo-functors. We say that $\ff{F}$ is \mbox{\emph{bi-left adjoint}} to $\ff{G}$ (equivalently that $\ff{G}$ is bi-right adjoint to $\ff{F}$) if for each $\ff{C}\in \cc{C}$, $\ff{D}\in \cc{D}$, there is an equivalence of categories $\cc{D}(\ff{F}\ff{C},\ff{D})\mr{\phi_{\ff{C},\ff{D}}} \cc{C}(\ff{C},\ff{G}\ff{D})$ in a way such that $\phi$ is a pseudo-natural transformation in each variable. In this case, we use the notation $\ff{F} \dashv_b \ff{G}$.
\end{sinnadaitalica}

\begin{remark}\label{equivalencia de adjuntos con mi def}
 It is straightforward to check that $\ff{F}$ is bi-left adjoint to $\ff{G}$ iff there exist pseudo-natural transformations $\ff{F}\ff{G}\Mr{\epsilon} id_{\cc{D}}$, $id_{\cc{C}}\Mr{\eta} \ff{G}\ff{F}$ such that $\epsilon \ff{F} \circ \ff{F}\eta=\ff{F}$ and $\ff{G}\epsilon \circ \eta \ff{G}=\ff{G}$. \cqd
\end{remark}

\begin{proposition}  \emph{\cite[9.16]{F}} \label{equivalencia de adjuntos fiore}
Let $\ff{F}: \vcenter{\xymatrix{\cc{C} \ar@<.75ex>[r] \ar@<-.75ex>@{<-}[r] & \cc{D}}}:\ff{G}$ be pseudo-functors. Then $\ff{F}$ is bi-left adjoint to $\ff{G}$ iff there exists a pseudo-natural transformation $id_{\cc{C}} \Mr{\eta} \ff{G}\ff{F}$ such that $\eta_{\ff{C}}$ is a bi-universal arrow from $\ff{C}$ to $\ff{G}\F\C \; \forall \ \ff{C}\in \cc{C}$. \cqd 
\end{proposition}

\begin{sinnadaitalica} {\bf 2-Equivalence.} \label{2equivalencia}
A 2-functor $\cc{C}\mr{\ff{F}}\cc{D}$ is said to be a \emph{\mbox{2-equivalence} of \mbox{2-categories}} if there exist a 2-functor $\cc{D}\mr{\ff{G}} \cc{C}$ and invertible 2-natural \mbox{transformations} $\ff{F}\ff{G}\Mr{\alpha} id_\cc{D}$ and $\ff{G}\ff{F}\Mr{\beta} id_\cc{C}$. $\ff{G}$ is said to be a \emph{quasi-inverse} for $\ff{F}$.
\end{sinnadaitalica}

\begin{sinnadaitalica} {\bf Pseudo-equivalence.} \label{pequivalencia}
A pseudo-functor $\cc{C}\mr{\ff{F}}\cc{D}$ is said to be a \emph{\mbox{pseudo-equivalence} of \mbox{2-categories}} if there exists a pseudo-functor $\cc{D}\mr{\ff{G}} \cc{C}$ and equivalence pseudo-natural \mbox{transformations} $\ff{F}\ff{G}\Mr{\alpha} id_\cc{D}$ and $\ff{G}\ff{F}\Mr{\beta} id_\cc{C}$. $\ff{G}$ is said to be a \emph{pseudo-quasi-inverse} for $\ff{F}$.
\end{sinnadaitalica}

Pseudo-equivalences are sometimes called bi-equivalences in the literature. See for example \cite{L} where \ref{peqsiipf&f} is mentioned. 

Often we have 2-functors that do not have a quasi-inverse but do have a pseudo-quasi-inverse and thus determine a pseudo-equivalence, see \ref{proppseudoeqapro}. 

\begin{proposition} \emph{\cite[1.11]{K1}} \label{eqsiif&f}
A 2-functor $\ff{F}:\cc{C} \mr{} \cc{D}$ is a
\mbox{2-equivalence} of 2-categories if and only if it is 2-fully-faithful and essentially surjective on objects.\cqd
\end{proposition}

\begin{proposition}\label{peqsiipf&f}
A pseudo-functor $\ff{F}:\cc{C} \mr{} \cc{D}$ is a
\mbox{pseudo-equivalence} of \mbox{2-categories} if and only if it is pseudo-fully-faithful and pseudo-essentially surjective on objects.
Moreover, $\ff{F}$ is essentially surjective on objects iff the pseudo-natural transformation $\alpha$ from \ref{pequivalencia} is invertible.
\end{proposition}

\begin{proof}
$\Rightarrow)$ Let $\cc{D}\mr{\ff{G}}\cc{C}$ be a pseudo-quasi-inverse for $\ff{F}$ and $\ff{F}\ff{G}\Mr{\alpha} id_{\cc{D}}$, $\ff{G}\ff{F}\Mr{\beta} id_\cc{C}$ equivalence pseudo-natural transformations as in \ref{pequivalencia}. Note that for each $\ff{C}\in \cc{C}$ and $\ff{D}\in \cc{D}$, $\alpha_{\ff{D}}$ and $\beta_{\ff{C}}$ are equivalences by \ref{equivalencias pointwise}.

Let's check first that $\ff{F}$ is pseudo-essentially surjective on objects: Given $\ff{D}\in \cc{D}$, $\ff{F}\ff{G}\ff{D}$ is equivalent to $\ff{D}$ via $\alpha_{\ff{D}}$.

Let's check now that $\ff{F}$ is pseudo-fully-faithful: To do that, we need to prove that for each $\ff{C},\ \ff{C}'\in \cc{C}$, $\ff{F}:\cc{C}(\ff{C},\ff{C}')\mr{} \cc{D}(\ff{F}\ff{C},\ff{F}\ff{C}')$ is an equivalence of categories. Recall that this is equivalent to prove that this morphisms are essentially surjective on objects and full and faithful in the 1-dimensional sense \cite{ML}. 

So let $\ff{C},\ \ff{C}'\in \cc{C}$. To check that $\ff{F}:\cc{C}(\ff{C},\ff{C}')\mr{} \cc{D}(\ff{F}\ff{C},\ff{F}\ff{C}')$ is full and faithful, we need to prove that $\ff{F}$ induces a bijection between the set of 2-cells of $\cc{C}$ between two fixed morphisms $\ff{C} \mrpair{\ff{f}}{\ff{g}} \ff{C}'$ and the set of 2-cells between $\ff{F}\ff{f}$ and $\ff{F}\ff{g}$. We are going to see first that this induced function is injective, so suppose that we have \mbox{$\ff{C} \cellpairrd{\ff{f}}{\theta}{\eta}{\ff{g}}\ff{C}'\in \cc{C}$} such that $\ff{F}\theta=\ff{F}\eta$. Then $\ff{G}\ff{F}\theta=\ff{G}\ff{F}\eta$ and so, since $\beta$ is pseudo-natural, we have the following equality:

$$ 
\vcenter{\xymatrix@C=-0pc{ \f \dcell{\theta} && \beta_{\ff{C}} \deq \\
\g && \beta_{\ff{C}} }}
\vcenter{\xymatrix@C=-0pc{ \quad = \quad }}
\vcenter{\xymatrix@C=-0pc{ \f \dl & \dc{\beta_{\f}} & \dr \beta_{\ff{C}} \\
\beta_{{\ff{C}}'} \deq && \G\F\f \dcellbymedio{\G\F\theta} \\
\beta_{{\ff{C}}'} \dl & \dc{\beta_{\g}\inv} & \dr \G\F\g \\
\g && \beta_{\ff{C}} }}
\vcenter{\xymatrix@C=-0pc{ \quad = \quad }}
\vcenter{\xymatrix@C=-0pc{ \f \dl & \dc{\beta_{\f}} & \dr \beta_{\ff{C}} \\
\beta_{{\ff{C}}'} \deq && \G\F\f \dcellbymedio{\G\F\eta} \\
\beta_{{\ff{C}}'} \dl & \dc{\beta_{\g}\inv} & \dr \G\F\g \\
\g && \beta_{\ff{C}} }}
\vcenter{\xymatrix@C=-0pc{ \quad = \quad }}
\vcenter{\xymatrix@C=-0pc{ \f \dcell{\eta} && \beta_{\ff{C}} \deq \\
\g && \beta_{\ff{C}} }}$$

Then, since $\beta_{\ff{C}}$ is an equivalence, we have that $\theta=\eta$.

In the same way, one can prove that the corresponding function induced by $\ff{G}$ is also injective. This is going to be useful to prove surjectivity. So, let $\ff{F}\ff{C}\cellrd{\ff{F}\ff{f}}{\rho}{\ff{F}\ff{g}} \ff{F}\ff{C}'$. Consider the following 2-cell $\mu$:

$$ 
\mu = \; \vcenter{\xymatrix@C=-0pc{  && \f \opdosuno{=} \\
\f \deq &&& id_{\C} \op{\cong} \\
\f \dl & \dc{\beta_{\f}} & \dr \beta_{\ff{C}} && \overline{\beta_{\ff{C}}} \deq \\
\beta_{{\ff{C}}'} \deq && \G\F\f \dcell{\G\rho} && \overline{\beta_{\ff{C}}} \deq \\
\beta_{{\ff{C}}'} \dl & \dc{\beta_{\g}\inv} & \dr \G\F\g && \overline{\beta_{\ff{C}}} \deq \\
\g \deq && \beta_{\ff{C}} && \overline{\beta_{\ff{C}}} \\
\g  &&& id_{\C} \cl{\cong} \\
&& \g \cldosuno{=} }}$$

\noindent where $\overline{\beta_{\ff{C}}}$ denotes a quasi-inverse for $\beta_{\C}$.

Then, since $\beta$ is pseudo-natural, we have the following equality:

$$ 
\vcenter{\xymatrix@C=-0pc{ \beta_{{\ff{C}}'} \deq && \G\F\f \dcellbymedio{\G\F\mu} \\
\beta_{{\ff{C}}'} && \G\F\g }}
\vcenter{\xymatrix@C=-0pc{ \quad = \quad }}
\vcenter{\xymatrix@C=-0pc{ \beta_{{\ff{C}}'} \dl & \dc{\beta_{\f}\inv} & \dr \G\F\f \\
\f \dcell{\mu} && \beta_{\ff{C}} \deq \\
\g \dl & \dc{\beta_{\g}} & \dr \beta_{\ff{C}} \\
\beta_{{\ff{C}}'} && \G\F\g }}
\vcenter{\xymatrix@C=-0pc{ \quad = \quad }}
\vcenter{\xymatrix@C=-0pc{ \beta_{{\ff{C}}'} \deq && \G\F\f \dcellb{\G\rho} \\
\beta_{{\ff{C}}'} && \G\F\g}}$$

And so, since $\beta_{\ff{C}'}$ is an equivalence, $\ff{G}\ff{F}\mu=\ff{G}\rho$. This implies that $\ff{F}\mu=\rho$ because of the injectivity of $\ff{G}$ that we have mentioned before.

Finally, to check that $\ff{F}:\cc{C}(\ff{C},\ff{C}')\mr{} \cc{D}(\ff{F}\ff{C},\ff{F}\ff{C}')$ is essentially surjective on objects, let $\ff{F}\ff{C}\mr{\ff{f}}\ff{F}\ff{C}'\in \cc{D}$. Consider $\ff{g}=\beta_{\ff{C}'}\ff{G}\ff{f}\overline{\beta_{\ff{C}}}:\ff{C}\mr{}\ff{C}' \in \cc{C}$. Then, since $\beta$ is pseudo-natural, we have an invertible 2-cell $\ff{G}\ff{F}\ff{g}\Mr{}\overline{\beta_{\ff{C}'}}\ff{g}\beta_{\ff{C}} \Mr{} \ff{G}\ff{f}$. This, plus the fact that $\ff{G}:\cc{D}(\ff{F}\ff{C},\ff{F}\ff{C}')\mr{} \cc{C}(\ff{G}\ff{F}\ff{C},\ff{G}\ff{F}\ff{C}')$ is full and faithful (this can be seen as we saw the equivalent assertion for $\ff{F}$), yields that there is an invertible 2-cell $\ff{F}\ff{g}\Mr{}\ff{f}$ which concludes the proof.

$\Leftarrow)$ Given $\ff{D}\in \cc{D}$, since $\ff{F}$ is pseudo-essentially surjective on objects, there exist $\ff{G}\ff{D}\in \cc{D}$ and an equivalence $\ff{F}\ff{G}\ff{D}\Mr{\alpha_{\ff{D}}}\ff{D}\in \cc{D}$.

Given $\ff{D}\mr{\ff{f}}\ff{D}'\in \cc{D}$, consider $\ff{F}\ff{G}\ff{D}\mr{\alpha_{\ff{D}}} \ff{D} \mr{\ff{f}} \ff{D}' \mr{\overline{\alpha_{\ff{D}'}}} \ff{F}\ff{G}\ff{D}'$. Then, since $\ff{F}$ is pseudo-fully-faithful, there exist $\ff{G}\ff{D} \mr{\ff{G}\ff{f}} \ff{G}\ff{D}'\in \cc{C}$ and an invertible 2-cell \mbox{$\ff{F}\ff{G}\ff{f}\Mr{\tilde{\alpha_{\ff{f}}}} \overline{\alpha_{\ff{D}'}} \ff{f} \alpha_{\ff{D}}$}. 

Given $\ff{D} \cellrd{\ff{f}}{\theta}{\ff{g}} \ff{D}' \in \cc{D}$, consider $ 
\vcenter{\xymatrix@C=-0pc@R=1pc{ && \F\G\f  \opdosdos{\tilde{\alpha}_{\f}} \\
\overline{\alpha_{\ff{D}'}} \deq && \f \dcell{\theta} && \alpha_{\ff{D}} \deq \\
\overline{\alpha_{\ff{D}'}} && \g && \alpha_{\ff{D}} \\
&& \cldosdos{\tilde{\alpha}_{\g}\inv} \F\G\g }}
$. Then, since $\ff{F}$ is pseudo-fully-faithful, there exists a unique 2-cell $\ff{G}\ff{D} \cellrd{\ff{G}\ff{f}}{\ff{G}\theta}{\ff{G}\ff{g}} \ff{G}\ff{D}' \in \cc{C}$ such that $$ 
\vcenter{\xymatrix@C=-0pc{ \F\G\f \dcellbymedio{\F\G\theta} \\ \F\G\g  }}
\vcenter{\xymatrix@C=-0pc{ \quad = \quad }}
\vcenter{\xymatrix@C=-0pc{ && \F\G\f  \opdosdos{\tilde{\alpha}_{\f}} \\
\overline{\alpha_{\ff{D}'}} \deq && \f \dcell{\theta} && \alpha_{\ff{D}} \deq \\
\overline{\alpha_{\ff{D}'}} && \g && \alpha_{\ff{D}} \\
&& \cldosdos{\tilde{\alpha}_{\g}\inv} \F\G\g  }}$$

To construct $id_{\ff{G}\ff{D}} \mr{\alpha^{\G}_{\ff{D}}} \ff{G}id_{\ff{D}}$, consider the following invertible 2-cell $\ff{F}\ff{G}\ff{D}\cellrd{\ff{F}\ff{G}id_{\ff{D}}}{\mu}{\ff{F}id_{\ff{G}\ff{D}}} \ff{F}\ff{G}\ff{D}$:

$$\mu = \;
\vcenter{\xymatrix@C=-0pc{ && \F\G id_{\D} \opdosdos{\tilde{\alpha}_{id_{\D}}} \\
\overline{\alpha_{\ff{D}}} && id_{\D} && \alpha_{\D} \\
&& id_{\F\G\D} \cldosdos{=} \dcellbymedio{\alpha^{\F}_{\G\D}} \\
&& \F id_{\G\D} }}$$

Then, since $\ff{F}$ is pseudo-fully-faithful, there exists a unique invertible 2-cell $\ff{G}\ff{D} \cellrd{\ff{G}id_{\ff{D}}}{\widetilde{\ff{G}_{\ff{D}}}}{id_{\ff{G}\ff{D}}} \ff{G}\ff{D}$ such that $\ff{F}\widetilde{\ff{G}_{\ff{D}}}=\mu$. Take $\alpha^{\G}_{\ff{D}}=\widetilde{\ff{G}_{\ff{D}}}^{-1}$.

Given $\ff{D} \mr{\ff{f}} \ff{D}' \mr{\ff{g}} \ff{D}''\in \cc{D}$, consider the following invertible 2-cell $\ff{F}\ff{G}\ff{D}\cellrd{\ff{F}(\ff{G}\ff{g}\ff{G}\ff{f})}{\eta}{\ff{F}\ff{G}\ff{g}\ff{f}} \ff{F}\ff{G}\ff{D}''$:

$$ 
\eta = \; \vcenter{\xymatrix@C=-0pc{ && && & \F(\G \g \G \f) \optrestres{\alpha^{\F}_{\G \f, \G \g}} \\
&& \F\G \g \opdosdos{\tilde{\alpha}_{\g}} && && && \F\G \f \opdosdos{\tilde{\alpha}_{\g}} \\
\overline{\alpha_{\ff{D}''}} \ardr && \g && \alpha_{\D'} & \dc{\cong} & \overline{\alpha_{\ff{D}'}} && \f && \alpha_{\D} \ardl \\
& \overline{\alpha_{\ff{D}''}} \ardrrrr &&& \g & \dc{\tilde{\alpha}_{\g\f}\inv} & \f &&& \alpha_{\D} \ardllll \\
&&&&& \F\G\g\f}}$$

Then, since $\ff{F}$ is pseudo-fully-faithful, there exists a unique invertible 2-cell $\ff{G}\ff{D} \cellrdb{\ff{G}\ff{g}}{\alpha^{\G}_{\ff{f},\ff{g}}}{\ff{G}\ff{g}\ff{f}} \ff{G}\ff{D}''$ such that $\ff{F}\alpha^{\G}_{\ff{f},\ff{g}}=\eta$. 

It can be checked that $\ff{G}$ defined by this data is a pseudo-functor.

Define $\alpha_{\ff{f}}= 
\vcenter{\xymatrix@C=-0.3pc@R=1pc{ && \f \opunodos{= \quad} && && \alpha_{\D} \deq \\
& id_{\D'} \op{\cong} &&& \f \deq && \alpha_{\D} \deq \\
\alpha_{\D'} \deq && \overline{\alpha_{\ff{D}'}} && \f && \alpha_{\D} \\
\alpha_{\D'} && && \F\G\f \cldosdos{\tilde{\alpha}_{\f}\inv} }}
$. It can be checked that $\alpha$ is a pseudo-natural transformation.

It only remains to define $\beta$: For $\ff{C}\in \cc{C}$, consider the equivalence $\ff{F}\ff{G}\ff{F}\ff{C}\mr{\alpha_{\ff{F}\ff{C}}}\ff{F}\ff{C}$. Then, since $\ff{F}$ is pseudo-fully-faithful, there exist an equivalence $\ff{G}\ff{F}\ff{C}\mr{\beta_{\ff{C}}} \ff{C} \in \cc{C}$ and an invertible 2-cell $\ff{F}\beta_{\ff{C}}\Mr{\gamma_{\ff{C}}} \alpha_{\ff{F}\ff{C}}$. For $\ff{C}\mr{\ff{h}}\ff{C}'\in \cc{C}$, consider the following 2-cell $\ff{F}\ff{G}\ff{F}\ff{C}\cellrd{\ff{F}(\ff{h}\beta_{\ff{C}})}{\rho}{\ff{F}(\beta_{\ff{C}'}\ff{G}\ff{F}\ff{h})} \ff{F}\ff{C}'$

$$\rho \; =  \quad
\vcenter{\xymatrix@C=-0pc{ &&&& \F(\h \beta_{\C}) \opdosdos{\alpha^{\F}_{\beta_{\C},\h}{}\inv} \\
&& \F\h \deq &&&& \F\beta_{\C} \dcell{\gamma_{\C}} \\
&& \F\h \opunodos{= \quad} &&&& \alpha_{\F \C} \deq \\
& id_{\F\C'} \op{\cong} &&& \F\h \deq && \alpha_{\F \C} \deq \\
\alpha_{\F \C'} \deq && \overline{\alpha_{\F \C'}} && \F\h && \alpha_{\F \C} \\
\alpha_{\F \C'} \dcellbymedio{\gamma_{\C'}\inv} && && \F\G\F\h \deq \cldosdos{\tilde{\alpha}_{\h}{}\inv} \\
\F \beta_{\C'} &&&& \F\G\F\h \\
&& \F(\beta_{\C'}\G\F\h) \cldosdos{\alpha_{\G\F\h,\beta_{\C'}}^{\F}} }}
$$

Then, since $\ff{F}$ is pseudo-fully-faithful, there exists a unique invertible 2-cell \mbox{$\ff{h}\beta_{\ff{C}} \Mr{\beta_{\ff{h}}} \beta_{\ff{C}'}\ff{G}\ff{F}\ff{h}$} such that $\ff{F}\beta_{\ff{h}}=\rho$.

It can be checked that $\beta$ is a pseudo-natural transformation.

The remaining assertion follows immediately from the proof.
\end{proof}

\begin{remark}
 The previous proof can be easily adapted to the case of \ref{eqsiif&f}. \cqd 
\end{remark}

\vspace{2ex}

\begin{remark}[] \cite[I,4.2.]{GRAY} \label{evaluation}
Evaluation determines a \emph{quasifunctor}
\mbox{$\cc{H}om_p(\cc{C},\cc{D}) \times \cc{C} \mr{ev} \cc{D}$} (in the sense of \cite[I,4.1.]{GRAY}), in particular, fixing a variable, it is a 2-functor in the other).
 In the strict case $\cc{H}om$, evaluation is actually a \mbox{2-bifunctor.} In the case of $p\cc{H}om_p(\cc{C},\cc{D})$, it is a pseudo-functor in each variable. \cqd
\end{remark}

\begin{remark}[]\cite[I,4.2]{GRAY} \label{Homisbifunctor}
Given 2-functors
$\cc{C}' \mr{\ff{H}_0} \cc{C}$ and $\cc{D} \mr{\ff{H}_1} \cc{D}'$, and $\ff{F}\cellrd{\theta}{\rho}{\eta}\ff{G}$  in $\cc{H}om_\epsilon(\cc{C},\cc{D})(\ff{F},\ff{G})$, the definition
$\cc{H}om_\epsilon(\ff{H}_0, \ff{H}_1)(\ff{F}\cellrd{\theta}{\rho}{\eta}\ff{G}) = \ff{H}_1\ff{F}\ff{H}_0\cellrd{\ff{H}_1\theta\ff{H}_0}{\ff{H}_1\rho\ff{H}_0}{\ff{H}_1\eta\ff{H}_0}\ff{H}_1\ff{G}\ff{H}_0$ determines a functor
\mbox{$\cc{H}om_\epsilon(\cc{C}, \cc{D})(\ff{F}, \ff{G}) \mr{}
\cc{H}om_\epsilon(\cc{C}', \cc{D}')(\ff{H}_1\ff{F}\ff{H}_0, \ff{H}_1\ff{G}\ff{H}_0)$}, and this assignation is bifunctorial in the variable $(\cc{C}, \cc{D})$ (here $\cc{H}om_\epsilon$ denotes either $\cc{H}om$ or $\cc{H}om_p$). Both constructions $\cc{H}om$ and $\cc{H}om_p$ \mbox{determine} a bifunctor
\mbox{$2\hbox{-}\cc{CAT}^{op} \times 2\hbox{-}\cc{CAT} \mr{} 2\hbox{-}\cc{CAT}$.} The same assertion holds for pseudo-functors (see \cite[I,4.20]{GRAY}).
\end{remark}

If $\cc{C}$ and $\cc{D}$ are 2-categories, the product 2-category
$\cc{C} \times \cc{D}$ is constructed in the usual way, and this together with the
\mbox{2-category } $\cc{H}om(\cc{C},\cc{D})$ determine a
symmetric cartesian closed structure as follows (see \cite[chapter 2]{K1} or \cite[I,2.3.]{GRAY}):

\begin{proposition}  \label{cartesianclosed}  The usual definitions determine an isomorphism of
\mbox{2-categories} :

$$
\cc{H}om(\cc{C},\,\cc{H}om(\cc{D},\,\cc{A})) \mr{\cong}
\cc{H}om(\cc{C} \times \cc{D},\,\cc{A}).
$$
Composing with the symmetry
$\cc{C} \times \cc{D} \mr{\cong} \cc{D} \times \cc{C}$ yields an isomorphism:
$$
\cc{H}om(\cc{C},\,\cc{H}om(\cc{D},\,\cc{A})) \mr{\cong}
\cc{H}om(\cc{D},\,\cc{H}om(\cc{C},\,\cc{A})).
$$

\vspace{-4ex}

\cqd
\end{proposition}



\begin{sinnadastandard} {\bf Notation.}
Let $\cc{C}$ be a 2-category, $\ff{C}\in \cc{C}$ and $\ff{D}\scellrd{\ff{f}}{\alpha}{\ff{g}}\ff{E}\in \cc{C}$.

\noindent \begin{itemize}
  \item[-] $\ff{f}_*$:  $\;\cc{C}(\ff{C},\ff{D})\mr{\ff{f}_*} \cc{C}(\ff{C},\ff{E})$, $\ff{f}_*(\ff{h} \mr{\beta} \ff{h}') = (\ff{f}\ff{h} \mr{\ff{f}\beta} \ff{f}\ff{h}')$.
						
  \item[-] $\ff{f}^*$: $\;\cc{C}(\ff{E},\ff{C})\mr{\ff{f}^*} \cc{C}(\ff{D},\ff{C})$, $\ff{f}^*(\ff{h} \mr{\beta} \ff{h}') = (\ff{h}\ff{f} \mr{\beta \ff{f}}\ff{h}'\ff{f})$.

  \item[-] $\alpha_*$: $\;\ff{f}_*\Mr{\alpha_*} \ff{g}_*$,
                   $(\alpha_*)_\ff{h}=\alpha \ff{h}$.

  \item[-] $\alpha^*$: $\;\ff{f}^*\Mr{\alpha^*} \ff{g}^*$
                   $(\alpha^*)_\ff{h}=\ff{h}\alpha $.

  \item[-] $\cc{C} \mr{\cc{C}(\ff{C},-)} \cc{C}at$:
$\cc{C}(\ff{C},-)(\ff{D}\scellrd{\ff{f}}{\alpha}{\ff{g}}\ff{E})  \;=\;  (\cc{C}(\ff{C},\ff{D})\scellrd{\ff{f}_*}{\alpha_*}{\ff{g}_*}\cc{C}(\ff{C},\ff{E}))$.

  \item[-] $\cc{C}^{op} \mr{\cc{C}(-,\ff{C})} \cc{C}at$:
$\cc{C}(-,\ff{C})(\ff{D}\scellrd{\ff{f}}{\alpha}{\ff{g}}\ff{E})  \;=\;  (\cc{C}(\ff{D},\ff{C})\scellrd{\ff{f}^*}{\alpha^*}{\ff{g}^*}\cc{C}(\ff{E},\ff{C}))$.
						
   \item[-] We will also denote by $\ff{f}^*$ the 2-natural transformation from $\cc{C}(\ff{E},-)$ to $\cc{C}(\ff{D},-)$ defined by $(\ff{f}^*)_\ff{C}=\ff{f}^*$.

    \item[-] We will also denote by $\ff{f}_*$ the 2-natural transformation from $\cc{C}(-,\ff{D})$ to $\cc{C}(-,\ff{E})$ defined by $(\ff{f}_*)_\ff{C}=\ff{f}_*$.

    \item[-] We will also denote by $\alpha^*$ the modification from $\ff{f}^*$ to $\ff{g}^*$ defined by $(\alpha^*)_\ff{C}=\alpha^*$.

    \item[-] We will also denote by $\alpha_*$ the modification from $\ff{f}_*$ to $\ff{g}_*$ defined by $(\alpha_*)_\ff{C}=\alpha_*$.
 \cqd 
       \end{itemize}
\end{sinnadastandard} 
 
\begin{sinnadaitalica}{\bf Yoneda 2-functors.}  \label{Y2functor}
Given a locally small 2-category $\cc{C}$, the \emph{Yoneda \mbox{2-functors}} are the following (note that each one is the other for the dual 2-category):

\vspace{1ex}

a. $\cc{C} \mr{\ff{y}^{(-)}} \cc{H}om(\cc{C},\cc{C}at)^{op}$,
$\ff{y}^{\ff{C}} = \cc{C}(\ff{C},-)$,
$\ff{y}^{\ff{f}} = \ff{f}^*$
$\ff{y}^{\alpha} = \alpha^*$.

b. $\cc{C} \mr{\ff{y}_{(-)}} \cc{H}om(\cc{C}^{op},\cc{C}at)$,
$\ff{y}_{\ff{C}} = \cc{C}(-,\ff{C})$,
$\ff{y}_{\ff{f}} = \ff{f}_*$
$\ff{y}_{\alpha} = \alpha_*$.
\end{sinnadaitalica}

 Recall the Yoneda Lemma for enriched categories over $\cc{C}at$. We consider explicitly only the case $a.$ of
\ref{Y2functor}.

\begin{proposition}[\textbf{Yoneda lemma}]\label{2Yoneda}

Given a locally small \mbox{2-category} $\cc{C}$, a 2-functor
$\ff{F}:\cc{C} \mr{} \cc{C}at$ and an object $\ff{C}\in \cc{C}$, there is an isomorphism of categories, natural in $\ff{F}$.
%
%

$$\xymatrix@R=.5pc{\cc{H}om(\cc{C},\cc{C}at)(\cc{C}(\ff{C},-),\ff{F}) & \ar[rr]^h && & \ff{F}\ff{C} \\
		      \theta \mr{\rho} \eta & \ar@{|->}[rr] &&& \theta_\ff{C}(id_\ff{C}) \mr{(\rho_\ff{C})_{id_\ff{C}}} \eta_\ff{C}(id_\ff{C})
}$$
\end{proposition}

\begin{proof}
The application $h$ has an inverse

%
%

$$\xymatrix@R=.5pc{\ff{F}\ff{C} & \ar[rr]^{\ell} &&& \cc{H}om(\cc{C},\cc{C}at)(\cc{C}(\ff{C},-),\ff{F}) \\
C \mr{f}  D & \ar@{|->}[rr] &&& \ell C \mr{\ell f} \ell D }$$

\noindent where $(\ell C)_\ff{D}(\ff{f}\Mr{\alpha}\ff{g})=\ff{F}\ff{f}(C)\mr{(\ff{F}\alpha)_C} \ff{F}\ff{g}(C)$ and $((\ell f)_\ff{D})_\ff{f}=\ff{F}\ff{f}(f)$.
\end{proof}
\begin{corollary} \label{yonedaff}
The Yoneda 2-functors in \ref{Y2functor} are 2-fully-faithful. \cqd
\end{corollary}

Beyond the theory of $\cc{C}at$-enriched categories, the lemma also holds for pseudo-functors and pseudo-natural transformations in the following way:

\begin{proposition}[\textbf{Pseudo-Yoneda lemma}]\label{pseudoYoneda}

Given a locally small \mbox{2-category} $\cc{C}$, a pseudo-functor (in particular, a 2-functor)
$\ff{F}:\cc{C} \mr{} \cc{C}at$ and an object $\ff{C}\in \cc{C}$, there is an  equivalence of categories, natural in $\ff{F}$.
%
%

$$\xymatrix@R=.5pc{p\cc{H}om_p(\cc{C},\cc{C}at)(\cc{C}(\ff{C},-),\ff{F}) & \ar[rr]^{\tilde{h}} && & \ff{F}\ff{C} \\
		      \theta \mr{\rho} \eta & \ar@{|->}[rr] &&& \theta_\ff{C}(id_\ff{C}) \mr{(\rho_\ff{C})_{id_\ff{C}}} \eta_\ff{C}(id_\ff{C})
}$$

\noindent Furthermore, the quasi-inverse $\tilde{\ell}$ is a section of $\tilde{h}$, $\tilde{h}\, \tilde{\ell} = id$.

\end{proposition}
\begin{proof}
$\tilde{h}$ and $\tilde{\ell}$ are defined as in \ref{2Yoneda}, but now $\tilde{\ell}$ is only a section quasi-inverse of $\tilde{h}$.
The details can be checked by the reader. One can find a guide in \cite{nlab} for the case of lax functors and bi-categories. We refer to the arguing and the notation there: In our case, the unit $\eta$ is an isomorphism because $\ff{F}$ is a pseudo-functor, and the counit $\epsilon$ is an isomorphism because $\alpha$ is pseudo-natural and the unitor $r$ is the equality.
\end{proof}

%

\begin{corollary} \label{repflexible}
For any locally small 2-category $\cc{C}$, and $\ff{C} \in \cc{C}$,  the inclusion
$\cc{H}om(\cc{C},\cc{C}at)(\cc{C}(\ff{C},-),\ff{F}) \mr{i} \cc{H}om_p(\cc{C},\cc{C}at)(\cc{C}(\ff{C},-),\ff{F})$ has a retraction
$\alpha$, natural in $\ff{F}$, $\alpha \, i = id$,
$i \, \alpha \cong id$, which determines an equivalence of categories.
\end{corollary}
\begin{proof}
Note that $i = \tilde{\ell} \, h$, then define $\alpha = \ell \, \tilde{h}$.
\end{proof}

\begin{corollary}
The Yoneda 2-functors in \ref{Y2functor} can be considered as \mbox{2-functors} landing in the $\cc{H}om_p$ 2-functor 2-categories. In this case, they are pseudo-fully-faithful.
\cqd
\end{corollary}
\subsection{Weak limits and colimits}\label{weak limits and colimits}

\vspace{1ex}

By \emph{weak} we understand any of the several ways universal properties can be relaxed in 2-categories. Note that pseudo-limits and pseudo-colimits (already considered in \cite{G3}) require isomorphisms, and have many advantages over bi-limits and bi-colimits, which only require equivalences. Their universal properties are both stronger and more convenient to use. On the other hand, in many situations bi-limits and bi-colimits are unavoidable and seems to be the right concept to consider.
The defining universal properties characterize bi-limits up to equivalence and pseudo-limits up to isomorphism.

\vspace{1ex}

\begin{sinnadastandard}{\bf Notation.} We consider pseudo-limits $\Lim{i\in \cc{I}}{\ff{F}i}$, and bi-limits
$\biLim{i\in \cc{I}}{\ff{F}i}$, of contravariant pseudo-functors, and their dual concepts, pseudo-colimits $\coLim{i\in \cc{I}}{\ff{F}i}$, and bi-colimits $\bicoLim{i\in \cc{I}}{\ff{F}i}$, of covariant \mbox{pseudo-functors.}
\end{sinnadastandard}

\begin{sinnadaitalica}{\bf Pseudo-cone.} \label{pseudo-cone}
Let $\ff{F}:\cc{I}^{op} \mr{} \cc{A}$ be a pseudo-functor and $\ff{A}$ an object of $\cc{A}$.
A \emph{pseudo-cone} for $\ff{F}$ with vertex $\ff{A}$ is a pseudo-natural transformation from the 2-functor which is constant at $\ff{A}$ to $\ff{F}$, i.e. it consists in a family of morphisms of $\cc{A}$, $\left\{\ff{A} \mr{\theta_i}\ff{F}i \right\}_{i\in \cc{I}}$
and a family of invertible 2-cells of $\cc{A}$,
\mbox{$\left\{\ff{F}u \theta_{j} \Mr{\theta_u} \theta_i \right\}_{i\mr{u} j \in \cc{I}}$}
  satisfying the following equations:

\begin{itemize}
 \item[PC0.] For each $i \in \cc{I}$,
$\; \theta_{id_{i}}\circ \alpha^{\F}_i \theta_i = id_{\theta_i}$, $\;$ i.e.

$$\vcenter{\xymatrix@C=-0pc{ id_{\F i} \dcell{\alpha^{\F}_i} && \theta_i \deq \\
\F id_i && \theta_i \\
& \theta_i \cl{\theta_{id_i}} }}
\vcenter{\xymatrix@C=-0pc{ \quad = \quad }}
\vcenter{\xymatrix@C=-0pc{ id_{\F i} && \theta_i \\ & \theta_i \cl{=} }}$$

\item[PC1.] For each $i \mr{u} j \mr{v} k \in \cc{I}$, $\; \theta_u \circ \ff{F}u\theta_v= \theta_{vu} \circ \alpha^{\F}_{u,v} \theta_k$, $\;$ i.e.



$$\vcenter{\xymatrix@C=-0pc{ \F u \deq && \F v && \theta_k \\
\F u &&& \theta_j \cl{\theta_v} \\
& \theta_i \clunodos{\theta_u} }}
\vcenter{\xymatrix@C=-0pc{ \quad = \quad }}
\vcenter{\xymatrix@C=-0pc{ \F u && \F v && \theta_k \deq \\
& \F (vu) \cl{\alpha^{\F}_{u,v}} &&&  \theta_k \\
&& \theta_i\clunodos{\theta_{vu}} \\
}}$$

%
%
%
	    
\item[PC2.] For each $i\cellrd{u}{\alpha}{v} j \in \cc{I}$, $\; \theta_u=\theta_v \circ \ff{F}\alpha \theta_j$, $\;$ i.e.


$$\vcenter{\xymatrix@C=-0pc{ \F u && \theta_j \\ & \theta_i \cl{\theta_u} }}
\vcenter{\xymatrix@C=-0pc{ \quad = \quad }}
\vcenter{\xymatrix@C=-0pc{ \F u \dcell{\F \alpha} && \theta_j \deq \\
\F v && \theta_j \\
& \theta_i \cl{\theta_v} }}$$

%
%
\end{itemize}
	  
A \emph{morphism of pseudo-cones} between $\theta$ and $\eta$ with the same vertex is a modification, i.e. a family of 2-cells of $\cc{A}$, $\left\{\theta_{i}\Mr{\rho_{i}} \eta_{i}\right\}_{i\in \cc{I}}$ satisfying the following equation:

\vspace{1ex}

\begin{itemize}
 \item[PCM.] For each $i\stackrel{u}{\rightarrow} j \in \cc{I}$, $\; \rho_i \circ \theta_u=\eta_u \circ \ff{F}u \rho_j$, $\;$ i.e.

$$\vcenter{\xymatrix@C=-0pc{\F u && \theta_j \\ & \theta_i \cl{\theta_u} \dcell{\rho_i} \\ & \eta_i }}
\vcenter{\xymatrix@C=-0pc{ \quad = \quad }}
\vcenter{\xymatrix@C=-0pc{ \F u \deq && \theta_j \dcell{\rho_j} \\ \F u && \eta_j \\ & \eta_i \cl{\eta_u}  }}$$
\end{itemize}

%
%
%
\vspace{1ex}

Pseudo-cones form a category
$\ff{PC}_\cc{A}(\ff{A},\ff{F}) = p\cc{H}om_p(\cc{I}^{op},\cc{A})(\ff{A}, \ff{F})$ furnished with a pseudo-cone $\ff{PC}_\cc{A}(\ff{A},\ff{F}) \mr{} \cc{A}(\ff{A},\, \ff{F}i)$,
 for the pseudo-functor
\mbox{$\cc{I}^{op} \mr{\cc{A}(\ff{A},\, \ff{F}(-))} \cc{CAT}$.}
\end{sinnadaitalica}

As a particular case, we have the notion of pseudo-cone over a 2-functor.

\begin{remark} \label{PCbifunctor} 
Since $p\cc{H}om_p(\cc{I}^{op},\cc{A})$ is a 2-category, it follows:

\begin{itemize}
 \item[a.] Pseudo-cones determine a 2-bifunctor
$(p\cc{H}om_p(\cc{I}^{op},\cc{A}) \times \cc{A})^{op}
\mr{\ff{PC}_\cc{A}} \cc{CAT}$.
\end{itemize}

From Remark \ref{Homisbifunctor}  it follows in particular:

\begin{itemize}
 \item[b.] A pseudo-functor $\cc{A} \mr{\ff{H}} \cc{B}$ induces a functor between the categories of pseudo-cones
$\ff{PC}_\cc{A}(\ff{F},\ff{A}) \mr{\ff{PC}_\ff{H}} \ff{PC}_\cc{B}(\ff{HF},\ff{HA})$.
\cqd
\end{itemize}
\end{remark}

\begin{sinnadaitalica}{\bf Pseudo-limit and bi-limit.} \label{colimits} $ $
The \emph{pseudo-limit} in $\cc{A}$ of the pseudo-functor \mbox{$\ff{F}:\cc{I}^{op}\mr{}\cc{A}$} is the universal pseudo-cone, denoted \mbox{$\Bigg\{\Lim{i\in \cc{I}}{\ff{F}i}\mr{\pi_i} \ff{F}i\Bigg\}_{i\in \cc{I}}$,} \mbox{$\Bigg\{
\vcenter{\xymatrix@C=-0pc@R=1pc{ \F u && \pi_j \\ & \pi_i \cl{\pi_u} }}
\Bigg\}_{i \mr{u} j \in \cc{I}}$,} in the sense that for each \mbox{$\ff{A}\in \cc{A}$,} post-composition with the $\pi_i$'s is an isomorphism of categories
\begin{equation}\label{isoplim}
\; \cc{A}(\ff{A},\Lim{i\in \cc{I}}{\ff{F}i}) \mr{\pi_*} \ff{PC}_\cc{A}(\ff{A},\ff{F}) 
\end{equation}

Equivalently, there is an isomorphism of categories
\mbox{$\cc{A}(\ff{A},\Lim{i\in \cc{I}}{\ff{F}i}) \mr{\cong} \Lim{i\in \cc{I}}{\cc{A}(\ff{A},\ff{F}i)}$} commuting with pseudo-cones. Remark that there is also an isomorphism of categories
\mbox{$\ff{PC}_\cc{A}(\ff{A},\ff{F}) \mr{\cong} \Lim{i\in \cc{I}}{\cc{A}(\ff{A},\ff{F}i)}$} (note that these isomorphisms are 2-natural in the variable $\A$).

\vspace{1ex}

Requiring $\pi_*$ to be an equivalence (which implies that also the other two isomorphisms above are equivalences) defines the notion of \mbox{\emph{bi-limit}} (note that these equivalences are pseudo-natural in the variable $\A$). Clearly, pseudo-limits are bi-limits.

\vspace{1ex}

We omit the explicit consideration of the dual concepts. \cqd
\end{sinnadaitalica}

As a particular case, we have pseudo-limits and bi-limits (and its dual concepts) of 2-functors.

\begin{remark}\label{isodeplimexplicito}
 As we are going to use the isomorphism \eqref{isoplim} in the following sections (and the equivalence in case of bi-limits), we are going to make the meaning of having them explicit. In the case of pseudo-limits it means that:

\begin{itemize}
 \item[-] Given a pseudo-cone $\left\{\ff{A} \mr{\theta_i} \ff{F}i\right\}_{i\in \cc{I}}$, $\left\{\ff{F}u \theta_j \Mr{\theta_u} \theta_i \right\}_{i\mr{u}j\in \cc{I}}$, there exists a unique morphism $\ff{A} \mr{\ff{f}}\Lim{i\in \cc{I}}{\ff{F}i} \in \cc{A}$ such that $\pi_i \ff{f}=\theta_i \ \forall\  i\in \cc{I}$ and $\pi_u\ff{f}=\theta_u \ \forall \ i\mr{u}j\in \cc{I}$.
 \item[-] And given a morphism of pseudo-cones $\left\{\theta_{i}\Mr{\rho_{i}} \eta_{i}\right\}_{i\in \cc{I}}$, there exists a unique 2-cell $\ff{A} \cellrd{\ff{f}}{\mu}{\ff{g}}\Lim{i\in \cc{I}}{\ff{F}i} \in \cc{A}$ such that $\pi_i \mu=\rho_i \ \forall\  i\in \cc{I}$.
\end{itemize}
 
\noindent In the case of bi-limits it means that:

\begin{itemize}
 \item[-] Given a pseudo-cone $\left\{\ff{A} \mr{\theta_i} \ff{F}i\right\}_{i\in \cc{I}}$, $\left\{\ff{F}u \theta_j \Mr{\theta_u} \theta_i \right\}_{i\mr{u}j\in \cc{I}}$, there exist a morphism $\ff{A} \mr{\ff{f}}\biLim{i\in \cc{I}}{\ff{F}i} \in \cc{A}$ and invertible 2-cells $\Bigg\{\ff{A}\cellrd{\pi_i \ff{f}}{\alpha_i}{\theta_i} \ff{F}i\Bigg\}_{i\in \cc{I}}$ such that $$\vcenter{\xymatrix@C=-0pc
       {
        \ff{F}u 
	& 
	& \pi_j 
	& & \ff{f} \did
	\\
	& \pi_i \cl{\pi_u} 
	& 
	& & \ff{f} 
	\\
	& & \theta_i \clunodos{\alpha_i}
        &
        }}
\vcenter{\xymatrix@C=-.6pc{\quad = \quad}} 
\vcenter{\xymatrix@C=-0pc
       {
        \ff{F}u \did
	&& \pi_j 
	& 
	& \ff{f} 
	\\
	\ff{F}u 
	&&
	& \theta_j \cl{\alpha_j} 
	&
	\\
	&&
        \theta_i \cldosuno{\theta_u}
        &
        }}
        \quad \forall \ i\mr{u}j \in \cc{I}.$$

\item[-]  And given a morphism of pseudo-cones $\left\{\theta_{i}\Mr{\rho_{i}} \eta_{i}\right\}_{i\in \cc{I}}$, there exists a unique 2-cell $\ff{A} \cellrd{\ff{f}}{\mu}{\ff{g}}\biLim{i\in \cc{I}}{\ff{F}i} \in \cc{A}$ such that $\pi_i \mu=\rho_i \ \forall\  i\in \cc{I}$. \qed
\end{itemize}

\end{remark}

%
%
%
%

\vspace{1ex}

It is well known that, in $\cc{C}at$-enriched theory, strict limits and colimits are performed pointwise (if they exists in the codomain category). Here we establish this fact for pseudo-limits and pseudo-colimits in both strict and pseudo 2-functor 2-categories. Abusing notation we can say that the formula
$(\coLim{i\in \cc{I}}{\ff{F}_i)(\ff{C}}) = \coLim{i\in \cc{I}}{\ff{F}_i(\ff{C}})$ holds in both 2-categories. The verification of this is straightforward but requires some care. We also checked that both pseudo-limits and bi-limits (and its dual concepts) are performed pointwise in the 2-category of pseudo-functors.

\begin{proposition}\label{pointwisebi-limit}
Let $\cc{I}\mr{\ff{F}}\cc{A}$, $i\mapsto \ff{F}_i$ be a pseudo-functor where
$\cc{A}$ is either $\cc{H}om(\cc{C},\cc{D})$ or $\cc{H}om_p(\cc{C},\cc{D})$.
For each $\ff{C} \in \cc{C}$ let $\ff{F}_i\ff{C} \mr{\lambda^\ff{C}_i} \ff{L}\ff{C}$ be a pseudo-colimit pseudo-cone in $\cc{D}$ for the pseudo-functor
$\cc{I} \mr{\ff{F}} \cc{A} \mr{ev(-, \ff{C})} \cc{D}$ (where $ev$ is evaluation, see \ref{evaluation}). Then $\ff{LC}$ is 2-functorial in $\ff{C}$ in such a way that $\lambda^\ff{C}_i$ becomes 2-natural and $\ff{F}_i \mr{\lambda_i} \ff{L}$ is a pseudo-colimit pseudo-cone in
$\cc{A}$ in both cases.
By duality the same assertion holds for pseudo-limits.
\end{proposition}


\begin{proof}
Given $\ff{C} \cellrd{\ff{f}}{\alpha}{\ff{g}}  \ff{D}$ in $\cc{C}$, evaluation determines a 2-cell in $p\cc{H}om_p(\cc{I}, \cc{D})$
$
\ff{F}\ff{C} \cellrd{\ff{F}\ff{f}}{\ff{F}\alpha}{\ff{F}\ff{g}}
\ff{F}\ff{D} =
ev(\ff{F}(\,\hbox{-}\,), \ff{C} \cellrd{\ff{f}}{\alpha}{\ff{g}}  \ff{D})$ 
(note that \mbox{$(\ff{FC})_i = \ff{F}_i\ff{C}$}, and similarly for $\ff{f}$, $\ff{g}$ and $\ff{\alpha}$). Then,  for each $\ff{X} \in \cc{D}$, it follows (from
%
%
\mbox{Remark \ref{PCbifunctor} a.)} that precomposing with this 2-cell determines a 2-cell (clearly 2-natural in the variable $\ff{X}$) in the right leg of the diagram below. Since the rows are isomorphisms, there is a unique 2-cell (also natural in the variable $\ff{X}$) in the left leg which makes the diagram commutative.
$$
\xymatrix@C=9.5ex
       {
        \cc{D}(\ff{LD}, \ff{X})
                     \ar[r]^<<<<<<<<<{(\lambda^\ff{D})^*}_<<<<<<<<{\cong}
                     \ar@<-2ex>[d]^{\;\Rightarrow}
                     \ar@<2ex>[d]
      & \ff{PC}_\cc{D}(\ff{F}\ff{D}, \ff{X})
                  \ar@<-2ex>[d]^{\;\Rightarrow}
                  \ar@<2ex>[d]
      \\
      \cc{D}(\ff{LC}, \ff{X})
                     \ar[r]^<<<<<<<<<{(\lambda^\ff{C})^*}_<<<<<<<<{\cong}
      & \ff{PC}_\cc{D}(\ff{F}\ff{C}, \ff{X})
      }
$$
By the Yoneda lemma (\ref{yonedaff}), the left leg is given by precomposing with a unique 2-cell in $\cc{D}$, that we denote
$\ff{L}\ff{C} \cellrd{\ff{L}\ff{f}}{\ff{L}\alpha}{\ff{L}\ff{g}}
                                                   \ff{L}\ff{D}$. It is clear by uniqueness that this determines a 2-functor
$\cc{C} \mr{\ff{L}} \cc{D}$.

Putting $\ff{X} = \ff{L}\ff{D}$ in the upper left corner and tracing the identity down the diagram yields the following commutative diagram of pseudo-cones in $\cc{D}$:
$$
\xymatrix@C=12ex@R=8ex
       {
        \ff{F}_i\ff{C}
               \ar[r]^<<<<<<<<<{\lambda_i^\ff{C}}
               \ar@<-2ex>[d]_{\ff{F}_i{\ff{f}}}
               \ar@<-2ex>[d]^<<<<<<{\ff{F}_i{\ff{\alpha}}}
                               ^<<<<<<<<<{\;\Rightarrow}
               \ar@<2ex>[d]^{\ff{F}_i{\ff{g}}}
      & \ff{L}\ff{C}
                     \ar@<-2ex>[d]_{\ff{L}{\ff{f}}}
               \ar@<-2ex>[d]^<<<<<<{\ff{L}{\ff{\alpha}}}
                               ^<<<<<<<<<{\;\Rightarrow}
                      \ar@<2ex>[d]^{\ff{L}{\ff{g}}}
     \\
       \ff{F}_i\ff{D}
                     \ar[r]^<<<<<<<<<{\lambda_i^\ff{D}}
      & \ff{L}\ff{D}
      }
$$
This shows that $\ff{L}$ is furnished with a pseudo-cone for $\ff{F}$ and that the $\lambda_i$ are \mbox{2-natural.} It only remains to check the universal property:

Let \mbox{$\cc{C} \mr{\ff{G}} \cc{D}$} be a 2-functor, consider the 2-functor $\cc{A} \mr{ev(-, \ff{C})} \cc{D}$. We have the following diagram, where the right leg is given by \mbox{Remark \ref{PCbifunctor} b.:}
$$
\xymatrix@C=8ex
       {
        \cc{A}(\ff{L}, \ff{G})
                     \ar[r]^<<<<<<<<<{\lambda^*}
                     \ar@<0.5ex>[d]_{ev(-, \ff{C})}
      & \ff{PC}_\cc{A}(\ff{F}, \ff{G})
                  \ar@<-0.5ex>[d]^{\ff{PC}_{ev(-, \ff{C})}}
      \\
      \cc{D}(\ff{LC}, \ff{GC})
                     \ar[r]^<<<<<<<{(\lambda^\ff{C})^*}_<<<<<<{\cong}
      & \ff{PC}_\cc{D}(\ff{F}\ff{C}, \ff{GC})
      }
$$
We prove now that the upper row is an isomorphism. Given
$\ff{F}_i \scellrd{\theta_i}{\rho_i}{\eta_i} \ff{G}$ in
$\ff{PC}_\cc{A}(\ff{F}, \ff{G})$, it follows there exists a unique
$\ff{L}\ff{C} \cellrd{\ff{\tilde{\theta}\ff{C}}}
                     {\ff{\tilde{\rho}\ff{C}}}
                     {\ff{\tilde{\eta}\ff{C}}}
                                                   \ff{G}\ff{C}$
in $\cc{D}(\ff{LC}, \ff{GC})$ such that
$\tilde{\rho}\ff{C}\, \lambda_i^\ff{C} = \rho_i\ff{C}$.
%
%
It is necessary to show that this 2-cell actually lives in $\cc{A}$. This has to be checked for any
$\ff{C} \scellrd{\ff{f}}{\alpha}{\ff{g}}  \ff{D}$ in $\cc{C}$. In both cases it can be done considering the isomorphism
$
\xymatrix@C=6.5ex
       {
        \cc{D}(\ff{LC}, \ff{\ff{G}\ff{D}})
                     \ar[r]^<<<<<<<{(\lambda^\ff{C})^*}_<<<<<<{\cong}
        & \ff{PC}_\cc{D}(\ff{F}\ff{C}, \ff{G}\ff{D}).
      }
$
\end{proof}

\begin{remark}\label{pointwise en phomp}
A similar proof gives the result for $p\cc{H}om_p(\cc{C},\cc{D})$. It also can be checked, by changing the arguments just a little bit that bi-limits and bi-colimits are performed pointwise in $p\cc{H}om_p(\cc{C},\cc{D})$. We leave the details to the reader. \cqd
\end{remark}

\begin{definition}\label{bi-tensor}
 Let $\cc{A}$ be a 2-category, $\ff{C} \in \cc{A}$ and $\ff{E} \in \cc{C}at$. We define the bi-tensor $\ff{E} \tilde{\otimes}_{\cc{A}} \ff{C}$ as the object of $\cc{A}$ such that $\forall$ $\ff{D} \in \cc{A}$, there is an equivalence of categories pseudo-natural in $\D$ 
 $$\cc{C}(\ff{E} \tilde{\otimes}_{\cc{A}} \ff{C}, \ff{D}) \simeq \cc{C}at(\ff{E}, \cc{A}(\ff{C},\ff{D})).$$
 If this equivalences are isomorphisms 2-natural in $\D$, we call it pseudo-tensor and we denote it by $\otimes$ instead of $\tilde{\otimes}$. Pseudo-tensors are in fact the tensors of the 2-category seen as a $\cc{C}at$-enriched category.
 
 We omit to make the dual concept explicit.
\end{definition}
%

It follows from the definition and the Yoneda lemmas (\ref{2Yoneda} and \ref{pseudoYoneda}):

\begin{proposition}\label{tensor funtorial}
For each category $\E$:
 \begin{enumerate}
  \item $\E \otimes_\cc{A} (-): \cc{A} \mr{} \cc{A}$ is a 2-functor.
  \item $\E \tilde{\otimes}_\cc{A} (-): \cc{A} \mr{} \cc{A}$ is a pseudo-functor. $\hfill \square$
 \end{enumerate}
\end{proposition}

From \ref{tensor funtorial}, it can be verified the pointwise nature of pseudo-tensors and bi-tensors in the 2-functor and pseudo-functor 2-categories:

\begin{proposition}\label{tensorptoapto}
\comw{A} 
\begin{enumerate}
 \item Let $\cc{A}$ be either $\cc{H}om(\cc{C},\cc{D})$ or $\cc{H}om_p(\cc{C},\cc{D})$, $\F\in \cc{A}$ and $\ff{E} \in \cc{C}at$.  
Then $\E \otimes_\cc{D} \F\X$ is a 2-functor in the variable $\X$ and determines a pseudo-tensor in $\cc{A}$. That is:

$$
 (\E \otimes_{\cc{A}} \F) (\ff{D}) = 
 \E  \otimes_{\cc{D}} \F \ff{D}
$$

\item Let $\cc{A}=p\cc{H}om_p(\cc{C},\cc{D})$, $\F\in \cc{A}$ and $\ff{E} \in \cc{C}at$. Then $\E \tilde{\otimes}_\cc{D} \F\X$ is a pseudo-functor in the variable $\X$ and determines a bi-tensor in $\cc{A}$. That is:

\begin{center}
$ \displaystyle
 (\E \tilde{\otimes}_{\cc{A}} \F) (\ff{D}) = 
\E  \tilde{\otimes}_{\cc{D}} \ff{F} \ff{D}$   
\end{center}
\vspace{-4ex}
\qed
\end{enumerate}
\end{proposition}

\vspace{1ex}

We make now precise what we do consider as \emph{preservation} properties of a \mbox{pseudo-functor}. We do it in the case of pseudo-colimits, bi-colimits, pseudo-tensors and bi-tensors, but the same clearly applies to dual concepts. 

\begin{definition} \label{preservation}
\comw{A}
\begin{enumerate}
 \item Let
$\cc{I} \mr{\ff{X}} \cc{C} \mr{\ff{H}} \cc{A}$ be any pseudo-functors. We say that $\ff{H}$ \emph{preserves} a pseudo-colimit (resp. bi-colimit) pseudo-cone \mbox{$\ff{X}_i \mr{\lambda_i} \ff{L}$} in
$\cc{C}$, if $\ff{H}\ff{X}_i \mr{\ff{H}\lambda_i} \ff{H}\ff{L}$  is a pseudo-colimit (resp. bi-colimit) pseudo-cone in $\cc{A}$. Equivalently, if the (usual) comparison arrow is an isomorphism (resp. an equivalence) in $\cc{A}$.

\item Let $\cc{C} \mr{\ff{H}} \cc{A}$ be any pseudo-functor. We say that $\ff{H}$ preserves a pseudo-tensor (respectively bi-tensor) $\E \otimes_{\cc{C}} \ff{C}$ in $\cc{C}$ (respectively $\E \tilde{\otimes}_{\cc{C}} \ff{C}$) if $\ff{H}(\E \otimes_{\cc{C}} \C)$ (respectively $\ff{H}(\E \tilde{\otimes}_{\cc{C}} \C)$) is the pseudo-tensor (respectively bi-tensor) $\E \otimes_{\cc{A}} \ff{H}\C$ (respectively $\E \tilde{\otimes}_{\cc{A}} \ff{H}\C$).
\end{enumerate}

\end{definition}

Note that by the very definition, 2-representable 2-functors preserve pseudo-limits and bi-limits. Also, from proposition \ref{pointwisebi-limit} it follows:

\begin{proposition} \label{yonedapreserves}
The Yoneda 2-functors in \ref{Y2functor} preserve \mbox{pseudo-limits.}
 \cqd
\end{proposition}

\vspace{1ex}

Recall that small pseudo-limits and pseudo-colimits indexed by a category of locally small categories exist and are locally small, as well that the 2-category $\cc{C}at$ of small categories has all small pseudo-limits and pseudo-colimits \mbox{(see for example \cite{G3}, \cite{BKP}, \cite{K2}).} 

\begin{sinnadastandard} \label{limincat}
We refer to the explicit construction of pseudo-limits of category valued \mbox{2-functors}, which is similar to the construction of pseudo-limits of category-valued functors in \cite[Expos\'e VI 6.]{G3}, see full details \mbox{in \cite{mitesis}.}
\end{sinnadastandard}

It is also key to our work the explicit construction of 2-filtered \mbox{pseudo-colimits}  of category valued 2-functors developed in \cite{DS}.
We recall this now.

Even though Dubuc and Street work with an alternate definition of 2-filtered \mbox{2-category} that is more suitable for their calculations, we are going to use the \mbox{following} equivalent one (see \cite{Data}) in the following sections: 

\begin{sinnadaitalica}\label{2filtered} {\bf 2-filtered.} \emph{\cite{K}}
Let $\cc{C}$ be a non-empty 2-category. $\cc{C}$ is said to be \emph{2-filtered} if the following axioms are satisfied:

\begin{enumerate}
 \item[F0.] Given two objects $\ff{C}$, $\ff{D}\in \cc{C}$, there exist an object $\ff{E} \in \cc{C}$ and arrows $\ff{C}\rightarrow \ff{E}$, $\ff{D}\rightarrow \ff{E}$.

\item[F1.] Given two arrows $\ff{C} \mrpair{\ff{f}}{\ff{g}} \ff{D}$, there exist an arrow $\ff{D}\mr{\ff{h}} \ff{E}$ and an invertible 2-cell $ \ff{C}\cellrd{\ff{h}\ff{f}}{\alpha}{\ff{h}\ff{g}} \ff{E}$.

\item[F2.] Given two 2-cells $\ff{C} \cellpairrd{}{\alpha}{\beta}{} \ff{D}$ there exists an arrow $\ff{D}\mr{\ff{h}} \ff{E}$ such that $\ff{h}\alpha=\ff{h}\beta$.
\end{enumerate}

\noindent The dual notion of 2-cofiltered 2-category is given by the duals of axioms F0, F1 and F2.
\end{sinnadaitalica}

\begin{sinnadastandard} \label{defLF} {\bf Construction LL.} \cite{DS}
Let $\cc{I}$ be a 2-filtered \mbox{2-category} and \mbox{$\ff{F}:\cc{I}\rightarrow \cc{C}at$} a \mbox{2-functor.}
We define a category $\mathcal{L}(\ff{F})$ in two steps as follows:

\vspace{1ex}

\noindent \emph{\bf First step} (\cite[Definition 1.5]{DS}):

\begin{itemize}
 \item[]Objects: $(C,i)$ with $C \in \ff{F}i$.

\item[] Premorphisms: A premorphism between $(C,i)$ and $(D,j)$ is a triple $(u,r,v)$ where $i \mr{u} k$, $j \mr{v} k$ in $\cc{I}$ and $\ff{F}(u)(C)\mr{r} \ff{F}(v)(D)$ in $\ff{F}k$.

\item[] Homotopies: An homotopy between two premorphisms $(u_1,r_1,v_1)$ and $(u_2,r_2,v_2)$ is a quadruple $(w_1,w_2,\alpha,\beta)$ where $k_1\mr{w_1} k$, $k_2\mr{w_2} k$ are \mbox{1-cells} of $\mathcal{I}$ and \mbox{$w_1v_1\mr{\alpha} w_2v_2$,}
$w_1u_1\mr{\beta} w_2u_2$ are invertible 2-cells of $\mathcal{I}$ such that the following diagram commutes in $\ff{F}k$:
$$
\xymatrix@C=7ex
          {
           \ff{F}(w_1)\ff{F}(u_1)(C) = \ff{F}(w_1u_1)(C)
                \ar[r]^{\ff{F}(\beta)_C}
                \ar@<-6ex>[d]_{\ff{F}(w_1)(r_1)}
          &
           \ff{F}(w_2u_2)(C) = \ff{F}(w_2)\ff{F}(u_2)(C)
                \ar@<6ex>[d]^{\ff{F}(w_2)(r_2)}
          \\
           \ff{F}(w_1)\ff{F}(v_1)(D) = \ff{F}(w_1v_1)(D)
                \ar[r]_{\ff{F}(\alpha)_D}
          &
           \ff{F}(w_2v_2)(D) = \ff{F}(w_2)\ff{F}(v_2)(D)
          }
$$
\end{itemize}

We say that two premorphisms $r_1, r_2$ are equivalent if there is an \mbox{homotopy} between them. In that case, we write $r _1 \sim r_2$.

\vspace{1ex}

Equivalence is indeed an equivalence relation, and premorphisms can be (non uniquely) composed. Up to equivalence, composition is independent of the choice of representatives and of the choice of the composition between them. Since associativity holds and
identities exist, the following actually does define a category.

\vspace{1ex}

\noindent \emph{\bf Second step} (\cite[Definition 1.13]{DS}):

\begin{itemize}
 \item[] Objects: $(C, \; i)$ with $C \in Fi$.

 \item[] Morphisms: equivalence classes of premorphisms.

 \item[] Composition: defined by composing representative premorphisms.
\end{itemize}
\end{sinnadastandard}

\begin{proposition}\emph{\cite[Theorem 1.19]{DS}} \label{construccionDS}
Let $\cc{I}$ be a 2-filtered \mbox{2-category,} \mbox{$\ff{F}:\cc{I}\rightarrow \cc{C}at$} a 2-functor, $i \mr{u} j$ in $\cc{I}$ and $C \mr{r} D \in \ff{F}i$.
The \mbox{following} formulas define a pseudo-cone $\ff{F} \Mr{\lambda} \cc{L}(\ff{F})$:
$$
\lambda_i(C) =(C,i)
\hspace{3ex}
\lambda_i(r) =[i,r,i]
\hspace{3ex}
(\lambda_u)_C =[u,\ff{F}u(C),j]
$$
which is a pseudo-colimit for the 2-functor $\ff{F}$. \cqd
\end{proposition}
\subsection{2-cofinal 2-functors}\label{2-cofinal 2-functors}

Propositions \ref{MJ} and \ref{phi} are key to prove reindexing properties for 2-pro-objects in section \ref{Mtrick}. In order to state and prove them, we give the following definition:

\begin{definition}
Let $\F:\cc{I}\rightarrow \cc{J}$ be a pseudo-functor (as a particular case, $\F$ might be a 2-functor) with $\cc{I}$ a 2-filtered 2-category. We say that $\F$ is \emph{2-cofinal} if it has the following properties:

\begin{itemize}

\item[CF0.] Given $j \in \cc{J}$, there exist $i \in \cc{I}$ and a morphism $j \rightarrow \F i \in \cc{J}$.
  
\item[CF1.] Given $j\in \cc{J}$, $i\in \cc{I}$ and $j\mrpair{a}{b}\F i \in \cc{J}$, there exist $i\stackrel{u}\rightarrow i'\ \in \cc{I}$ and an invertible 2-cell $j \cellrd{\F(u)a}{\alpha}{\F(u)b}\F i' \in \cc{J}$. 

\item[CF2.] Given $j\in \cc{J}$, $i\in \cc{I}$ and $j\cellpairrd{a}{\alpha}{\beta}{b} \F i \in \cc{J}$, there exists $i\stackrel{u}\rightarrow i'\ \in \cc{I}$ such that $\F(u)\alpha=\F(u)\beta$. 

\end{itemize}

\end{definition}

\begin{remark}\label{cofinal implica filtrante}
If $\F:\cc{I}\rightarrow \cc{J}$ is a 2-cofinal pseudo-functor, then $\cc{J}$ is also 2-filtered.
\end{remark}

\begin{proof}
The proof is straightforward.
\end{proof}

\begin{proposition}\label{obs1}
Let $\F:\cc{I}\rightarrow \cc{J}$ be a 2-cofinal pseudo-functor. Then, for each $\vcenter{\xymatrix@R=-.5pc{ & \F i\\
                                                                                                         j\ar[ru]^{a} \ar[rd]_{b} & \\
                                                                                                          & \F i'}} \in \cc{J}$, 
there are morphisms $\vcenter{\xymatrix@R=-.5pc{i\ar[rd]^u & \\
                                    & i'' \\
                                    i'\ar[ru]_v & }} \in \cc{I}$ and an invertible 2-cell $$\vcenter{\xymatrix@R=.5pc{& \F i\ar[rd]^{\F u} & \\
                                                                                     j \ar[ru]^{a} \ar[rd]_{b} \ar@{}[rr]|{\Downarrow \; \alpha} &  & \F i''\\
                                                                                     & \F i' \ar[ru]_{\F v} & }}$$
\end{proposition}

\begin{proof}
It is straightforward from F0 and CF1.
\end{proof}

\begin{corollary}\label{corolario2}
Let $\F:\cc{I}\rightarrow \cc{J}$ be a 2-cofinal pseudo-functor. Given $\F i \stackrel{a}\rightarrow \F i' \in \cc{J}$, there are morphisms $\vcenter{\xymatrix@R=-.5pc{i\ar[rd]^u & \\
                                    & i'' \\
                                    i'\ar[ru]_v & }} \in \cc{I}$ and an invertible 2-cell $\vcenter{\xymatrix@R=.5pc{& \F i \ar[rd]^{\F u} & \\
                                                                                     \F i \ar[ru]^{id} \ar[rd]_{a} \ar@{}[rr]|{\Downarrow \; \alpha} &  & \F i''\\
                                                                                     & \F i' \ar[ru]_{\F v} & }}$ (i.e. an invertible 2-cell $\F i\cellrd{\F u}{\alpha}{\F(v)a}\F i'' $).
\cqd
\end{corollary}

\begin{lemma}\label{lemita3}
Let $\F:\cc{I}\rightarrow \cc{J}$ be a 2-cofinal pseudo-functor. Then, given $i\mrpair{u}{v}i' \in \cc{I}$ and an invertible 2-cell $\F i \cellrd{\F u}{\alpha}{\F v}\F i' \in \cc{J}$, there exist $i'\stackrel{w}\rightarrow i'' \in \cc{I}$ and an invertible 2-cell $i\cellrd{wu}{\delta}{wv}i'' \in \cc{I}$ such that $\F(w)\alpha=\F\delta$. 
\end{lemma}

\begin{proof}
It is straightforward from F1 and CF2.
\end{proof}

\vspace{2ex}

The following lemmas are used in the proof of \ref{T_Cat}.

\begin{lemma}\label{lemita1}
Let $\cc{J}$ be a 2-filtered 2-category and $\ff{G}:\cc{J}\rightarrow \cc{C}at$ a 2-functor. Let $(y,j)\stackrel{[a,r,b]}\longrightarrow (y',j')$ be a morphism in $\coLim{j\in \cc{J}}{\G j}$, $a'$, $b'$, $c$ morphisms in $\cc{J}$ and $\alpha$, $\beta$ invertible 2-cells as in the following diagram:

$$\vcenter{\xymatrix@R=.35pc@C=.8pc{j \ar@/^1.5pc/[rrrdd]^{a'} \ar[rdd]^{a} \ar@{}[rrrd]|{\Downarrow \; \alpha} & & & \\ 
                                                               & & & \\
                                                               & j'' \ar[rr]^{c} & & j''' \\
                                                               \ar@{}[rrr]|{\Downarrow \; \beta} & & & \\
                                                               j' \ar@/_1.5pc/[rrruu]_{b'} \ar[ruu]_{b} & & & }}$$
                                                               
Then $[a,r,b]=[a',s,b']$ where $s$ is defined as the composition $\G(\beta)_{y'}\G(c)(r)\G(\alpha)_y$
\end{lemma}
\begin{proof}
It is straightforward from Lemma 1.18 of \cite{DS}.
\end{proof}

\begin{corollary}\label{corolario1}
As a particular case (take $a'=ca$, $b'=cb$ and $\alpha$ and $\beta$ the corresponding identities), we have that given $(a,r,b)$ and $c$ as in the previous proposition, $[a,r,b]=[ca,\G(c)(r),cb]$.   
\cqd
\end{corollary}

\begin{lemma}\label{lemita2}
Let $\cc{J}$ be a 2-filtered 2-category and $\G:\cc{J}\rightarrow \cc{C}at$ a 2-functor. Let $(y,j)\stackrel{[a,r,b]}\longrightarrow (y',j')$  be a morphism in  $\coLim{j\in \cc{J}}{\G j}$ and suppose that we have a configuration as follows:

$$\vcenter{\xymatrix@R=.8pc{
j \ar@/^1.5pc/[rr]^{a_0} \ar[rd]^{a} \ar@{}[rr]|{\cong \; \Downarrow \; \alpha} & & j_0 \ar[rd]^{a_2} & \\ 
 & j'' \ar[ru]^{a_1} \ar[rd]_{b_1} \ar@{}[rr]|{\cong \; \Downarrow \; \gamma} & & j_2 \\
j' \ar@/_1.5pc/[rr]_{b_0} \ar[ru]_{b} \ar@{}[rr]|{\cong \; \Uparrow \; \beta} &  & j_1 \ar[ru]_{b_2} & \\  }}$$ 

Then $[a,r,b]=[a_2a_0,s,b_2b_0]$ where $s$ is defined as the composition $\G(b_2\beta^{-1})_{y'}\G(\gamma b)_{y'}\G(a_2a_1)(r)\G(a_2\alpha)_y$.
\end{lemma}

\begin{proof}
In \ref{lemita1}, take $a':=a_2a_0$, $b':=b_2b_0$, $c:=a_2a_1$, $\alpha:=a_2\alpha$ and \mbox{$\beta:=b_2\beta^{-1}\circ\gamma b$.}
\end{proof}

\begin{theorem}\label{T_Cat}
Let $\F: \cc{I}\rightarrow \cc{J}$ be a 2-cofinal 2-functor and $\ff{G}:\cc{J}\rightarrow \cc{C}at$ a 2-functor. Then the canonical morphism 
$$\coLim{i\in \cc{I}}{\ff{G}\F i}\stackrel{\ff{h}}\longrightarrow \coLim{j\in \cc{J}}{\ff{G}j}$$

\noindent is an equivalence of categories.
\end{theorem}

\begin{proof}
First of all, let's note that $\ff{h}(x,i)=(x,\F i)\ \forall\ (x,i)\in \coLim{i\in \cc{I}}{\G\F i}$ and \mbox{$\ff{h}([u,r,v])=[\F(u),r,\F(v)]\ \forall\ (x,i)\stackrel{[u,r,v]}\longrightarrow (x',i')\in \coLim{i\in \cc{I}}{\G\F i}$.}

Now, recall that is enough to check that $\ff{h}$ is essentially surjective on objects and full and faithful (\cite{ML}). 

\begin{itemize}
  \item[-] $\ff{h}$ is essentially surjective on objects: Let $(y,j)\in \coLim{j\in \cc{J}}{\G j}$. By CF0, \mbox{$\exists \ j\stackrel{a}\rightarrow \F i\in \cc{J}$} and clearly $\ff{h}(\G(a)(y),i)\cong(y,j)$ in $\coLim{j\in \cc{J}}{\G j}$. 
\item[-] $\ff{h}$ is full: Let $(x,\F i)\stackrel{[a,r,b]}\longrightarrow (x',\F i'))\in \coLim{j\in \cc{J}}{\G j}$ where $\vcenter{\xymatrix@R=-0.5pc{\F i \ar[rd]^{a} & \\ & j \\ \F i' \ar[ru]_{b} & }} \in \cc{J}$ and \mbox{$\G(a)(x)\mr{r}\G(b)(x')$.} By CF0, $\exists \ j\stackrel{c}\rightarrow \F i''\in \cc{J}$. Then, by \ref{corolario1}, \mbox{$[a,r,b]=[ca,\G(c)(r),cb]$,} so without loss of generality, we can suppose $j=\F i''$.

By using \ref{corolario2} for $a$ and $b$ respectively and by \ref{obs1}, we have the following \mbox{configuration:}

$$\xymatrix@R=.8pc{
\F i  \ar@/^1.5pc/[rr]^{\F u_0 } \ar[rd]^{a} \ar@{}[rr]|{\cong \; \Downarrow \; \alpha } & & \F i_0  \ar[rd]^{\F u_2 } & \\ 
 & \F i''  \ar[ru]^{\F u_1 } \ar[rd]_{\F v_1 }  \ar@{}[rr]|{\cong \; \Downarrow \; \gamma} && \F i_2  \\
\F i'  \ar@/_1.5pc/[rr]_{\F v_0 } \ar[ru]_{b} \ar@{}[rr]|{\cong \; \Uparrow \; \beta} && \F i_1  \ar[ru]_{\F v_2 } & \\  }$$

Now, by \ref{lemita2}, $[a,r,b]=\ff{h}([u_2u_0,s,v_2v_0])$ for some $s$.

\item[-] $\ff{h}$ is faithful:
Suppose $(x,i)\mrpair{[u_0,r_0,v_0]}{[u_1,r_1,v_1]}(x',i')\in \coLim{i\in \cc{I}}{\G\F i}$ (where $\vcenter{\xymatrix@R=-0.5pc{i\ar[rd]^{u_0} & \\ & i_0 \\ i'\ar[ru]_{v_0} & }}$ and $\vcenter{\xymatrix@R=-0.5pc{i\ar[rd]^{u_1} & \\ & i_1 \\ i'\ar[ru]_{v_1} & }} \in \cc{I}$) such that $\ff{h}([u_0,r_0,v_0])=\ff{h}([u_1,r_1,v_1])$ in $\coLim{j\in \cc{J}}{\G j}$. Then, there exists $\vcenter{\xymatrix@R=-0.5pc{\F i_0 \ar[rd]^{a} & \\ & j \\ \F i_1 \ar[ru]_{b} & }}$ and invertible 2-cells $a\F u_0 \stackrel{\alpha_0}\Rightarrow b\F u_1$, $a\F v_0 \stackrel{\beta_0}\Rightarrow b\F v_1 \in \cc{J}$ such that the following diagram commutes

\begin{equation}\label{*}
\xymatrix{\G(a)\G(\F(u_0))(x)\ar[r]^{\G(\alpha_0)_x} \ar[d]_{\G(a)(r_0)} & \G(b)\G(\F(u_1))(x) \ar[d]^{\G(b)(r_1)} \\
          \G(a)\G(\F(v_0))(x')\ar[r]^{\G(\beta_0)_{x'}} & \G(b)\alpha(\F(v_1))(x')  }
\end{equation}

By CF0, $\exists \ j\stackrel{c}\rightarrow \F i_2 \in \cc{J}$. By using \ref{corolario2} two times for $\F i_0 \stackrel{ca}\rightarrow\F i_2$ and \mbox{$\F i_1 \stackrel{cb}\rightarrow\F i_2$} respectively, we have $\vcenter{\xymatrix@R=-0.5pc{i_0\ar[rd]^{u_2} & \\ & i_3 \\ i_2\ar[ru]_{v_2} & }}$, $\vcenter{\xymatrix@R=-0.5pc{i_1\ar[rd]^{u_3} & \\ & i_4 \\ i_2\ar[ru]_{v_3} & }} \in \cc{I}$ and invertible \mbox{2-cells} \mbox{$\F u_2 \stackrel{\alpha_1}\Rightarrow\F (v_2) ca$} and $\F u_3 \stackrel{\beta_1}\Rightarrow\F(v_3) cb \in \cc{J}$.

Now, consider $\vcenter{\xymatrix@R=-0.5pc{& \F i_3  \\ \F i_2 \ar[ru]^{\F v_2 }\ar[rd]_{\F v_3 } & \\ & \F i_4 }} \in \cc{J}$. By \ref{obs1}, we have $\vcenter{\xymatrix@R=-0.5pc{i_3\ar[rd]^{u_4} & \\ & i_5 \\ i_4\ar[ru]_{v_4} & }} \in \cc{I}$ and an invertible 2-cell $\F u_4 \F v_2 \stackrel{\gamma_0}\Rightarrow \F v_4 \F v_3 \in \cc{J}$.

Then we have the following configuration

$$\xymatrix@R=.8pc@C=.5pc{\F i \ar[rd]^{\F u_0 } \\
                          & \F i_0 \ar[ddr]^{a}\ar@/^1.5pc/[rrrrd]^{\F u_2 } \\ 
                          \F i' \ar[ru]_{\F v_0 } & & \ar@{}[rr]|{\cong \; \Downarrow \; \alpha_1} &  & & \F i_3  \ar[rd]^{\F u_4 } \\ 
                          & & j \ar[rr]^{c} & & \F i_2  \ar[rd]_{\F v_3 } \ar[ru]^{\F v_2 } \ar@{}[rr]|{\cong \; \Downarrow \; \gamma_0 } & & \F i_5  \\ 
                          \F i  \ar[rd]^{\F u_1 } & & \ar@{}[rr]|{\cong \; \Downarrow \; \beta_1} & & & \F i_4  \ar[ru]_{\F v_4 } \\
                          & \F i_1  \ar[ruu]_{b} \ar@/_1.5pc/[rrrru]_{\F u_3 } \\
                          \F i' \ar[ru]_{\F v_1 }}$$

Then we have $\vcenter{\xymatrix@R=-0.5pc{i_0\ar[rd]^{u_4u_2} \\ & i_5 \\ i_1\ar[ru]_{v_4u_3}}} \in \cc{I}$ and invertible 2-cells $\F(u_5u_0)\stackrel{\alpha_2}\Rightarrow \F(v_5u_1)$, \mbox{$\F(u_5v_0)\stackrel{\beta_2}\Rightarrow \F(v_5v_1) \in \cc{J}$} given by the following compositions:

$$\xymatrix@C=-0pc{\F u_4  \did && \F u_2  \op{\alpha_1} && \F u_0  \did \\
						\F u_4  \ar@<-2pt>@{-}[d]+<4pt,8pt> \ar@{}[dr]|{\gamma_0} & \F v_2  \ar@<2pt>@{-}[d]+<-4pt,8pt> & c \did & \ \ a \ \ \ar@<-2pt>@{-}[d]+<4pt,8pt> \ar@{}[dr]|{\alpha_0} & \F u_0  \ar@<2pt>@{-}[d]+<-4pt,8pt> \\
						\F v_4  \did & \F v_3  & c & \ \ b \ \ & \F u_1  \did \\
						\F v_4  && \F u_3  \cl{{\beta_1}^{-1}} && \F u_1   } \ \xymatrix@C=-.4pc{\\ \\,} \ \ \ \
\xymatrix@C=-0pc{\F u_4  \did && \F u_2  \op{\alpha_1} && \F v_0  \did \\
									\F u_4  \ar@<-2pt>@{-}[d]+<4pt,8pt> \ar@{}[dr]|{\gamma_0} & \F v_2  \ar@<2pt>@{-}[d]+<-4pt,8pt> & c \did & \ \ a \ \ \ar@<-2pt>@{-}[d]+<4pt,8pt> \ar@{}[dr]|{\beta_0} & \F v_0  \ar@<2pt>@{-}[d]+<-4pt,8pt> \\
						\F v_4  \did & \F v_3  & c & \ \ b \ \ & \F v_1  \did \\
						\F v_4  && \F u_3  \cl{{\beta_1}^{-1}} && \F v_1   }$$

It can be checked that $(\F u_5,\F(v_5)\alpha_2,\beta_2)$ is an homotopy between $(\F u_0,r_0,\F v_0)$ and $(\F u_1,r_1,\F v_1)$, so without loss of generality, we can suppose from the beginning that the homotopy between $\h([u_0,r_0,v_0])$ and $\h([u_1,r_1,v_1])$ has the form $(\F u,\F v,\alpha_0,\beta_0))$ where $\vcenter{\xymatrix@R=-0.5pc{i_0\ar[rd]^{u} \\ & i_2 \\ i_1\ar[ru]_v}} \in \cc{I}$.

Let's apply \ref{lemita3} to the triplets $uu_0, vu_1, \alpha_0$ and $uv_0,vv_1,\beta_0$ respectively. So, we have $i_2\stackrel{w_0}\rightarrow i_3$, $i_2\stackrel{z_0}\rightarrow i_4$ and invertible 2-cells $i \cellrd{w_0uu_0}{\delta_0}{w_0vu_1} i_3$, $i' \cellrd{z_0uv_0}{\nu_0}{z_0vv_1} i_4 \in \cc{I}$ such that $\F(w_0)\alpha_0=\delta_0$ and $\F(z_0)\beta_0=\nu_0$.

Then, by \ref{obs1}, there exist morphisms $\vcenter{\xymatrix@R=-0.5pc{i_3\ar[rd]^{w_1} \\ & i_5 \\ i_4\ar[ru]_{z_1}}} \in \cc{I}$ and an invertible 2-cell $\F(i_2) \cellrd{\F(w_1w_0)}{\alpha}{\F(z_1z_0)} \F(i_5) \in \cc{J}$ and, by \ref{lemita3}, we have $i_5\stackrel{w_2}\rightarrow i''$ and an invertible 2-cell $i_2\cellrd{w_2w_1w_0}{\delta}{w_2z_1z_0} i'' \in \cc{I}$ such that $\F(w_2)\alpha=\delta$. 

It can be checked that $(w_2w_1w_0u,w_2z_1z_0v,(\delta vu_1)\circ (w_2w_1\delta_0),(w_2z_1\nu_0)\circ (\delta uv_0))$ is an homotopy between $(u_0,r_0,v_0)$ and $(u_1,r_1,v_1)$ which concludes the proof. 
\end{itemize}                     
\end{proof}

The purpose of the following is twofold. First to construct a cofinite and filtered poset with a unique initial object $\ff{M}(\cc{J})$ associated to a 2-filtered 2-category (\ref{defMJ} and \ref{MJ}). And second, to prove that there is a 2-cofinal 2-functor $\ff{M}(\cc{J})\stackrel{\F}\rightarrow \cc{J}$ (\ref{phi}). This is a \mbox{2-categorical} version  of a result of Deligne \cite[Expose I,8.1.6]{G2}, see also \cite{EH} Mardesick trick. This results are key to prove reindexing properties of 2-pro-objects in section \ref{Mtrick}.    

\begin{definition}\label{defMJ}

\noindent \begin{enumerate}
\item A diagram in a 2-category $\cc{J}$ is a functor $\ff{C} \mr{f} \cc{J}$ from a category $\ff{C}$ to the underlying category of $\cc{J}$. It is said to be finite if $\ff{C}$ is a finite category.
 
 \item $\ff{C}\stackrel{f}\longrightarrow \cc{J}$ is a subdiagram of $\ff{D}\stackrel{g}\longrightarrow \cc{J}$ if there is an injective (on objects and on morphisms) functor $\ff{C}\stackrel{h}\longrightarrow \ff{D}$ such that $gh=f$. If $h$ is an isomorphism of categories we say that the diagrams are isomorphic.
\end{enumerate}
\end{definition}

\begin{remark}
Final objects $c \in \ff{C}$, $d \in \ff{D}$ correspond under isomorphism of diagrams. That is, $h(c) = d$ (thus $f(c) = g(d)$ in $\cc{J}$).
\cqd
\end{remark}

\begin{definition}
Let $\cc{J}$ be a 2-category. We denote by $\ff{M}(\cc{J})$ the poset of equivalence classes (under isomorphism) of finite diagrams over 
$\cc{J}$ ordered by the subdiagram relation (in the sense of subsets, not injections). We assume that all index categories $\ff{C}$ in $\ff{M}(\cc{J})$ have a chosen empty final object denoted 
$*_{\ff{C}}$.  
\end{definition}

\begin{proposition}\label{MJ}
Let $\cc{J}$ be a 2-filtered 2-category. $\ff{M}(\cc{J})$ is cofinite, filtered and has a unique initial object.
\end{proposition}

\begin{proof}
Clearly $\ff{M}(\cc{J})$ is cofinite and has a unique initial object. Let's check that it is filtered: 
Let \mbox{$\ff{C}\stackrel{f}\longrightarrow \cc{J}$} and 
$\ff{D}\stackrel{g}\longrightarrow \cc{J}\in \ff{M}(\cc{J})$. 
Consider the category $\ff{E}$ disjoint union of $\ff{C}$ and $\ff{D}$, and an additional object $*$ together with one morphism  $c \to *$ from each object $ c\in \ff{C}$ or $c\in \ff{D}$. Clearly $* = *_\ff{E}$. Since $\cc{J}$ is 2-filtered, we have 
$\vcenter
     {
      \xymatrix@R=-.5pc
            { f(*_\ff{C}) \ar[rd]^a & \\ 
           &  j 
            \\ g(*_\ff{D})\ar[ru]_b & 
            }
     }
\in \cc{J}$.  
For each $c\in \ff{C}$ (resp. $c \in \ff{D}$), 
$\exists \, ! \  c \stackrel{r_c}\rightarrow *_\ff{C}$ 
(resp. $\exists \, ! \, c \stackrel{r_c}\rightarrow *_\ff{D}$). 
Consider the diagram $\ff{E}\stackrel{h}\longrightarrow \cc{J}$ defined by $h = f$ on $\ff{C}$, $h = g$ on $\ff{D}$, $h(*)=j$, 
$h(c \to *) = a \circ f(r_c)$ for $c\in \ff{C}$, and 
$h(c \to *) = b \circ g(r_c)$ for $c\in \ff{D}$.

It is clear that this diagram is above 
$\ff{C}\stackrel{f}\longrightarrow \cc{J}$ and 
$\ff{D}\stackrel{g}\longrightarrow \cc{J}$.
\end{proof}

\begin{proposition}\label{phi}
Let $\cc{J}$ be a 2-filtered 2-category. There is a 2-cofinal 2-functor $\ff{M}(\cc{J})\stackrel{\F}\rightarrow \cc{J}$ where $\ff{M}(\cc{J})$ is the poset defined in \ref{defMJ} (we are considering $\ff{M}(\cc{J})$ as a trivial 2-category). 
\end{proposition}

\begin{proof}
The 2-functor $\F$ is defined as follows:

\begin{itemize}
 \item[-] $\F(\ff{C}\stackrel{f}\longrightarrow \cc{J})=f(*_\ff{C})$.

 \item[-] If $\ff{C}\stackrel{f}\longrightarrow \cc{J}$ is a subdiagram of 
$\ff{D}\stackrel{g}\longrightarrow \cc{J}$ via $\ff{C}\mr{\ff{h}}\ff{D}$, 
$\exists \ !\ h(*_\ff{C}) \stackrel{r}\rightarrow *_\ff{D}$. Then 
$\F(\ff{C}\stackrel{f}\rightarrow \cc{J}$ $\;\leq\;\ff{D}\stackrel{g}\rightarrow \cc{J})=f(*_\ff{C})=g(h(*_\ff{C})) \stackrel{g(r)}\rightarrow g(*_\ff{D})$.

 \item[-] The 2-cells are the identities so they go to the corresponding identities.     

\end{itemize}
Let's check that $\F$ is 2-cofinal:

\begin{itemize}
 \item[CF0.] Let $j \in \cc{J}$. Then $\F(\{*\}\stackrel{j}\rightarrow \cc{J}) = j$. 

\item[CF1.] Let $j \in \cc{J}$, 
$\ff{C}\stackrel{f}\longrightarrow \cc{J}\in \ff{M}(\cc{J})$ and $j\mrpair{a}{b}f(*_\ff{C}) \in \cc{J}$. Since $\cc{J}$ is 2-filtered, we have $f(*_\ff{C})\stackrel{e}\rightarrow j' \in \cc{J}$ and an invertible 2-cell $j \cellrd{e a}{\alpha} {e b} j' \in \cc{J}$. Consider the category 
$\ff{D}$ disjoint union of $\ff{C}$ and $\{*\}$, with a morphism from each object of $\ff{C}$ to $*$. Clearly $* = *_\ff{D}$. Consider the diagram 
$\ff{D}\stackrel{g}\longrightarrow \cc{J}$ where $g$ is defined by $g = f$ in $\ff{C}$, $g(*) = j'$, and $g(c \to *) =e \circ f(r_c)$, where $r_c$ is the unique morphism $c \to *_\ff{C} $ in $\ff{C}$.
$\ff{C}\stackrel{f}\longrightarrow \cc{J}$ is a subdiagram of 
$\ff{D}\stackrel{g}\longrightarrow \cc{J}$ and 
$\ff{C}\stackrel{f}\longrightarrow \cc{J}$ $\;\leq\;$ 
$\ff{D}\stackrel{g}\longrightarrow \cc{J}$ is sent by $\F$ to 
$f(*_\ff{C})\stackrel{e}\rightarrow g(*) = j'$.

\item[CF2.] Let $j\in \cc{J}$, $\ff{C}\stackrel{f}\longrightarrow \cc{J}$ in $\ff{M}(\cc{J})$, and $j \cellpairrd{a}{\alpha}{\beta}{b} f(*_\ff{C})$ in 
$\cc{J}$. Since $\cc{J}$ is 2-filtered, we have $f(*_\ff{C})\stackrel{e}\rightarrow j'$ in $\cc{J}$ such that 
$e\,\alpha = e\,\beta$. The proof follows in the same way that for CF1.
\end{itemize}
\end{proof}
\subsection{2-functor associated to a pseudo-functor}\label{A sombrero}

In this subsection we establish a result of independent interest and that will be needed to prove that $\Prop{C}$ is a closed 2-bmodel 2-category (see \ref{2-closed}) provided that $\cc{C}$ is (\ref{Propde2modelos}). Our construction of $\hat{\ff{A}}$ and $\ff{T}$ are inspired in the constructions for the same purpose that can be found in \cite{GRAY} or \cite{DH}. We think the construction made in \cite{GRAY} has a slight mistake because the value of $\ff{T}$ in 2-cells is not considered. 
Our case is simpler because we are only interested in the case of filtered categories and we consider pseudo-functors instead of lax-functors. A reference to the validity of this result is made in \cite{L}. We are going to use the results of this subsection only for cofinite filtered posets. 

\begin{proposition}\label{proposition_a}
 Any category $\ff{A}$ together with a class $\ff{B}$ of pairs of arrows $A \mrpair{f}{g} B \in \ff{A}$ closed under composition and containing all pairs with $f=g$ (note that $\ff{B}$ is a category) determine a 2-category $\hat{\ff{A}}$ as follows:
 
 Objects and arrows are those of $\ff{A}$ and we add a 2-cell $A \cellrd{f}{\theta_{g,f}}{g} B$ for each pair $A \mrpair{f}{g} B \in \ff{B}$, subject to the equations

\begin{itemize}
  \item[-] $\theta_{f,f}=id_f$
  \item[-] $\theta_{h,g}\circ \theta_{g,f}=\theta_{h,f}$ \begin{equation}\label{ecuacionestheta} \end{equation}
  \vspace{-1.1cm}   
  \item[-] $\theta_{g',f'}\theta_{g,f}=\theta_{g'g,f'f}$
 \end{itemize}

 Note that $\theta_{g,f}^{-1}=\theta_{f,g}$ (every 2-cell is invertible).  \cqd 
\end{proposition}

A lax-functor is defined by the same data that a pseudo-functor but without requiring the structural 2-cells to be invertible.

Let $\cc{A}\mr{\ff{F}}\cc{C}$ be a lax-functor. Then, given any tuple of composable arrows, iterating structural 2-cells determines 2-cells from the composition of the values of $\ff{F}$ to the value of $\ff{F}$ in the composition. It easily follows from the associativity axiom that all possible iterations are equal. Thus:

\begin{proposition}\label{lemitadepseudofuntores}
Given any tuple $\ff{f}=(\ff{f}_1,\ff{f}_2,...,\ff{f}_n)$ of composable arrows \mbox{$\ff{A}_0\mr{\ff{f}_1}\ff{A}_1...\mr{\ff{f}_n} \ff{A}_{n} \in \cc{A}$}, there is a well defined (structural) 2-cell $\ff{F}\ff{A}_0 \cellrd{\ff{F}\ff{f}_n...\ff{F}\ff{f}_1}{\theta_{\ff{f}}}{\ff{F}(\ff{f}_n...\ff{f}_1)} \ff{F}\ff{A}_n$. \cqd  
\end{proposition}

\begin{proposition}
 Let $\cc{A}\mr{\ff{F}}\cc{C}$ be a pseudo-functor and $\ff{f}=(\ff{f}_1,\ff{f}_2,...,\ff{f}_n)$, \mbox{$\ff{g}=(\ff{g}_1,\ff{g}_2,...,\ff{g}_m)$} with $\ff{A}\mr{\ff{f}_1}\ff{A}_1...\mr{\ff{f}_n} \ff{B}$, $\ff{A}\mr{\ff{g}_1}\ff{B}_1...\mr{\ff{g}_m} \ff{B} \in \cc{A}$ be such that $\ff{f}_n ...\ff{f}_2 \ff{f}_1=\ff{g}_m ...\ff{g}_2 \ff{g}_1$. Then there is a well defined 2-cell $\ff{F}\ff{A}\cellrd{\ff{F}\ff{f}_n ... \ff{F}\ff{f}_1}{\theta_{\ff{g},\ff{f}}}{\ff{F}\ff{g}_m ...\ff{F}\ff{g}_1} \ff{F}\ff{B}$. These 2-cells satisfy equations \eqref{ecuacionestheta}. 

\end{proposition}

\begin{proof}
Define $\theta_{\ff{g},\ff{f}}= \theta_{\ff{g}}^{-1} \circ \theta_{\ff{f}}$.
\end{proof}


\begin{proposition}\label{Asombrero}
 Let $\ff{A}$ be a category. There exist a 2-category $\hat{\ff{A}}$ and a pseudo-functor $\ff{A}\mr{\ff{T}}\hat{\ff{A}}$ such that for each 2-category $\cc{C}$, there is an isomorphism of \mbox{2-categories} 
 
 $$\cc{H}om_p(\hat{\ff{A}},\cc{C}) \mr{\ff{T}^*} p\cc{H}om_p(\ff{A},\cc{C})$$
 
 $$\xymatrix{\ff{A} \ar[rr]^{\ff{T}} \ar[dr]_{\F} 
             && \hat{\ff{A}} \ar@{-->}[dl]^{\exists ! \hat{\F}}
             \\
             & \cc{C}
             & }$$

Furthermore, if $\ff{A}$ is filtered, then $\ff{T}$ is 2-cofinal. 
\end{proposition}

\begin{proof}
 We define $\hat{\ff{A}}$ as follows:

\begin{itemize}
 \item[-] Objects of $\hat{\ff{A}}$ are the objects of $\ff{A}$.
 
 \item[-] Morphisms of $\hat{\ff{A}}$ are tuples of composable morphisms of $\ff{A}$. More explicitly, a morphism $A\mr{\f}B$ is a tuple $\f=(f_1,f_2,...,f_n)$ with $A\mr{f_1}A_1...\mr{f_n} B$, $n\geq 0$. 
 
 \item[-] We consider the empty tuple $\emptyset=(-)$ corresponding to $n=0$ as an arrow $A\mr{\emptyset_{A}} A$ for every object $A\in \ff{A}$. 

 \item[-] Composition is given by reverse juxtaposition (with identities $id_A=\emptyset_A$), i.e. $(g_1,g_1,...,g_m)(f_1,f_2,...,f_n)=(f_1,f_2,...,f_n,g_1,g_2,...,g_m)$.

\end{itemize}

We then apply the construction of proposition \ref{proposition_a} with $A \mrpair{\f}{\g} B \in \ff{B}$ iff \mbox{$f_n ... f_1=g_m ... g_1$}, $\f=\emptyset_A$, $\g=(id_A)$ or $\f=(id_A)$, $\g=\emptyset_A$.

\vspace{1ex}

We define $\ff{A}\mr{\ff{T}}\hat{\ff{A}}$ as follows:
 
\begin{itemize}
 \item[-] $\ff{T} A=A$

 \item[-] $\ff{T}(A\mr{f}B)=(f)$. 

\end{itemize}

Note that since $A\mrpair{\emptyset_A}{(id_A)}A \in \ff{B}$ and $A \mrpair{(f,g)}{(gf)} C \in \ff{B}$. Then we have invertible 2-cells $\alpha^{\ff{T}}_A=\theta_{(id_A),\emptyset_A}: id_A\Mr{}(id_A)$ and for $A\mr{f}B\mr{g}C$, $\alpha^{\ff{T}}_{f,g}=\theta_{(gf),(f,g)}: (g)(f)\Mr{} (gf)$.
It immediately follows that $\ff{T}$ given by this data is actually a pseudo-functor.
 
 \vspace{2ex}
 
 Let's check that $\ff{T}^*$ is an isomorphism of 2-categories:
 
\begin{itemize}
 \item[-] On objects: Let $\F \in p\cc{H}om_p(\ff{A},\cc{C})$. We define $\hat{\F}$ as follows:
 
 $$\hat{\F}A=\F A, \quad \hat{\F}(A\mr{\f}B)=\F f_n...\F f_1, \quad \hat{\F}\theta_{\g,\f}=\theta_{\g}^{-1}\circ \theta_{\f} \quad \hbox{ (see \ref{lemitadepseudofuntores}).}$$
 
 It can be easily checked that this data defines a 2-functor $\hat{\F}$ which is unique such that $\hat{\F}\ff{T}=\F$. 
 
 \item[-] On morphisms: Let $\ff{F} \mr{\mu} \ff{G} \in p\cc{H}om_p(\ff{A},\cc{C})$. We define $\hat{\mu}_A=\mu_A$ and $\hat{\mu}_{\f}=\mu_{f_n}\ff{F}f_{n-1}...\ff{F}f_1 \circ \ff{G}f_n \mu_{f_{n-1}} \ff{F}f_{n-1}...\ff{F}f_1\circ ... \circ \ff{G}f_n \ff{G}f_{n-1} ... \ff{G}f_2 \mu_{f_1}$. It can be easily checked that this data gives the unique pseudo-natural transformation such that $\hat{\mu}\ff{T}=\mu$.
 
 \item[-] On 2-cells: Let $\ff{F} \cellrd{\mu}{\rho}{\mu'} \ff{G} \in p\cc{H}om_p(\ff{A},\cc{C})$. We define $\hat{\rho}_A=\rho_A$. It can be easily checked that this data gives the unique modification such that $\hat{\rho} \ff{T}=\rho$.
\end{itemize}
 

Finally, let's check that $\ff{T}$ is 2-cofinal in case $\ff{A}$ is filtered. CF0 is clear and CF2 is vacuous since $\ff{A}$ is a category. CF1: Given $A\in \hat{\ff{A}}$, $B\in \ff{A}$ and two morphisms $A \mrpair{\f=(f_1,...,f_n)}{\g=(g_1,...,g_{m})} B \in \hat{\ff{A}}$, since $\ff{A}$ is filtered, $\exists \ B\mr{h}C \in \ff{A}$ such that $hf_1...f_n=hg_1...g_{m}$. Then $B \mrpair{(f_1,...,f_n,h)}{(g_1,...,g_m,h)} C  \in \ff{B}$ and thus we have an invertible 2-cell $\theta_{(h)\g,(h)\f}: (h)\f \Mr{} (h) \g \in \hat{\ff{A}}$.
 \end{proof}

\begin{remark}
In particular, from \ref{cofinal implica filtrante}, we have that if $\ff{A}$ is filtered, then $\hat{\ff{A}}$ is 2-filtered. 
$\hfill \square$

\end{remark} 
 
\begin{remark}
 $\widehat{\ff{A}^{op}}=\hat{\ff{A}}^{op}$. \cqd
\end{remark}

\begin{corollary}\label{pointwise en sombrero}
 Let $\ff{A}$ be a category. Then the 2-category $\cc{H}om_p(\hat{\ff{A}},\cc{C})$ has all bi-limits of pseudo-functors and bi-cotensors and they are computed pointwise. The dual assertion also holds.
\end{corollary}

\begin{proof}
 The proof follows immediately from \ref{Asombrero} plus \ref{pointwisebi-limit} and \ref{tensorptoapto}.
\end{proof}
\subsection{Further results.}\label{further results}

A. Joyal pointed to us the notion of \emph{flexible} functors, related with some of our results on pseudo-colimits of representable 2-functors. We \mbox{recall} now this notion since it bears some significance for the concept of \mbox{2-pro-object} developed in this thesis. Any \mbox{2-pro-object} determines a 2-functor which is flexible, and some of our results find their right place stated in the context of flexible 2-functors.

\vspace{1ex}

\begin{sinnadastandard}
{\bf Warning.} \emph{In this subsection 2-categories are assumed to be locally small, except the illegitimate constructions $\cc{H}om$ and $\cc{H}om_p$.}
\end{sinnadastandard}

The inclusion $\cc{H}om(\cc{C},\cc{C}at) \mr{i} \cc{H}om_p(\cc{C},\cc{C}at)$ has a left adjoint $(-)' \dashv i$, we refer the reader to \cite{BKP}. The 2-natural counit of this adjunction $\ff{F}' \Mr{\varepsilon_\ff{F}} \ff{F}$ is an equivalence in $\cc{H}om_p(\cc{C},\cc{C}at)$, with a section given by the pseudo-natural unit $\ff{F} \Mr{\eta_\ff{F}} \ff{F'}$,
$\varepsilon_\ff{F} \eta_\ff{F} = 1_\ff{F}$,
$\eta_\ff{F} \varepsilon_\ff{F} \cong 1_\ff{F'}$, \mbox{[\citealp{BKP}, Proposition 4.1.]}

\begin{definition}[]\emph{[\citealp{BKP}, Proposition 4.2]} \label{flexible}
A 2-functor $\cc{C} \mr{\ff{F}} \cc{C}at$ is \emph{flexible} if the counit
$\ff{F'} \Mr{\varepsilon_\ff{F}} \ff{F}$ has a 2-natural section $\ff{F} \Mr{\lambda} \ff{F'}$,
$\varepsilon_\ff{F} \lambda = 1_\ff{F}$,
\mbox{$\lambda \varepsilon_\ff{F} \cong 1_\ff{F'}$}, which determines an equivalence in $\cc{H}om(\cc{C},\cc{C}at)$.
\end{definition}

We state now a useful characterization of flexible 2-functors $\ff{F}$ independent of the left adjoint $(-)'$, the proof will appear elsewhere \cite{DD}.

\begin{proposition} \label{flexiblechar}
A 2-functor $\cc{C} \mr{\ff{F}} \cc{C}at$ is flexible
$\iff$ for all \mbox{2-functors} $\ff{G}$, the inclusion $\cc{H}om(\cc{C},\cc{C}at)(\ff{F},\ff{G}) \mr{i_\ff{G}} \cc{H}om_p(\cc{C},\cc{C}at)(\ff{F},\ff{G})$ has a retraction
$\alpha_\ff{G}$ natural in $\ff{G}$, $\alpha_\ff{G} i_\ff{G} = id$, $i_\ff{G} \alpha_\ff{G} \cong id$, which determines an equivalence of categories. \cqd
\end{proposition}
Let $\cc{H}om(\cc{C},\cc{C}at)_f$ and  $\cc{H}om_p(\cc{C},\cc{C}at)_f$ be the subcategories whose objects are the flexible 2-functors. We have the following corollaries:
\begin{corollary} \label{flexible2=p}
The 2-categories $\cc{H}om(\cc{C},\cc{C}at)_f$ and $\cc{H}om_p(\cc{C},\cc{C}at)_f$ are \emph{pseudoequivalent} in the sense they have the same objects and retract equivalent hom categories. \cqd
\end{corollary}
By \ref{peqsiipf&f} the inclusion 2-functor
$\cc{H}om(\cc{C},\cc{C}at)_f \mr{} \cc{H}om_p(\cc{C},\cc{C}at)_f$ has the identity (on objects) as a retraction pseudo-quasi-inverse, with the equality as the invertible pseudo-natural transformation $\ff{F} \mr{=} \ff{F}$ in  $\cc{H}om_p(\cc{C},\cc{C}at)_f$.

An important property of flexible 2-functors, false in general, is the \mbox{following}:
\begin{corollary}\label{eqencadaCeseq}
Let $\theta:\ff{G}\Rightarrow \ff{F}\in \cc{H}om(\cc{C},\cc{C}at)_f$ be such that $\theta_\ff{C}:\ff{G}\ff{C}\rightarrow \ff{F}\ff{C}$ is an equivalence of categories for each $\ff{C}\in \cc{C}$. Then, $\theta$ is an equivalence in $\cc{H}om(\cc{C},\cc{C}at)_f$.
\end{corollary}
\begin{proof}
It is easy to check that there is a pseudo-natural transformation \mbox{$\eta':\ff{F}\Rightarrow \ff{G}$} such that $\theta\eta'\cong \ff{F}$ and $\eta'\theta\cong \ff{G}$ in $\cc{H}om_p(\ff{F},\ff{F})$ and $\cc{H}om_p(\ff{G},\ff{G})$ respectively. Now, by \ref{flexiblechar}, there is a 2-natural transformation $\eta:\ff{F}\Rightarrow \ff{G}$ such that $\eta\cong \eta'$ in $\cc{H}om_p(\ff{F},\ff{G})$. Then, $\theta\eta\cong \ff{F}$ and $\eta\theta\cong \ff{G}$ in $\cc{H}om(\ff{F},\ff{F})$ and $\cc{H}om(\ff{G},\ff{G})$ respectively and so $\theta$ is an equivalence in $\cc{H}om(\cc{C},\cc{C}at)$.
\end{proof}

\begin{proposition} \label{colimflexible}
Small pseudo-colimits of flexible 2-functors are \mbox{flexible.}
\end{proposition}
\begin{proof}
Let $\ff{F} = \coLim{i \in \cc{I}}{\ff{F} i}$, where each $\ff{F} i$ is flexible, and let $\ff{G}$ be any other 2-functor. Set $\cc{A} = \cc{H}om(\cc{C},\cc{C}at)$ and
$\cc{A}_p = \cc{H}om_p(\cc{C},\cc{C}at)$. Then:
$$
\cc{A}(\ff{F}, \ff{G}) \cong
\Lim{i \in \cc{I}}{\cc{A}(\ff{F} i,\, \ff{G})} \mr{i}
\Lim{i \in \cc{I}}{\cc{A}_p(\ff{F} i,\, \ff{G})} \cong
\cc{A}_p(\ff{F}, \ff{G}).
$$
The two isomorphisms are given by definition  \ref{colimits}. The arrow $i$ is the pseudo-limit of the equivalences with retraction quasi-inverses corresponding to each $\ff{F} i$. It is not difficult to check that $i$  is also such an equivalence.
\end{proof}
It follows also from \ref{flexiblechar} that the pseudo-Yoneda lemma (\ref{pseudoYoneda}, \ref{repflexible}) says that the representable 2-functors are flexible, so we have:
\begin{corollary}\label{pseudo=2forpro}
Small pseudo-colimits of representable 2-functors are \mbox{flexible.} \cqd
\end{corollary}
Note that  \ref{colimflexible} and \ref{pseudo=2forpro} hold for any pseudo-colimit that may exist.

\pagebreak
\begin{center}
 {\bf Resumen en castellano de la secci\'on \ref{prelims}}
\end{center}

En esta secci\'on se fija la notaci\'on que se va a usar a lo largo de toda la tesis y se enuncian las definiciones y los resultados b\'asicos de la teor\'ia de 2-categor\'ias necesarios para este trabajo. 

La mayor\'ia de estos resultados son conocidos. Para aquellos que no hemos encontrado en la literatura, damos demostraciones detalladas. 

En \ref{weak limits and colimits} probamos que los pseudo-l\'imites (c\'onicos) en las 2-categor\'ias de 2-funtores $\cc{H}om(\cc{C},\cc{D})$, $\cc{H}om_p(\cc{C},\cc{D})$ y $p\cc{H}om_p(\cc{C},\cc{D})$ (definici\'on \ref{ccHom}) y los bi-l\'imites en $p\cc{H}om_p(\cc{C},\cc{D})$ se calculan punto a punto. 

En \ref{2-cofinal 2-functors} definimos la noci\'on de pseudo-funtor 2-cofinal entre 2-categor\'ias y probamos ciertas propiedades que usaremos en la secci\'on \ref{Mtrick} para demostrar las propiedades de 
reindexaci\'on de 2-pro-objetos. All\'i construimos un poset cofinito y filtrante con un \'unico objeto inicial $\ff{M}(\cc{J})$ asociado a una 2-categor\'ia 2-filtrante (\ref{defMJ} and \ref{MJ}) y probamos que se tiene un 2-funtor 2-cofinal $\ff{M}(\cc{J})\stackrel{\F}\rightarrow \cc{J}$ (\ref{phi}). 

En \ref{A sombrero} construimos un 2-funtor asociado via un pseudo-funtor 2-cofinal a un pseudo-funtor dado. Este resultado tiene inter\'es independiente y ser\'a usado en la secci\'on \ref{2-modelos}. Nuestra construcci\'on de $\hat{\ff{A}}$ y $\ff{T}$ fueron inspiradas en las construcciones hechas en \cite{GRAY} o \cite{DH}.

A. Joyal nos se\~nal\'o la noci\'on de funtores flexibles, relacionada con algunos resultados de esta tesis acerca de pseudo-col\'imites de 2-funtores representables. Recordamos en \ref{further results} esta noci\'on ya que tiene relevancia para el concepto de \mbox{2-pro-objeto} desarrollado en esta tesis. Todo \mbox{2-pro-objeto} determina un 2-funtor flexible, y algunos de nuestros resultados tienen su enunciado correcto en el contexto de 2-funtores flexibles.

\pagebreak

\section{2-Pro-objects}\label{2-Pro-objects}

{\bf Warning:} \emph{In this section 2-categories are assumed to be locally small, except illegitimate constructions as $\cc{H}om$, $\cc{H}om_p$ for large $\cc{C}$
or
$2\hbox{-}\cc{CAT}$.}

\vspace{1ex}

Some of the main results of this thesis are in this section. In \ref{def2pro} we define the \mbox{2-category} of 2-pro-objects of a 2-category $\cc{C}$ and establish the basic formula for morphisms and 2-cells of this \mbox{2-category.} Then, in \ref{lemas2pro}, we develop the notion of a morphism and a \mbox{2-cell} in $\cc{C}$ \emph{representing} a morphism and a 2-cell in $\Pro{C}$ respectively, inspired in the 1-dimensional notion of an arrow  representing a morphism of pro-objects found in \cite{AM}. We use this in \ref{pseudo-limitsen2pro} to construct the 2-filtered 2-category that serves as the index \mbox{2-category} for the 2-cofiltered pseudo-limit of \mbox{2-pro-objects.} This is also inspired in a \mbox{construction} for the same purpose found in \cite{AM}. We were forced to appeal to this complicated construction because the conceptual treatment of this problem found in \cite{G2} does not apply in the \mbox{2-categorical} case. This is because a \mbox{2-functor} is not the pseudo-colimit 
indexed by its 2-category of elements of 2-representable 2-functors. 
Finally, in \ref{pu2pro} we prove the universal property of $\Pro{C}$.

\subsection{Definition of the 2-category of 2-pro-objects}\label{def2pro}

In this subsection we define the 2-category of 2-pro-objects of a fixed \mbox{2-category} and prove its basic properties. A 2-pro-object over a \mbox{2-category} $\cc{C}$ will be a small \mbox{2-cofiltered} diagram in $\cc{C}$ and it will be the pseudo-limit of its own diagram in the \mbox{2-category} $\Pro{C}$.

\begin{definition} \label{2proc}
Let $\cc{C}$ be a 2-category. We define the 2-category of \mbox{2-pro-objects} of $\cc{C}$, which we denote by $2$-$\cc{P}ro(\cc{C})$, as follows:
\begin{enumerate}
\item
Its objects are the 2-functors
$\cc{I}^{op}\mr{\ff{X}}\cc{C}$,
$\ff{X} = (\ff{X}_i,\, \ff{X}_u,\, \ff{X}_\alpha)_{i,\, u,\, \alpha \in \cc{I}}$, with
$\cc{I}$ a small \mbox{2-filtered} 2-category. Often we are going to abuse the notation by saying \mbox{$\ff{X} = (\ff{X}_i)_{i\in \cc{I}}$.}
\item
If $\ff{X}=(\ff{X}_i)_{i\in \cc{I}}$ and $\ff{Y}=(\ff{Y}_j)_{j\in \cc{J}}$ are two 2-pro-objects,
$$
\Pro{C}(\ff{X},\ff{Y})=\cc{H}om(\cc{C},\, \cc{C}at)^{op}(\Lim{i \in \cc{I}}{\cc{C}(\ff{X}_i,-)}, \Lim{j\in \cc{J}}{\cc{C}(\ff{Y}_j,-)})
$$
$$
=\cc{H}om(\cc{C},\, \cc{C}at)(\coLim{j \in \cc{J}}{\cc{C}(\ff{Y}_j,-)},\coLim{i \in \cc{I}}{\cc{C}(\ff{X}_i,-)})
$$
\end{enumerate}

Compositions are given by the corresponding compositions in the \mbox{2-category} $\cc{H}om(\cc{C},\, \cc{C}at)^{op}$ so it is easy to check that $2$-$\cc{P}ro(\cc{C})$ is indeed a 2-category.
\end{definition}

\begin{sinnadastandard} {\bf Notation.}
 We are going to use the subindex notation to denote the evaluation of 2-pro-objects.
\end{sinnadastandard}

\begin{proposition}\label{eqconloscolimderep}
By definition there is a 2-fully-faithful \mbox{2-functor}
\mbox{$2\hbox{-}\cc{P}ro(\cc{C}) \mr{\ff{L}} \cc{H}om(\cc{C},\, \cc{C}at)^{op}$.}
Thus, there is a contravariant \mbox{2-equivalence} of \mbox{2-categories} \mbox{$2$-$\cc{P}ro(\cc{C}) \mr{\ff{L}} \cc{H}om(\cc{C},\, \cc{C}at)^{op}_{fc}$}, where $\cc{H}om(\cc{C},\, \cc{C}at)_{fc}$ stands for the full subcategory of $\cc{H}om(\cc{C},\, \cc{C}at)$ whose objects are those \mbox{2-functors} which are small 2-filtered pseudo-colimits of \mbox{representable} 2-functors. \mbox{However,} it is important to note that this equivalence is not injective on \mbox{objects.} \cqd
\end{proposition}

From Corollary \ref{pseudo=2forpro} it follows:

\begin{proposition} \label{proisflexible}
For any 2-pro-object $\ff{X}$, the corresponding \mbox{2-functor} $\ff{LX}$ is flexible. \cqd

\end{proposition}

\begin{remark}\label{proppseudoeqapro}
If we use pseudo-natural transformations to define morphisms of \mbox{2-pro-objects} we obtain a 2-category $2$-$\cc{P}ro_p(\cc{C})$, which anyway, by \ref{proisflexible}, results pseudoequivalent  (see \ref{flexible2=p}) to
$2$-$\cc{P}ro(\cc{C})$, with the same objects and retract equivalent hom categories. We think our choice of morphisms, which is much more convenient to use, will prove to be the good one for the applications. Nevertheless, this other version is unavoidable to prove that $\Pro{C}$ has a closed 2-bmodel structure (see section \ref{2-modelos}) due to the nature of the axioms of closed 2-bmodel 2-category where commutativities are non-strict but only holds up to invertible 2-cells.
\end{remark}

\begin{remark}
The assertion from \ref{eqconloscolimderep} also holds replacing $\Pro{C}$ for $\Prop{C}$ and $\cc{H}om(\cc{C},\cc{C}at)$ for $\cc{H}om_p(\cc{C},\cc{C}at)$. \cqd
\end{remark}

Next we establish the basic formula which is essential in many \mbox{computations} in the $2$-category $2$-$\cc{P}ro(\cc{C})$:
\begin{proposition}\label{iso}
There is an isomorphism of categories:
\addtocounter{equation}{-1}
\begin{equation} \label{basica2}
  \Pro{C}(\ff{X},\ff{Y})\cong \Lim{j\in \cc{J}}{\coLim{i\in \cc{I}}
  {\cc{C}(\ff{X}_i,\ff{Y}_j)}}
\end{equation}
\end{proposition}
\begin{proof}
\begin{multline*}
\Pro{C}(\ff{X},\ff{Y}) \;=\; \cc{H}om(\cc{C},\, \cc{C}at)(\coLim{j \in \cc{J}}{\cc{C}(\ff{Y}_j,-)}, \;\coLim{i \in \cc{I}}{\cc{C}(\ff{X}_i,-))} \;\cong\;
\\
\Lim{j\in \cc{J}}{\cc{H}om(\cc{C},\, \cc{C}at)(\cc{C}(\ff{Y}_j,-), \;\coLim{i \in \cc{I}}{\cc{C}(\ff{X}_i,-))}} \;\cong\;
\Lim{j\in \cc{J}}{\coLim{i\in \cc{I}}{\cc{C}(\ff{X}_i,\ff{Y}_j)}}
\end{multline*}

The first isomorphism is due to \ref{colimits} and the second one to \ref{2Yoneda}.
\end{proof}

\begin{remark}\label{formula en prop}
 In the case of $\Prop{C}$, formula \eqref{basica2} is an equivalence of categories instead of an isomorphism since the second $\cong$ is only an equivalence (see \ref{pseudoYoneda}). \cqd
\end{remark}

\begin{corollary}
The 2-category $\Pro{C}$ is locally small. \cqd
\end{corollary}

\begin{corollary} \label{CisPro}
There is a canonical 2-fully-faithful 2-functor
\mbox{$\cc{C}\mr{c} 2$-$\cc{P}ro(\cc{C})$} which sends an object of $\cc{C}$ into the corresponding \mbox{2-pro-object} with index \mbox{2-category} $\{*\}$. Since this 2-functor is also injective on objects, we can identify $\cc{C}$ with a 2-full subcategory of
$2\hbox{-}\cc{P}ro(\cc{C})$.
\cqd
\end{corollary}

Where there is no risk of confusion, we will omit to indicate notationally this identification. By the very definition of $2\hbox{-}\cc{P}ro(\cc{C})$ it follows:

\begin{proposition}\label{X=lim}
If $\ff{X}=(\ff{X}_i)_{i\in \cc{I}}$ is any 2-pro-object of $\cc{C}$, then
\mbox{$\ff{X}=\Lim{i\in \cc{I}}{\ff{X}_i}\;$} in \mbox{$2$-$\cc{P}ro(\cc{C})$.} $\ff{X}$ is equipped with a pseudo-cone structure, $\left\{\ff{X} \mr{\pi_i} \ff{X}_i\right\}_{i\in \cc{I}}$, $\left\{\ff{X}_u \, \pi_j\Mr{\pi_u} \pi_i\right\}_{i \mr{u} j \, \in \cc{I}}$.

 \vspace{1ex}

Under the isomorphism
$\Pro{C}(\ff{X},\, \ff{X}_i) \cong \coLim{k \in
\cc{I}}{\cc{C}(\ff{X}_k,\, \ff{X}_i})$ \eqref{basica2},
projections \mbox{$\ff{X} \mr{\pi_i} \ff{X}_i$}
correspond to objects $(id_{\ff{X}_i},\, i)$ in construction \ref{defLF}. \cqd

\end{proposition}

\begin{remark}
 The previous proposition also holds in $\Prop{C}$. \cqd
\end{remark}

Note that from proposition \ref{X=lim} it follows:

\begin{remark} \label{XisLim}
Given any two pro-objects $\ff{X}, \, \ff{Z} \, \in $ 2-$\cc{P}ro(\cc{C})$, there is an isomorphism of categories
2-$\cc{P}ro(\cc{C})(\ff{Z},\, \ff{X}) \mr{\cong}
\ff{PC}_{2\hbox{-}\cc{P}ro(\cc{C})}(\ff{Z},\, c\ff{X})$, where
$\ff{PC}_{2\hbox{-}\cc{P}ro(\cc{C})}$ is the category of pseudo-cones for the 2-functor $c\ff{X}$ with \mbox{vertex
$\ff{Z}$.}
\end{remark}
\vspace{1ex}

It is important to note that when $\Lim{i\in \cc{I}}{\ff{X}_i}$ exists in $\cc{C}$, this pseudo-limit would not be isomorphic to $\ff{X}$ in $2$-$\cc{P}ro(\cc{C})$. In general, the functor $c$ does not preserve 2-cofiltered pseudo-limits, in fact, it will preserve them only when $\cc{C}$ is already $\Pro{C}$, that is, when $c$ is an equivalence.

\subsection{Lemmas to compute with 2-pro-objects.}\label{lemas2pro}

In this subsection, we establish technical lemmas to be used in computations with 2-pro-objects.

\begin{definition} \label{representa} ${}$

\begin{enumerate}
 \item Let $\ff{X} \mr{\ff{f}} \ff{Y}$ be an arrow in $2$-$\cc{P}ro(\cc{C})$. We say that a pair $(\ff{r},\varphi)$ \emph{represents} $\ff{f}$, if $\varphi$ is an invertible $2$-cell $\ff{X} \cellrd{\ff{r} \, \pi_i}{\varphi}{\pi_j \, \ff{f}} \ff{Y}_j$. That is, if
we have the following diagram in $2$-$\cc{P}ro(\cc{C})$:
$$
\xymatrix@C=0.8pc@R=0.8pc
           {
            \ff{X} \ar@{}[ddrr]|{\cong \; \Downarrow \; \varphi} \ar[rr]^{\pi_i}  \ar[dd]_{\ff{f}}
            & & \ff{X}_i \ar[dd]^{\ff{r}}
           \\
            & 
            &
           \\ \ff{Y} \ar[rr]_{\pi_j}
            & & \ff{Y}_j
           }
$$

\vspace{-2ex}

\item Let $\ff{X} \cellrd{\ff{f}}{\alpha}{\ff{g}} \ff{Y} \in \Pro{C}$ and
$\;\ff{X}_i \cellrd{\ff{r}}{\theta}{\ff{s}} \ff{Y}_j \in \cc{C}$ as in the following diagram:
$$
\xymatrix@C=9ex@R=7ex
        {
         \ff{X}   \ar@<1.2ex>[d]^{\ff{g}}
             \ar@{}@<-1.3ex>[d]^{\stackrel{\alpha}{\Rightarrow}}
             \ar@<-1.1ex>[d]_{\ff{f}}
             \ar[r]^{\pi_i}
         & \ff{X}_i \ar@<1.2ex>[d]^{\ff{s}}
             \ar@{}@<-1.3ex>[d]^{\stackrel{\theta}{\Rightarrow}}
             \ar@<-1.1ex>[d]_{\ff{r}}
        \\
         \ff{Y} \ar[r]_{\pi_j}             
         & \ff{Y}_j
        }
$$
We say that $(\theta, \ff{r}, \varphi, \ff{s}, \psi)$ \emph{represents} $\alpha$ if $(\ff{r},\varphi)$ represents $\ff{f}$, $(\ff{s},\psi)$ represents $\ff{g}$,
and the following equality holds in $2$-$\cc{P}ro(\cc{C})$:
$$\vcenter{\xymatrix@C=0ex
          {
           \ff{r} \dcell{\theta} && \pi_i \deq
          \\
           \ff{s} \dl & \dc{\psi} & \pi_i \dr
          \\
           \pi_j && \ff{g} 
          }}
\vcenter{\xymatrix{\comw{\pi_j \ff{f}} \\ = \\ \comw{\pi_j \ff{f}} }}
\vcenter{\xymatrix@C=0ex
          {
           \ff{r} \dl & \dc{\varphi} & \dr \pi_i
          \\
           \pi_j  \deq && \ff{f} \dcell{\alpha} 
          \\
           \pi_j && \ff{g}
          }}$$

          
That is, $\theta \pi_i = \pi_j \alpha$  ``modulo'' a pair of invertible 2-cells $\varphi,\, \psi$.

\vspace{1ex}


\end{enumerate}

\end{definition}

\begin{remark}
Same definitions may be given in $\Prop{C}$.  \cqd
\end{remark}

\begin{proposition} \label{idrepresenta} ${}$
Let $\ff{X} = (\ff{X}_i)_{i\in \cc{I}}$ and
$\ff{Y} = (\ff{Y}_j)_{j\in \cc{J}}$ be any two 2-pro-objects.

\begin{enumerate}

\item Let $\ff{X} \mr{\ff{f}} \ff{Y} \in \Pro{C}$. Then, for any $j \in \cc{J}$ there exist
$i \in \cc{I}$ and $\ff{X}_i \mr{\ff{r}} \ff{Y}_j \in \cc{C}$, such that $(\ff{r},id)$ represents $\ff{f}$.

\item Let $\ff{X} \cellrd{\ff{f}}{\alpha}{\ff{g}} \ff{Y} \in \Pro{C}$. Then, for any $j \in \cc{J}$ there exist
$i \in \cc{I}$, $\ff{X}_i \cellrd{\ff{r}}{\theta}{\ff{s}} \ff{Y}_j \in \cc{C}$, and \mbox{appropriate} invertible 2-cells $\varphi$ and $\psi$ such that $(\theta,\ff{r},\varphi,\ff{s},\psi)$ represents $\alpha$.

Observe that in case $\alpha$ is invertible, one can choose a representative with an invertible $\theta$.
\end{enumerate}

\end{proposition}

\begin{proof}

Consider $\ff{X}\cellrd{\pi_j \ff{f}}{\pi_j\alpha}{\pi_j \ff{g}}\ff{Y}_j$ and use formula \ref{iso} plus the constructions of pseudo-limits (\ref{limincat}) and  2-filtered pseudo-colimits of categories (\ref{defLF}).
\end{proof}

 From the previous proposition plus the pseudo-equivalence of \ref{proppseudoeqapro}, it follows:

 \begin{proposition}\label{idrepresentap}
Let $\ff{X} = (\ff{X}_i)_{i\in \cc{I}}$ and
$\ff{Y} = (\ff{Y}_j)_{j\in \cc{J}}$ be any two 2-pro-objects.

\begin{enumerate}

\item Let $\ff{X} \mr{\ff{f}} \ff{Y} \in \Prop{C}$. Then, for any $j \in \cc{J}$ there exist
$i \in \cc{I}$, $\ff{X}_i \mr{\ff{r}} \ff{Y}_j \in \cc{C}$ and an invertible 2-cell $\varphi$, such that $(\ff{r},\varphi)$ represents $\ff{f}$.

\item Let $\ff{X} \cellrd{\ff{f}}{\alpha}{\ff{g}} \ff{Y} \in \Prop{C}$. Then, for any $j \in \cc{J}$ there exist
$i \in \cc{I}$, $\ff{X}_i \cellrd{\ff{r}}{\theta}{\ff{s}} \ff{Y}_j \in \cc{C}$, and \mbox{appropriate} invertible 2-cells $\varphi$ and $\psi$ such that $(\theta,\ff{r},\varphi,\ff{s},\psi)$ represents $\alpha$.

\end{enumerate}\cqd
\end{proposition}

\begin{lemma}\label{lema1}
Let
$\ff{X}=(\ff{X}_i)_{i\in \cc{I}}$ be a 2-pro-object, let $\ff{X}_i\mr{\ff{r}} \ff{C}$,
\mbox{$\ff{X}_{j}\mr{\ff{s}} \ff{C}\in \cc{C}$}, and
$\ff{X} \cellrd{\ff{r} \pi_i}{\alpha}{\ff{s} \pi_{j}} \ff{C}\in 2$-$\cc{P}ro(\mathcal{C})$. Then, $\exists \vcenter{\xymatrix@R=-0.5pc{i \ar[rd]^{u} &\\& k \\ j\ar[ru]_{v}}} \in \cc{I}$ and $\ff{X}_{k} \cellrd{\ff{r} \ff{X}_u}{\theta}{\ff{s} \ff{X}_v} \ff{C} \in \cc{C}$ such that \mbox{$\alpha \circ \fr \pi_u=\s \pi_v \circ \theta \pi_k$} in $\Pro{C}$:

$$ \vcenter{\xymatrix@R=.75pc@C=2.5pc{ & \X_k \ar[dd]^{\X_v} \ar[r]^{\X_u} \ar@{}[ddl]|{\Downarrow \; \pi_v} \ar@{}[ddr]|{\Downarrow \ \theta} & \X_i \ar[dd]_{\fr} \\
X \ar[ur]^{\pi_k} \ar[dr]_{\pi_j} \\
& \X_j \ar[r]_{\s} & \C }}
\vcenter{\xymatrix{ \quad = \quad }}
\vcenter{\xymatrix@C=2.5pc{ & \X_k \ar@/_3ex/[dl]_{\pi_k}  \ar@{}[dl]|{\stackrel{\pi_u}{\Rightarrow}} \ar[d]^{\X_u} \\
\X \ar[d]_{\pi_j} \ar@{}[dr]|{\Downarrow \; \alpha} \ar[r]^{\pi_i} & \X_i \ar[d]^{\fr} \\
\X_j \ar[r]_{\s} & \C }}
\vcenter{\xymatrix@C=-.4pc{\quad \quad i.e. \quad \quad  }}  $$


$$\vcenter{\xymatrix@C=-0pc{ \fr \dl & \dc{\theta} & \dr \X_u && \pi_k \deq \\
\s \deq & &  \X_v && \pi_k \\
\s &&& \pi_j \cl{\pi_v}  }}
\vcenter{\xymatrix@C=-0pc{ \quad = \quad }}
\vcenter{\xymatrix@C=-0pc{ \fr \deq & &  \X_u && \pi_k \\
\fr \dl & \dcr{\alpha} && \dr \pi_i \cl{\pi_u} \\
\s &&& \pi_j }}
$$  
Observe that in case $\alpha$ is invertible, one can choose $\theta$ to be invertible.
\end{lemma}

\begin{proof}
By formula $\ref{iso}$ and the construction of 2-filtered pseudo-colimits (\ref{defLF}), $\alpha$ is given by $(\ff{r},i)\mr{[u,\theta,v]} (\ff{s},j)\in \coLim{i\in \cc{I}}{\cc{C}(\ff{X}_i,\ff{C})}$ . Thus, $\exists \vcenter{\xymatrix@R=-0.5pc{i \ar[rd]^{u} &\\& k \\ j\ar[ru]_{v}}} \in \cc{I}$ and $\ff{X}_{k} \cellrd{\ff{r} \ff{X}_u}{\theta}{\ff{s} \ff{X}_v} \ff{C} \in \cc{C}$ such that \mbox{$\alpha \circ \fr \pi_u=\s \pi_v \circ \theta \pi_k$}, as we wanted to prove.


\end{proof}

The following is an immediate consequence of \cite[Lemma 2.2.]{DS}

\begin{remark} \label{u=v}
If $i=j$, then one can choose $u=v$. \cqd
\end{remark}

\begin{remark}\label{lema1p}
From the previous proposition plus the pseudo-equivalence of \ref{proppseudoeqapro}, it follows that the previous lemma also holds in $\Prop{C}$, and so also \ref{lemorphismaraM_f} and \ref{lema2}. \cqd
\end{remark}

\vspace{1ex}

The following two lemmas will be used to prove reindexing properties of \mbox{2-pro-objects} in section \ref{Mtrick} and will be also needed in section \ref{2-modelos}:

\begin{lemma}\label{lemorphismaraM_f}
 Let $\ff{X}\mr{\ff{f}} \ff{Y}$ be a morphism in $\Pro{C}$, $\ff{X}_i\mr{\ff{r}}\ff{Y}_j$, $\ff{X}_{i'}\mr{\ff{s}}\ff{Y}_{j'}\in \cc{C}$ and $\varphi$, $\psi$ invertible 2-cells in $\cc{C}$ such that $(\ff{r},\varphi)$ and $(\ff{s},\psi)$ both represent $\ff{f}$ and there are morphisms $i\mr{u} i'\in \cc{I}$, $j\mr{a}j' \in \cc{J}$. Then there are morphisms $\vcenter{\xymatrix@R=-0.5pc{i\ar[rd]^v \\
                           & i''\\
                           i' \ar[ru]_w}} \in \cc{I}$ 
and an invertible 2-cell $\vcenter{\xymatrix@R=1pc@C=-.5pc{\ff{Y}_{a} \ar@{-}[dr] &&\ff{s} \dc{\theta} & &\ff{X}_w \ar@{-}[dl] \\
& \ff{r} & & \ff{X}_v }}$ $\in \cc{C} $ such that $\varphi \circ \fr \pi_v \circ \theta \pi_{i''}=\pi_a \f \circ \Y_a \psi \circ \Y_a \s \pi_w$, i.e.


$$\vcenter{\xymatrix@C=-0pc{ \Y_a \ar@{-}[dr] && \s \dc{\theta} && \X_w \ar@{-}[dl] & \pi_{i''} \deq \\
& \fr \deq && \X_v  && \pi_{i''} \\
& \fr \dl & \dcr{\varphi} && \dr \pi_i \cl{\pi_v} \\
& \pi_j &&& \f }}
\vcenter{\xymatrix@C=-0pc{ \quad = \quad }}
\vcenter{\xymatrix@C=-0pc{ \Y_a \deq && \s \deq && \X_w && \pi_{i''} \\
\Y_a \deq && \s \dl & \dcr{\psi} && \dr \pi_{i'} \cl{\pi_w} \\
\Y_a && \pi_{j'} &&& \f \deq \\
& \pi_j \cl{\pi_a} &&&& \f}}$$


\end{lemma}

\begin{proof}
 In \ref{lema1}, take $\ff{X}:=\ff{X}$, $\ff{X}_i:=\ff{X}_i$, $\ff{X}_j:=\ff{X}_{i'}$, $\ff{C}:=\ff{Y}_j$, $\ff{r}:=\ff{r}$, $\ff{r}:=\ff{Y}_a \ff{s}$ and $\alpha:=
 \vcenter{\xymatrix@R=1.2pc@C=-0.1pc{ & \fr \dl & \dcr{\varphi} && \dr \pi_i \\
 & \pi_j \op{\pi_a\inv} &&& \f \deq \\
 \Y_a \deq && \pi\sj \dl & \dc{\psi\inv} & \dr \f \\
 \Y_a && \s && \pi\si }}$ .
\end{proof}

\begin{lemma}\label{lemorphismaraM_D}
 Let $\left\{\ff{X}\mr{\ff{f}_l} \ff{Y}^l\right\}_{l=1,...,k}$ be a finite family of morphisms in $\Pro{C}$ with $\ff{X}$ indexed by $\cc{I}$ and for every $l=1,...,k$, $\ff{Y}^l$ indexed by $\cc{J}$. Consider $\left\{\ff{X}_i\mr{\ff{r}_l}\ff{Y}^l_j\right\}_{l=1,...,k}$, $\left\{\ff{X}_{i'}\mr{\ff{s}_l}\ff{Y}^l_{j'}\right\}_{l=1,...,k}\in \cc{C}$ and invertible 2-cells $\left\{\varphi_l\right\}_{l=1,...k}$, $\left\{\psi_l\right\}_{l=1,...k}$ in $\cc{C}$ such that $(\ff{r}_l,\varphi_l)$ and $(\ff{s}_l,\psi_l)$ both represent $\ff{f}_l\ \forall \ l=1,...,k$ and there are morphisms \mbox{$i\mr{u} i'\in \cc{I}$, $j\mr{a}j' \in \cc{J}$.} Then there are morphisms $\vcenter{\xymatrix@R=-0.5pc{i\ar[rd]^v \\
                           & i''\\
                           i' \ar[ru]_w}} \in \cc{I}$ 
and invertible \mbox{2-cells} $\vcenter{\xymatrix@R=1pc@C=-.5pc{\ff{Y}^l_{a} \ar@{-}[dr] &&\ff{s}_l \dc{\theta_l} & &\ff{X}_w \ar@{-}[dl] \\
                                                     & \ff{r}_l & & \ff{X}_v }} \in \cc{C} \ \forall \ l=1,...k$ such that for each $l=1,...k$, \mbox{$\varphi_l \circ \fr_l \pi_v \circ \theta_l \pi_{i''}=\pi_a \f \circ \Y_a^l \psi_l \circ \Y_a^l \s_l \pi_w$,} i.e.


$$\vcenter{\xymatrix@C=-0pc{ \Y_a^l \ar@{-}[dr] && \s_l \dc{\theta_l} && \X_w \ar@{-}[dl] & \pi_{i''} \deq \\
& \fr_l \deq && \X_v  && \pi_{i''} \\
& \fr_l \dl & \dcr{\varphi_l} && \dr \pi_i \cl{\pi_v} \\
& \pi_j &&& \f_l }}
\vcenter{\xymatrix@C=-0pc{ \quad = \quad }}
\vcenter{\xymatrix@C=-0pc{ \Y_a^l \deq && \s_l \deq && \X_w && \pi_{i''} \\
\Y_a^l \deq && \s_l \dl & \dcr{\psi_l} && \dr \pi_{i'} \cl{\pi_w} \\
\Y_a^l && \pi_{j'} &&& \f_l \deq \\
& \pi_j \cl{\pi_a} &&&& \f_l}}$$


\end{lemma}

\begin{proof}
We are going to proceed by induction in $k$. 
For $k=1$, use \ref{lemorphismaraM_f}.

\noindent $k \Rightarrow k+1$:
by inductive hypothesis, $\exists \ \vcenter{\xymatrix@R=-0.5pc{i \ar[rd]^{v_0} \\
                                                                      & i_0 \\
                                                                      i' \ar[ru]_ {w_0}}} \in \cc{I}$ and, for each $l=1,...k$, an invertible 2-cell $\vcenter{\xymatrix@R=1pc@C=-.5pc{\ff{Y}^l_{a} \ar@{-}[dr] &&\ff{s}_l \dc{\tilde{\theta}_l} & &\ff{X}_{w_0} \ar@{-}[dl] \\
                                                     & \ff{r}_l & & \ff{X}_{v_0} }} \in \cc{C}$ such that 

\begin{equation}\label{hastak}
\vcenter{\xymatrix@C=-0pc{ \Y_a^l \ar@{-}[dr] && \s_l \dc{\tilde{\theta}_l} && \X_{w_0} \ar@{-}[dl] & \pi_{i_0} \deq \\
& \fr_l \deq && \X_{v_0}  && \pi_{i_0} \\
& \fr_l \dl & \dcr{\varphi_l} && \dr \pi_i \cl{\pi_{v_0}} \\
& \pi_j &&& \f_l }}
\vcenter{\xymatrix@C=-0pc{ \quad = \quad }}
\vcenter{\xymatrix@C=-0pc{ \Y_a^l \deq && \s_l \deq && \X_{w_0} && \pi_{i_0} \\
\Y_a^l \deq && \s_l \dl & \dcr{\psi_l} && \dr \pi_{i'} \cl{\pi_{w_0}} \\
\Y_a^l && \pi_{j'} &&& \f_l \deq \\
& \pi_j \cl{\pi_a} &&&& \f_l}}
\end{equation}
                            
\noindent Also, $\exists \ \vcenter{\xymatrix@R=-0.5pc{i \ar[rd]^{v_1} \\
                                                                      & i_1 \\
                                                                      i' \ar[ru]_ {w_1}}} \in \cc{I}$ and an invertible 2-cell $\vcenter{\xymatrix@R=1pc@C=-.8pc{\ff{Y}^{k+1}_{a} \ar@{-}[dr] &&\ff{s}_{k+1} \dc{\tilde{\theta}_{k+1}} & &\ff{X}_{w_1} \ar@{-}[dl] \\
                                                     & \ff{r}_{k+1} & & \ff{X}_{v_1} }} \in \cc{C}$ such that

\begin{equation}\label{k+1}
\vcenter{\xymatrix@C=-0pc{ \Y_a^{k+1} \ar@{-}[dr] && \s_{k+1} \dc{\tilde{\theta}_{k+1}} && \X_{w_1} \ar@{-}[dl] & \pi_{i_1} \deq \\
& \fr_{k+1} \deq && \X_{v_1}  && \pi_{i_1} \\
& \fr_{k+1} \dl & \dcr{\varphi_{k+1}} && \dr \pi_i \cl{\pi_{v_1}} \\
& \pi_j &&& \f_{k+1} }}
\vcenter{\xymatrix@C=-0pc{ \quad = \quad }}
\vcenter{\xymatrix@C=-0pc{ \Y_a^{k+1} \deq && \s_{k+1} \deq && \X_{w_1} && \pi_{i_1} \\
\Y_a^{k+1} \deq && \s_{k+1} \dl & \dcr{\psi_{k+1}} && \dr \pi_{i'} \cl{\pi_{w_1}} \\
\Y_a^{k+1} && \pi_{j'} &&& \f_{k+1} \deq \\
& \pi_j \cl{\pi_a} &&&& \f_{k+1}}}
\end{equation}  

\noindent Since $\cc{I}$ is 2-filtered, $\exists \ \vcenter{\xymatrix@R=-0.5pc{i_0 \ar[rd]^{v_2} \\
                                                                      & i_2 \\
                                                                      i_1 \ar[ru]_ {w_2}}}$, morphisms $i_2\mr{u_0}i_3$, $i_2\mr{u_1}i_4$ and \mbox{invertible} 2-cells $i \cellrd{ w_2 v_1}{\mu_0}{u_0 v_2 v_0} i_3$, $i'\cellrd{ v_2 w_0}{\mu_1}{u_1 w_2 w_1} i_4 \in \cc{I}$, $\vcenter{\xymatrix@R=-0.5pc{i_3 \ar[rd]^{v_3} \\
                                                                      & i_5 \\
                                                                      i_4 \ar[ru]_ {w_3}}} $, a morphism $i_5\mr{u_2} i''$ and an \mbox{invertible} 2-cell $i_2 \cellrd{ w_3 u_1}{\mu_2}{u_2 v_3 u_0} i'' \in \cc{I}$.
\vspace{1ex}

Consider $v=u_2 v_3 u_0 v_2 v_0$ and $w=u_2 w_3 u_1 v_2 w_0$, for each $l=1,..,k$ 
$$\theta_l= \!\!\!\!\! 
\vcenter{\xymatrix@C=-0.3pc{\ff{Y}^l_{a} \ar@{-}[dr] &&\ff{s}_l \dc{\tilde{\theta}_l} && \ff{X}_{w_0} \ar@{-}[dl] && \X_{v_2} \deq && \X_{u_1} \deq && \X_{w_3} \deq && \X_{u_2} \deq \\
                           & \ff{r}_l \deq && \ff{X}_{v_0} \deq &&& \X_{v_2} \deq && \X_{u_1} \dl && \X_{w_3} \dc{\X_{\mu_2}} && \dr \X_{u_2} \\
			   & \fr_l && \X_{v_0} &&& \X_{v_2} && \X_{u_0} && \X_{v_3} && \X_{u_2} }}
\mbox{ and } \theta_{k+1}= \!\!\!\!\!\! 
\vcenter{\xymatrix@C=-0.3pc{ \Y_a^{k+1} \deq && \!\!\! \s_{k+1} \!\!\! \deq && \X_{w_0} \dl && \X_{v_2} \dc{\X_{\mu_1}} && \X_{u_1} \dr&&  \X_{w_3} \deq && \X_{u_2} \deq \\
\Y_a^{k+1} \ar@{-}[dr] && \!\!\! \s_{k+1} \!\!\! \dc{\tilde{\theta}_{k+1}} && \X_{w_1} \ar@{-}[dl] && \X_{w_2} \deq && \X_{u_1} \deq && \Xwt \deq && \Xud \deq \\
& \fr\kk \deq && \Xvu \deq &&& \Xwd \deq && \Xuu \dl && \Xwt \dc{\X_{\mu_2}} && \Xud \dr \\
& \fr\kk \deq && \Xvu \dl && \dcr{\X_{\mu_0}} & \Xwd && \Xuc \dr && \Xvt \deq && \Xud \deq \\
& \fr\kk && \Xvc &&& \Xvd && \Xuc && \Xvt && \Xud}}$$

Let's check that this data satisfies the desired property:

For $l=1,...,k$: 


$$\vcenter{\xymatrix@C=-0.3pc{  
\ff{Y}_a^l \ar@{-}[dr] &&  \ff{s}_l \dc{\tilde{\theta}_l} && \ff{X}_{w_0} \ar@{-}[dl] && \ff{X}_{v_2} \did && \ff{X}_{u_1} \did && \ff{X}_{w_3} \did && \ff{X}_{u_2} \did && \pi_{i''} \did  \\
& \ff{r}_l \did && \ff{X}_{v_0} \did &&& \ff{X}_{v_2} \did && \ff{X}_{u_1} \dl && \ff{X}_{w_3} \dc{\ff{X}_{\mu_2}} && \ff{X}_{u_2} \dr && \pi_{i''} \did  \\
& \ff{r}_l \did && \ff{X}_{v_0} \did &&& \ff{X}_{v_2} \did && \ff{X}_{u_0} \did && \ff{X}_{v_3} \did && \ff{X}_{u_2} && \pi_{i''}   \\
& \ff{r}_l \did && \ff{X}_{v_0} \did &&& \ff{X}_{v_2} \did && \ff{X}_{u_0} \did && \ff{X}_{v_3} \ar@{-}[drr] && \dc{\pi_{v_3}} & \pi_{i_5} \cl{\pi_{u_2}} \ar@{-}[dl] \\
& \ff{r}_l \did && \ff{X}_{v_0} \did &&& \ff{X}_{v_2} \did && \ff{X}_{u_0} \ar@{-}[drr] && \dc{\pi_{u_0}} && \pi_{i_3} \ar@{-}[dll] \\
& \ff{r}_l \did && \ff{X}_{v_0} \did &&& \ff{X}_{v_2} \ar@{-}[drr] && \dc{\pi_{v_2}} && \pi_{i_2} \ar@{-}[dll] \\
& \ff{r}_l \did && \ff{X}_{v_0} \ar@{-}[drrr] && \dcr{\pi_{v_0}} &&& \pi_{i_0} \ar@{-}[dll] \\
& \ff{r}_l \dl && \dcr{\varphi_l} &&& \dr \pi_i \\
& \pi_j &&&&& \f_l}}
\vcenter{\xymatrix@C=-.3pc{\ \ = \ \ \ \ } }
\vcenter{\xymatrix@C=-0.3pc{
\Y_a^l \deq && \s_l \deq && \Xwc \deq && \Xvd \deq && \Xuu \deq && \Xwt \deq && \Xud && \pi_{i''} \\
\Y_a^l \deq && \s_l \deq && \Xwc \deq && \Xvd \deq && \Xuu \deq && \Xwt \ar@{-}[drr] && \dc{\pi_{w_3}} & \ar@{-}[dl] \pi_{i_5} \cl{\pi_{u_2}} \\
\Y_a^l \deq && \s_l \deq && \Xwc \deq && \Xvd \deq && \Xuu \ar@{-}[drr] && \dc{\pi_{u_1}} && \ar@{-}[dll] \pi_{i_4} \\
\Y_a^l \deq && \s_l \deq && \Xwc \deq && \Xvd \ar@{-}[drr] && \dc{\pi_{v_2}} && \ar@{-}[dll] \pi_{i_2} \\
\Y_a^l \deq && \s_l \deq && \Xwc \ar@{-}[drr] && \dc{\pi_{w_0}} && \ar@{-}[dll] \pi_{i_0} \\
\Y_a^l \deq && \s_l \dl && \dc{\psi_l} && \dr \pi_{i'} \\
\Y_a^l && \pi_{j'} &&&& \f_l \deq \\
& \pi_j \cl{\pi_a} &&&&& \f_l}}$$

\vspace{2ex}

\noindent Where the equality is due to \eqref{hastak} plus axiom PC2 of pseudo-cones.


\vspace{2ex}

\noindent For $l=k+1$:


\vspace{-4cm}

$$\vcenter{\xymatrix@R=.9pc@C=-0.3pc{
\Y_a^{k+1} \deq && \!\!\!\!\! \s_{k+1} \!\!\!\!\! \deq && \Xwc \dl && \Xvd \dc{\X_{\mu_1}} && \dr \Xuu && \Xwt \deq && \Xud \deq && \pi_{i''} \deq \\
\Y_a^{k+1} \ar@{-}[dr] && \!\!\!\!\! \s_{k+1} \!\!\!\!\! \dc{\tilde{\theta}\kk} && \Xwu \ar@{-}[dl] && \Xwd \deq && \Xuu \deq && \Xwt \deq && \Xud \deq && \pi_{i''} \deq \\
& \fr\kk \deq && \Xvu \deq &&& \Xwd \deq && \Xuu \dl && \Xwt \dc{\X_{\mu_2}} && \dr \Xud && \pi_{i''} \deq \\
& \fr\kk \deq && \Xvu \dl && \dcr{\X_{\mu_0}} & \Xwd && \dr \Xuc && \Xvt \deq && \Xud \deq && \pi_{i''} \deq  \\
& \fr\kk \deq && \Xvc \deq &&& \Xvd \deq && \Xuc \deq && \Xvt \deq && \Xud && \pi_{i''} \\
& \fr\kk \deq && \Xvc \deq &&& \Xvd \deq && \Xuc \deq && \Xvt \ardrr & \dcr{\pi_{v_3}} && \ardl \!\!\! \pi_{i_5} \!\!\! \clb{\pi_{u_2}} \\
& \fr\kk \deq && \Xvc \deq &&& \Xvd \deq && \Xuc \ardrr && \dc{\pi_{u_0}} && \ardll \pi_{i_3} \\
& \fr\kk \deq && \Xvc \deq &&& \Xvd \ardrr && \dc{\pi_{v_2}} && \ardll \pi_{i_2} \\
& \fr\kk \deq && \Xvc \ardrr &&& \dc{\pi_{v_0}} && \ardll \pi_{i_0} \\
& \fr\kk \dl && \dcr{\varphi\kk} &&& \dr \pi_i \\
& \pi_j &&&&& \f\kk }}
\vcenter{\xymatrix@C=-.3pc{ = } }
\vcenter{\xymatrix@R=.9pc@C=-0.3pc{ \Y_a^{k+1} \deq && \!\!\!\!\! \s_{k+1} \!\!\!\!\! \deq && \Xwc \dl && \Xvd \dc{\X_{\mu_1}} && \dr \Xuu && \Xwt \deq && \Xud \deq && \pi_{i''} \deq \\
\Y_a^{k+1} \ar@{-}[dr] && \!\!\!\!\! \s_{k+1} \!\!\!\!\! \dc{\tilde{\theta}\kk} && \Xwu \ar@{-}[dl] && \Xwd \deq && \Xuu \deq && \Xwt \deq && \Xud \deq &&  \pi_{i''} \deq \\
& \fr\kk \deq && \Xvu \deq &&& \Xwd \deq  && \Xuu \dl && \Xwt \dc{\X_{\mu_2}} && \dr \Xud &&  \pi_{i''} \deq \\
& \fr\kk \deq && \Xvu \deq &&& \Xwd \deq  && \Xuc \deq && \Xvt \deq && \Xud && \pi_{i''} \\
& \fr\kk \deq && \Xvu \deq &&& \Xwd \deq  && \Xuc \deq && \Xvt \ardrr & \dcr{\pi_{v_3}} && \ardl \!\!\! \pi_{i_5} \!\!\! \clb{\pi_{u_2}} \\
& \fr\kk \deq && \Xvu \deq &&& \Xwd \deq  && \Xuc \ardrr && \dc{\pi_{u_0}} && \ardll \pi_{i_3} \\
& \fr\kk \deq && \Xvu \deq &&& \Xwd \ardrr && \dc{\pi_{w_2}} && \ardll \pi_{i_2} \\
& \fr\kk \deq && \Xvu \ardrrr && \dcr{\pi_{v_1}} &&& \ardll \pi_{i_1} \\
& \fr\kk \dl && \dcr{\varphi\kk} &&& \dr \pi_i \\
& \pi_j &&&&& \f\kk }}
= $$

\vspace{-1cm}

$$
\vcenter{\xymatrix@R=.9pc@C=-0.3pc{ \Y_a^{k+1} \deq && \!\!\!\!\! \s_{k+1} \!\!\!\!\! \deq && \Xwc \dl && \Xvd \dc{\X_{\mu_1}} && \dr \Xuu && \Xwt \deq && \Xud \deq && \pi_{i''} \deq \\
\Y_a^{k+1} \ar@{-}[dr] && \!\!\!\!\! \s_{k+1} \!\!\!\!\! \dc{\tilde{\theta}\kk} && \Xwu \ar@{-}[dl] && \Xwd \deq && \Xuu \deq && \Xwt \deq && \Xud \deq && \pi_{i''} \deq \\
& \fr\kk \deq && \Xvu \deq &&&  \Xwd \deq && \Xuu \deq && \Xwt \deq && \Xud && \pi_{i''} \\
& \fr\kk \deq && \Xvu \deq &&&  \Xwd \deq && \Xuu \deq && \Xwt \ardrr & \dcr{\pi_{w_3}} && \ardl \!\!\! \pi_{i_5} \!\!\!  \clb{\pi_{u_2}} \\
= & \fr\kk \deq && \Xvu \deq &&&  \Xwd \deq && \Xuu \ardrr && \dc{\pi_{u_1}} && \ardll \pi_{i_4} \\
& \fr\kk \deq && \Xvu \deq &&& \Xwd \ardrr && \dc{\pi_{w_2}} && \ardll \pi_{i_2} \\
& \fr\kk \deq && \Xvu \ardrrr && \dcr{\pi_{v_1}} &&& \ardll \pi_{i_1} \\
& \fr\kk \dl && \dcr{\varphi\kk} &&& \dr \pi_i \\
& \pi_j &&&&& \f\kk }}
\vcenter{\xymatrix@C=-.3pc{ = } }
\vcenter{\xymatrix@R=.9pc@C=-0.3pc{ \Y_a^{k+1} \deq && \!\!\!\!\! \s_{k+1} \!\!\!\!\! \deq && \Xwc \dl && \Xvd \dc{\X_{\mu_1}} && \dr \Xuu && \Xwt \deq && \Xud \deq && \pi_{i''} \deq \\
\Y_a^{k+1} \deq && \!\!\!\!\! \s_{k+1} \!\!\!\!\! \deq && \Xwu \deq && \Xwd \deq && \Xuu \deq && \Xwt \deq && \Xud && \pi_{i''}  \\
\Y_a^{k+1} \deq && \s_{k+1} \deq && \Xwu \deq && \Xwd \deq && \Xuu \deq && \Xwt \ardrr && \dc{\pi_{w_3}} & \ardl \!\!\! \pi_{i_5} \!\!\! \clb{\pi_{u_2}} \\
\Y_a^{k+1} \deq && \s_{k+1} \deq && \Xwu \deq && \Xwd \deq && \Xuu \ardrr && \dc{\pi_{u_1}} && \ardll \pi_{i_4} \\
\Y_a^{k+1} \deq && \s_{k+1} \deq && \Xwu \deq && \Xwd \ardrr && \dc{\pi_{w_2}} && \ardll \pi_{i_2} &&&& \\
\Y_a^{k+1} \deq && \s_{k+1} \deq && \Xwu \ardrr && \dc{\pi_{w_1}} && \ardll \pi_{i_1} \\
\Y_a^{k+1} \deq && \s\kk \dl && \dc{\psi\kk} && \dr \pi_{i'} \\
\Y_a^{k+1} && \pi_{j'} &&&& \f\kk \deq \\
& \pi_j \cl{\pi_a} &&&&& \f\kk }} 
= $$

$$\vcenter{\xymatrix@C=-.3pc{ = } }
\vcenter{\xymatrix@R=1pc@C=-0.3pc{ \Y_a^{k+1} \deq && \s_{k+1} \deq && \Xwc \deq && \Xvd \deq && \Xuu \deq && \Xwt \deq && \Xud && \pi_{i''} \\
\Y_a^{k+1} \deq && \s_{k+1} \deq && \Xwc \deq && \Xvd \deq && \Xuu \deq && \Xwt \ardrr && \dc{\pi_{w_3}} & \ardl \pi_{i_5} \cl{\pi_{u_2}} \\
\Y_a^{k+1} \deq && \s_{k+1} \deq && \Xwc \deq && \Xvd \deq && \Xuu \ardrr && \dc{\pi_{u_1}} && \ardll \pi_{i_4} \\
\Y_a^{k+1} \deq && \s_{k+1} \deq && \Xwc \deq && \Xvd \ardrr && \dc{\pi_{v_2}} && \ardll \pi_{i_2} \\
\Y_a^{k+1} \deq && \s_{k+1} \deq && \Xwc \ardrr && \dc{\pi_{w_0}} && \ardll \pi_{i_0} \\
\Y_a^{k+1} \deq && \s\kk \dl && \dc{\psi\kk} && \dr \pi_{i'} \\
\Y_a^{k+1} && \pi_{j'} &&&& \f\kk \deq \\
& \pi_j \cl{\pi_a} &&&&& \f\kk }}$$

\noindent Where the first, the second and the last equalities are due to elevators calculus plus axiom PC2 of pseudo-cones and the third one holds by elevators calculus plus \eqref{k+1}.
\end{proof}

\begin{lemma}\label{lema4}

Let $\ff{X}=(\ff{X}_i)_{i\in \cc{I}}$ be a 2-pro-object and $\ff{X}_i\cellpairrd{\ff{f}}{\theta}{\theta'}{\ff{g}}\ff{C} \in \cc{C}$ such that $\theta\pi_i=\theta'\pi_i$ in $2$-$\cc{P}ro(\cc{C})$. Then $\exists \ i\mr{u} i' \in \cc{I}$ such that $\theta \ff{X}_u=\theta' \ff{X}_u$.
\end{lemma}

\begin{proof}
It follows from \ref{basica2} and [\citealp{DS}, Lemma 1.20.]
\end{proof}

The following lemma will also be used in section \ref{Mtrick}:

\begin{lemma}\label{lema2paraM_D}
Let $\ff{X}=(\ff{X}_i)_{i\in \cc{I}}$ be a 2-pro-object and $\Bigg\{\ff{X}_i\cellpairrd{\ff{f}_l}{\theta_l}{\theta'_l}{\ff{g}_l}\ff{C}\Bigg\}_{l=1,...k} \in \cc{C}$ be such that $\theta_l\pi_i=\theta'_l\pi_i \ \forall \ l=1,...k$ in $2$-$\cc{P}ro(\cc{C})$. Then $\exists \ i\mr{u} i' \in \cc{I}$ such that $\theta_l \ff{X}_u=\theta'_l \ff{X}_u$ \mbox{$\forall \ l=1,...k$.}
\end{lemma}

\begin{proof}
We are going to proceed by induction in $k$. 
For $k=1$, use \ref{lema4}.

\vspace{1ex}

\noindent $k \Rightarrow k+1$: 
By inductive hypothesis, $\exists \ i\mr{u_0} i_0 \in \cc{I}$ such that $\theta_l \ff{X}_{u_0}=\theta'_l \ff{X}_{u_0} \ \forall \ l=1,...,k$ and $\exists \ i\mr{u_1} i_1 \in \cc{I}$ such that $\theta_{k+1} \ff{X}_{u_1}=\theta'_{k+1} \ff{X}_{u_1}$.

Since $\cc{I}$ is 2-filtered, we have morphisms $i_0 \mr{v_0} i'$, $i_1 \mr{v_1} i'$ and an invertible 2-cell $i \cellrd{v_0 u_0}{\mu}{ v_1 u_1} i' \in \cc{I}$. 

It is easy to check that $u=v_0 u_0$ satisfies the desired property. 
\end{proof}

\begin{remark}
The previous two lemmas (and so the following one) also hold in $\Prop{C}$ (have in mind the pseudo-equivalence of \ref{proppseudoeqapro}). \cqd
\end{remark}

\begin{lemma}\label{lema5}
Let $\ff{X} \cellrd{\ff{f}}{\alpha}{\ff{g}} \ff{Y}$ in $2$-$\cc{P}ro(\cc{C})$ and
$\ff{X}_i \cellpairrd{\ff{r}}{\theta}{\theta'}{\ff{s}} \ff{Y}_j$ in $\cc{C}$ such that $(\theta,\,\ff{r},\,\varphi,\,\ff{s},\,\psi)$ and $(\theta',\,\ff{r},\,\varphi,\,\ff{s},\,\psi)$ both represent $\alpha$. Then, there exists $i \mr{u} i' \in \cc{I}$ such that $\theta \ff{X}_u=\theta' \ff{X}_u$.
\end{lemma}

\begin{proof}
Since both $(\theta,\ff{r},\varphi,\ff{s},\psi)$ and $(\theta',\ff{r},\varphi,\ff{s},\psi)$ represents $\alpha$, and $\varphi\hbox{, }\psi$ are invertible, it follows that  $\theta\pi_i=\theta'\pi_i$. Then, by \ref{lema4}, there exists $i \mr{u} i' \in \cc{I}$ such that $\theta \ff{X}_u=\theta' \ff{X}_u$.
\end{proof}

\begin{lemma}\label{lema2}
Let $\ff{X} \cellrd{\ff{f}}{\alpha}{\ff{g}} \ff{Y} \in $ $2$-$\cc{P}ro(\cc{C})$, $(\ff{r},\varphi)$ representing $\ff{f}$, \mbox{$\ff{X}_i\mr{\ff{r}} \ff{Y}_j$} and $(\ff{s},\psi)$ representing $\ff{g}$, $\ff{X}_{i'}\mr{\ff{s}} \ff{Y}_j$. Then, $\exists \vcenter{\xymatrix@R=-0.5pc{i \ar[rd]^{u} &\\& k \\ i'\ar[ru]_{v}}} \in \cc{I}$ and $\ff{X}_{k} \cellrd{\ff{r}\ff{X}_u}{\theta}{\ff{s}\ff{X}_v} \ff{Y}_j \in \cc{C}$  such that \mbox{$(\theta,\; \ff{r} \ff{X}_u,\; \ff{r} \pi_u \circ \varphi, \; \ff{s}\ff{X}_v, \; \ff{s} \pi_v \circ \psi)$} represents $\alpha$.
Observe that in case $\alpha$ is invertible, one can choose $\theta$ to be invertible.
\end{lemma}
\begin{proof}
In lemma $\ref{lema1}$, take $\ff{C}=\ff{Y}_j$, and
$\alpha = 
\vcenter{\xymatrix@R=1pc@C=-0pc{  \fr \dl & \dc{\varphi} & \dr \pi_i  \\
\pi_j \deq && \f \dcell{\alpha} \\
\pi_j \dl & \dc{\psi^{-1}} & \dr \ff{g} \\
\s && \pi_{i'}  }}$. Then, $\exists \vcenter{\xymatrix@R=-0.5pc{i \ar[rd]^{u} &\\& k \\ i'\ar[ru]_{v}}} \in \cc{I}$ and $\ff{X}_{k} \cellrd{\ff{r}\ff{X}_u}{\theta}{\ff{s}\ff{X}_v} \ff{Y}_j \in \cc{C}$ such that 
 $$\vcenter{\xymatrix@C=-0pc{ \fr \dl & \dc{\theta} & \dr \X_u && \pi_k \deq \\
 \s \deq && \X_v && \pi_k \\
 \s &&& \pi_{i'} \cl{\pi_v}  }}
\vcenter{\xymatrix@C=-0pc{ \quad = \quad }}
\vcenter{\xymatrix@C=-0pc{ \fr \deq && \X_u && \pi_k \\
\fr \dl & \dcr{\varphi} && \dr \pi_i \cl{\pi_u} \\
\pi_j \deq &&& \f \dcell{\alpha} \\
\pi_j \dl & \dcr{\psi^{-1}} && \dr \ff{g} \\
\s &&& \pi_{i'} }}, \hbox{ i.e. }
 \vcenter{\xymatrix@C=-0pc{ \fr \dl & \dc{\theta} & \dr \X_u && \pi_k \deq \\
 \s \deq && \X_v && \pi_k \\
 \s \dl & \dcr{\psi} && \dr \pi_{i'} \cl{\pi_v}  \\
 \pi_j &&& \ff{g} }}
\vcenter{\xymatrix@C=-0pc{ \quad = \quad }}
\vcenter{\xymatrix@C=-0pc{ \fr \deq && \X_u && \pi_k \\
\fr \dl & \dcr{\varphi} && \dr \pi_i \cl{\pi_u} \\
\pi_j \deq &&& \f \dcell{\alpha} \\
\pi_j &&& \ff{g} }}$$

%

\vspace{1ex}

This proves that $(\theta,\; \ff{r} \ff{X}_u,\; \ff{r} \pi_u \circ \varphi, \; \ff{s}\ff{X}_v, \; \ff{s} \pi_v \circ \psi)$ \mbox{represents $\alpha$.}
\end{proof}
From remark \ref{u=v} we have:
\begin{remark}
If $i=i'$, then one can choose $u=v$. \cqd
\end{remark}

\subsection{2-cofiltered pseudo-limits in $2\hbox{-}\cc{P}ro(\cc{C})$.}\label{pseudo-limitsen2pro}

Let $\cc{J}$ be a small  2-filtered 2-category and
$\cc{J}^{op} \mr{\ff{X}}  2\hbox{-}\cc{P}ro(\cc{C})$ a \mbox{2-functor,}
\mbox{$\ff{X}^j = (\ff{X}^j_i)_{i\in \cc{I}_j}$,} $\cc{I}_j^{op} \mr{\ff{X}^j} \cc{C}$. Recall (\ref{X=lim}) that for each $j$ in $\cc{J}$, $\ff{X}^j$ is equipped with a pseudo-limit pseudo-cone $\left\{\pi^j_i\right\}_{i \in \cc{I}_j}$,
$\left\{\pi^j_u\right\}_{i \mr{u} i' \in \cc{I}_j}$ for the 2-functor
$\ff{X}^j$. We are using the supra-index notation to denote the evaluation of $\X$.

We are going to construct a 2-pro-object which is going to be the pseudo-limit of
$\ff{X}$ in $2\hbox{-} \cc{P}ro(\cc{C})$. First we construct its index category.

\begin{definition}\label{kequis}
Let $\cc{K}_\ff{X}$ be the 2-category consisting on:
\begin{enumerate}
	\item[] 0-cells of $\cc{K}_\ff{X}$: $(i,j)$, where $j\in \cc{J}$, $i\in \cc{I}_j$.
	\item[] 1-cells of $\cc{K}_\ff{X}$: $(i,j)\mr{(a,\ff{r},\varphi)} (i',j')$, where $j\smr{a} j'\in \cc{J}$, $\ff{X}_{i'}^{j'}\smr{\ff{r}} \ff{X}_{i}^{j} \in \cc{C}$ are such that $(\ff{r},\varphi)$ represents $\ff{X}^a$.
	\item[] 2-cells of $\cc{K}_\ff{X}$: $(a,\ff{r},\varphi)\Mr{(\alpha,\theta)}(b,\ff{s},\psi)$, where $a \Mr{\alpha} b \in \cc{J}$ and $(\theta,\ff{r},\varphi,\ff{s},\psi)$ represents $\ff{X}^\alpha$.
\end{enumerate}

The 2-category structure is given as follows:

$$
\xymatrix@C=6pc
             {
              (i,j) \ar@<4.8ex>[r]^{(a,\ff{r},\varphi)}
                    \ar@{}@<3.5ex>[r]|{\Downarrow(\alpha,\theta)}
                    \ar[r]^{(b,\ff{s},\psi)}
                    \ar@{}@<-1.3ex>[r]|{\Downarrow(\beta,\eta)}
                    \ar@<-4.8ex>[r]^{(c,\ff{t},\phi)}
                 &
             (i',j') \ar@<4.8ex>[r]^{(a',\ff{r}',\varphi')}
                     \ar@{}@<3.5ex>[r]|{\Downarrow(\alpha',\theta')}
                     \ar[r]^{(b', \ff{s}',\psi')}
                     \ar@{}@<-1.3ex>[r]|{\Downarrow(\beta',\eta')}
                     \ar@<-4.8ex>[r]^{(c', \ff{t}',\phi')}
                 &
             (i'',j'')
             }
$$

\begin{enumerate}
\item[-] $(a',\ff{r}',\varphi')(a,\ff{r},\varphi)=(a' a,\ff{r}\ff{r}', 
\vcenter{\xymatrix@C=-0.2pc@R=0.5pc{  \fr \deq && \fr' \dl & \dc{\varphi'} & \dr \pi\sii \\
\fr \dl & \dc{\varphi} & \dr \pi\si && \X\ta \deq \\
\pi_i \deq && \X^a && \X\ta \\
\pi_i &&& \X^{a'a} \cl{=} }}  
)$

\item[-] $(\alpha',\theta')(\alpha,\theta)=(\alpha'\alpha,\theta\theta')$

\item[-] $(\beta,\eta)\circ (\alpha,\theta)=(\beta\circ \alpha,\eta\circ\theta)$

\end{enumerate}
One can easily check that the structure so defined is indeed a 2-category, which  is clearly small.\end{definition}

\begin{proposition}
The 2-category $\cc{K}_\ff{X}$ is 2-filtered.
\end{proposition}

\vspace{-2ex}

\begin{proof}

\begin{itemize}
 \item[F0.] Let $(i,j)$,$(i',j')\in \cc{K}_\ff{X}$. Since $\cc{J}$ is 2-filtered, $\exists \vcenter{\xymatrix@R=-0.5pc{j \ar[rd]^{a} &\\& j'' \\ j'\ar[ru]_{b}}}\in \cc{J}$. By \ref{idrepresenta}, \mbox{$\exists\  \ff{X}_{i_1}^{j''}\mr{\ff{r}_1}{} \ff{X}_i^j$} and $ \ff{X}_{i_2}^{j''}\mr{\ff{r}_2}{} \ff{X}_{i'}^{j'} \in \cc{C}$ such that $(\ff{r}_1,id)$ represents $\ff{X}^{a}$ and $(\ff{r}_2,id)$ represents $\ff{X}^{b}$.
Since $\cc{I}_{j''}$ is 2-filtered, $\exists \vcenter{\xymatrix@R=-0.5pc{i_1 \ar[rd]^{u} &\\ & i'' \\ i_2\ar[ru]_{v}}}\in \cc{I}_{j''}$. Then, we have the following situation in $\cc{K}_\ff{X}$ which concludes the proof of axiom $F0$:
$$
\xymatrix@R=.3pc@C=6pc
         {
           (i,j) \hspace{2ex} \ar[rd]^*[l]
                 {
                  \hspace{-5ex} (a,\ff{r}_1 \ff{X}^{j''}_{u}, \ff{r}_1 \pi_u^{j''} ) 
                 }
           &
           \\
           & \hspace{3ex} (i'',j'')
           \\ (i',j') \hspace{2ex} \ar[ru]_*[l]
                 {
                  \hspace{-5ex} (b,\ff{r}_2 \ff{X}^{j''}_{v}, \ff{r}_2 \pi_v^{j''} ) 
                 }
          }
$$

\noindent Note that
$\ff{r}_1 \pi_u^{j''} =\vcenter{\xymatrix@R=1pc@C=-0.2pc{ \fr_1 \deq && \X_u\tjj && \pi\sii\tjj \\
		    \fr_1 \dl & \dcr{=} && \dr \pi_{i_1}\tjj \cl{\pi_u\tjj} \\
		    \pi_i^j &&& \X^a }}$
and 
$\ff{r}_2 \pi_v^{j''} =\vcenter{\xymatrix@R=1pc@C=-0.2pc{ \fr_2 \deq && \X_v\tjj && \pi\sii\tjj \\
		    \fr_2 \dl & \dcr{=} && \dr \pi_{i_2}\tjj \cl{\pi_v\tjj} \\
		    \pi\si\tj &&& \X^b }}$

\item[F1.] Let $\xymatrix{(i,j)\ar@<1ex>[r]^{(a,\ff{r},\varphi)} \ar@<-1ex>[r]_{(b,\ff{s},\psi)} & (i',j')}\in \cc{K}_\ff{X}$. Since $\cc{J}$ is 2-filtered, \mbox{$\exists \  j'\mr{c} j''$} and an invertible 2-cell $j \cellrd{ca}{\alpha}{cb} j'' \in \cc{J}$. By \ref{idrepresenta}, \mbox{$\exists \ \ff{X}_{k}^{j''}\mr{\ff{t}} \ff{X}_{i'}^{j'} \in \cc{C}$} such that $(\ff{t},id)$ \mbox{represents} $\ff{X}^c$. Then $(\ff{r}\ff{t}, 
\vcenter{\xymatrix@C=-0.2pc@R=0.5pc{  \fr \deq && \ft \dl & \dc{=} & \dr \pi_k\tjj \\
\fr \dl & \dc{\varphi} & \dr \pi\si\tj && \X^c \deq \\
\pi_i^j \deq && \X^a && \X^c \\
\pi_i^j &&& \X\tca \cl{=} }}  
)$ represents $\ff{X}^{ca}$ and $(\ff{s}\ff{t}, 
\vcenter{\xymatrix@C=-0.2pc@R=0.5pc{  \s \deq && \ft \dl & \dc{=} & \dr \pi_k\tjj \\
\s \dl & \dc{\psi} & \dr \pi\si\tj && \X^c \deq \\
\pi_i^j \deq && \X^b && \X^c \\
\pi_i^j &&& \X\tca \cl{=} }}  
)$ \mbox{represents} $\ff{X}^{cb}$, so, by \ref{lema2},
there \mbox{exists} \mbox{$k \mr{w} i''\in \cc{I}_{j''}$} and an invertible 2-cell \mbox{$\ff{X}^{j''}_{i''} \cellrd{\ff{r}\ff{t}\ff{X}^{j''}_{w}}{\theta}{\ff{s}\ff{t}\ff{X}^{j''}_{w}} \ff{X}^j_i \in \cc{C}$} such that
\mbox{$(\theta,\,\ff{r}\ff{t}\ff{X}^{j''}_{w},\, 
\vcenter{\xymatrix@C=-0.2pc@R=0.5pc{  \fr \deq && \ft \deq & \X_w\tjj && \pi\sii\tjj \\
\fr \deq && \ft \dl & \dc{=} & \dr \pi_k\tjj \cl{\pi_w\tjj} \\
\fr \dl & \dc{\varphi} & \dr \pi\si\tj && \X^c \deq \\
\pi_i^j \deq && \X^a && \X^c \\
\pi_i^j &&& \X\tca \cl{=} }}
,\,\ff{s}\ff{t}\ff{X}^{j''}_{w},\, 
\vcenter{\xymatrix@C=-0.2pc@R=0.5pc{  \s \deq && \ft \deq & \X_w\tjj && \pi\sii\tjj \\
\s \deq && \ft \dl & \dc{=} & \dr \pi_k\tjj \cl{\pi_w\tjj} \\
\s \dl & \dc{\psi} & \dr \pi\si\tj && \X^c \deq \\
\pi_i^j \deq && \X^b && \X^c \\
\pi_i^j &&& \X\tcb \cl{=} }}
)$}
represents $\ff{X}^{\alpha}$.
Then we have an invertible 2-cell in $\cc{K}_\ff{X}$ $\xymatrix{(i,j)\ar@<1.5ex>[rr]^{(c,\ff{t}\ff{X}^{j''}_{w},t\pi_w^{j''})(a,\ff{r},\varphi)} \ar@<-1.5ex>[rr]_{(c,\ff{t}\ff{X}^{j''}_{w},\ff{t}\pi_{w}^{j''})(b,\ff{s},\psi)} \ar@{}[rr]|{ \Downarrow \; (\alpha,\theta)} && (i'',j'')}$ which concludes the proof of axiom $F1$.

\item[F2.] Let $\xymatrix@C=2.5pc{(i,j)\ar@<1.5ex>[rr]^{(a,\ff{r},\varphi)} \ar@<-1.5ex>[rr]_{(b,\ff{s}, \psi)} \ar@{}[rr]|{\Downarrow \; (\alpha,\theta) \;\; \Downarrow \; (\alpha',\theta')} & & (i',j')} \in \cc{K}_\ff{X}$. Since $\cc{J}$ is \mbox{2-filtered,} \mbox{$\exists\ j' \mr{c} j'' \in \cc{J}$} such that \mbox{$c\alpha=c\alpha'$}. Also, by \ref{idrepresenta}, $\exists\  \ff{X}_{k}^{j''}\mr{\ff{t}} \ff{X}_{i'}^{j'} \in \cc{C}$ such that $(\ff{t},id)$ represents $\ff{X}^c$. Then, it is easy to check that $(\ff{t},\ff{t},id,\ff{t},id)$ represents $\ff{X}^c$ and therefore we have that $(\theta \ff{t},\ff{r}\ff{t}, 
\vcenter{\xymatrix@C=-0.2pc@R=0.5pc{  \fr \deq && \ft \dl & \dc{=} & \dr \pi_k\tjj \\
\fr \dl & \dc{\varphi} & \dr \pi\si\tj && \X^c \deq \\
\pi_i^j \deq && \X^a && \X^c \\
\pi_i^j &&& \X\tca \cl{=} }}  ,\ff{s}\ff{t}, 
\vcenter{\xymatrix@C=-0.2pc@R=0.5pc{  \s \deq && \ft \dl & \dc{=} & \dr \pi_k\tjj \\
\s \dl & \dc{\psi} & \dr \pi\si\tj && \X^c \deq \\
\pi_i^j \deq && \X^b && \X^c \\
\pi_i^j &&& \X\tca \cl{=} }} )$ and $(\theta'\ff{t},\ff{r}\ff{t}, 
\vcenter{\xymatrix@C=-0.2pc@R=0.5pc{  \fr \deq && \ft \dl & \dc{=} & \dr \pi_k\tjj \\
\fr \dl & \dc{\varphi} & \dr \pi\si\tj && \X^c \deq \\
\pi_i^j \deq && \X^a && \X^c \\
\pi_i^j &&& \X\tca \cl{=} }}, \ff{s}\ff{t}, 
\vcenter{\xymatrix@C=-0.2pc@R=0.5pc{  \s \deq && \ft \dl & \dc{=} & \dr \pi_k\tjj \\
\s \dl & \dc{\psi} & \dr \pi\si\tj && \X^c \deq \\
\pi_i^j \deq && \X^b && \X^c \\
\pi_i^j &&& \X\tca \cl{=} }} )$ both represent $\ff{X}^{c\alpha}$: 


$$\vcenter{\xymatrix@C=0pc@R=1.5pc{  \fr \dcell{\theta} && \ft \deq && \pi_k\tjj \deq \\
\s \deq && \ft \dl & \dc{=} & \dr \pi_k\tjj \\
\s \dl & \dc{\psi} & \dr \pi\si\tj && \X^c \deq \\
\pi_i^j \deq && \X^b && \X^c \\
\pi_i^j &&& \X\tcb \cl{=} }}
\vcenter{\xymatrix@C=0pc{ \quad = \quad}}
\vcenter{\xymatrix@C=0pc@R=1.5pc{ \fr \deq&& \ft \dl & \dc{=} & \dr \pi_k\tjj \\
\fr \dl & \dc{\varphi} & \dr \pi\si\tj && \X^c \deq \\
\pi_i^j \deq && \X^a \dcell{\X^{\alpha}} && \X^c \deq \\
\pi_i^j \deq && \X^b && \X^c \\
\pi_i^j  &&& \X\tcb \cl{=} }}
\vcenter{\xymatrix@C=0pc{ \quad = \quad}}
\vcenter{\xymatrix@C=0pc@R=1.5pc{ \fr \deq&& \ft \dl & \dc{=} & \dr \pi_k\tjj \\
\fr \dl & \dc{\varphi} & \dr \pi\si\tj && \X^c \deq \\
\pi_i^j \deq && \X^a && \X^c \\
\pi_i^j \deq &&& \X\tca \cl{=} \dcellb{\X^{c\alpha}} \\
\pi_i^j  &&& \X\tcb}}$$

\noindent where the first equality is due to elevators calculus plus the fact that $(\theta,\ff{r},\varphi,\ff{s},\psi)$ represents $\ff{X}^\alpha$.
				  
Then, by \ref{lema5}, $\exists \, k \mr{w} i'' \in \cc{I}_{j''}$ such that \mbox{$\theta  \ff{t} \ff{X}^{j''}_{w}=\theta'\ff{t}\ff{X}^{j''}_{w}$}, so \mbox{$(c,\ff{t}\ff{X}^{j''}_{w},\ff{t}\pi_{w})(\alpha,\theta)=(c,\ff{t}\ff{X}^{j''}_{w},t\pi_{w})(\alpha',\theta')$,} which concludes the proof of axiom $F2$.
\end{itemize}
\end{proof}

\begin{proposition}\label{teo} Let $\widetilde{\ff{X}}$ be the 2-pro-object
$\cc{K}_\ff{X}^{op} \mr{\widetilde{\ff{X}}} \cc{C}$ defined by
\mbox{$\widetilde{\ff{X}}_{(i, j)} = \ff{X}^j_i$,}
$\widetilde{\ff{X}}_{(a,\ff{r}, \varphi)} = \ff{r}$, and
$\widetilde{\ff{X}}_{(\alpha, \theta)} = \theta$.
Then the following equation holds in $2$-$\cc{P}ro(\cc{C})$:
$$\widetilde{\ff{X}}=\Lim{j\in \cc{J}}{\ff{X}^j}$$
\end{proposition}

\vspace{-3ex}

\begin{proof}
Let $\ff{Z} \in$ 2-$\cc{P}ro(\cc{C})$, and
$\Big\{\ff{Z}\mr{\ff{h}_j} \ff{X}^j\Big\}_{j\in \cc{J}}$,
\mbox{$\Big\{\ff{X}^a \ff{h}_{j'} \Mr{\ff{h}_a} \ff{h}_j \Big\}_{j \mr{a}j' \in \cc{J}}$}
be a pseudo-cone for $\ff{X}$ with vertex $\ff{Z}$ (\ref{pseudo-cone}). Given
\mbox{$(i,j)\mr{(a,\ff{r},\varphi)} (i',j') \; \in \cc{K}_\ff{X}$}, the definitions
$\ff{h}_{(i,j)} = \pi^j_i \ff{h}_j$ and
$\ff{h}_{(a,\ff{r},\varphi)} = 
\vcenter{\xymatrix@C=-0.2pc@R=.5pc{ \fr \dl & \dc{\varphi} & \dr \pi\si\tj && \h\sj \deq \\
\pi_i^j \deq && \X^a && \h\sj \\
\pi_i^j &&& \h_j \cl{\h_a} }}
$ determine a pseudo-cone for $c\widetilde{\ff{X}}$ (where $c$ is the morphism of \ref{CisPro}) with vertex $\ff{Z}$:

\begin{itemize}
 \item[PC0.] It is straightforward.

\item[PC1.] Given $(i,j) \mr{(a,\ff{r},\varphi)} (i',j') \mr{(b, \ff{s}, \psi)} (i'',j'') \in \cc{K}_{\X}$,
$$\vcenter{ \xymatrix@C=0pc@R=1.5pc{  \fr \deq && \s \dl & \dc{\psi} & \dr \pi\sii\tjj && \h\sjj \deq \\
\fr \deq && \pi\si\tj \deq && \X^b && \h\sjj \\
\fr \dl & \dc{\varphi} & \dr \pi\si\tj &&& \h\sj \deq \cl{\h_b} \\
\pi_i^j \deq && \X^a &&& \h\sj \\
\pi_i^j &&&& \h_j \cldosuno{\h_a} }}
\vcenter{ \xymatrix@C=-0pc{ \  = \  } }
\vcenter{ \xymatrix@C=0pc@R=1.5pc{  \fr \deq && \s \dl & \dc{\psi} & \dr \pi\sii\tjj && \h\sjj \deq \\
\fr \dl & \dc{\varphi} & \dr \pi\si\tj && \X^b \deq && \h\sjj \deq \\
\pi_i^j \deq && \X^a && \X^b && \h\sjj \deq \\
\pi_i^j \deq &&& \X\tba \cl{=} &&& \h\sjj \\
\pi_i^j &&&& \h_j \clunodos{\h\sba} }}$$

\noindent where the equality is due to elevators calculus plus the fact that $\ff{h}$ is a pseudo-cone.
			     
\item[PC2.] Given $(i,j) \cellrd{(a,\ff{r},\varphi)}{(\alpha,\theta)}{(b,\ff{s},\psi)} (i',j') \in \cc{K}_{\X}$,

$$\vcenter{ \xymatrix@C=0pc@R=1.5pc{ \fr \dl & \dc{\varphi} & \dr \pi\si\tj && \h\sj \deq \\
\pi_i^j \deq && \X^a && \h\sj \\
\pi_i^j &&& \h_j \cl{\h_a} }}
\vcenter{ \xymatrix@C=0pc{  \quad = \quad  }}
\vcenter{ \xymatrix@C=0pc@R=1.5pc{  \fr \dl & \dc{\varphi} & \dr \pi\si\tj && \h\sj \deq \\
\pi_i^j \deq && \X^a \dcell{\X^{\alpha}} && \h\sj \deq \\
\pi_i^j \deq && \X^b && \h\sj \\
\pi_i^j &&& \h_j \cl{\h_b} }}
\vcenter{ \xymatrix@C=0pc{  \quad = \quad  }}
\vcenter{ \xymatrix@C=0pc@R=1.5pc{ \fr \dcell{\theta} && \pi\si\tj \deq && \h\sj \deq \\
\s \dl & \dc{\psi} & \dr \pi\si\tj && \h\sj \deq \\
\pi_i^j \deq && \X^b && \h\sj \\
\pi_i^j &&& \h_j \cl{\h_b} }}$$

\noindent where the first equality is due to the fact that $\ff{h}$ is a pseudo-cone and the second one is valid because $(\theta,\ff{r},\varphi, \ff{s},\psi)$ represents $\ff{X}^\alpha$.				
				
It is straightforward to check that this extends to a functor, that we denote $p$ (for the isomorphism below see \ref{XisLim}):
$$
\ff{PC}_{2\hbox{-}\cc{P}ro(\cc{C})}(\ff{Z},\ff{X}) \mr{p} \ff{PC}_{2\hbox{-}\cc{P}ro(\cc{C})}(\ff{Z},c\widetilde{\ff{X}})
 \, \cong \,2\hbox{-} \cc{P}ro(\cc{C})(\ff{Z},\, \ff{\widetilde{X}})
$$
\end{itemize}

The proposition follows if $p$ is an isomorphism. In the sequel we prove that this is the case.

\vspace{1ex}

\noindent \emph{1. $p$ is bijective on objects}:
Let
$$
    {\Big\{\ff{Z}\mr{\ff{h}_{(i,j)}} \ff{X}^j_i\Big\}_{(i,j) \in \cc{K}_\ff{X}}, \hspace{2ex}
    }
\Big\{\ff{\widetilde{X}}_{(a,\ff{r},\varphi)} \ff{h}_{(i',j')}
     = \ff{r} \, \ff{h}_{(i',j')} \Mr{\ff{h}_{(a,\ff{r},\varphi)}}
     \ff{h}_{(i,j)}\Big\}_
     {(i,j) \mr{(a,\ff{r},\varphi)} (i',j') \in \cc{K}_\ff{X}}
$$
be a pseudo-cone for $c\ff{\widetilde{X}}$ with vertex $\ff{Z}$ (\ref{pseudo-cone}).

Check that for each $j \in \cc{J}$,
$\Big\{\ff{Z}\mr{\ff{h}_{(i,j)}}\ff{X}_i^j \Big\}_{i\in \cc{I}_j}$
together  with
\mbox{$\Big\{\ff{h}_u = \ff{h}_{(j,\ff{X}_{u}^j,\pi_u^j)}: \ff{X}_u^j \ff{h}_{(i',j')}\Mr{} \ff{h}_{(i,j)} \Big\}_{i\mr{u} i'\in \cc{I}_j}$} is a pseudo-cone for $\ff{X}^j$.
Then, since $\Big\{\ff{X}^j \mr{\pi^j_i} \ff{X}^j_i\Big\}_{i \in \cc{I}_j}$, $\Big\{\ff{X}^j_u \pi^j_{i'} \Mr{\pi_u^j} \pi^j_i\Big\}_{i\mr{u}i' \in \cc{I}_j}$ is a pseudo-limit pseudo-cone, it follows that there exists a unique $\ff{Z} \mr{\ff{h}_j} \ff{X}^j$ such that
\begin{equation}\label{*1}
 \forall i\in \cc{I}_j \;\; \pi_i^j\ff{h}_j=\ff{h}_{(i,j)} \;\;and\;\; \forall \; i \mr{u} i'\in \cc{I}_j \ \;\; \pi_u^j\ff{h}_j=\ff{h}_{u}.
\end{equation}
It only remains to define the 2-cells of the pseudo-cone structure. That is, for each
$j \mr{a} j' \in \cc{J}$, we need invertible 2-cells
$\ff{X}^a \ff{h}_{j'} \Mr{\ff{h}_a} \ff{h}_j $,
 such that $\{\ff{h}_j\}_{j\in \cc{J}}$ together with
 $\left\{\ff{h}_a\right\}_{j \mr{a} j' \in \cc{J}}$ form a pseudo-cone for $\ff{X}$ with vertex $\ff{Z}$.

Consider the pseudo-cone
 $\Big\{\ff{X}^j \mr{\pi^j_i} \ff{X}^j_i \Big\}_{i \in \cc{I}_j}$. Then the compositions 
 $\Bigg\{\ff{Z} \mrpair{\pi_i^j\ff{h}_j}{\pi_i^j\ff{X}^a \ff{h}_{j'}}
                                           \ff{X}_i^j\Bigg\}_{i\in \cc{I}_j}$
determine two pseudo-cones for $\ff{X}^j$ with vertex $\ff{Z}$.

\vspace{1ex}

\noindent {\bf Claim 1}
 \emph
   { Let $(\ff{r},\varphi)$ and $(\ff{s},\psi)$ be two pairs
     representing $\ff{X}^a$ as follows:
    $$
     \xymatrix@C=0.8pc@R=0.8pc
           {
            \ff{X}^{j'} \ar@{}[rrdd]|{\cong \; \Downarrow \; \varphi} \ar[rr]^{\pi^{j'}_{i'}}  \ar[dd]_{\ff{X}^a}
            & & \ff{X}^{j'}_{i'} \ar[dd]^{\ff{r}}
           \\
            & 
            &
           \\  \ff{X}^j \ar[rr]_{\pi^j_i}
            & & \ff{X}^j_i
           }
     \hspace{5ex}
     \xymatrix@C=0.8pc@R=0.8pc
           {
            \ff{X}^{j'} \ar@{}[rrdd]|{\cong \; \Downarrow \; \psi} \ar[rr]^{\pi^{j'}_{i''}}  \ar[dd]_{\ff{X}^a}
            & & \ff{X}^{j'}_{i''} \ar[dd]^{\ff{s}}
           \\
            & 
            &
           \\ \ff{X}^j \ar[rr]_{\pi^j_i}
            & & \ff{X}^j_i
           }
$$
Then, $\ff{h}_{(a,\fr,\varphi)} \circ \varphi^{-1}\ff{h}_{j'}= \ff{h}_{(a,\s,\psi)} \circ \psi^{-1}\ff{h}_{j'}$, i.e.
$$\vcenter{\xymatrix@C=-0pc{ \pi_i^j \dl & \dc{\varphi^{-1}} & \dr \X^a && \h\sj \deq \\
\fr \ardr && \pi\si\tj \dc{\h\st{a}{r}{\varphi}} && \ardl \h\sj \\
& \pi_i^j && \h_j}}
\vcenter{\xymatrix@C=0pc{\ = \ \ } }
 \vcenter{\xymatrix@C=-0.5pc{ \pi_i^j \dl & \dc{\psi^{-1}} & \dr \X^a && \h\sj \deq \\
 \s \ardr && \pi\sii\tj \dc{\h\st{a}{s}{\psi}} && \h\sj \ardl \\
 & \pi_i^j && \h_j }}$$
}
\hfill (proof below).

\vspace{1ex}

\noindent {\bf Claim 2}
 \emph
      {
       For each $i \in \cc{I}_j$, let $(\ff{r}, \varphi)$ be a pair representing
       $\ff{X}^a$, and set
       \mbox{$\rho_i= \ff{h}_{(a,\fr,\varphi)} \circ \varphi^{-1}\ff{h}_{j'} =
       \vcenter{\xymatrix@R=.5pc@C=-0.4pc{ \pi_i^j \dl & \dc{\varphi^{-1}} & \dr \X^a && \h\sj \deq \\
\fr \ardr && \pi\si\tj \dc{\h\st{a}{r}{\varphi}} && \ardl \h\sj \\
& \pi_i^j && \h_j}} 
$.}
       Then, $\left\{\rho_i\right\}_{i \in \cc{I}_j}$ determines an isomorphism of
       pseudo-cones
       $\Bigg\{\ff{Z} \cellrd{\pi_i^j\ff{h}_j}{\rho_i}
        {\pi_i^j\ff{X}^a \ff{h}_{j'}} \ff{X}^j_i\Bigg\}_{i\in \cc{I}_j}$
      }
\hfill\mbox{(proof below)}.

Since $\Big\{\ff{X}^j \mr{\pi^j_i} \ff{X}^j_i\Big\}_{i \in \cc{I}_j}$, $\Big\{\ff{X}^j_u \pi^j_{i'} \Mr{\pi_u^j} \pi^j_i\Big\}_{i\mr{u}i' \in \cc{I}_j}$ is a pseudo-limit pseudo-cone, the functor
\mbox{$2$-$\cc{P}ro(\cc{C})(\ff{Z}, \ff{X}^j) \mr{(\pi^j)_*} \ff{PC}_{2\hbox{-}\cc{P}ro(\cc{C})}(\ff{Z}, \ff{X}^j)$}
is an isomorphism of categories $\forall \ j\in \cc{J}$. Then, from Claim 2 it follows that there are invertible 2-cells
\mbox{$\ff{Z} \cellrd{\ff{h}_j}{h_a}{\ff{X}^a\ff{h}_{j'}} \ff{X}^j \in 2$-$\cc{P}ro(\cc{C})$} such that $\rho_i=\pi_i^j \ff{h}_a \ \forall \ i\in \cc{I}_j \ \forall \ j\in \cc{J}$. Let's check that $\Big\{\ff{Z}\mr{\ff{h}_j} \ff{X}^j\Big\}_{j\in \cc{J}}$ together with $\Big\{\ff{X}^a \ff{h}_{j'}\Mr{\ff{h}_a} \ff{h}_j\Big\}_{j\mr{a} j' \in \cc{J}}$ is a pseudo-cone over $\ff{X}$:

\begin{itemize}
 \item[PC0.] By Claim 1, in Claim 2 we can take $\ff{r}=id$ and $\varphi=id$, so $\rho_i=id$ and therefore $\ff{h}_{id}=id$.

\item[PC1.] Given $j\mr{a}j'\mr{b}j''\in \cc{J}$ and $i\in \cc{I}_j$, by Claim 1, in Claim 2 we can take $(\ff{r},id)$ representing $\ff{X}^a$, $\ff{X}_{i'}^{j'}\mr{\ff{r}}\ff{X}_i^j$, $(\ff{s},id)$ representing $\ff{X}^b$, $\ff{X}_{i''}^{j''}\mr{\ff{s}}\ff{X}_{i'}^{j'}$ and $(\ff{r}\ff{s},id)$ representing $\ff{X}^{ba}$. Then

$$\vcenter{\xymatrix@C=-0.3pc@R=1.5pc{ \pi_i^j \deq && \X^a \deq && \X^b && \h\sjj \\
\pi_i^j \deq && \X^a &&& \h\sj \cl{\h_b} \\
\pi_i^j &&&& \h_j \cldosuno{\h_a} }}
\vcenter{\xymatrix@C=-0.3pc{  = \quad }}
\vcenter{\xymatrix@C=-0.3pc@R=1.5pc{ \pi_i^j \deq && \X^a \deq & \X^b && \h\sjj \\
\pi_i^j \dl & \dc{=} & \dr \X^a && \h\sj \deq \cl{\h_b} \\ 
\fr \ardr && \pi\si\tj \dc{\st{a}{r}{id}} && \h\sj \ardl \\
& \pi_i^j && \h_j }}
\vcenter{\xymatrix@C=-0.3pc{ \quad = \quad }}
\vcenter{\xymatrix@C=-0.4pc@R=1.5pc{ \pi_i^j \dl & \dc{=} & \dr \X^a && \X^b \deq && \h\sjj \deq \\
\fr \deq && \pi\si\tj \dl & \dc{=} & \dr \X^b && \h\sjj \deq \\
\fr \deq && \s \ardr && \pi\sii\tjj \dc{\h\st{b}{\s}{id}} && \h\sjj \ardl \\
\fr \ardr &&& \pi\si\tj \dc{\h\st{a}{\fr}{id}} && \h\sj \ardl \\
& \pi_i^j &&& \h_j}}
\vcenter{\xymatrix@C=-0.3pc{  = \quad }}$$

$$\vcenter{\xymatrix@C=-0.3pc{ \quad = \quad }}
\vcenter{\xymatrix@C=-0.3pc@R=1.5pc{ \pi_i^j \dl & \dc{=} & \dr \X^a && \X^b \deq && \h\sjj \deq \\
\fr \deq && \pi\si\tj \dl & \dc{=} & \dr \X^b && \h\sjj \deq \\
\fr \ardr && \s & \dc{\h\st{ba}{\fr\s}{id}} & \pi\sii\tjj && \h\sjj \ardl \\ 
& \pi_i^j &&&& \h_j }}
\vcenter{\xymatrix@C=-0.3pc{ \quad = \quad }}
\vcenter{\xymatrix@C=-0.3pc@R=1.5pc{ \pi_i^j \deq && \X^a \ar@{-}[drr] && \X^b \dc{\h\sba} && \h\sjj \ar@{-}[dll] \\
\pi_i^j &&&& \h_j }}$$

\noindent where the first, the second and the last equalities hold by definition of $\ff{h}_{(a,\ff{r},id)}$, $\ff{h}_{(b,\ff{s},id)}$ and $\ff{h}_{(ba,\ff{rs},id)}$ respectively plus some elevators calculus; and the third equality is due to the fact that $\ff{h}$ is a pseudo-cone.

Since we checked this for any $i\in \cc{I}_j$, it follows:

$$\vcenter{\xymatrix@C=-0pc@R=1.5pc{ \X^a \deq && \X^b && \h\sjj \\
\X^a &&& \h\sj \cl{\h_b} \\
&& \h_j \cldosuno{\h_a} }}
\vcenter{\xymatrix@C=-0pc{ \quad = \quad }}
\vcenter{\xymatrix@C=-0pc@R=1.5pc{ \X^a \ar@{-}[drr] && \X^b \dc{\h\sba} && \h\sjj \ar@{-}[dll] \\
&& \h_j }}$$

\item[PC2.] Given $j\cellrd{a}{\alpha}{b}j'\in \cc{J}$ and $i\in \cc{I}_j$, there is $\ff{X}_{i'}^{j'}\cellrd{\ff{r}}{\theta}{\ff{s}}\ff{X}_i^j$ and appropriate invertible 2-cells $\varphi$, $\psi$ such that $(\theta,\ff{r},\varphi,\ff{s},\psi)$ represents $\ff{X}^{\alpha}$. By Claim 1, in Claim 2 we can take those representatives of $\ff{X}^a$ and $\ff{X}^b$ and then:


$$\vcenter{\xymatrix@C=-0pc@R=1.5pc{ \pi_i^j \deq && \X^a && \h\sj \\
\pi_i^j &&& \h_j \cl{\h_a} }}
\vcenter{\xymatrix@C=-0pc{ \quad = \quad }}
\vcenter{\xymatrix@C=-0pc@R=1.5pc{ \pi_i^j \dl & \dc{\varphi^{-1}} & \dr \X^a && \h\sj \deq \\
\fr \ardr && \pi\si\tj \dc{\h\st{a}{\fr}{\varphi}} && \h\sj \ardl \\
& \pi_i^j && \h_j }}
\vcenter{\xymatrix@C=-0pc{ \quad = \quad }}
\vcenter{\xymatrix@C=-0pc@R=1.5pc{ \pi_i^j \dl & \dc{\varphi^{-1}} & \dr \X^a && \h\sj \deq \\
\fr \dcell{\theta} && \pi\si\tj \deq && \h\sj \deq \\
\s \ardr && \pi\si\tj \dc{\h\st{b}{\s}{\psi}} && \h\sj \ardl \\
& \pi_i^j && \h_j }}
\vcenter{\xymatrix@C=-0pc{ \quad = \quad }}$$

$$\vcenter{\xymatrix@C=-0pc{ \quad = \quad }}
\vcenter{\xymatrix@C=-0pc@R=1.5pc{ \pi_i^j \deq && \X^a \dcell{\X^{\alpha}} && \h\sj \deq \\
\pi_i^j \dl & \dc{\psi^{-1}} & \dr \X^b && \h\sj \deq \\
\s \ardr && \pi\si\tj \dc{\h\st{b}{\s}{\psi}} && \h\sj \ardl \\
& \pi_i^j && \h_j }}
\vcenter{\xymatrix@C=-0pc{ \quad = \quad }}
\vcenter{\xymatrix@C=-0pc@R=1.5pc{ \pi_i^j \deq && \X^a \dcell{\X^{\alpha}} && \h\sj \deq \\
\pi_i^j \deq && \X^b && \h\sj \\
\pi_i^j &&& \h_j \cl{\h_b} }}$$


\noindent where the first and last equalities hold by definition of $\ff{h}_{(a,\ff{r},\varphi)}$ and $\ff{h}_{(b,\ff{s},\psi)}$ respectively, the second equality is due to the fact that $\ff{h}$ is a pseudo-cone and the third one is valid because $(\theta,\ff{r},\varphi,\ff{s},\psi)$ represents $\ff{X}^\alpha$. 
			    
Since we checked this for any $i\in \cc{I}_j$, it follows:

$$\vcenter{\xymatrix@C=-0pc@R=1.5pc{ \X^a && \h\sj \\
& \h_j \cl{\h_a} }}
\vcenter{\xymatrix@C=-0pc{ \quad = \quad }}
\vcenter{\xymatrix@C=-0pc@R=1.5pc{ \X^a \dcell{\X^{\alpha}} && \h\sj \deq \\
\X^b && \h\sj \\
& \h_j \cl{\h_b} }}$$
\end{itemize}

\noindent \emph{2. $p$ is full and faithful}:
Let $\Bigg\{\ff{Z} \cellrd{\pi_i^j \ff{h}_j}{\rho_{(i,j)}}{\pi_i^j \ff{m}_j} \ff{X}_i^j \Bigg\}_{(i,j)\in \cc{K}_\ff{X}}$
be a morphism of pseudo-cones for $\tilde{\ff{X}}$.
It is easy to check that for each $j\in \cc{J}$, $\Bigg\{\ff{Z}  \cellrd{\pi_i^j \ff{h}_j}{\rho_{(i,j)}}{\pi_i^j \ff{m}_j} \ff{X}_i^j \Bigg\}_{i\in \cc{I}_j}$ is a morphism of pseudo-cones for $\ff{X}^j$.  Then arguing as above, there exists a unique morphism \mbox{$\ff{Z} \cellrd{\ff{h}_j}{\rho_j}{\ff{m}_j} \ff{X}^j \in 2$-$\cc{P}ro(\cc{C})$} such that for each $i\in \cc{I}_j$, $\pi_i^j\rho_j=\rho_{(i,j)}$. It only remains to prove that $\{\rho_j\}_{j\in \cc{J}}$ is a morphism of pseudo-cones: 

\begin{itemize}
 \item[PCM.] Given $j\mr{a}j'\in \cc{J}$ and $i\in \cc{I}_j$, by Claim 1, in Claim 2 we can take $(\ff{r},id)$ representing $\ff{X}^a$, $\ff{X}_{i'}^{j'}\mr{\ff{r}}\ff{X}_{i}^{j}$ and then:

$$\vcenter{\xymatrix@C=-0pc@R=1.5pc{ \pi_i^j \deq && \X^a && \h\sj \\
\pi_i^j \deq &&& \h_j \cl{\h_a} \dcell{\rho_j} \\
\pi_i^j &&& \m_j }}
\vcenter{\xymatrix@C=-0pc{ \quad = \quad }}
\vcenter{\xymatrix@C=-0pc@R=1.5pc{ \pi_i^j \deq && \X^a && \h\sj \\
\pi_i^j \dl & \dcr{\rho\sd{i}{j}} && \dr \h_j \cl{\h_a} \\
\pi_i^j &&& \m_j }}
\vcenter{\xymatrix@C=-0pc{ \quad = \quad }}
\vcenter{\xymatrix@C=-0pc@R=1.5pc{ \pi_i^j \dl & \dc{=} & \dr \X^a && \h\sj \deq \\
\fr \deq && \pi\si\tj \dl & \dc{\rho\sd{i'}{j'}} & \dr \h\sj \\
\fr \dl & \dc{=} & \dr \pi\si\tj && \m\sj \deq \\
\pi_i^j \deq && \X^a && \m\sj \\
\pi_i^j &&& \m_j \cl{\m_a} }}
\vcenter{\xymatrix@C=-0pc{ \quad = \quad }}$$

$$\vcenter{\xymatrix@C=-0pc{ \quad = \quad }}
\vcenter{\xymatrix@C=-0pc@R=1.5pc{ \pi_i^j \dl & \dc{=} & \dr \X^a && \h\sj \deq \\
\fr \deq && \pi\si\tj \deq && \h\sj \dcell{\rho\sj} \\
\fr \dl & \dc{=} & \dr \pi\si\tj && \m\sj \deq \\
\pi_i^j \deq && \X^a && \m\sj \\
\pi_i^j && & \m_j \cl{\m_a} }}
\vcenter{\xymatrix@C=-0pc{ \quad = \quad }}
\vcenter{\xymatrix@C=-0pc@R=1.5pc{ \pi_i^j \deq &&  \X^a \deq && \h\sj \dcell{\rho\sj} \\
 \pi_i^j \deq &&  \X^a && \m\sj \\
 \pi_i^j  &&& \m_j \cl{\m_a} }}$$


\noindent where the second equality is valid because $\rho$ is a morphism of pseudo-cones.
			      
Since we checked this for any $i\in \cc{I}_j$, it follows:                                   

$$\vcenter{\xymatrix@C=-0pc@R=1.5pc{ \X^a && \h\sj \\ & \h_j \cl{\h_a} \dcell{\rho_j} \\ & \m_j}}
\vcenter{\xymatrix@C=-0pc{ \quad = \quad }}
\vcenter{\xymatrix@C=-0pc@R=1.5pc{ \X^a \deq && \h\sj \dcell{\rho\sj} \\ \X^a && \m\sj \\ & \m_j \cl{\m_a} }}$$

\end{itemize}
\end{proof}

\noindent \emph{Proof of Claim 1}. 
First assume that $i'=i''$ and $(\ff{r},\varphi)$, $(\ff{s},\psi)$ are related by a \mbox{2-cell} 
$(i,j)\cellrd{(a,\ff{r},\varphi)}{(a,\theta)}{(a,\ff{s},\psi)}(i',j')$ in $\cc{K}_{\ff{X}}$. 
Then:
$$\vcenter{\xymatrix@C=-0pc@R=1.5pc{ \pi_i^j \dl & \dc{\varphi^{-1}} & \dr \X^a && \h\sj \deq \\
\fr \ardr && \pi\si\tj \dc{\h\st{a}{\fr}{\varphi}} && \h\sj \ardl \\
& \pi_i^j && \h_j }}
\vcenter{\xymatrix@C=-0pc{ \quad = \quad }}
\vcenter{\xymatrix@C=-0pc@R=1.5pc{ \pi_i^j \dl & \dc{\psi^{-1}} & \dr \X^a && \h\sj \deq \\ 
\s \dcell{\theta^{-1}} && \pi\si\tj \deq && \h\sj \deq \\
\fr \ardr && \pi\si\tj \dc{\h\st{a}{\fr}{\varphi}} && \h\sj \ardl \\
& \pi_i^j && \h_j}}
\vcenter{\xymatrix@C=-0pc{ \quad = \quad }}
\vcenter{\xymatrix@C=-0pc@R=1.5pc{ \pi_i^j \dl & \dc{\psi^{-1}} & \dr \X^a && \h\sj \deq \\ 
\s \ardr && \pi\si\tj \dc{\h\st{a}{\s}{\psi}} && \h\sj \ardl \\
& \pi_i^j && \h_j }}$$

\noindent where the first equality holds because $\theta$ represents $id$ (the identity of $\ff{X}^a$), and the second one is valid because $\ff{h}$ is a pseudo-cone.

The general case reduces to this one as follows: 
we have 
$
\vcenter
    {
     \xymatrix@R=-.5pc
        {
         & (i',j') 
         \\
         (i,j) \ar[ru]^{(a,\ff{r},\varphi)}  \ar[dr]_{(a,\ff{s},\psi)} 
         \\
		 & (i'',j') 
		 }
    }\in \cc{K}_{\X}
$.
Take 
$
\vcenter
    {
     \xymatrix@R=-.5pc
         {
            i'\ar[rd]^u 
         \\
 		  & k 
		 \\
		    i'' \ar[ru]_v
		  }
    }
$
in 	$\cc{I}_j$. This yields a \mbox{particular} instance of lemma \ref{lema2}:

$$
\xymatrix@C=9ex@R=7ex
        {
         \ff{X}^{j'}   \ar@<1ex>[d]^{\ff{X}^a}
             \ar@{}[d]|{\stackrel{id}{\Rightarrow}}
             \ar@<-1ex>[d]_{\ff{X}^a}
             \ar[r]^{\pi_k^{j'}}
         & \ff{X}_k^{j'} \ar@<1ex>[d]^{\ff{s}\ff{X}_v^{j'}}
             \ar@<-1ex>[d]_{\ff{r}\ff{X}_u^{j'}}
        \\
         \ff{X}^j \ar[r]_{\pi_i^j}             
         & \ff{X}_i^j
        }
$$

with $(\ff{r}\ff{X}_u^{j'}, 
\vcenter{\xymatrix@C=-0.2pc@R=0.5pc{ \fr \deq && \X_u\tj && \X_k\tj \\
		    \fr \dl & \dcr{\varphi} && \dr \pi\si\tj \cl{\pi_u\tj} \\
		    \pi_i^j &&& \X^a }}
		    )$ and 
$(\ff{s}\ff{X}_v^{j'}, 
\vcenter{\xymatrix@C=-0.2pc@R=0.5pc{ \s \deq && \X_v\tj && \X_k\tj \\
		    \s \dl & \dcr{\psi} && \dr \pi\sii\tj \cl{\pi_v\tj} \\
		    \pi_i^j &&& \X^a }}
		    )$ 
both representing $\ff{X}^a$. 
It follows that there exists 
$k\stackrel{w}{\rightarrow} k' \in \cc{I}_{j'}$ and 
$\ff{X}_{k'}^{j'}\cellrd{\ff{r}\ff{X}_u^{j'}\ff{X}_w^{j'}}{\theta}{\ff{s}\ff{X}_v^{j'}\ff{X}_w^{j'}} \ff{X}_i^j \in \cc{C}$ such that 
\mbox{$(\theta,\; \ff{r}\ff{X}_u^{j'}\ff{X}_w^{j'},\; 
\vcenter{\xymatrix@C=-0.2pc@R=0.5pc{ \fr \deq && \X_u\tj \deq && \X_w\tj && \pi\sk\tj \\
\fr \deq && \X_u\tj &&& \pi_k\tj \cl{\pi_w\tj} \\
\fr \dl && \dc{\varphi} && \dr \pi\si\tj \cldosuno{\pi_u\tj} \\
\pi_i^j &&&& \X^a }},\; \ff{s}\ff{X}_v^{j'}\ff{X}_w^{j'},\; 
\vcenter{\xymatrix@C=-0.2pc@R=0.5pc{ \s \deq && \X_v\tj \deq && \X_w\tj && \pi\sk\tj \\
\s \deq && \X_v\tj &&& \pi_k\tj \cl{\pi_w\tj} \\
\s \dl && \dc{\psi} && \dr \pi\sii\tj \cldosuno{\pi_v\tj} \\
\pi_i^j &&&& \X^a }})$} 
represents  $id$ (the identity of $\ff{X}^a$).

\noindent Considering $(\ff{r}\ff{X}_u^{j'}\ff{X}_w^{j'},\; 
\vcenter{\xymatrix@C=-0.2pc@R=0.5pc{ \fr \deq && \X_u\tj \deq && \X_w\tj && \pi\sk\tj \\
\fr \deq && \X_u\tj &&& \pi_k\tj \cl{\pi_w\tj} \\
\fr \dl && \dc{\varphi} && \dr \pi\si\tj \cldosuno{\pi_u\tj} \\
\pi_i^j &&&& \X^a }}
)$ and 
\mbox{$(\ff{s}\ff{X}_v^{j'}\ff{X}_w^{j'},\; 
\vcenter{\xymatrix@C=-0.2pc@R=0.5pc{ \s \deq && \X_v\tj \deq && \X_w\tj && \pi\sk\tj \\
\s \deq && \X_v\tj &&& \pi_k\tj \cl{\pi_w\tj} \\
\s \dl && \dc{\psi} && \dr \pi\sii\tj \cldosuno{\pi_v\tj} \\
\pi_i^j &&&& \X^a }}
)$}
both representing $\ff{X}^a$, we have a situation that corresponds to the previous case. 
Thus:

$$\vcenter{\xymatrix@C=0.2pc@R=1.5pc{ \pi_i^j \dl && \dc{\varphi\inv} && \dr \X^a &&&& \h\sj \deq \\
\fr \deq &&&& \pi\si\tj \opdosuno{\! \! \! (\pi_u\tj)\inv} &&&& \h\sj \deq \\
\fr \deq && \X_u\tj \deq &&& \pi_k\tj \opb{(\pi_w\tj)\inv} &&& \h\sj \deq \\
\fr \ardr && \X_u\tj && \X_w\tj \dc{\h_(a,\fr \X_u\tj \X_w\tj, I)} && \pi\sk\tj && \h\sj \ardl \\
& \pi_i^j &&&&&& \h_j }}
\vcenter{\xymatrix@C=-0pc{  = }}
\vcenter{\xymatrix@C=0.2pc@R=1.5pc{  \pi_i^j \dl && \dc{\psi\inv} && \dr \X^a &&&& \h\sj \deq \\
\s \deq &&&& \pi\sii\tj \opdosuno{\! \! \! (\pi_v\tj)\inv} &&&& \h\sj \deq \\
\s \deq && \X_v\tj \deq &&& \pi_k\tj \opb{(\pi_w\tj)\inv} &&& \h\sj \deq \\
\s \ardr && \X_v\tj && \X_w\tj \dc{\h_(a,\s \X_v\tj \X_w\tj, II)} && \pi\sk\tj && \h\sj \ardl \\
& \pi_i^j &&&&&& \h_j }}$$

\noindent where $I=\vcenter{\xymatrix@C=-0.2pc@R=0.5pc{ \fr \deq && \X_u\tj \deq && \X_w\tj && \pi\sk\tj \\
\fr \deq && \X_u\tj &&& \pi_k\tj \cl{\pi_w\tj} \\
\fr \dl && \dc{\varphi} && \dr \pi\si\tj \cldosuno{\pi_u\tj} \\
\pi_i^j &&&& \X^a }}$ and $II=\vcenter{\xymatrix@C=-0.2pc@R=0.5pc{ \s \deq && \X_v\tj \deq && \X_w\tj && \pi\sk\tj \\
\s \deq && \X_v\tj &&& \pi_k\tj \cl{\pi_w\tj} \\
\s \dl && \dc{\psi} && \dr \pi\sii\tj \cldosuno{\pi_v\tj} \\
\pi_i^j &&&& \X^a }} $

\vspace{1ex}

Then, since $\ff{h}$ is a pseudo-cone, we have that

$$\vcenter{\xymatrix@C=0.2pc@R=1.5pc{ \pi_i^j \dl && \dc{\varphi\inv} && \dr \X^a &&&& \h\sj \deq \\
\fr \deq &&&& \pi\si\tj \opdosuno{ \! \! \! \! (\pi_u\tj)\inv } &&&& \h\sj \deq \\
\fr \deq && \X_u\tj \deq &&& \pi_k\tj \opb{(\pi_w\tj)\inv} &&& \h\sj \deq \\
\fr \deq && \X_u\tj \ardr && \X_w\tj & \dc{\h\st{j'}{\X_u\tj \X_w\tj}{\X_u\tj \pi_w\tj}} & \pi\sk\tj && \h\sj \ardl \\
\fr \ardr &&& \pi\si\tj \dcr{\h\st{a}{\fr}{\varphi}} &&&& \h\sj \ardl \\
& \pi_i^j &&&&& \h_j}}
\vcenter{\xymatrix@C=-0pc{  = }}
\vcenter{\xymatrix@C=0.2pc@R=1.5pc{  \pi_i^j \dl && \dc{\psi\inv} && \dr \X^a &&&& \h\sj \deq \\
\s \deq &&&& \pi\sii\tj \opdosuno{ \! \! \! \! (\pi_v\tj)\inv} &&&& \h\sj \deq \\
\s \deq && \X_v\tj \deq &&& \pi_k\tj \opb{(\pi_w\tj)\inv} &&& \h\sj \deq \\
\s \deq && \X_v\tj \ardr && \X_w\tj & \dc{\h\st{j'}{\X_v\tj \X_w\tj}{\X_v\tj \pi_w\tj}} & \pi\sk\tj && \h\sj \ardl \\
\s \ardr &&& \pi\sii\tj \dcr{\h\st{a}{\s}{\psi}} &&&& \h\sj \ardl \\
& \pi_i^j &&&&& \h_j }}$$

From \ref{*1} and the fact that $\ff{X}^j$ is a pseudo-cone, it follows that  
$\vcenter{\xymatrix@C=0.2pc@R=1.5pc{ \fr \deq &&&& \pi\si\tj \opdosuno{\! \! \! \! (\pi_u\tj)\inv } &&&& \h\sj \deq \\
\fr \deq && \X_u\tj \deq &&& \pi_k\tj \opb{(\pi_w\tj)\inv} &&& \h\sj \deq \\
\fr \deq && \X_u\tj \ardr && \X_w\tj & \dc{\h\st{j'}{\X_u\tj \X_w\tj}{\X_u\tj \pi_w\tj}} & \pi\sk\tj && \h\sj \ardl \\
\fr &&& \pi\si\tj &&&& \h\sj }}$ and 
$\vcenter{\xymatrix@C=0.2pc@R=1.5pc{  \s \deq &&&& \pi\sii\tj \opdosuno{\! \! \! \! (\pi_v\tj)\inv} &&&& \h\sj \deq \\
\s \deq && \X_v\tj \deq &&& \pi_k\tj \opb{(\pi_w\tj)\inv} &&& \h\sj \deq \\
\s \deq && \X_v\tj \ardr && \X_w\tj & \dc{\h\st{j'}{\X_v\tj \X_w\tj}{\X_v\tj \pi_w\tj}} & \pi\sk\tj && \h\sj \ardl \\
\s &&& \pi\sii\tj &&&& \h\sj }}$ are identities. So 
$$\vcenter{\xymatrix@C=-0pc{ \pi_i^j \dl & \dc{\varphi^{-1}} & \dr \X^a && \h\sj \deq \\
\fr \ardr && \pi\si\tj \dc{\h\st{a}{r}{\varphi}} && \ardl \h\sj \\
& \pi_i^j && \h_j}}
\vcenter{\xymatrix@C=0pc{\ = \ \ } }
 \vcenter{\xymatrix@C=-0.5pc{ \pi_i^j \dl & \dc{\psi^{-1}} & \dr \X^a && \h\sj \deq \\
 \s \ardr && \pi\sii\tj \dc{\h\st{a}{s}{\psi}} && \h\sj \ardl \\
 & \pi_i^j && \h_j }}$$
 as we wanted to prove.
\qed

\vspace{1ex}

\noindent \emph{Proof of Claim 2}.
Given any $i \mr{u} k \in \cc{I}_j$, we have to check the PCM equation in \ref{pseudo-cone}. Given the pair $(\ff{s}, \psi)$ used to define $\rho_k$, it is possible to choose a pair $(\ff{r}, \varphi)$ to define $\rho_i$ in such a way that the equation holds. This arguing is justified by Claim 1.
\qed

\begin{remark}\label{limitesenproparapseudo}
A similar proof can be done in case $\ff{X}$ is only a pseudo-functor. Replace the equality 2-cell $\vcenter{\xymatrix@C=-.5pc@R=1pc{\X^a && \X^b \\ & \X^{ba} \cl{=} }} $ by the structure 2-cell $\vcenter{\xymatrix@C=-.5pc@R=1pc{\X^a && \X^b \\ & \X^{ba} \clb{\alpha^{\X}_{a,b}} }} $ in the elevators. \cqd
\end{remark}

\begin{theorem}\label{cerradaporpseudo-limites}
$2$-$\cc{P}ro(\cc{C})$ is closed under small 2-cofiltered \mbox{pseudo-limits.} Considering the equivalence in \ref{eqconloscolimderep}, it follows that the inclusion
\mbox{$\cc{H}om(\cc{C},\, \cc{C}at)_{fc} \;\subset\; \cc{H}om(\cc{C},\, \cc{C}at)$} is closed under small 2-filtered \mbox{pseudo-colimits} \cqd
\end{theorem}

\begin{proof}
 It is immediate from \ref{teo}.
\end{proof}

Having \ref{limitesenproparapseudo} in mind, from the fact that $\Prop{C}$ is pseudo-equivalent to $\Pro{C}$ it follows easily that:

\begin{corollary}
 $\Prop{C}$ is closed under small 2-cofiltered \mbox{bi-limits} of pseudo-functors.  \cqd
\end{corollary}
 
 \subsection{Universal property of $\Pro{C}$}\label{pu2pro}
In this subsection we prove for \mbox{2-pro-objects} the universal property established for pro-objects in [\citealp{G2}, Ex. I, Prop. 8.7.3.].
Consider the 2-functor $\cc{C} \mr{c} 2$-$\cc{P}ro(\cc{C})$ of Corollary \ref{CisPro} and a 2-pro-object $\ff{X} = (\ff{X}_i)_{i\in \cc{I}}$. Given a 2-functor $\cc{C} \mr{\ff{F}} \cc{E}$ into a 2-category closed under
small \mbox{2-cofiltered} pseudo-limits, we can naively extend
$\ff{F}$ into a 2-cofiltered pseudo-limit preserving 2-functor
$2\hbox{-} \cc{P}ro(\cc{C}) \mr{\widehat{\ff{F}}} \cc{E}$ by defining
$\widehat{\ff{F}}\ff{X} = \Lim{i \in \cc{I}}{\ff{F}\ff{X}_i}$.
This is just part of a 2-equivalence of 2-categories that we develop with the necessary precision in this subsection. First the universal property  should be wholly established for $\cc{E} = \cc{C}at$, and only afterwards can be lifted to any 2-category $\cc{E}$ closed under small 2-cofiltered pseudo-limits.

\begin{lemma}\label{casoCat}
 Let $\cc{C}$ be a 2-category and $\ff{F}:\cc{C}\mr{} \cc{C}at$ a \mbox{2-functor.} Then, there exist a 2-functor
 $\widehat{\ff{F}}: 2\hbox{-}\cc{P}ro(\cc{C}) \mr{} \cc{C}at$ that preserves small 2-cofiltered pseudo-limits, and an isomorphism
 $\widehat{\ff{F}}c \mr{\cong} \ff{F}$ in $\cc{H}om(\cc{C},\cc{C}at)$.
\end{lemma}

\vspace{1ex}

\begin{proof}
Let $\ff{X}=(\ff{X}_i)_{i\in \cc{I}}\in 2\hbox{-}\cc{P}ro(\cc{C})$ be a 2-pro-object. Define:

\vspace{1ex}

\noindent
$
\widehat{\ff{F}}\ff{X} \;=\;
(\cc{H}om(\cc{C},\cc{C}at)(-,\ff{F})\circ \ff{L})\ff{X} \;=\;
\cc{H}om(\cc{C},\cc{C}at)(\coLim{i\in \cc{I}}{\cc{C}(\ff{X}_i,-)},\ff{F})
\mr{\cong}
$

\hfill
$
\mr{\cong} \Lim{i\in \cc{I}}{\cc{H}om(\cc{C},\cc{C}at)(\cc{C}(\ff{X}_i,-),\ff{F})}
\mr{\cong} \Lim{i\in \cc{I}}{\ff{F}\ff{X}_i}.
$

\vspace{1ex}

Where $\ff{L}$ is the 2-functor of \ref{eqconloscolimderep}, the first isomorphism is by definition of pseudo-colimit \ref{colimits}, and the second is due to the Yoneda isomorphism \ref{2Yoneda}. Since it is a 2-equivalence, the \mbox{2-functor} $\ff{L}$ preserves any pseudo-limit. Then by Corollary \ref{cerradaporpseudo-limites} it follows that the composite
$\cc{H}om(\cc{C},\cc{C}at)(-,\ff{F})\circ \ff{L}$ preserves small 2-cofiltered pseudo-limits
 \end{proof}

\begin{theorem} \label{pseudouniversalcat}
Let $\cc{C}$ be any 2-category. Then, pre-composition with
\mbox{$\cc{C} \mr{c} 2\hbox{-}\cc{P}ro(\cc{C})$}  is a 2-equivalence of 2-categories:
$$
\xymatrix@R=.7pc@C=.5pc
 {
  & \cc{H}om(2\hbox{-}\cc{P}ro(\cc{C}),\cc{C}at)_+ \ar[rrrr]^>>>>>>>>{c^*}
  & & & & \comw{\ff{X}\ff{X}} \cc{H}om(\cc{C},\cc{C}at) \comw{\ff{X} \ff{X}}
 }
$$
(where $\cc{H}om(2\hbox{-}\cc{P}ro(\cc{C}),\cc{C}at)_+$ stands for the full subcategory whose objects are those \mbox{2-functors} that preserve small 2-cofiltered pseudo-limits).
\end{theorem}
\begin{proof}
We will check that the 2-functor $c^*$ is essentially surjective on objects and \mbox{2-fully-faithful} (see \ref{eqsiif&f}):

\begin{itemize}
 \item[-] \emph{Essentially surjective on objects}: It follows from lemma \ref{casoCat}.

 \item[-] \emph{2-fully-faithful}: We will check that if $\ff{F}$ and $\ff{G}$ are 2-functors from $2\hbox{-}\cc{P}ro(\cc{C})$ to $\cc{C}at$ that preserve small 2-cofiltered pseudo-limits, then
\begin{equation} \label{hrest}
\cc{H}om(2\hbox{-}\cc{P}ro(\cc{C}),\cc{C}at)_+(\ff{F},\ff{G})
                    \mr{c^*} \cc{H}om(\cc{C},\cc{C}at)(\ff{F}c, \ff{G}c)
\end{equation}
is an isomorphism of categories.

Let $\ff{F}c \cellrd{\theta c}{\mu}{\eta c} \ff{G}c \in \cc{H}om(\cc{C},\cc{C}at)(\ff{F}c, \ff{G}c)$.
It can be easily checked that composites
$\Bigg\{\ff{F}\ff{X} \mr{\ff{F} \pi_i} \ff{F}\ff{X}_i
\cellrd{\theta_{\ff{X}_i}}{\mu_{\ff{X}_i}}{\eta_{\ff{X}_i}} \ff{G}\ff{X}_i\Bigg\}_{i\in \cc{I}}$
 determine two pseudo-cones for $\ff{G}\ff{X}$ together with a morphism of pseudo-cones. Since $\ff{G}$ preserves small \mbox{2-cofiltered} pseudo-limits, post-composing with
 $\ff{GX} \mr{\ff{G}\pi_i} \ff{GX}_i$ is an isomorphism of categories
 $\cc{C}at(\ff{FX}, \ff{GX}) \mr{(\ff{G} \pi)_*}
 \ff{PC}_{\cc{C}at}(\ff{FX}, \ff{GX})$.
 It follows that there \mbox{exists} a unique \mbox{2-cell} in $\cc{C}at$,
 $\ff{F}\ff{X}
\cellrd{\theta'_{\ff{X}}}{\mu'_{\ff{X}}}{\eta'_{\ff{X}}} \ff{G}\ff{X}$,
such that
$\ff{G}\pi_i \theta'_{\ff{X}} \,=\, \theta_{\ff{X}_i} \ff{F} \pi_i$,
$\ff{G}\pi_i \eta'_{\ff{X}} \,=\, \eta_{\ff{X}_i} \ff{F} \pi_i$, and
$\ff{G}\pi_i \mu'_{\ff{X}} \,=\, \mu_{\ff{X}_i} \ff{F} \pi_i$,
$\forall i\in \cc{I}$. It is not difficult to check that $\theta'_{\ff{X}}$, $\eta'_{\ff{X}}$ are in fact 2-natural on $\ff{X}$, and that $\mu'_{\ff{X}}$ is a modification. Clearly $\theta' c = \theta$, $\eta' c = \eta$, and
$\mu' c = \mu$. Thus \ref{hrest} is an isomorphism of categories.
\end{itemize}
\end{proof}

\begin{lemma}\label{es-sury}
 Let $\cc{C}$ be a 2-category, $\cc{E}$ a 2-category closed under
 small \mbox{2-cofiltered} pseudo-limits and $\ff{F}: \cc{C} \mr{} \cc{E}$ a 2-functor. Then, there exists a 2-functor
 \mbox{$\widehat{\ff{F}}: 2\hbox{-}\cc{P}ro(\cc{C}) \mr{} \cc{E}$}
 that preserves small 2-cofiltered pseudo-limits, and an isomorphism
 $\widehat{\ff{F}}c \mr{\cong} \ff{F}$ in $\cc{H}om(\cc{C},\cc{E})$.
 \end{lemma}

\begin{proof}
If $\ff{X}=(\ff{X}_i)_{i\in \cc{I}}\in 2\hbox{-}\cc{P}ro(\cc{C})$, define
$\widehat{\ff{F}}\ff{X}=\Lim{i\in \cc{I}}{\ff{F}\ff{X}_i}$. We will prove that this is the object function part of a 2-functor, and that this 2-functor has the rest of the properties asserted in the proposition.

Consider the composition
$\ff{y}_{(-)} \,\ff{F}:\,\cc{C}\mr{\ff{F}} \cc{E}\mr{\ff{y}_{(-)}}
 \cc{H}om(\cc{E}^{op},\cc{C}at)$,
 where $\ff{y}_{(-)}$ is the Yoneda 2-functor (\ref{Y2functor}). Under the isomorphism \ref{cartesianclosed} this corresponds to a \mbox{2-functor}
\mbox{$\cc{E}^{op} \mr{} \cc{H}om(\cc{C},\cc{C}at)$.} Composing this
\mbox{2-functor} with a quasi-inverse  $\widehat{(-)}$ for the \mbox{2-equivalence} in \ref{pseudouniversalcat}, we obtain a \mbox{2-functor}
$\cc{E}^{op} \mr{} \cc{H}om(2\hbox{-}\cc{P}ro(\cc{C}),\cc{C}at)_+$,
which in turn corresponds to a 2-functor
$ 2\hbox{-}\cc{P}ro(\cc{C}) \mr{\widetilde{\ff{F}}} \cc{H}om(\cc{E}^{op},\cc{C}at)$. The 2-functor $\widetilde{\ff{F}}$ preserves small 2-cofiltered pseudo-limits because they are computed pointwise in
$\cc{H}om(\cc{E}^{op},\cc{C}at)$ (\ref{pointwisebi-limit}). By chasing the isomorphisms one can check that we have the following diagram:
\begin{equation} \label{efetilde}
\xymatrix@C=0.5pc@R=0.5pc
           {
            { }
            & &  { }
           \\
            & {\widetilde{\ff{F}} c \mr{\cong} \ff{y}_{(-)} \ff{F}},
            &
           \\ { }
            & &  { }
           }
\hspace{6ex}
\xymatrix@C=0.5pc@R=0.5pc
           {
            \cc{C} \ar[rr]^<<<<<<<{c}  \ar[dd]_{\ff{F}}
            & &  2\hbox{-}\cc{P}ro(\cc{C}) \ar[dd]^{\widetilde{\ff{F}}}
           \\
            & \Downarrow \cong
            &
           \\ \cc{E} \ar[rr]_<<<<<{\ff{y}_{(-)}}
            & &  \cc{H}om(\cc{E}^{op},\cc{C}at)
           }
\end{equation}
Consider the following chain of isomorphisms (the first and the third because $\widetilde{\ff{F}}$ and $\ff{y}_{(-)}$ preserve pseudo-limits (\ref{yonedapreserves}), and the middle one given by \ref{efetilde}):
$$
\widetilde{\ff{F}}\ff{X} =
\widetilde{\ff{F}} \Lim{i \in \cc{I}}{\ff{X}_i} \mr{\cong}
\Lim{i \in \cc{I}}{\widetilde{\ff{F}} c \ff{X}_i} \mr{\cong}
\Lim{i \in \cc{I}}{\ff{y}_{(-)} \ff{F}\ff{X}_i} \ml{\cong}
\ff{y}_{(-)} \Lim{i \in \cc{I}}{\ff{F} \ff{X}_i}.
$$
This shows that $\widetilde{\ff{F}}\ff{X}$ is in the essential image of
$\ff{y}_{(-)}$. Since $\ff{y}_{(-)}$ is 2-fully-faithful (\ref{yonedaff}), it follows there is a factorization
\mbox{$\ff{y}_{(-)} \widehat{\ff{F}} \mr{\cong} \widetilde{\ff{F}}$,} given by a \mbox{2-functor}
$2\hbox{-}\cc{P}ro(\cc{C}) \mr{\widehat{\ff{F}}} \ff{E}$. Clearly
$\widehat{\ff{F}}$ preserves small \mbox{2-cofiltered} \mbox{pseudo-limits}. We have
$
\ff{y}_{(-)} \widehat{\ff{F}} c  \mr{\cong}
\widetilde{\ff{F}} c \mr{\cong}
\ff{y}_{(-)} \ff{F}.
$
Finally, the fully-faithfulness of $\ff{y}_{(-)}$ provides an isomorphism
$\widehat{\ff{F}}c \mr{\cong} \ff{F}$. This finishes the proof.
\end{proof}
Exactly the same proof of theorem \ref{pseudouniversalcat} applies with an arbitrary \mbox{2-category $\cc{E}$} in place of $\cc{C}at$, and we have:
\begin{theorem} \label{pseudouniversal}
Let $\cc{C}$ be any 2-category, and $\cc{E}$ a 2-category closed under small \mbox{2-cofiltered} pseudo-limits. Then, pre-composition with
\mbox{$\cc{C} \mr{c} 2\hbox{-}\cc{P}ro(\cc{C})$}  is a 2-equivalence of 2-categories:
$$
\xymatrix@R=.7pc@C=.3pc
 {
  & \cc{H}om(2\hbox{-}\cc{P}ro(\cc{C}),\cc{E})_+ \ar[rrrr]^>>>>>>{c^*}
  & & & & \comw{\ff{X}\ff{X}} \cc{H}om(\cc{C},\cc{E}) \comw{\ff{X} \ff{X}}
 }
$$
Where $\cc{H}om(2\hbox{-}\cc{P}ro(\cc{C}),\cc{E})_+$ stands for the full subcategory whose objects are those \mbox{2-functors} that preserve small 2-cofiltered pseudo-limits.
\cqd
\end{theorem}

From theorem \ref{pseudouniversal} it follows automatically the pseudo-functoriality of the assignment of the 2-category $2\hbox{-}\cc{P}ro(\cc{C})$ to each 2-category $\cc{C}$, and in such a way that $c$ becomes a pseudo-natural transformation. But we can do better:

\vspace{1ex}

If we put
$\cc{E} = \Pro{D}$ in \ref{pseudouniversal} it follows there is a 2-functor \mbox{(post-composing} with $c$ followed by a quasi-inverse in \ref{pseudouniversal})
\begin{equation} \label{ufa!}
\xymatrix@R=.7pc@C=.3pc
 {
  &  \cc{H}om(\cc{C},\cc{D})
\ar[rrrr]^>>>>>>{\widehat{(-)}}
  & & & & \comw{\ff{X}\ff{X}}
  \cc{H}om(2\hbox{-}\cc{P}ro(\cc{C}),\,
               2\hbox{-}\cc{P}ro(\cc{D}))_+
           \comw{\ff{X} \ff{X}} \hspace{-2.5ex},
 }
\end{equation}
and for each 2-functor $\cc{C} \mr{\ff{F}} \cc{D}$, a diagram:
\begin{equation} \label{pseudosquare}
\xymatrix@C=0.8pc@R=0.8pc
           {
            \Pro{C} \ar@{}[rrdd]|{\cong \; \Downarrow} \ar[rr]^{\widehat{\ff{F}}}
            & &  \Pro{D}
            \\
            & 
            &
            \\
            \cc{C} \ar[rr]_{\ff{\ff{F}}} \ar[uu]^{c}
            & & \cc{D} \ar[uu]_{c}
           }
\end{equation}
Given any 2-pro-object $\ff{X} \in \Pro{C}$, set
$2\hbox{-}\cc{P}ro(\ff{F})(\ff{X}) = \widehat{\ff{F}}\ff{X}$. It is straightforward to check that this determines a 2-functor
$$
\xymatrix@C=5ex@R=0.5pc
           {
            \Pro{C} \ar[rr]^
            {2\hbox{-}\cc{P}ro(\ff{F})}
            & &  \Pro{D}
           }
$$
making diagram \ref{pseudosquare} commutative. It follows we have an \mbox{isomorphism}
\mbox{$\widehat{\ff{\ff{F}}}\ff{X} \mr{\cong}
                   2\hbox{-}\cc{P}ro(\ff{F})(\ff{X})$} 2-natural in $\ff{X}$. This shows that the 2-functor $2\hbox{-}\cc{P}ro(\ff{F})$ preserves small 2-cofiltered pseudo-limits because
$\widehat{\ff{F}}$ does. Also, it follows that
$2\hbox{-}\cc{P}ro(\ff{F})$ determines a 2-functor as in \ref{ufa!}. In conclusion, denoting now by $2\hbox{-}\cc{CAT}$ the 2-category of locally small 2-categories (see \ref{2CAT}) we have:
\begin{theorem}
The definition $2\hbox{-}\cc{P}ro(\ff{F})(\ff{X}) = \widehat{\ff{F}}\ff{X}$ determines a \mbox{2-functor}
$$
2\hbox{-}\cc{P}ro(\hbox{-}): 2\hbox{-}\cc{C}\cc{A}\cc{T} \mr{} 2\hbox{-}\cc{C}\cc{A}\cc{T}_+
$$
in such a way that $c$ becomes a 2-natural transformation (where $2\hbox{-}\cc{C}\cc{A}\cc{T}_+$ is the full sub 2-category of locally small 2-categories closed under small \mbox{2-cofiltered} pseudo limits and small pseudo-limit preserving 2-functors). \cqd
\end{theorem}

\pagebreak

\begin{center}
 {\bf Resumen en castellano de la secci\'on \ref{2-Pro-objects}}
\end{center}

En esta secci\'on se encuentran algunos de los resultados claves de este trabajo. En \ref{def2pro}, dada una 2-categor\'ia $\cc{C}$ definimos la 2-categor\'ia $2$-$\cc{P}ro(\cc{C})$ cuyos objetos llamamos 2-pro-objetos. Un 2-pro-objeto de $\cc{C}$ es un 2-funtor a valores en $\cc{C}$ (o diagrama en $\cc{C}$)
indexado por una 2-categor\'ia 2-cofiltrante y ser\'a el pseudo-l\'imite de su propio diagrama en la 2-categor\'ia $\Pro{C}$. Tambi\'en en \ref{def2pro}, establecemos la f\'ormula b\'asica que describe los morfismos y las 2-celdas entre \mbox{2-pro-objetos} en t\'erminos de pseudo-l\'imites y pseudo-col\'imites de las categor\'ias de morfismos de $\cc{C}$.

En \ref{lemas2pro}, establecemos ciertos lemas t\'ecnicos que nos permiten operar con 2-pro-objetos en las secciones siguientes. 

En \ref{pseudo-limitsen2pro}, dada $\cc{J}$ una 2-categor\'ia 2-filtrante y un funtor
$\cc{J}^{op} \mr{\ff{X}}  2\hbox{-}\cc{P}ro(\cc{C})$, construimos un 2-pro-objeto que ser\'a el pseudo-l\'imite de $\ff{X}$ en $2\hbox{-} \cc{P}ro(\cc{C})$. Para esto, primero construimos una 2-categor\'ia \mbox{2-filtrante} que sirve como 2-categor\'ia de \'indices para el pseudo-l\'imite (Definici\'on \ref{kequis} y proposici\'on \ref{teo}). 

Finalmente, en \ref{pu2pro}, enunciamos y demostramos la propiedad universal de $2$-$\cc{P}ro(\cc{C})$ (Teorema \ref{pseudouniversal}) establecida para pro-objetos en [\citealp{G2}, Ex. I, Prop. 8.7.3.]. Considerar el 2-funtor $\cc{C} \mr{c} 2$-$\cc{P}ro(\cc{C})$ del corolario \ref{CisPro} y un 2-pro-objeto $\ff{X} = (\ff{X}_i)_{i\in \cc{I}}$. Dado un 2-funtor $\cc{C} \mr{\ff{F}} \cc{E}$ a una 2-categor\'ia cerrada por pseudo-l\'imites 2-cofiltrantes, podemos extender $\ff{F}$ a un 2-funtor que preserva pseudo-l\'imites 2-cofiltrantes
$2\hbox{-} \cc{P}ro(\cc{C}) \mr{\widehat{\ff{F}}} \cc{E}$ definiendo
$\widehat{\ff{F}}\ff{X} = \Lim{i \in \cc{I}}{\ff{F}\ff{X}_i}$.
Esto es solo una parte de una 2-equivalencia de 2-categor\'ias que desarrollamos aqu\'i. Primero debemos desarrollar completamente la propiedad universal para $\cc{E} = \cc{C}at$, y solo despu\'es de esto, puede ser traspasada a una 2-categor\'ia cualquiera $\cc{E}$ cerrada por pseudo-l\'imites 2-cofiltrantes.

Tambi\'en consideramos en esta secci\'on la 2-categor\'ia $\Prop{C}$ que es ``retract pseudo-equivalent'' a $\Pro{C}$, \ref{proppseudoeqapro}, hecho que se sigue de que los 2-funtores a valores en $\cc{C}at$ asociados a 2-pro-objetos son flexibles. Esta 2-categor\'ia ser\'a esencial en la secci\'on \ref{2-modelos} y probar\'a ser interesante en s\'i misma.

\pagebreak

\section{Reindexing properties for 2-pro-objects}\label{Mtrick}

In this section we prove some reindexing properties for the 2-categories $\Pro{C}$ and $\Prop{C}$ that will be used to determine closed 2-bmodel structures on them (see \ref{2-closed}) as Edwards and Hastings do in \cite{EH} for $\ff{Pro}(\ff{C})$ in the 1-dimensional case. The reindexing properties for $\ff{Pro}(\ff{C})$ can be found in \cite{AM} or \cite{G2}.

\subsection{Reindexing for objects}\label{objetos}

\begin{proposition}\label{reindexing}

Let $\ff{X}=(\ff{X}_j)_{j\in \cc{J}}$ be a 2-pro-object and $\F:\cc{I}\rightarrow \cc{J}$ be a 2-cofinal \mbox{2-functor} with $\cc{I}$ a 2-filtered 2-category. Then, the 2-pro-object $\ff{X}_\F=(\ff{X}_{\F(i)})_{i\in \cc{I}}$ is equivalent to $\ff{X}$ in $\Pro{C}$.   
\end{proposition}

\begin{proof}
First note that the 2-pro-objects $\ff{X}$ and $\ff{X}_\F$ are equivalent if the canonical 2-natural transformation $\coLim{i \in \cc{I}}{\cc{C}(\ff{X}_{\F(i)},-)}\stackrel{\theta}\Rightarrow \coLim{j\in \cc{J}}{\cc{C}(\ff{X}_j,-)}$ is an equivalence in $\cc{H}om(\cc{C},\cc{C}at)$.

Now, for each $\ff{C}\in \cc{C}$, consider the 2-functor $\xymatrix@R=0.3pc{\cc{J}\ar[r] & \cc{C}at \\ j\ar@{|->}[r] & \cc{C}(\ff{X}_j,\ff{C})}$. Then, by \ref{T_Cat}, $\coLim{i\in \cc{I}}{\cc{C}(\ff{X}_{\F(i)},\ff{C})}\stackrel{\theta_\ff{C}}\rightarrow \coLim{j\in \cc{J}}{\cc{C}(\ff{X}_j,\ff{C})}$ is an equivalence of categories $\forall \ff{C}\in \cc{C}$ and so, by \ref{colimflexible} and \ref{eqencadaCeseq}, $\theta$ is an equivalence.
\end{proof}

\begin{remark}
If we denote the equivalence given by the previous proposition $\ff{X}\mr{\ff{f}}\ff{X}_\F$ and its quasi-inverse $\overline{\ff{f}}$, then $\overline{\ff{f}}\ff{f}=id_{\ff{X}}$.   
\end{remark}

\begin{corollary}\label{reindexingparaobjetos}

It follows from \ref{phi} that every 2-pro-object $\ff{X}=(\ff{X}_j)_{j\in \cc{J}}$ is equivalent in  
$\Pro{C}$ to a 2-pro-object indexed by the cofinite filtered poset with a unique initial object
$\ff{M}(\cc{J})$ via a 2-cofinal 2-functor $\ff{M}(\cc{J}) \mr{\F} \cc{J}$. 

\end{corollary}

\begin{proof}
It is immediate from \ref{phi} and \ref{reindexing}. 
\end{proof}

Since every morphism in $\Pro{C}$ is a morphism in $\Prop{C}$, \ref{reindexing} and \ref{reindexingparaobjetos} also hold in $\Prop{C}$. However, it is worth mentioning that a proof similar to the one for \ref{reindexing} can be done for $\Prop{C}$ and \ref{colimflexible}, \ref{eqencadaCeseq} wouldn't be needed because of \ref{equivalencias pointwise}. 

%
%

\subsection{Reindexing for morphisms}\label{morfismos}

\begin{definition}\label{M_f}
Let $\ff{f}$ be a morphism from $\ff{X}=(\ff{X}_i)_{i\in \cc{I}}$ to $\ff{Y}=(\ff{Y}_j)_{j\in \cc{J}}$ in $\Pro{C}$. We are going to denote by $\cc{M}_\ff{f}$ the following 2-category:

Objects are the pairs $(\ff{r},\varphi)$ that represent $\ff{f}$. 

Morphisms $(\ff{r},\varphi)\rightarrow(\ff{s},\psi)$ ($\ff{r}:\ff{X}_i\rightarrow \ff{Y}_j$, $\ff{s}:\ff{X}_{i'}\rightarrow \ff{Y}_{j'}$) are triplex $(u,a,\theta)$ where $i\stackrel{u}\rightarrow i'\in\cc{I}$, $j\stackrel{a}\rightarrow j'\in\cc{J}$ and $\theta$ is an invertible 2-cell $\vcenter{\xymatrix@C=0pc{\ff{Y}_{a} \dl & \dc{\theta} & \ff{s} \dr \\
\ff{r} & & \ff{X}_u}}$ such that 


$$\vcenter{\xymatrix@C=-0pc{ \Y_a \dl & \dc{\theta} & \dr \s && \pi\si \deq \\
\fr \deq && \X_u && \pi\si \\
\fr \dl & \dcr{\varphi} && \dr \pi_i \cl{\pi_u} \\
\pi_j &&& \f }}
\vcenter{\xymatrix@C=-0pc{ \quad = \quad }}
\vcenter{\xymatrix@C=-0pc{ \Y_a \deq && \s \dl & \dc{\psi} & \dr \pi\si \\
\Y_a && \pi\sj && \f \deq \\
& \pi_j \cl{\pi_a} &&& \f }}$$

A 2-cell $(\ff{r},\varphi)\cellrd{(u,a,\theta)}{}{(v,b,\eta)}(\ff{s},\psi)$ is a pair $(\mu,\alpha)$ where $i\cellrd{u}{\mu}{v}i' \in \cc{I}$ and $j\cellrd{a}{\alpha}{b}j' \in \cc{J}$ are such that 

$$\vcenter{\xymatrix@C=-.3pc{\ff{Y}_a \dcellb{\ff{Y}_\alpha} & & \ff{s} \deq \\
	                     \ff{Y}_{b} \dl & \dc{\eta} & \ff{s} \dr \\
	                     \ff{r} & & \ff{X}_{v} }}
\vcenter{\xymatrix@C=-.3pc{\quad \quad = \quad \quad  }						}
\vcenter{\xymatrix@C=-.3pc{\ff{Y}_a \dl & \dc{\theta} &\ff{s} \dr \\
		           \ff{r} \deq & & \ff{X}_{u} \dcellb{\ff{X}_\mu} \\
		           \ff{r} & & \ff{X}_{v} }}$$

Identities and compositions are defined in the obvious way. 
\end{definition}

\begin{lemma}\label{M_f 2-filt}
Let $\ff{f}$ be a morphism from $\ff{X}=(\ff{X}_i)_{i\in \cc{I}}$ to $\ff{Y}=(\ff{Y}_j)_{j\in \cc{J}}$ in $\Pro{C}$. The \mbox{2-category} $\cc{M}_\ff{f}$ is 2-filtered and the 2-functors $\vcenter{\xymatrix@R=.3pc{\cc{M}_\ff{f}\ar[r] & \cc{I} \\ (\ff{r},\varphi)\ar@{|->}[r] & i}}$, $\vcenter{\xymatrix@R=.3pc{\cc{M}_\ff{f}\ar[r] & \cc{J} \\ (\ff{r},\varphi)\ar@{|->}[r] & j}}$ are 2-cofinal.  
\end{lemma}

\begin{proof}
$\cc{M}_\ff{f}$ is 2-filtered:

\begin{itemize}
\item[F0:] Let $(\ff{r},\varphi)$, $(\ff{s},\psi)\in \cc{M}_\ff{f}$ ($\ff{r}:\ff{X}_i\rightarrow \ff{Y}_j$, $\ff{s}:\ff{X}_{i'}\rightarrow \ff{Y}_{j'}$). Since $\cc{J}$ is 2-filtered, we have $\vcenter{\xymatrix@R=-0.5pc{j\ar[rd]^a & \\ & j'' \\ j'\ar[ru]_{b}}}\in \cc{J}$ and, by \ref{idrepresenta} and the fact that $\cc{I}$ is 2-filtered, we have $\ff{X}_{i''}\mr{\ff{t}}\ff{Y}_{j''} \in \cc{C}$ and an invertible 2-cell $\epsilon$ such that $(\ff{t},\epsilon)$ represents $\ff{f}$ and there are morphisms $\vcenter{\xymatrix@R=-0.5pc{i\ar[dr]^u \\
                          & i'' \\
                          i' \ar[ur]_ v}}\in \cc{I}$. 
Then we have $(\ff{r},\varphi)$ and $(\ff{t},\epsilon)$ both representing $\ff{f}$ equipped with morphisms $i\mr{u}i'' \in \cc{I}$, $j\mr{a}j'' \in \cc{J}$. So, by \ref{lemorphismaraM_f}, there are morphisms $\vcenter{\xymatrix@R=-0.5pc{i\ar[rd]^{u_0} \\
                          & \tilde{i} \\ i''\ar[ru]_{v_0}}}\in \cc{I}$ and an invertible 2-cell $\vcenter{\xymatrix@C=-.5pc{\ff{Y}_{a} \ar@{-}[dr] &&\ff{t} \dc{\theta} && \ff{X}_{v_0} \ar@{-}[dl] \\
& \ff{r} & & \ff{X}_{u_0} }}$ $\in \cc{C} $ such that 

   
$$\vcenter{\xymatrix@C=-0pc{ \Y_a \ardr && \ft \dc{\theta} && \Xvc \ardl && \pi\sit \deq \\
& \fr \deq && \Xuc &&& \pi\sit \\
& \fr \dl & \dcr{\varphi} && \dr \pi_i \clunodos{\piuc} \\
& \pi_j &&& \f }}
\vcenter{\xymatrix@C=-0pc{ \quad = \quad }}
\vcenter{\xymatrix@C=-0pc{ \Y_a \deq && \ft \deq && \Xvc && \pi\sit \\
\Y_a \deq && \ft \dl & \dcr{\eps} && \dr \pi\sii \cl{\pivc} \\
\Y_a && \pi\sjj &&& \f \deq \\
& \pi_j \cl{\pi_a} &&&& \f }}$$

Then we have $(\ff{s}, \psi)$ and $(\ff{t}\ff{X}_{v_0},  
\vcenter{\xymatrix@C=-0.2pc@R=1pc{ \ft \deq && \Xvc && \pi\sit \\
\ft \dl & \dcr{\eps} && \dr \pi\sii \cl{\pivc} \\
\pi\sjj &&& \f}}
)$ both representing $\ff{f}$ equipped with morphisms $i'\mr{v_0 v}\tilde{i} \in \cc{I}$, $j'\mr{b}j'' \in \cc{J}$. So, by \ref{lemorphismaraM_f}, there are morphisms $\vcenter{\xymatrix@R=-0.5pc{i'\ar[rd]^{u_1} \\
                          & i_0 \\
                          \tilde{i}\ar[ru]_{v_1}}}\in \cc{I}$ and an invertible 2-cell $\vcenter{\xymatrix@C=-.5pc{\ff{Y}_{b} \ar@{-}[dr] & \ff{t} \ar@{}[d]^{\eta}&\ff{X}_{v_0} & \ff{X}_{v_1} \ar@{-}[dl] \\
& \ff{s} & \ff{X}_{u_1} }}$ $\in \cc{C} $ such that 

$$\vcenter{\xymatrix@C=-0pc{  \Y_b \ardr && \ft & \dc{\eta} & \Xvc && \Xvu \ardl && \piic \deq \\
& \s \deq &&&& \Xuu &&& \piic \\
& \s \dl && \dcr{\psi} &&& \dr \pi\si \clunodos{\piuu} \\
& \pi\sj &&&&& \f }}
\vcenter{\xymatrix@C=-0pc{ \quad = \quad }}
\vcenter{\xymatrix@C=-0pc{ \Y_b \deq && \ft \deq && \Xvc \deq && \Xvu && \piic \\
\Y_b \deq && \ft \deq && \Xvc &&& \pi\sit \cl{\pivu} \\
\Y_b \deq && \ft \dl && \dc{\eps} && \dr \pi\sii \cldosuno{\pivc} \\
\Y_b && \pi\sjj &&&& \f \deq \\
& \pi\sj \cl{\pi_b} &&&&& \f }}$$

It can be checked that there are morphisms in $\cc{M}_\ff{f}$ $\vcenter{\xymatrix@R=-0.5pc{(\ff{r},\varphi) \ar[rrd]^<<<<<<<<<<{(v_1 u_0,a,\theta \ff{X}_{v_1})}\\
                          & &(\ff{t}\ff{X}_{v_0}\ff{X}_{v_1}, I)\\
                          (\ff{s},\psi)\ar[rru]_<<<<<<<<<<{(u_1,b,\eta)}}}$.   

\noindent where 
$$I=\vcenter{\xymatrix@C=-0pc{    \ft \deq && \Xvc \deq && \Xvu && \piic \\
\ft \deq && \Xvc &&& \pi\sit \cl{\pivu} \\
\ft \dl && \dc{\eps} && \dr \pi\sii \cldosuno{\pivc} \\
\pi\sjj &&&& \f }}$$                          
                          
\item[F1:] Let $(\ff{r},\varphi)\mrpair{(u,a,\theta)}{(v,b,\eta)}(\ff{s},\psi)\in \cc{M}_\ff{f}$. Since $\cc{J}$ is 2-filtered, we have $j'\stackrel{c}\rightarrow j''$ and an invertible 2-cell $ca\stackrel{\alpha}\Rightarrow cb \in \cc{J}$ and, by \ref{idrepresenta} and the fact that $\cc{I}$ is 2-filtered, we have $\ff{X}_{i''}\mr{\ff{t}} \ff{Y}_{j''} \in \cc{C}$ and an invertible 2-cell $\epsilon$ such that $(\ff{t},\epsilon)$  represents $\ff{f}$ and there are a morphism $i'\mr{w}i''$ and an invertible 2-cell $wu\Mr{\mu}wv \in \cc{I}$. 

Then we have $(\ff{s},\psi)$ and $(\ff{t},\epsilon)$ both representing $\ff{f}$ equipped with morphisms \mbox{$i'\mr{w}i'' \in \cc{I}$,} $j'\mr{c}j'' \in \cc{J}$. So, by \ref{lemorphismaraM_f}, there are morphisms $\vcenter{\xymatrix@R=-0.5pc{i'\ar[rd]^{u_0} \\
                          & i_0 \\
                          i''\ar[ru]_{v_0}}}\in \cc{I}$ and an invertible 2-cell $\vcenter{\xymatrix@C=-.5pc{\ff{Y}_{c} \ar@{-}[dr] &&\ff{t} \dc{\tilde{\theta}} & &\ff{X}_{v_0} \ar@{-}[dl] \\
& \ff{s} & & \ff{X}_{u_0} }}$ $\in \cc{C} $ such that 

\begin{equation}\label{A126}
\vcenter{\xymatrix@C=-0pc{ \Y_c \ardr && \ft \dc{\tilde{\theta}} && \Xvc \ardl && \piic \deq \\
& \s \deq && \Xuc &&& \piic \\
& \s \dl & \dcr{\psi} && \dr \pi\si \clunodos{\piuc} \\
& \pi\sj &&& \f }}
\vcenter{\xymatrix@C=-0pc{ \quad = \quad }}
\vcenter{\xymatrix@C=-0pc{ \Y_c \deq && \ft \deq && \Xvc && \piic \\
\Y_c \deq && \ft \dl & \dcr{\eps} && \dr \pi\sii \cl{\pivc} \\
\Y_c && \pi\sjj &&& \f \deq \\
& \pi\sj \cl{\pi_c} &&&& \f }}
\end{equation}
                            
Plus, since $\cc{I}$ is 2-filtered, there is a morphism $i_0\mr{\tilde{w}}\tilde{i}$ and an invertible 2-cell $\tilde{w}u_0 \Mr{\tilde{\mu}} \tilde{w}v_0 w \in \cc{I}$.

Let's check that we a have a morphism in $\cc{M}_\ff{f}$ \mbox{$(\ff{s},\psi)\mr{(\tilde{w}v_0 w,c,\ff{s}\ff{X}_{\tilde{\mu}}\circ \tilde{\theta}\ff{X}_{\tilde{w}})}(\ff{t}\ff{X}_{v_0}\ff{X}_{\tilde{w}}, 
\vcenter{\xymatrix@C=-0.2pc@R=1pc{    \ft \deq && \Xvc \deq && \X\swt && \pi\sit \\
\ft \deq && \Xvc &&& \piic \cl{\pi\swt} \\
\ft \dl && \dc{\eps} && \dr \pi\sii \cldosuno{\pivc} \\
\pi\sjj &&&& \f }}
)$:}

$$\vcenter{\xymatrix@C=-0.3pc{ \Y_c \ardr && \ft \dcr{\tilde{\theta}} &&& \Xvc \ardl && \X\swt \deq &&& \pi\sit \deq \\
& \s \deq &&& \Xuc \ardl & \dcr{\X\smut} && \X\swt \ardr &&& \pi\sit \deq \\
& \s \deq  && \X_w \ardrrr && \Xvc & \dc{\pi_{\tilde{w} v_0 w}} && \X\swt && \pi\sit \ardllll \\
& \s \dl && \dcr{\psi} &&& \dr \pi\si \\
& \pi\sj &&&&& \f }}
\vcenter{\xymatrix@C=-0pc{ \quad = \quad }}
\vcenter{\xymatrix@C=-0.2pc{ \Y_c \ardr && \ft \dcr{\tilde{\theta}} &&& \Xvc \ardl && \X\swt \deq &&& \pi\sit \deq \\
& \s \deq &&& \Xuc \ardl & \dcr{\X\smut} && \X\swt \ardr &&& \pi\sit \deq \\
& \s \deq && \X_w \deq && \Xvc \deq &&& \X\swt && \pi\sit \\
& \s \deq && \X_w \deq && \Xvc &&&& \piic \cl{\pi\swt} \\
& \s \deq && \X_w &&&& \pi\sii \cldosdos{\pivc} \\
& \s \dl && \dc{\psi} && \dr \pi\si \cldosdos{\pi_w} \\
& \pi\sj &&&& \f }}
\vcenter{\xymatrix@C=-0pc{ = \quad }}$$

$$\vcenter{\xymatrix@C=-0pc{ \quad = \quad }}
\vcenter{\xymatrix@C=-0pc{ \Y_c \deq && \ft \deq && \Xvc \deq && \X\swt && \pi\sit \\
\Y_c \ardr && \ft \dc{\tilde{\theta}} && \Xvc \ardl &&& \piic \deq \cl{\pi\swt} \\
& \s \deq && \Xuc &&&& \piic \\
& \s \dl && \dc{\psi} && \dr \pi\si \cldosdos{\piuc} \\
& \pi\sj &&&& \f }}
\vcenter{\xymatrix@C=-0pc{ \quad = \quad }}
\vcenter{\xymatrix@C=-0pc{ \Y_c \deq && \ft \deq && \Xvc \deq && \X\swt && \pi\sit \\
\Y_c \deq && \ft \deq && \Xvc &&& \piic \cl{\pi\swt} \\
\Y_c \deq && \ft \dl && \dc{\eps} && \dr \pi\sii \cldosuno{\pivc} \\
\Y_c && \pi\sjj &&&& \f \deq \\
& \pi_j \cl{\pi_c} &&&&& \f }}$$

\noindent where the first equality follows from axiom PC1, the second one holds by elevators calculus plus axiom PC2 and the last one is due to \eqref{A126}. 

Now observe that

$$\vcenter{\xymatrix@C=-0.4pc{ \Y_a \dl & \dc{\Y_{\alpha}} & \dr \Y_c &&& \ft \deq && \Xvc \deq && \X\swt \deq && \quad & \pi\sit \deq \\
\Y_b \deq && \Y_c \ardr && \dcr{\tilde{\theta}} & \ft && \Xvc \ardl && \X\swt \deq &&& \pi\sit \deq \\
\Y_b \deq && & \s \deq &&& \Xuc \ardl & \dcr{\X\smut} && \X\swt \ardr &&& \pi\sit \deq \\
\Y_b \dl & \dcr{\eta} && \dr \s && \X_w \deq && \Xvc \deq &&& \X\swt \deq && \pi\sit \deq \\
\fr &&& \X_v && \X_w && \Xvc &&& \X\swt && \pi\sit }}
\vcenter{\xymatrix@C=-0pc{ \quad = \quad }}
\vcenter{\xymatrix@C=-0.2pc{ \Y_a \dl & \dc{\Y_{\alpha}} & \dr \Y_c && \ft \deq && \Xvc \deq && \X\swt \deq &&&& \pi\sit \deq \\
\Y_b \deq && \Y_c \ardr && \ft \dc{\tilde{\theta}} && \Xvc \ardl && \X\swt \deq &&&& \pi\sit \deq \\
\Y_b \deq && & \s \deq && \Xuc \deq &&& \X\swt &&&& \pi\sit \\
\Y_b \deq &&& \s \deq && \Xuc &&& && \piic \cldosdos{\pi\swt} \\
\Y_b \deq &&& \s \deq && &&& \pi\si \cltresdos{\piuc} \opdosdos{\pi_w\inv} \\
\Y_b \deq &&& \s \deq &&& \X_w \deq &&&& \pi\sii \opdosuno{\pivc\inv} \\
\Y_b \deq &&& \s \deq &&& \X_w \deq && \Xvc \deq &&& \piic \op{\pi\swt\inv} \\
\Y_b \dl & \dcr{\eta} && \dr \s &&& \X_w \deq && \Xvc \deq && \X\swt \deq && \pi\sit \deq \\
\fr &&& \X_v &&& \X_w && \Xvc && \X\swt && \pi\sit }}
\vcenter{\xymatrix@C=-0pc{ \!\!\!\!\!\! = }}
$$

$$\vcenter{\xymatrix@C=-0.3pc{  \Y_a \dl & \dc{\Y_{\alpha}} & \dr \Y_c && \ft \deq && \Xvc \deq && \X\swt &&&& \pi\sit \\
\Y_b \deq && \Y_c \deq && \ft \deq && \Xvc && && \piic \cldosdos{\pi\swt} && \\
\Y_b \deq && \Y_c \deq && \ft \dl && \dc{\eps} && \dr \pi\sii \cldosdos{\pivc} \\
\Y_b \deq && \Y_c && \pi\sjj &&&& \f \deq \\
\Y_b \deq &&& \pi\sj \cl{\pi_c} \dl && \dcr{\psi\inv} &&& \dr \f \\
\Y_b \dl & \dcr{\eta} && \dr \s &&&&& \pi\si \deq \\
\fr \deq &&& \X_v \deq &&&&& \pi\si \opdosdos{\pi_w\inv} \\
\fr \deq &&& \X_v \deq &&& \X_w \deq &&&& \pi\sii \opdosuno{\pivc\inv} \\
\fr \deq &&& \X_v \deq &&& \X_w \deq && \Xvc \deq &&& \piic \opb{\pi\swt\inv} \\
\fr &&& \X_v &&& \X_w && \Xvc && \X\swt && \pi\sit }}
\vcenter{\xymatrix@C=-0pc{  =  }}
\vcenter{\xymatrix@C=-0.3pc{\Y_a \dl & \dc{\Y_{\alpha}} & \dr \Y_c && \ft \deq && \Xvc \deq && \X\swt &&&& \pi\sit \\
\Y_b \deq && \Y_c \deq && \ft \deq && \Xvc && && \piic \cldosdos{\pi\swt} && \\
\Y_b \deq && \Y_c \deq && \ft \dl && \dc{\eps} && \dr \pi\sii \cldosdos{\pivc} \\
\Y_b \deq && \Y_c && \pi\sjj &&&& \f \deq \\
\Y_b &&& \pi\sj \cl{\pi_c} &&&&& \f \deq \\
& \pi_j \clunodos{\pi_b} \dl &&& \dcr{\varphi\inv} &&&& \dr \f \\
& \fr \deq &&&&&&& \pi_i \opdosdos{\pi_v\inv} \\
& \fr \deq &&&&& \X_v \deq &&&& \pi\si \opdosdos{\pi_w\inv} \\
& \fr \deq &&&&& \X_v \deq && \X_w \deq &&&& \pi\sii \opdosuno{\pivc\inv} \\
& \fr \deq &&&&& \X_v \deq && \X_w \deq && \Xvc \deq &&& \piic \opb{\pi\swt\inv} \\
& \fr &&&&& \X_v && \X_w && \Xvc && \X\swt && \pi\sit}}
\vcenter{\xymatrix@C=-0pc{ = \quad }}
$$                    

$$\vcenter{\xymatrix@C=-0.3pc{  \Y_a \deq && \Y_c \deq && \ft \deq && \Xvc \deq && \X\swt &&&& \pi\sit \\
\Y_a \deq && \Y_c \deq && \ft \deq && \Xvc && && \piic \cldosdos{\pi\swt} && \\
\Y_a \deq && \Y_c \deq && \ft \dl && \dc{\eps} && \dr \pi\sii \cldosdos{\pivc} \\
\Y_a \deq && \Y_c && \pi\sjj &&&& \f \deq \\
\Y_a &&& \pi\sj \cl{\pi_c} &&&&& \f \deq \\
& \pi_j \clunodos{\pi_a} \dl &&& \dcr{\varphi\inv} &&&& \dr \f \\
& \fr \deq &&&&&&& \pi_i \opdosdos{\pi_u\inv} \\
& \fr \deq &&&&& \X_u \deq &&&& \pi\si \opdosdos{\pi_w\inv} \\
& \fr \deq &&&&& \X_u \dl & \dc{\X_{\mu}} & \dr \X_w &&&& \pi\sii \deq \\
& \fr \deq &&&&& \X_v \deq && \X_w \deq &&&& \pi\sii \opdosuno{\pivc\inv} \\
& \fr \deq &&&&& \X_v \deq && \X_w \deq && \Xvc \deq &&& \piic \opb{\pi\swt\inv} \\
& \fr &&&&& \X_v && \X_w && \Xvc && \X\swt && \pi\sit }}
\vcenter{\xymatrix@C=-0pc{\!\!\!\!\!\!\!\!  = \!\!\!\! }}
\vcenter{\xymatrix@C=-0.3pc{ \Y_a \deq && \Y_c \deq && \ft \deq && \Xvc \deq && \X\swt &&&& \pi\sit \\
\Y_a \deq && \Y_c \ardr && \ft \dc{\tilde{\theta}} && \ardl \Xvc && && \piic \deq \cldosdos{\pi\swt} && \\
\Y_a \deq &&& \s \deq && \Xuc &&&&& \piic \\
\Y_a \deq &&& \s \dl && \dcr{\psi} &&& \dr \pi\si \cltresdos{\piuc} \\
\Y_a &&& \pi\sj &&&&& \f \deq \\
& \pi_j \clunodos{\pi_a} \dl &&& \dcr{\varphi\inv} &&&& \dr \f \\
& \fr \deq &&&&&&& \pi_i \opdosdos{\pi_u\inv} \\
& \fr \deq &&&&& \X_u \deq &&&& \pi\si \opdosdos{\pi_w\inv} \\
& \fr \deq &&&&& \X_u \dl & \dc{\X_{\mu}} & \dr \X_w &&&& \pi\sii \deq \\
& \fr \deq &&&&& \X_v \deq && \X_w \deq &&&& \pi\sii \opdosuno{\pivc\inv} \\
& \fr \deq &&&&& \X_v \deq && \X_w \deq && \Xvc \deq &&& \piic \opb{\pi\swt\inv} \\
& \fr &&&&& \X_v && \X_w && \Xvc && \X\swt && \pi\sit }}
\vcenter{\xymatrix@C=-0pc{ \!\!\!\!\!\!\!\!\! = \quad }}
$$   

$$\vcenter{\xymatrix@C=-0.2pc{ \Y_a \deq && \Y_c \ardr && \ft \dc{\tilde{\theta}} && \ardl \Xvc && \X\swt \deq && \pi\sit \deq \\
\Y_a \deq &&& \s \deq && \Xuc \deq &&& \X\swt && \pi\sit \\
\Y_a \deq &&& \s \deq && \Xuc &&&& \piic \cl{\pi\swt} \\
\Y_a \dl & \dcr{\theta} && \dr \s &&&& \pi\si \deq \cldosdos{\piuc} \\
\fr \deq &&& \X_u \deq &&&& \pi\si \opdosdos{\pi_w\inv} \\
\fr \deq &&& \X_u \deq && \X_w \deq &&&& \pi\sii \opdosuno{\pivc\inv} \\
\fr \deq &&& \X_u \deq && \X_w \deq && \Xvc \deq &&& \piic \op{\pi\swt\inv} \\
\fr \deq &&& \X_u \dl & \dc{\X_{\mu}} & \dr \X_w && \Xvc \deq && \X\swt \deq && \pi\sit \deq \\
\fr &&& \X_v && \X_w && \Xvc && \X\swt && \pi\sit }}
\vcenter{\xymatrix@C=-0pc{  =  }}
\vcenter{\xymatrix@C=-0.2pc{ \Y_a \deq && \Y_c \ardr && \dcr{\tilde{\theta}} & \ft && \Xvc \ardl && \X\swt \deq &&& \pi\sit \deq \\
\Y_a \deq &&& \s \deq &&& \Xuc \ardl & \dcr{\X\smut} && \X\swt \ardr &&& \pi\sit \deq \\
\Y_a \dl & \dcr{\theta} && \dr \s && \X_w \deq && \Xvc \deq &&& \X\swt \deq && \pi\sit \deq \\
\fr \deq &&& \X_u \dl & \dc{\X_{\mu}} & \dr \X_w && \Xvc \deq &&& \X\swt \deq && \pi\sit \deq \\
\fr &&& \X_v && \X_w && \Xvc &&& \X\swt && \pi\sit }}$$

\noindent where the first equality is due to axiom PC2, the second one and the fifth one require some elevators calculus plus \eqref{A126}, the third holds because $(v,b,\eta)$ is a morphism in $\cc{M}_{\ff{f}}$, the fourth one and the last one are valid by elevators calculus plus axiom PC2 and the sixth one is due to elevators calculus plus the fact that $(u,a, \theta)$ is a morphism in $\cc{M}_{\ff{f}}$.

Then, by \ref{lema4}, there exist a morphism $\tilde{i}\mr{w_0}i_1 \in \cc{I}$ such that 

$$\vcenter{\xymatrix@C=-0.5pc{ \ff{Y}_a \dl & \dc{\ff{Y}_\alpha} & \ff{Y}_c \dr &&& \ff{t} \deq &&& \ff{X}_{v_0} \deq && \ff{X}_{\tilde{w}} \deq & \ff{X}_{w_0} \deq \\
                    \ff{Y}_b \deq & & \ff{Y}_c \dl & && \ff{t} \ar@{}[d]^{\tilde{\theta}} &&&  \ff{X}_{v_0} \dr && \ff{X}_{\tilde{w}} \deq  & \ff{X}_{w_0} \deq\\
                    \ff{Y}_b \deq & & && \ff{s} \deq &&  \ff{X}_{u_0} \dr && \dc{\ff{X}_{\tilde{\mu}}} && \ff{X}_{\tilde{w}} \dl & \ff{X}_{w_0} \deq  \\
                    \ff{Y}_b \dl && \dc{\eta} & & \ff{s} \dr && \ff{X}_{w} \deq && \ff{X}_{v_0} \deq && \ff{X}_{\tilde{w}} \deq & \ff{X}_{w_0} \deq  \\
                    \ff{r} & &&& \ff{X}_{v} & &\ff{X}_{w} & & \ff{X}_{v_0} & &\ff{X}_{\tilde{w}} & \ff{X}_{w_0}}}
\vcenter{\xymatrix@C=-.3pc{\quad \quad = \quad \quad  }}
\vcenter{\xymatrix@C=-.5pc{ \ff{Y}_a \deq & & & \ff{Y}_c \ar@{-}[dr] && \ff{t} \dc{\tilde{\theta}} && \ff{X}_{v_0} \ar@{-}[dl] & \ff{X}_{\tilde{w}} \deq & \ff{X}_{w_0} \deq\\
                    \ff{Y}_a \deq & & & & \ff{s} \deq & &  \ff{X}_{u_0} \dr & \dc{\ff{X}_{\tilde{\mu}}} & \ff{X}_{\tilde{w}} \dl & \ff{X}_{w_0} \deq\\
                    \ff{Y}_a \dl & & \dc{\theta} & & \ff{s} \dr & & \ff{X}_{w} \deq & \ff{X}_{v_0} \deq & \ff{X}_{\tilde{w}} \deq & \ff{X}_{w_0} \deq\\
                    \ff{r} \deq & & & & \ff{X}_{u} \dl & \dc{\ff{X}_\mu} & \ff{X}_{w} \dr & \ff{X}_{v_0} \deq & \ff{X}_{\tilde{w}} \deq & \ff{X}_{w_0} \deq \\
                    \ff{r} & & & & \ff{X}_{v} & & \ff{X}_{w} & \ff{X}_{v_0} & \ff{X}_{\tilde{w}} & \ff{X}_{w_0} }}$$

We can conclude that we a have a morphism in $\cc{M}_\ff{f}$ \mbox{$(\ff{s},\psi)\mr{(w_0\tilde{w}v_0 w,c,\ff{s}\ff{X}_{\tilde{\mu}}\ff{X}_{w_0}\circ \tilde{\theta}\ff{X}_{\tilde{w}}\ff{X}_{w_0})}(\ff{t}\ff{X}_{v_0}\ff{X}_{\tilde{w}}\ff{X}_{w_0}, 
\vcenter{\xymatrix@C=-0.2pc@R=1pc{ \ft \deq && \Xvc \deq && \X\swt \deq && \Xwc && \piiu \\
\ft \deq && \Xvc \deq && \X\swt &&& \pi\sit \cl{\piwc} \\
\ft \deq && \Xvc &&&& \piic \cldosuno{\pi\swt} \\
\ft \dl && \dc{\eps} && \dr \pi\sii \cldosdos{\pivc} \\
\pi\sjj &&&& \f}}
)$} and an invertible 2-cell in $\cc{M}_\ff{f}$, 

\noindent $(w_0\tilde{w}v_0 w,c,\ff{s}\ff{X}_{\tilde{\mu}}\ff{X}_{w_0}\circ \tilde{\theta}\ff{X}_{\tilde{w}}\ff{X}_{w_0})(u,a,\theta)\Mr{(w_0\tilde{w}v_0 \mu,\alpha)}$

$\hfill (w_0\tilde{w}v_0 w,c,\ff{s}\ff{X}_{\tilde{\mu}}\ff{X}_{w_0}\circ \tilde{\theta}\ff{X}_{\tilde{w}}\ff{X}_{w_0})(v,b,\eta)$.

\item[F2:] Let $(\ff{r},\varphi)\cellpairrd{(u,a,\theta)}{(\mu,\alpha)}{(\rho,\beta)}{(v,b,\eta)}(\ff{s},\psi)\in \cc{M}_\ff{f}$. Since $\cc{J}$ is 2-filtered, we have $j'\stackrel{c}\rightarrow j''\in \cc{J}$ such that $c\alpha=c\beta$ and, by \ref{idrepresenta} and the fact that $\cc{I}$ is 2-filtered, we have \mbox{$\ff{t}:\ff{X}_{i''}\rightarrow \ff{Y}_{j''} \in \cc{C}$} and an invertible 2-cell $\epsilon$ such that $(\ff{t},\epsilon)$ represents $\ff{f}$ and there is a morphism \mbox{$i'\mr{w}\tilde{i}\in \cc{I}$} such that $w\mu=w\rho$.

Now, since $\cc{I}$ is 2-filtered, we have morphisms $\vcenter{\xymatrix@R=-0.5pc{\tilde{i} \ar[rd]^{u_0} \\
                                                                               & i_0 \\
                                                                               i'' \ar[ru]_{v_0}}}\in \cc{I}$.
                                                                               
Then we have $(\ff{s},\psi)$ and $(\ff{t}\ff{X}_{v_0}, \vcenter{\xymatrix@C=-0.2pc@R=1pc{ \ft \deq && \Xvc && \pi\sit \\
\ft \dl & \dcr{\eps} && \dr \pi\sii \cl{\pivc} \\
\pi\sjj &&& \f}})$ both representing $\ff{f}$ equipped with morphisms $i'\mr{u_0 w} i_0 \in \cc{I}$, $j'\mr{c}j'' \in \cc{J}$. Then, by \ref{lemorphismaraM_f}, we have morphisms $\vcenter{\xymatrix@R=-0.5pc{i' \ar[rd]^{u_1} \\
                                                                               & i_1 \\
                                                                               i_0 \ar[ru]_{v_1}}}\in \cc{I}$ and an invertible 2-cell $\vcenter{\xymatrix@C=-.5pc{\ff{Y}_{c} \ar@{-}[dr] &\ff{t}
                                                                               \ar@{}[d]^{\tilde{\theta}} & \ff{X}_{v_0} & \ff{X}_{v_1}\ar@{-}[dl] \\
& \ff{s} & \ff{X}_{u_1} }}$ $\in \cc{C} $ such that 

$$\vcenter{\xymatrix@C=-0pc{ \Y_c \ardr && \ft & \dc{\tilde{\theta}} & \Xvc && \Xvu \ardl && \piiu \deq \\
& \s \deq &&&& \Xuu &&& \piiu \\
& \s \dl && \dcr{\psi} &&& \dr \pi\si \clunodos{\piuu} \\
& \pi\sj &&&&& \f }}
\vcenter{\xymatrix@C=-0pc{ \quad = \quad }}
\vcenter{\xymatrix@C=-0pc{ \Y_c \deq && \ft \deq && \Xvc \deq && \Xvu && \piiu \\
\Y_c \deq && \ft \deq && \Xvc &&& \piic \cl{\pivu} \\
\Y_c \deq && \ft \dl && \dc{\eps} && \dr \pi\sii \cldosuno{\pivc} \\
\Y_c && \pi\sjj &&&& \f \deq \\
& \pi\sj \cl{\pi_c} &&&&& \f}}$$

Plus, since $\cc{I}$ is 2-filtered, we have a morphism $i_1\mr{\tilde{w}}i_2$ and an invertible 2-cell $\tilde{w}u_1\Mr{\tilde{\mu}} \tilde{w}v_1 u_0 w \in \cc{I}$.

It can be checked that there is a morphism in $\cc{M}_\ff{f}$ \mbox{$(\ff{s},\psi)\mr{(\tilde{w}v_1 u_0 w,c,\ff{s}\ff{X}_{\tilde{\mu}}\circ \tilde{\theta}\ff{X}_{\tilde{w}})} (\ff{t}\ff{X}_{v_0}\ff{X}_{v_1}\ff{X}_{\tilde{w}}, 
\vcenter{\xymatrix@C=-0.2pc@R=1pc{ \ft \deq && \Xvc \deq && \Xvu \deq && \X\swt && \piiu \\
\ft \deq && \Xvc \deq && \Xvu &&& \piiu \cl{\pi\swt} \\
\ft \deq && \Xvc &&&& \piic \cldosuno{\pivu} \\
\ft \dl && \dc{\eps} && \dr \pi\sii \cldosdos{\pivc} \\
\pi\sjj &&&& \f}}
)$} such that \mbox{$(\tilde{w}v_1 u_0 w,c,\ff{s}\ff{X}_{\tilde{\mu}}\circ \tilde{\theta}\ff{X}_{\tilde{w}})(u,a,\theta)=(\tilde{w}v_1 u_0 w,c,\ff{s}\ff{X}_{\tilde{\mu}}\circ \tilde{\theta}\ff{X}_{\tilde{w}})(v,b,\eta)$.}

\end{itemize}

$\vcenter{\xymatrix@R=.3pc{\cc{M}_\ff{f}\ar[r] & \cc{I} \\ (\ff{r},\varphi)\ar@{|->}[r] & i}}$ is 2-cofinal:

\begin{itemize}
\item[CF0:] Let $i\in \cc{I}$ and let $\ff{X}_{i'}\mr{\ff{r}} \ff{Y}_j \in \cc{C}$ such that $(\ff{r},id)$ represents $\ff{f}$. Since $\cc{I}$ is 2-filtered, we have $\vcenter{\xymatrix@R=-0.5pc{i\ar[rd]^{u} & \\ & i'' \\ i'\ar[ru]_{v}}}\in \cc{I}$. It is straightforward to check that $(\ff{r}\ff{X}_{v},\ff{r}\pi_{v})$ together with $i\mr{u}i'' \in \cc{I}$ proves CF0.

\item[CF1:] Let $i\in \cc{I}$, $(\ff{r},\varphi)\in \cc{M}_\ff{f}$ ($\ff{r}:\ff{X}_{i'}\rightarrow \ff{Y}_j$) and $i\mrpair{u}{v}i'\in \cc{I}$. Since $\cc{I}$ is 2-filtered, we have $i'\stackrel{w}\rightarrow i''$ and an invertible 2-cell $wu\stackrel{\mu}\Rightarrow wv \in \cc{I}$. It is straightforward to check that $(\ff{r},\varphi)\mr{ (w,id,id)}(\ff{r}\ff{X}_w, 
\vcenter{\xymatrix@C=-0.2pc@R=1pc{ \fr \deq && \X_w && \pi\sii \\ \fr \dl & \dcr{\varphi} && \dr \pi\si \cl{\pi_w} \\ \pi\sj &&& \f  }}
)$ proves CF1. 

\item[CF2:] Let $i\in \cc{I}$, $(\ff{r},\varphi)\in \cc{M}_\ff{f}$ ($\ff{r}:\ff{X}_{i'}\rightarrow \ff{Y}_j$) and $i\cellpairrd{u}{\mu}{\rho}{v}i' \in \cc{I}$. Since $\cc{I}$ is \mbox{2-filtered}, we have $i'\stackrel{w}\rightarrow i'' \in \cc{I}$ such that $w\mu=w\rho$. It is straightforward to check that \mbox{$(\ff{r},\varphi)\mr{ (w,id,id)}(\ff{r}\ff{X}_w, \vcenter{\xymatrix@C=-0.2pc@R=1pc{ \fr \deq && \X_w && \pi\sii \\ \fr \dl & \dcr{\varphi} && \dr \pi\si \cl{\pi_w} \\ \pi\sj &&& \f  }} )$} proves CF2. 

\end{itemize}

$\vcenter{\xymatrix@R=.3pc{\cc{M}_\ff{f}\ar[r] & \cc{J} \\ (\ff{r},\varphi)\ar@{|->}[r] & j}}$ is 2-cofinal:

\begin{itemize}
\item[CF0:] Let $j\in \cc{J}$. By \ref{idrepresenta}, we have $\ff{r}:\ff{X}_i\rightarrow \ff{Y}_j \in \cc{C}$ such that $(\ff{r},id)$ represents $\ff{f}$. This clearly proves CF0.

\item[CF1:] Let $j_0\in \cc{J}$, $(\ff{r},\varphi)\in \cc{M}_\ff{f}$ ($\ff{r}:\ff{X}_i\rightarrow \ff{Y}_j$) and $j_0\mrpair{a}{b}j \in \cc{J}$. Since $\cc{J}$ is 2-filtered, we have $j\stackrel{c}\rightarrow j'$ and an invertible 2-cell $ca\stackrel{\alpha}\Rightarrow cb \in \cc{J}$. Now, by \ref{idrepresenta}, we have $\ff{s}:\ff{X}_{i'}\rightarrow \ff{Y}_{j'} \in \cc{C}$ such that $(\ff{s},id)$ represents $\ff{f}$. 

From the proof of the fact that $\cc{M}_\ff{f}$ is 2-filtered, we have morphisms $\vcenter{\xymatrix@R=-0.5pc{(\ff{r},\varphi)\ar[rd]^{(u,c,\theta)} & \\ & (\ff{t},\psi)\\ (\ff{s},id)\ar[ru]_{(v,id,\eta)}}}$. It is straightforward to check that $(\ff{r},\varphi)\stackrel{(u,c,\theta)}\rightarrow (\ff{t},\psi)$ proves CF1. 

\item[CF2:] Let $j_0\in \cc{J}$, $(\ff{r},\varphi)\in \cc{M}_\ff{f}$ ($\ff{r}:\ff{X}_i\rightarrow \ff{Y}_j$) and $j_0\cellpairrd{a}{\alpha}{\beta}{b}j \in \cc{J}$. Since $\cc{J}$ is 2-filtered, we have $j\stackrel{c}\rightarrow j' \in \cc{J}$ such that $c\alpha=c\beta$. Now, by \ref{idrepresenta}, we have $\ff{s}:\ff{X}_{i'}\rightarrow \ff{Y}_{j'} \in \cc{C}$ such that $(\ff{s},id)$ represents $\ff{f}$. 

From the proof of the fact that $\cc{M}_\ff{f}$ is 2-filtered, we have morphisms $\vcenter{\xymatrix@R=-0.5pc{(\ff{r},\varphi)\ar[rd]^{(u,c,\theta)} & \\ & (\ff{t},\psi)\\ (\ff{s},id)\ar[ru]_{(v,id,\eta)}}}$. It is straightforward to check that $(\ff{r},\varphi)\stackrel{(u,c,\theta)}\rightarrow (\ff{t},\psi)$ proves CF2. 
\end{itemize}

\end{proof}

\begin{proposition}\label{reindexingparamorfismos}
Every morphism of 2-pro-objects $\ff{X}=(\ff{X}_i)_{i\in \cc{I}}\stackrel{\ff{f}}\rightarrow \ff{Y}=(\ff{Y}_j)_{j\in \cc{J}}$ can be represented up to equivalence by a 2-pro-object $\{\ff{X}'_m\stackrel{\ff{f}_m}\rightarrow \ff{Y}'_m\}_{m\in \cc{M}}$ in $2$-$\cc{P}ro(\cc{H}om_p(\ff{2},\cc{C}))$, i.e. $\exists$ a 2-filtered 2-category $\cc{M}$, 2-pro-objects $\ff{X}'=(\ff{X}'_m)_{m\in \cc{M}}$, $\ff{Y}'=(\ff{Y}'_m)_{m\in \cc{M}}$ and a morphism $\X' \mr{\f'} \Y'$ 
such that the following diagram commutes in $2$-$\cc{P}ro(\cc{C})$ up to isomorphism:

\begin{equation}\label{cuadradoMTflechas}
\xymatrix@R=3pc@C=3pc{\ff{X} \ar@{}[rd]|{\cong} \ar[r]^{\ff{f}} \ar[d]_\simeq & \ff{Y} \ar[d]^\simeq \\ \ff{X}' \ar[r]_{\ff{f}'} & \ff{Y}'} 
\end{equation}

\end{proposition}

\begin{proof}
Take $\cc{M}=\cc{M}_\ff{f}$ as defined in \ref{M_f}, $\ff{X}'_{(\ff{r},\varphi)}=\ff{X}_i$ and $\ff{Y}'_{(\ff{r},\varphi)}=\ff{Y}_j$ (if $\ff{r}:\ff{X}_i\rightarrow \ff{Y}_j$) and $\ff{f}_{(\ff{r},\varphi)}=\ff{r}$.

Since $\vcenter{\xymatrix@R=.3pc{\cc{M}_\ff{f}\ar[r] & \cc{I} \\ (\ff{r},\varphi)\ar@{|->}[r] & i}}$ and $\vcenter{\xymatrix@R=.3pc{\cc{M}_\ff{f}\ar[r] & \cc{J} \\ (\ff{r},\varphi)\ar@{|->}[r] & j}}$ are 2-cofinal, by \ref{reindexing}, $\ff{X}'$ is equivalent to $\ff{X}$ and $\ff{Y}'$ is equivalent to $\ff{Y}$. 
It is straightforward to check that diagram \eqref{cuadradoMTflechas} commutes up to the isomorphism given by $\varphi$ and the universal property of $\ff{Y}'$. 
\end{proof}

The previous proposition can be also stated as follows:

\begin{remark}\label{reindexing morfismos como lifting}
Every object $\f\in \cc{H}om(\ff{2},\Pro{C})$ have a lifting to $\cc{H}om(\cc{M}^{op},\cc{C})$ up to equivalence for some 2-filtered 2-category $\cc{M}$.

$$\xymatrix{ & \cc{H}om(\cc{M}^{op},\cc{C}) \ar[d]^{inc}_>>>>{\simeq \quad \;\;} \\ 
             \ff{2} \ar[ru]^{\f'} \ar[r]_{\f} & \Pro{C}}$$

\cqd
\end{remark}

\begin{corollary}\label{mardesictrickparaflechas}
Let $\ff{X}=(\ff{X}_i)_{i\in \cc{I}}\stackrel{\ff{f}}\rightarrow \ff{Y}=(\ff{Y}_j)_{j\in \cc{J}}\in \Pro{C}$. There exists a cofinite filtered poset with a unique initial object $\ff{J}$, and a morphism $\ff{X}'\stackrel{\ff{f}'}\rightarrow \ff{Y}'\in \cc{H}om(\ff{J}^{op},\cc{C})$ such that the following diagram commutes in $\Pro{C}$ up to isomorphism:

\begin{equation}\label{cuadradoMTflechasbis}
\xymatrix@R=3pc@C=3pc{\ff{X} \ar@{}[rd]|{\cong}\ar[r]^{\ff{f}} \ar[d]_\simeq & \ff{Y} \ar[d]^\simeq \\ \ff{X}' \ar[r]_{\ff{f}'} & \ff{Y}'} 
\end{equation}

Equivalently every object $\f\in \cc{H}om(\ff{2},\Pro{C})$ have a lifting to $\cc{H}om(\ff{J}^{op},\cc{C})$ up to equivalence for some cofinite and filtered poset with a unique initial object $\ff{J}$.

$$\xymatrix{ & \cc{H}om(\ff{J}^{op},\cc{C}) \ar[d]^{inc}_>>>>{\simeq \quad \;\;} \\ 
             \ff{2} \ar[ru]^{\f'} \ar[r]_{\f} & \Pro{C}}$$
\end{corollary}

\begin{proof}

Consider $\tilde{\f'}$ given by \ref{reindexingparamorfismos} and consider the following diagram: 

$$\xymatrix{& \cc{H}om(\cc{M}^{op},\cc{C}) \ar[rd]^{(\F^{op})^*} \ar[dd]|{\stackrel[\comw{a}]{\comw{a}}{\scriptstyle inc}} \\
            \ff{2} \ar@{}[r]|{\simeq} \ar[ru]^{\tilde{\f'}} \ar[rd]_{\f} & \ar@{}[r]|>>>>>>>>>{\simeq} & \cc{H}om(\ff{M}(\cc{M})^{op},\cc{C}) \ar[ld]^{inc} \\
            & \Pro{C}}$$

\noindent where $\ff{M}(\cc{M})\stackrel{\F}\rightarrow \cc{M}$ is the one given by \ref{phi}. Note that the equivalence in the right triangle is because $\F$ is 2-cofinal.

Then take $\ff{J}=\ff{M}(\cc{M})$ and $ \f'=(\F^{op})^* \tilde{\f'}$.



%
\end{proof}

The lifting property \ref{mardesictrickparaflechas} also holds for $\Prop{C}$:

\begin{corollary}\label{mardesictrickparaflechasenprop}
Let $\ff{X}=(\ff{X}_i)_{i\in \cc{I}}\stackrel{\ff{f}}\rightarrow \ff{Y}=(\ff{Y}_j)_{j\in \cc{J}}\in \Prop{C}$. There exist a cofinite filtered poset with a unique initial object $\ff{J}$, and a morphism $\ff{X}'\stackrel{\ff{f}'}\rightarrow \ff{Y}'\in \cc{H}om(\ff{J},\cc{C})$ such that the following diagram commutes in $\Prop{C}$ up to isomorphism:

\begin{equation}\label{cuadradoMTflechasbis}
\xymatrix@R=3pc@C=3pc{\ff{X} \ar@{}[rd]|{\cong} \ar[r]^{\ff{f}} \ar[d]_\cong & \ff{Y} \ar[d]^\cong \\ \ff{X}' \ar[r]_{\ff{f}'} & \ff{Y}'} 
\end{equation}

Equivalently every object $\f\in \cc{H}om(\ff{2},\Prop{C})$ have a lifting to $\cc{H}om_p(\ff{J}^{op},\cc{C})$ up to equivalence for some cofinite and filtered poset with a unique initial object $\ff{J}$.

$$\xymatrix{ & \cc{H}om_p(\ff{J}^{op},\cc{C}) \ar[d]^{inc}_>>>>{\simeq \quad \;\;} \\ 
             \ff{2} \ar[ru]^{\f'} \ar[r]_{\f} & \Prop{C}}$$

\end{corollary}

\begin{proof}
 By \ref{proppseudoeqapro}, there exist $\ff{X}\mr{\tilde{\ff{f}}}\ff{Y} \in \Pro{C}$ and an invertible 2-cell \mbox{$\ff{f} \Mr{\alpha} \tilde{\ff{f}}\in \Prop{C}$.} Apply \ref{mardesictrickparaflechas} to $\tilde{\ff{f}}$ to obtain a cofinite filtered poset with a unique initial object $\ff{J}$, and a morphism $\ff{X}'\stackrel{\ff{f}'}\rightarrow \ff{Y}'\in \cc{H}om(\ff{J}^{op} ,\cc{C})$ such that the following diagram commutes in $\Pro{C}$ up to an invertible 2-cell $\gamma$:

\begin{equation}\label{cuadradoMTflechasbis}
\xymatrix@R=3pc@C=3pc{\ff{X} \ar@{}[rd]|{\cong \; \Downarrow \; \gamma} \ar[r]^{\tilde{\ff{f}}} \ar[d]_{\ff{a}\simeq} & \ff{Y} \ar[d]^{\ff{b}\simeq} \\ \ff{X}' \ar[r]_{\ff{f}'} & \ff{Y}'} 
\end{equation}

Then we have 

$$\vcenter{\xymatrix@R=3pc@C=1.5pc{\ff{X}\ar[rr]^{\ff{f}} \ar[d]_{\ff{a}\simeq} & 
            \ar@{}[d]|{\cong \; \Downarrow \; \gamma \circ \ff{b}\alpha  } &
            \ff{Y} \ar[d]^{\ff{b}\simeq} \\ 
            \ff{X}' \ar[rr]_{\ff{f}'} &&
            \ff{Y}'}}
            \vcenter{\xymatrix{ \in \Prop{C}}}$$

\noindent as we wanted to prove.
%
%
\end{proof}

\subsection{Reindexing for diagrams}\label{diagramas}

The following proposition is a generalization of \ref{reindexingparamorfismos}.

\begin{proposition}\label{diagrama de diagramas}
Let $\Delta\mr{\D}\Pro{C}$ be a finite diagram with commutation relations and no loops in $2$-$\cc{P}ro(\cc{C})$. Then $\D$ can be represented up to equivalence by a 2-pro-object over $\cc{H}om_p(\Delta,\cc{C})$, i.e. there exists an inverse 2-filtered system of diagrams $\{\Delta \mr{\D_k}\cc{C}\}_{k\in \cc{M}}$ in $\cc{C}$ such that the diagram induced by the $\D_k$'s in $\Pro{C}$ is equivalent to $\D$ up to isomorphism. 

Equivalently, every object $\D\in \cc{H}om(\Delta,\Pro{C})$ have a lifting to $\cc{H}om(\cc{M}^{op},\cc{C})$ up to equivalence for some 2-filtered 2-category $\cc{M}$.

$$\xymatrix{ & \cc{H}om(\cc{M}^{op},\cc{C}) \ar[d]^{inc}_>>>>{\simeq \quad \;\;} \\ 
             \Delta \ar[ru]^{\D'} \ar[r]_{\D} & \Pro{C}}$$

\end{proposition}

\begin{proof}
We are going to proceed by induction in the amount of vertices of $\Delta$. The initial case is trivial.

Now, suppose that we have proved the proposition for diagrams with $n-1$ vertices and let $\Delta\mr{\D}\Pro{C}$ be a diagram with $n$ vertices. Let $x$ be an initial vertex of $\Delta$ and let $\Delta'$ be the diagram resulting by taking $x$ out of $\Delta$ and $\D'$ the induced diagram in $\Pro{C}$. By inductive hypothesis, there exists an inverse 2-filtered system of diagrams $\{\D'_j\}_{j\in \cc{J}}$ such that the diagram induced by the $\D'_j$'s in $\Pro{C}$ is equivalent to $\D'$. Let $\ff{X}=(\ff{X}_i)_{i\in \cc{I}}$ be the object of $\Pro{C}$ corresponding to $x$ by $\D$ and let $\{\ff{X}\mr{f_l}\ff{Y}^l\}_{l=1,...,m}$ be all the morphisms in $\Pro{C}$ corresponding to morphisms from $x$ to some other vertex of $\Delta$ when we apply $\D$. In what follows, we are going to abuse the notation by using $\ff{Y}^l$ for the corresponding objects via the equivalence between $\D'$ and the diagram induced by the $\D'_j$'s and $f_l$ for the composition of the previous $f_l$ with the corresponding equivalence. 

 Define $\cc{M}$ as the following 2-category:

Objects are $m$-tuples of pairs $(\ff{r}_l,\varphi_l)$ with $\ff{X}_i\mr{\ff{r}_l}\ff{Y}^l_j \in \cc{C}$ such that $i\in \cc{I}$, $j\in \cc{J}$ and $(\ff{r}_l,\varphi_l)$ represents $\ff{f}_l\ \forall \ l=1,...,m$. 

Morphisms $\{(\ff{r}_l,\varphi_l)\}_{l=1,...,m}\mr{}\{(\ff{s}_l,\psi_l)\}_{l=1,...,m}$ ($\ff{r}_l:\ff{X}_i\rightarrow \ff{Y}^l_j$, $\ff{s}_l:\ff{X}_{i'}\rightarrow \ff{Y}^l_{j'}$) are triplex $(u,a,\{\theta_l\}_{l=1,...,m})$ where $i\stackrel{u}\rightarrow i'\in\cc{I}$, $j\stackrel{a}\rightarrow j'\in\cc{J}$ and $\forall \ l=1,...,m\ \  \theta_l$ is an invertible 2-cell $\vcenter{\xymatrix@C=0pc{\ff{Y}^l_{a} \dl & \dc{\theta_l} & \ff{s}_l \dr \\
                               \ff{r}_l & & \ff{X}_u}}$ such that 

%

$$\vcenter{\xymatrix@C=-0pc{ \Yal \dl & \dc{\theta_l} & \dr \s_l && \pi\si \deq \\
\fr_l \deq && \X_u && \pi\si \\
\fr_l \dl & \dcr{\varphi_l} && \dr \pi_i \cl{\pi_u} \\
\pi_j &&& \f_l }}
\vcenter{\xymatrix@C=-0pc{ \quad = \quad }}
\vcenter{\xymatrix@C=-0pc{ \Yal \deq && \s_l \dl & \dc{\psi_l} & \dr \pi\si \\
\Yal && \pi\sj && \f_l \deq \\
& \pi_j \cl{\pi_a} &&& \f_l }}$$

A 2-cell $\{(\ff{r}_l,\varphi_l)\}_{l=1,...,m}\cellrd{(u,a,\{\theta_l\}_{l=1,...,m})}{}{(v,b,\{\eta_l\}_{l=1,...,m})}\{(\ff{s}_l,\psi_l)\}_{l=1,...,m}$ is a pair $(\mu,\alpha)$ where $i\cellrd{u}{\mu}{v}i' \in \cc{I}$, $j\cellrd{a}{\alpha}{b}j' \in \cc{J}$ and $\forall \ l=1,...,m$

$$\vcenter{\xymatrix@C=-.3pc{\ff{Y}^l_a \dcellb{\ff{Y}^l_\alpha} & & \ff{s}_l \deq \\
	                     \ff{Y}^l_{b} \dl & \dc{\eta_l} & \ff{s}_l \dr \\
	                     \ff{r}_l & & \ff{X}_{v} }}
\vcenter{\xymatrix@C=-.3pc{\quad \quad = \quad \quad  }						}
\vcenter{\xymatrix@C=-.3pc{\ff{Y}^l_a \dl & \dc{\theta_l} &\ff{s}_l \dr \\
		           \ff{r}_l \deq & & \ff{X}_{u} \dcellb{\ff{X}_\mu} \\
		           \ff{r}_l & & \ff{X}_{v} }}$$

Identities and compositions are defined in the obvious way. 

In the following, we are going to prove that $\cc{M}$ is 2-filtered and the 2-functors $\vcenter{\xymatrix@R=.3pc{\cc{M}\ar[r] & \cc{I} \\ \{(\ff{r}_l,\varphi_l)\}_{l=1,...,m}\ar@{|->}[r] & i}}$ and $\vcenter{\xymatrix@R=.3pc{\cc{M}\ar[r] & \cc{J} \\ \{(\ff{r}_l,\varphi_l)\}_{l=1,...,m}\ar@{|->}[r] & j}}$ are 2-cofinal. 

$\cc{M}$ is 2-filtered:

\begin{itemize}
\item[F0:] Let $\{(\ff{r}_l,\varphi_l)\}_{l=1,...,m}$, $\{(\ff{s}_l,\psi_l)\}_{l=1,...,m} \in \cc{M}$ ($\ff{r}_l:\ff{X}_i\rightarrow \ff{Y}^l_j$, $\ff{s}_l:\ff{X}_{i'}\rightarrow \ff{Y}^l_{j'}$). Since $\cc{J}$ is \mbox{2-filtered}, we have $\vcenter{\xymatrix@R=-0.5pc{j\ar[rd]^a & \\ & j'' \\ j'\ar[ru]_{b}}} \in \cc{J}$ and, by \ref{idrepresenta} and the fact that $\cc{I}$ is \mbox{2-filtered}, we have $\ff{X}_{i''}\mr{\ff{t}_l}\ff{Y}^l_{j''} \in \cc{C}$ and invertible 2-cells $\epsilon_l$ such that $\forall \ l=1,..,m \ (\ff{t}_l,\epsilon_l)$ represents $\ff{f}_l$ and there are morphisms $\vcenter{\xymatrix@R=-0.5pc{i\ar[dr]^u \\
                          & i'' \\
                          i' \ar[ur]_ v}} \in \cc{I}$. Observe that, in this case, we are using the fact that $\cc{I}$ is 2-filtered as in \ref{M_f 2-filt} but also to achieve that all $\ff{t}_l$ have the same source.
                          
Then we have $\{(\ff{r}_l,\varphi_l)\}_{l=1,...,m}$ and $\{(\ff{t}_l,\epsilon_l)\}_{l=1,...,m}$ such that $\forall \ l=1,...,m$, $(\ff{r}_l,\varphi_l)$ and $(\ff{t}_l,\epsilon_l)$ both represent $\ff{f}_l$ and there are morphisms $i\mr{u}i'' \in \cc{I}$, $j\mr{a}j'' \in \cc{J}$. So, by \ref{lemorphismaraM_D}, there are morphisms $\vcenter{\xymatrix@R=-0.5pc{i\ar[rd]^{u_0} \\
                                                                                                                                                             & \tilde{i} \\
                                                                                                                                                             i''\ar[ru]_{v_0}}} \in \cc{I}$ and invertible 2-cells $\vcenter{\xymatrix@C=-.5pc{\ff{Y}^l_{a} \ar@{-}[dr] &&\ff{t}_l \dc{\theta_l} & &\ff{X}_{v_0} \ar@{-}[dl] \\
                                              & \ff{r}_l & & \ff{X}_{u_0} }}$ $\in \cc{C} $ such that 

$$\vcenter{\xymatrix@C=-0pc{ \Yal \ardr && \ft_l \dc{\theta_l} && \Xvc \ardl && \pi\sit \deq \\
& \fr_l \deq && \Xuc &&& \pi\sit \\
& \fr_l \dl & \dcr{\varphi_l} && \dr \pi_i \clunodos{\piuc} \\
& \pi_j &&& \f_l }}
\vcenter{\xymatrix@C=-0pc{ \quad = \quad }}
\vcenter{\xymatrix@C=-0pc{ \Yal \deq && \ft_l \deq && \Xvc && \pi\sit \\
\Yal \deq && \ft_l \dl & \dcr{\eps_l} && \dr \pi\sii \cl{\pivc} \\
\Yal && \pi\sjj &&& \f_l \deq \\
& \pi_j \cl{\pi_a} &&&& \f_l }}$$

Then we have $\{(\ff{s}_l,\psi_l)\}_{l=1,...,m}$ and $\{(\ff{t}_l\ff{X}_{v_0}, 
\vcenter{\xymatrix@C=-0.2pc@R=1pc{ \ft_l \deq && \Xvc && \pi\sit \\
\ft_l \dl & \dcr{\eps_l} && \dr \pi\sii \cl{\pivc} \\
\pi_j &&& \f_l}}
)\}_{l=1,...,m}$ such that \mbox{$\forall \ l=1,...,m$ $(\ff{s}_l,\psi_l)$} and $(\ff{t}_l\ff{X}_{v_0},  
\vcenter{\xymatrix@C=-0.2pc@R=1pc{ \ft_l \deq && \Xvc && \pi\sit \\
\ft_l \dl & \dcr{\eps_l} && \dr \pi\sii \cl{\pivc} \\
\pi_j &&& \f_l}}
)$ both represent $\ff{f}_l$ and there are morphisms $i'\mr{v_0 v}\tilde{i} \in \cc{I}$, $j'\mr{b}j'' \in \cc{J}$. So, by \ref{lemorphismaraM_D}, there are morphisms $\vcenter{\xymatrix@R=-0.5pc{i'\ar[rd]^{u_1} \\
                                                    & i_0 \\
                                                    \tilde{i}\ar[ru]_{v_1}}}\in \cc{I}$ and invertible 2-cells $\vcenter{\xymatrix@C=-.5pc{\ff{Y}^l_{b} \ar@{-}[dr] & \ff{t}_l \ar@{}    [d]^{\eta_l}&\ff{X}_{v_0} & \ff{X}_{v_1} \ar@{-}[dl] \\
                                                                                                               & \ff{s}_l & \ff{X}_{u_1} }}$ $\in \cc{C} $ such that 

$$\vcenter{\xymatrix@C=-0pc{ \Ybl \ardr && \ft_l & \dc{\eta_l} & \Xvc && \Xvu \ardl && \piic \deq \\
& \s_l \deq &&&& \Xuu &&& \piic \\
& \s_l \dl && \dcr{\psi_l} &&& \dr \pi\si \clunodos{\piuu} \\
& \pi\sj &&&&& \f_l}}
\vcenter{\xymatrix@C=-0pc{ \quad = \quad }}
\vcenter{\xymatrix@C=-0pc{ \Ybl \deq && \ft_l \deq && \Xvc \deq && \Xvu && \piic \\
\Ybl \deq && \ft_l \deq && \Xvc &&& \pi\sit \cl{\pivu} \\
\Ybl \deq && \ft_l \dl && \dc{\eps_l} && \dr \pi\sii \cldosuno{\pivc} \\
\Ybl && \pi\sjj &&&& \f_l \deq \\
& \pi\sj \cl{\pi_b} &&&&& \f_l }}$$

It can be checked that there are morphisms in $\cc{M}$ $\vcenter{\xymatrix@R=0pc{\{(\ff{r}_l,\varphi_l)\}_{l=1,...,m} \ar[rrd]^<<<<<<<<<<{(v_1 u_0,a,\{\theta_l   \ff{X}_{v_1}\}_{l=1,...,m})}\\
                                                                                   & &\{(\ff{t}_l\ff{X}_{v_0}\ff{X}_{v_1}, I
                                                                                   )\}_{l=1,...,m}\\
                                                                                   \{(\ff{s}_l,\psi_l)\}_{l=1,...,m}\ar[rru]_<<<<<<<<<<{(u_1,b,\{\eta_l\}_{l=1,...,m})}}}$.   

\noindent where $$I=\vcenter{\xymatrix@C=-0pc{ \ft_l \deq && \Xvc \deq && \Xvu && \piic \\
\ft_l \deq && \Xvc &&& \pi\sit \cl{\pivu} \\
\ft_l \dl & \dcr{\eps_l} && \dr \pi\sii \clunodos{\pivc} \\
\pi\sjj &&& \f_l}}$$
                                                                                   
\item[F1:] Let $\{(\ff{r}_l,\varphi_l)\}_{l=1,...,m}\mrpair{(u,a,\{\theta_l\}_{l=1,...,m})}{(v,b,\{\eta_l\}_{l=1,...,m})}\{(\ff{s}_l,\psi_l)\}_{l=1,...,m}\in \cc{M}$. Since $\cc{J}$ is 2-filtered, we have $j'\stackrel{c}\rightarrow j'' $ and an invertible 2-cell $ca\stackrel{\alpha}\Rightarrow cb \in \cc{J}$ and, by \ref{idrepresenta} and the fact that $\cc{I}$ is 2-filtered, we have morphisms $\ff{X}_{i''}\mr{\ff{t}_l} \ff{Y}^l_{j''} \in \cc{C}$ and invertible 2-cells $\epsilon_l$ such that $\forall \ l=1,...,m \ (\ff{t}_l,\epsilon_l)$ represents $\ff{f}_l$ and there is a morphism $i'\mr{w}i''$ and an invertible 2-cell $wu\Mr{\mu}wv \in \cc{I}$. Have in mind the same observation made in the proof of axiom F0. 

Then we have $\{(\ff{s}_l,\psi_l)\}_{l=1,...,m}$ and $\{(\ff{t}_l,\epsilon_l)\}_{l=1,...,m}$ such that $\forall \ l=1,..,m$ $(\ff{s}_l,\psi_l)$ and $(\ff{t}_l,\epsilon_l)$ both represent $\ff{f}_l$ and there are morphisms $i'\mr{w} i'' \in \cc{I}$, $j'\mr{c} j'' \in \cc{J}$. So, by \ref{lemorphismaraM_D}, there are morphisms $\vcenter{\xymatrix@R=-0.5pc{i'\ar[rd]^{u_0} \\
                                                                                                                                                                 & i_0 \\
                                                                                                                                                                 i''\ar[ru]_{v_0}}} \in \cc{I}$ and invertible 2-cells $\vcenter{\xymatrix@C=-.5pc{\ff{Y}^l_{c} \ar@{-}[dr] &&\ff{t}_l \dc{\tilde{\theta_l}} & &\ff{X}_{v_0} \ar@{-}[dl] \\
                                               & \ff{s}_l & & \ff{X}_{u_0} }}$ $\in \cc{C} $ such that 
\begin{equation}\label{A139}
\vcenter{\xymatrix@C=-0pc{ \Ycl \ardr && \ft_l \dc{\tilde{\theta}_l} && \Xvc \ardl && \piic \deq \\
& \s_l \deq && \Xuc &&& \piic \\
& \s_l \dl & \dcr{\psi_l} && \dr \pi\si \clunodos{\piuc} \\
& \pi\sj &&& \f_l }}
\vcenter{\xymatrix@C=-0pc{ \quad = \quad }}
\vcenter{\xymatrix@C=-0pc{ \Ycl \deq && \ft_l \deq && \Xvc && \piic \\
\Ycl \deq && \ft_l \dl & \dcr{\eps_l} && \dr \pi\sii \cl{\pivc} \\
\Ycl && \pi\sjj &&& \f_l \deq \\
& \pi\sj \cl{\pi_a} &&&& \f_l }}
\end{equation}

Plus, since $\cc{I}$ is 2-filtered, there is a morphism $i_0\mr{\tilde{w}}\tilde{i}$ and an invertible 2-cell $\tilde{w}u_0 \Mr{\tilde{\mu}} \tilde{w}v_0 w \in \cc{I}$.

Let's check that we a have a morphism in $\cc{M}$ \mbox{$\{(\ff{s}_l,\psi_l)\}_{l=1,...,m}\mr{(\tilde{w}v_0 w,c,\{\ff{s}_l\ff{X}_{\tilde{\mu}}\circ \tilde{\theta_l}\ff{X}_{\tilde{w}}\}_{l=1,...,m})}\{(\ff{t}_l\ff{X}_{v_0}\ff{X}_{\tilde{w}}, 
\vcenter{\xymatrix@C=-0.2pc@R=1pc{ \ft_l \deq && \Xvc \deq && \X\swt && \pi\sit \\
\ft_l \deq && \Xvc &&& \piic \cl{\pi\swt} \\
\ft_l \dl & \dcr{\eps_l} && \dr \pi\sii \clunodos{\pivc} \\
\pi\sjj &&& \f_l}}
)\}_{l=1,...,m}$:}

$$\vcenter{\xymatrix@C=-0.2pc{ \Ycl \ardr && \ft_l \dcr{\tilde{\theta}_l} &&& \Xvc \ardl && \X\swt \deq &&& \pi\sit \deq \\
& \s_l \deq &&& \Xuc \ardl & \dcr{\X\smut} && \X\swt \ardr &&& \pi\sit \deq \\
& \s_l \deq  && \X_w \ardrrr && \Xvc & \dc{\pi_{\tilde{w} v_0 w}} && \X\swt && \pi\sit \ardllll \\
& \s_l \dl && \dcr{\psi_l} &&& \dr \pi\si \\
& \pi\sj &&&&& \f_l }}
\vcenter{\xymatrix@C=-0pc{  =  }}
\vcenter{\xymatrix@C=-0.2pc{ \Ycl \ardr && \ft_l \dcr{\tilde{\theta}_l} &&& \Xvc \ardl && \X\swt \deq &&& \pi\sit \deq \\
& \s_l \deq &&& \Xuc \ardl & \dcr{\X\smut} && \X\swt \ardr &&& \pi\sit \deq \\
& \s_l \deq && \X_w \deq && \Xvc \deq &&& \X\swt && \pi\sit \\
& \s_l \deq && \X_w \deq && \Xvc &&&& \piic \cl{\pi\swt} \\
& \s_l \deq && \X_w &&&& \pi\sii \cldosdos{\pivc} \\
& \s_l \dl && \dc{\psi_l} && \dr \pi\si \cldosdos{\pi_w} \\
& \pi\sj &&&& \f_l }}
\vcenter{\xymatrix@C=-0pc{  =  }}$$

$$\vcenter{\xymatrix@C=-0pc{ \quad = \quad }}
\vcenter{\xymatrix@C=-0pc{ \Ycl \deq && \ft_l \deq && \Xvc \deq && \X\swt && \pi\sit \\
\Ycl \ardr && \ft_l \dc{\tilde{\theta}_l} && \Xvc \ardl &&& \piic \deq \cl{\pi\swt} \\
& \s_l \deq && \Xuc &&&& \piic \\
& \s_l \dl && \dc{\psi_l} && \dr \pi\si \cldosdos{\piuc} \\
& \pi\sj &&&& \f_l }}
\vcenter{\xymatrix@C=-0pc{ \quad = \quad }}
\vcenter{\xymatrix@C=-0pc{ \Ycl \deq && \ft_l \deq && \Xvc \deq && \X\swt && \pi\sit \\
\Ycl \deq && \ft_l \deq && \Xvc &&& \piic \cl{\pi\swt} \\
\Ycl \deq && \ft_l \dl && \dc{\eps_l} && \dr \pi\sii \cldosuno{\pivc} \\
\Ycl && \pi\sjj &&&& \f_l \deq \\
& \pi_j \cl{\pi_c} &&&&& \f_l }}$$

\noindent where the first equality follows from axiom PC1, the second one holds by elevators calculus plus axiom PC2 and the last one is due to \eqref{A139}. 
                              
Now observe that

$$\vcenter{\xymatrix@C=-0.4pc{ \Yal \dl & \dc{\Y_{\alpha}^l} & \dr \Ycl &&& \ft_l \deq && \Xvc \deq && \X\swt \deq && \quad & \pi\sit \deq \\
\Ybl \deq && \Ycl \ardr && \dcr{\tilde{\theta}_l} & \ft_l && \Xvc \ardl && \X\swt \deq &&& \pi\sit \deq \\
\Ybl \deq && & \s_l \deq &&& \Xuc \ardl & \dcr{\X\smut} && \X\swt \ardr &&& \pi\sit \deq \\
\Ybl \dl & \dcr{\eta_l} && \dr \s_l && \X_w \deq && \Xvc \deq &&& \X\swt \deq && \pi\sit \deq \\
\fr_l &&& \X_v && \X_w && \Xvc &&& \X\swt && \pi\sit }}
\vcenter{\xymatrix@C=-0pc{ \quad = \quad }}
\vcenter{\xymatrix@C=-0.2pc{ \Yal \dl & \dc{\Y_{\alpha}^l} & \dr \Ycl && \ft_l \deq && \Xvc \deq && \X\swt \deq &&&& \pi\sit \deq \\
\Ybl \deq && \Ycl \ardr && \ft_l \dc{\tilde{\theta}_l} && \Xvc \ardl && \X\swt \deq &&&& \pi\sit \deq \\
\Ybl \deq && & \s_l \deq && \Xuc \deq &&& \X\swt &&&& \pi\sit \\
\Ybl \deq &&& \s_l \deq && \Xuc &&& && \piic \cldosdos{\pi\swt} \\
\Ybl \deq &&& \s_l \deq && &&& \pi\si \cltresdos{\piuc} \opdosdos{\pi_w\inv} \\
\Ybl \deq &&& \s_l \deq &&& \X_w \deq &&&& \pi\sii \opdosuno{\pivc\inv} \\
\Ybl \deq &&& \s_l \deq &&& \X_w \deq && \Xvc \deq &&& \piic \op{\pi\swt\inv} \\
\Ybl \dl & \dcr{\eta_l} && \dr \s_l &&& \X_w \deq && \Xvc \deq && \X\swt \deq && \pi\sit \deq \\
\fr_l &&& \X_v &&& \X_w && \Xvc && \X\swt && \pi\sit }}
\vcenter{\xymatrix@C=-0pc{ \!\!\!\!\!\! = \quad }}
$$

$$\vcenter{\xymatrix@C=-0.4pc{  \Yal \dl & \dc{\Y_{\alpha}^l} & \dr \Ycl && \ft_l \deq && \Xvc \deq && \X\swt &&&& \pi\sit \\
\Ybl \deq && \Ycl \deq && \ft_l \deq && \Xvc && && \piic \cldosdos{\pi\swt} && \\
\Ybl \deq && \Ycl \deq && \ft_l \dl && \dc{\eps_l} && \dr \pi\sii \cldosdos{\pivc} \\
\Ybl \deq && \Ycl && \pi\sjj &&&& \f_l \deq \\
\Ybl \deq &&& \pi\sj \cl{\pi_c} \dl && \dcr{\psi_l\inv} &&& \dr \f_l \\
\Ybl \dl & \dcr{\eta_l} && \dr \s_l &&&&& \pi\si \deq \\
\fr_l \deq &&& \X_v \deq &&&&& \pi\si \opdosdos{\pi_w\inv} \\
\fr_l \deq &&& \X_v \deq &&& \X_w \deq &&&& \pi\sii \opdosuno{\pivc\inv} \\
\fr_l \deq &&& \X_v \deq &&& \X_w \deq && \Xvc \deq &&& \piic \opb{\pi\swt\inv} \\
\fr_l &&& \X_v &&& \X_w && \Xvc && \X\swt && \pi\sit }}
\vcenter{\xymatrix@C=-0pc{  =  }}
\vcenter{\xymatrix@C=-0.4pc{\Yal \dl & \dc{\Y_{\alpha}^l} & \dr \Ycl && \ft_l \deq && \Xvc \deq && \X\swt &&&& \pi\sit \\
\Ybl \deq && \Ycl \deq && \ft_l \deq && \Xvc && && \piic \cldosdos{\pi\swt} && \\
\Ybl \deq && \Ycl \deq && \ft_l \dl && \dc{\eps_l} && \dr \pi\sii \cldosdos{\pivc} \\
\Ybl \deq && \Ycl && \pi\sjj &&&& \f_l \deq \\
\Ybl &&& \pi\sj \cl{\pi_c} &&&&& \f_l \deq \\
& \pi_j \clunodos{\pi_b} \dl &&& \dcr{\varphi_l\inv} &&&& \dr \f_l \\
& \fr_l \deq &&&&&&& \pi_i \opdosdos{\pi_v\inv} \\
& \fr_l \deq &&&&& \X_v \deq &&&& \pi\si \opdosdos{\pi_w\inv} \\
& \fr_l \deq &&&&& \X_v \deq && \X_w \deq &&&& \pi\sii \opdosuno{\pivc\inv} \\
& \fr_l \deq &&&&& \X_v \deq && \X_w \deq && \Xvc \deq &&& \piic \opb{\pi\swt\inv} \\
& \fr_l &&&&& \X_v && \X_w && \Xvc && \X\swt && \pi\sit}}
\vcenter{\xymatrix@C=-0pc{ = }}
$$                    

$$\vcenter{\xymatrix@C=-0.4pc{  \Yal \deq && \Ycl \deq && \ft_l \deq && \Xvc \deq && \X\swt &&&& \pi\sit \\
\Yal \deq && \Ycl \deq && \ft_l \deq && \Xvc && && \piic \cldosdos{\pi\swt} && \\
\Yal \deq && \Ycl \deq && \ft_l \dl && \dc{\eps_l} && \dr \pi\sii \cldosdos{\pivc} \\
\Yal \deq && \Ycl && \pi\sjj &&&& \f_l \deq \\
\Yal &&& \pi\sj \cl{\pi_c} &&&&& \f_l \deq \\
& \pi_j \clunodos{\pi_a} \dl &&& \dcr{\varphi_l\inv} &&&& \dr \f_l \\
& \fr_l \deq &&&&&&& \pi_i \opdosdos{\pi_u\inv} \\
& \fr_l \deq &&&&& \X_u \deq &&&& \pi\si \opdosdos{\pi_w\inv} \\
& \fr_l \deq &&&&& \X_u \dl & \dc{\X_{\mu}} & \dr \X_w &&&& \pi\sii \deq \\
& \fr_l \deq &&&&& \X_v \deq && \X_w \deq &&&& \pi\sii \opdosuno{\pivc\inv} \\
& \fr_l \deq &&&&& \X_v \deq && \X_w \deq && \Xvc \deq &&& \piic \opb{\pi\swt\inv} \\
& \fr_l &&&&& \X_v && \X_w && \Xvc && \X\swt && \pi\sit }}
\vcenter{\xymatrix@C=-0pc{  =  }}
\vcenter{\xymatrix@C=-0.4pc{ \Yal \deq && \Ycl \deq && \ft_l \deq && \Xvc \deq && \X\swt &&&& \pi\sit \\
\Yal \deq && \Ycl \ardr && \ft_l \dc{\tilde{\theta}_l} && \ardl \Xvc && && \piic \deq \cldosdos{\pi\swt} && \\
\Yal \deq &&& \s_l \deq && \Xuc &&&&& \piic \\
\Yal \deq &&& \s_l \dl && \dcr{\psi_l} &&& \dr \pi\si \cltresdos{\piuc} \\
\Yal &&& \pi\sj &&&&& \f_l \deq \\
& \pi_j \clunodos{\pi_a} \dl &&& \dcr{\varphi_l\inv} &&&& \dr \f_l \\
& \fr_l \deq &&&&&&& \pi_i \opdosdos{\pi_u\inv} \\
& \fr_l \deq &&&&& \X_u \deq &&&& \pi\si \opdosdos{\pi_w\inv} \\
& \fr_l \deq &&&&& \X_u \dl & \dc{\X_{\mu}} & \dr \X_w &&&& \pi\sii \deq \\
& \fr_l \deq &&&&& \X_v \deq && \X_w \deq &&&& \pi\sii \opdosuno{\pivc\inv} \\
& \fr_l \deq &&&&& \X_v \deq && \X_w \deq && \Xvc \deq &&& \piic \opb{\pi\swt\inv} \\
& \fr_l &&&&& \X_v && \X_w && \Xvc && \X\swt && \pi\sit }}
\vcenter{\xymatrix@C=-0pc{ \!\!\!\!\!\!\!\!\! = \quad }}
$$   

$$\vcenter{\xymatrix@C=-0.2pc{ \Yal \deq && \Ycl \ardr && \ft_l \dc{\tilde{\theta}_l} && \ardl \Xvc && \X\swt \deq && \pi\sit \deq \\
\Yal \deq &&& \s_l \deq && \Xuc \deq &&& \X\swt && \pi\sit \\
\Yal \deq &&& \s_l \deq && \Xuc &&&& \piic \cl{\pi\swt} \\
\Yal \dl & \dcr{\theta} && \dr \s_l &&&& \pi\si \deq \cldosdos{\piuc} \\
\fr_l \deq &&& \X_u \deq &&&& \pi\si \opdosdos{\pi_w\inv} \\
\fr_l \deq &&& \X_u \deq && \X_w \deq &&&& \pi\sii \opdosuno{\pivc\inv} \\
\fr_l \deq &&& \X_u \deq && \X_w \deq && \Xvc \deq &&& \piic \op{\pi\swt\inv} \\
\fr_l \deq &&& \X_u \dl & \dc{\X_{\mu}} & \dr \X_w && \Xvc \deq && \X\swt \deq && \pi\sit \deq \\
\fr_l &&& \X_v && \X_w && \Xvc && \X\swt && \pi\sit }}
\vcenter{\xymatrix@C=-0pc{  =  }}
\vcenter{\xymatrix@C=-0.2pc{ \Yal \deq && \Ycl \ardr && \dcr{\tilde{\theta}_l} & \ft_l && \Xvc \ardl && \X\swt \deq &&& \pi\sit \deq \\
\Yal \deq &&& \s_l \deq &&& \Xuc \ardl & \dcr{\X\smut} && \X\swt \ardr &&& \pi\sit \deq \\
\Yal \dl & \dcr{\theta_l} && \dr \s_l && \X_w \deq && \Xvc \deq &&& \X\swt \deq && \pi\sit \deq \\
\fr_l \deq &&& \X_u \dl & \dc{\X_{\mu}} & \dr \X_w && \Xvc \deq &&& \X\swt \deq && \pi\sit \deq \\
\fr_l &&& \X_v && \X_w && \Xvc &&& \X\swt && \pi\sit }}$$


\noindent where the first equality is due to axiom PC2, the second one and the fifth one require some elevators calculus plus \eqref{A139}, the third holds because $(v,b,\eta_l)$ is a morphism in $\cc{M}$, the fourth one and the last one are valid by elevators calculus plus axiom PC2 and the sixth one is due to elevators calculus plus the fact that $(u,a \theta_l)$ is a morphism in $\cc{M}$.

Then, by \ref{lema2paraM_D}, there exist a morphism $\tilde{i}\mr{w_0}i_1 \in \cc{I}$ such that 

$$\vcenter{\xymatrix@C=-0.5pc{ \ff{Y}^l_a \dl & \dc{\ff{Y}^l_\alpha} & \ff{Y}^l_c \dr &&& \ff{t}_l \deq &&& \ff{X}_{v_0} \deq && \ff{X}_{\tilde{w}} \deq & \ff{X}_{w_0} \deq \\
                    \ff{Y}^l_b \deq & & \ff{Y}^l_c \dl & && \ff{t}_l \ar@{}[d]^{\tilde{\theta}_l} &&&  \ff{X}_{v_0} \dr && \ff{X}_{\tilde{w}} \deq  & \ff{X}_{w_0} \deq\\
                    \ff{Y}^l_b \deq & & && \ff{s}_l \deq &&  \ff{X}_{u_0} \dr && \dc{\ff{X}_{\tilde{\mu}}} && \ff{X}_{\tilde{w}} \dl & \ff{X}_{w_0} \deq  \\
                    \ff{Y}^l_b \dl && \dc{\eta_l} & & \ff{s}_l \dr && \ff{X}_{w} \deq && \ff{X}_{v_0} \deq && \ff{X}_{\tilde{w}} \deq & \ff{X}_{w_0} \deq  \\
                    \ff{r}_l & &&& \ff{X}_{v} & &\ff{X}_{w} & & \ff{X}_{v_0} & &\ff{X}_{\tilde{w}} & \ff{X}_{w_0}}}
\vcenter{\xymatrix@C=-.3pc{\quad \quad = \quad \quad  }}
\vcenter{\xymatrix@C=-.5pc{ \ff{Y}^l_a \deq & & & \ff{Y}^l_c \ar@{-}[dr] && \ff{t}_l \dc{\tilde{\theta}_l} && \ff{X}_{v_0} \ar@{-}[dl] & \ff{X}_{\tilde{w}} \deq & \ff{X}_{w_0} \deq\\
                    \ff{Y}^l_a \deq & & & & \ff{s}_l \deq & &  \ff{X}_{u_0} \dr & \dc{\ff{X}_{\tilde{\mu}}} & \ff{X}_{\tilde{w}} \dl & \ff{X}_{w_0} \deq\\
                    \ff{Y}^l_a \dl & & \dc{\theta_l} & & \ff{s}_l \dr & & \ff{X}_{w} \deq & \ff{X}_{v_0} \deq & \ff{X}_{\tilde{w}} \deq & \ff{X}_{w_0} \deq\\
                    \ff{r}_l \deq & & & & \ff{X}_{u} \dl & \dc{\ff{X}_\mu} & \ff{X}_{w} \dr & \ff{X}_{v_0} \deq & \ff{X}_{\tilde{w}} \deq & \ff{X}_{w_0} \deq \\
                    \ff{r}_l & & & & \ff{X}_{v} & & \ff{X}_{w} & \ff{X}_{v_0} & \ff{X}_{\tilde{w}} & \ff{X}_{w_0} }}$$

We can conclude that we a have a morphism in $\cc{M}$ \mbox{$\{(\ff{s}_l,\psi_l)\}_{l=1,...,m} \!\!\!\!\!\!\!\!\!\!\!\!\!\!\!\!\!\!\! \mr{(w_0\tilde{w}v_0 w,c,\{\ff{s}_l\ff{X}_{\tilde{\mu}}\ff{X}_{w_0}\circ \tilde{\theta_l}\ff{X}_{\tilde{w}}\ff{X}_{w_0}\}_{l=1,...,m})} \!\!\!\! \{(\ff{t}_l\ff{X}_{v_0}\ff{X}_{\tilde{w}}\ff{X}_{w_0}, 
\vcenter{\xymatrix@C=-0.2pc@R=1pc{ \ft_l \deq && \Xvc \deq && \X\swt \deq && \Xwc && \piiu \\
\ft_l \deq && \Xvc \deq && \X\swt &&& \pi\sit \cl{\piwc} \\
\ft_l \deq && \Xvc &&&& \piic \cldosuno{\pi\swt} \\
\ft_l \dl && \dc{\eps_l} && \dr \pi\sii \cldosdos{\pivc} \\
\pi\sjj &&&& \f_l}}
)\}_{l=1,...,m}$} and an invertible 2-cell in $\cc{M}$ 

\noindent $((w_0\tilde{w}v_0 w,c,\{\ff{s}_l\ff{X}_{\tilde{\mu}}\ff{X}_{w_0}\circ \tilde{\theta_l}\ff{X}_{\tilde{w}}\ff{X}_{w_0}\}_{l=1,...,m})(u,a,\{\theta_l\}_{l=1,...,m})\Mr{(w_0\tilde{w}v_0 \mu,\alpha)}$

$\hfill ((w_0\tilde{w}v_0 w,c,\{\ff{s}_l\ff{X}_{\tilde{\mu}}\ff{X}_{w_0}\circ \tilde{\theta_l}\ff{X}_{\tilde{w}}\ff{X}_{w_0}\}_{l=1,...,m})(v,b,\{\eta\}_{l=1,...,m})$.

\item[F2:] Let $\{(\ff{r}_l,\varphi_l)\}_{l=1,...,m}\cellpairrd{(u,a,\{\theta_l\}_{l=1,...,m})}{(\mu,\alpha)}{(\rho,\beta)}{(v,b,\{\eta_l\}_{l=1,...,m})}\{(\ff{s}_l,\psi_l)\}_{l=1,...,m}\in \cc{M}$. Since $\cc{J}$ is 2-filtered, we have $j'\stackrel{c}\rightarrow j'' \in \cc{J}$ such that $c\alpha=c\beta$ and, by \ref{idrepresenta} and the fact that $\cc{I}$ is 2-filtered, we have $\{\ff{X}_{i''}\mr{\ff{t}_l}\ff{Y}^l_{j''}\}_{l=1,...,m}$ and invertible 2-cells $\epsilon_l$ such that $\forall \ l=1,...,m\ (\ff{t}_l,\epsilon_l)$ represents $\ff{f}_l$ and there is a morphism $i'\mr{w}\tilde{i}\in \cc{I}$ such that $w\mu=w\rho$.

Now, since $\cc{I}$ is 2-filtered, we have morphisms $\vcenter{\xymatrix@R=-0.5pc{\tilde{i} \ar[rd]^{u_0} \\
                                                                               & i_0 \\
                                                                               i'' \ar[ru]_{v_0}}}\in \cc{I}$.
                                                                               
Then we have $\{(\ff{s}_l,\psi_l)\}_{l=1,...,m}$ and $\{(\ff{t}_l\ff{X}_{v_0}, 
\vcenter{\xymatrix@C=-0.2pc@R=1pc{ \ft_l \deq && \Xvc && \piic \\ \ft_l \dl & \dcr{\eps_l} && \dr \pi\sii \cl{\pivc} \\ \pi\sjj &&& \f_l  }}
)\}_{l=1,...,m}$ such that \mbox{$\forall \ l=1,...,m$} $(\ff{s}_l,\psi_l)$ and $(\ff{t}_l\ff{X}_{v_0}, 
\vcenter{\xymatrix@C=-0.2pc@R=1pc{ \ft_l \deq && \Xvc && \pi\sit \\
\ft_l \dl & \dcr{\eps_l} && \dr \pi\sii \cl{\pivc} \\
\pi_j &&& \f_l}}
)$ both represent $\ff{f}_l$ and there are morphisms $i'\mr{u_0 w} i'' \in \cc{I}$, $j'\mr{c}j'' \in \cc{J}$. Then, by \ref{lemorphismaraM_D}, we have morphisms $\vcenter{\xymatrix@R=-0.5pc{i' \ar[rd]^{u_1} \\
                                    & i_1 \\
                                    i_0 \ar[ru]_{v_1}}}\in \cc{I}$ and invertible 2-cells $\vcenter{\xymatrix@C=-.5pc{\ff{Y}^l_{c} \ar@{-}[dr] &\ff{t}_l
                                                                                 \ar@{}[d]^{\tilde{\theta_l}} & \ff{X}_{v_0} & \ff{X}_{v_1}\ar@{-}[dl] \\
                                                                                 & \ff{s}_l & \ff{X}_{u_1} }}$ $\in \cc{C} $ such that 

$$\vcenter{\xymatrix@C=-0pc{ \Ycl \ardr && \ft_l & \dc{\tilde{\theta}_l} & \Xvc && \Xvu \ardl && \piiu \deq \\
& \s_l  \deq &&&& \Xuu &&& \piiu \\
& \s_l  \dl && \dcr{\psi_l} &&& \dr \pi\si \clunodos{\piuu} \\
& \pi\sj &&&&& \f_l }}
\vcenter{\xymatrix@C=-0pc{ \quad = \quad }}
\vcenter{\xymatrix@C=-0pc{ \Ycl \deq && \ft_l \deq && \Xvc \deq && \Xvu && \piiu \\
\Ycl \deq && \ft_l \deq && \Xvc &&& \piic \cl{\pivu} \\
\Ycl \deq && \ft_l \dl && \dc{\eps_l } && \dr \pi\sii \cldosuno{\pivc} \\
\Ycl && \pi\sjj &&&& \f_l \deq \\
& \pi\sj \cl{\pi_c} &&&&& \f_l}}$$

Plus, since $\cc{I}$ is 2-filtered, we have a morphism $i_1\mr{\tilde{w}}i_2$ and an invertible 2-cell $\tilde{w}u_1\Mr{\tilde{\mu}} \tilde{w}v_1 u_0 w \in \cc{I}$.

It can be checked that there is a morphism in $\cc{M}$ \mbox{$\{(\ff{s}_l,\psi_l)\}_{l=1,...,m} \!\!\!\!\!\!\!\!\!\!\!\!\!\!\!\!\!\!\! \mr{(\tilde{w}v_1 u_0 w,c,\{\ff{s}_l\ff{X}_{\tilde{\mu}}\circ \tilde{\theta_l}\ff{X}_{\tilde{w}}\}_{l=1,...,m})} \!\!\!\! \{(\ff{t}_l\ff{X}_{v_0}\ff{X}_{v_1}\ff{X}_{\tilde{w}}, 
\vcenter{\xymatrix@C=-0pc{ \ft_l \deq && \Xvc \deq && \Xvu \deq && \X\swt && \piid \\
\ft_l \deq && \Xvc \deq && \Xvu &&& \piiu \cl{\pi\swt} \\
\ft_l \deq && \Xvc &&&& \piic \cldosuno{\pivu} \\
\ft_l \dl && \dc{\eps_l} && \dr \pi\sii \cldosdos{\pivc} \\
\pi\sjj &&&& \f_l}}
)\}_{l=1,...,m}$} such that 

\noindent $(\tilde{w}v_1 u_0 w,c,\{\ff{s}_l\ff{X}_{\tilde{\mu}}\circ \tilde{\theta_l}\ff{X}_{\tilde{w}}\}_{l=1,...,m})(u,a,\{\theta_l\}_{l=1,...,m})=$

$\hfill (\tilde{w}v_1 u_0 w,c,\{\ff{s}_l\ff{X}_{\tilde{\mu}}\circ \tilde{\theta_l}\ff{X}_{\tilde{w}}\}_{l=1,...,m})(v,b,\{\eta_l\}_{l=1,...,m})$.

\end{itemize}

$\vcenter{\xymatrix@R=.3pc{\cc{M}\ar[r] & \cc{I} \\ \{(\ff{r}_l,\varphi_l)\}_{l=1,...,m}\ar@{|->}[r] & i}}$ is 2-cofinal:

\begin{itemize}
\item[CF0:] Let $i\in \cc{I}$ and let $\forall \ l=1,...,m \ \ff{X}_{i'}\mr{\ff{r}_l} \ff{Y}^l_j \in \cc{C}$ and invertible 2-cells $\varphi_l$ such that $(\ff{r}_l,\varphi_l)$ represents $\ff{f}_l$. Note that we are using the fact that $\cc{I}$ is 2-filtered to make all $\ff{r}_l$'s have the same source. Also because $\cc{I}$ is 2-filtered, we have \mbox{$\vcenter{\xymatrix@R=-0.5pc{i\ar[rd]^{u} & \\ & i'' \\ i'\ar[ru]_{v}}} \in \cc{I}$.} It is straightforward to check that $\{(\ff{r}_l\ff{X}_{v}, 
\vcenter{\xymatrix@C=-0.2pc@R=1pc{ \fr_l \deq && \X_v && \pi\sii \\
\fr_l \dl & \dcr{\varphi_l} && \dr \pi\si \cl{\pi_v} \\
\pi_j &&& \f_l}}
)\}_{l=1,...,m}$ together with \mbox{$i\mr{u}i''$} proves CF0.

\item[CF1:] Let $i\in \cc{I}$, $\{(\ff{r}_l,\varphi_l)\}_{l=1,...,m}\in \cc{M}$ ($\ff{r}_l:\ff{X}_{i'}\rightarrow \ff{Y}^l_j$) and $i\mrpair{u}{v}i'\in \cc{I}$. Since $\cc{I}$ is \mbox{2-filtered,} we have $i'\stackrel{w}\rightarrow i''$ and an invertible 2-cell $wu\stackrel{\mu}\Rightarrow wv \in \cc{I}$. It is straightforward to check that $\{(\ff{r}_l,\varphi_l)\}_{l=1,...,m}\mr{ (w,id,\{id\}_{l=1,...,m})}\{(\ff{r}_l\ff{X}_w, 
\vcenter{\xymatrix@C=-0.2pc@R=1pc{ \fr_l \deq && \X_w && \pi\sii \\ \fr_l \dl & \dcr{\varphi_l} && \dr \pi\si \cl{\pi_w} \\ \pi\sj &&& \f_l  }}
)\}_{l=1,...,m}$ proves CF1. 

\item[CF2:] Let $i\in \cc{I}$, $\{(\ff{r}_l,\varphi_l)\}_{l=1,...,m}\in \cc{M}$ ($\ff{r}_l:\ff{X}_{i'}\rightarrow \ff{Y}^l_j$) and $i\cellpairrd{u}{\mu}{\rho}{v}i' \in \cc{I}$. Since $\cc{I}$ is 2-filtered, we have $i'\stackrel{w}\rightarrow i'' \in \cc{I}$ such that $w\mu=w\rho$. It is straightforward to check that $\{(\ff{r}_l,\varphi_l)\}_{l=1,...,m}\mr{(w,id,\{id\}_{l=1,...,m})}\{(\ff{r}_l\ff{X}_w, \vcenter{\xymatrix@C=-0.2pc@R=1pc{ \fr_l \deq && \X_w && \pi\sii \\ \fr_l \dl & \dcr{\varphi_l} && \dr \pi\si \cl{\pi_w} \\ \pi\sj &&& \f_l  }} )\}_{l=1,...,m}$ proves CF2. 

\end{itemize}

$\vcenter{\xymatrix@R=.3pc{\cc{M}\ar[r] & \cc{J} \\ \{(\ff{r}_l,\varphi_l)\}_{l=1,...,m}\ar@{|->}[r] & j}}$ is 2-cofinal:

\begin{itemize}
\item[CF0:] Let $j\in \cc{J}$. By \ref{idrepresenta} and the fact that $\cc{I}$ is 2-filtered, we have \mbox{$\forall \ l=1,...,m$} \mbox{$\ff{r}_l:\ff{X}_i\rightarrow \ff{Y}^l_j \in \cc{C}$} and invertible 2-cells $\varphi_l$ such that $(\ff{r}_l,\varphi_l)$ represents $\ff{f}_l$. $\{(\ff{r}_l,\varphi_l)\}_{l=1,...,m}$ clearly proves CF0.

\item[CF1:] Let $j_0\in \cc{J}$, $\{(\ff{r}_l,\varphi_l)\}_{l=1,...,m}\in \cc{M}$ ($\ff{r}_l:\ff{X}_i\rightarrow \ff{Y}^l_j$) and $j_0\mrpair{a}{b}j \in \cc{J}$. Since $\cc{J}$ is 2-filtered, we have $j\stackrel{c}\rightarrow j' \in \cc{J}$ and an invertible 2-cell $ca\stackrel{\alpha}\Rightarrow cb \in \cc{J}$. Now, by \ref{idrepresenta}, we have $\ff{s}:\ff{X}_{i'}\rightarrow \ff{Y}_{j'} \in \cc{C}$ such that $(\ff{s},id)$ represents $\ff{f}$. 

From the proof of the fact that $\cc{M}$ is 2-filtered, we have morphisms $\vcenter{\xymatrix@R=0pc{\{(\ff{r}_l,\varphi_l)\}_{l=1,...,m}\ar[rd]^{\quad \quad (u,c,\{\theta_l\}_{l=1,...,m})} & \\                             & \{(\ff{t}_l,\psi_l)\}_{l=1,...,m}\\ \{(\ff{s}_l,id)\}_{l=1,...,m} \ar[ru]_{\quad \quad (v,id,\{\eta_l\}_{l=1,...,m})}}}$. It is straightforward to check that $\{(\ff{r}_l,\varphi_l)\}_{l=1,...,m}\stackrel{(u,c,\{\theta_l\}_{l=1,...,m})}\longrightarrow \{(\ff{t}_l,\psi_l)\}_{l=1,...,m}$ proves CF1. 

\item[CF2:] Let $j_0\in \cc{J}$, $\{(\ff{r}_l,\varphi_l)\}_{l=1,...,m}\in \cc{M}$ ($\ff{r}_l:\ff{X}_i\rightarrow \ff{Y}^l_j$) and $j_0\cellpairrd{a}{\alpha}{\beta}{b}j \in \cc{J}$. Since $\cc{J}$ is 2-filtered, we have $j\stackrel{c}\rightarrow j' \in \cc{J}$ such that $c\alpha=c\beta$. Now, by \ref{idrepresenta} and the fact that $\cc{I}$ is 2-filtered, we have $\forall \ l=1,...,m\ \ \ff{s}_l:\ff{X}_{i'}\rightarrow \ff{Y}^l_{j'} \in \cc{C}$ and invertible 2-cells $\psi_l$ such that $(\ff{s}_l,\psi_l)$ represents $\ff{f}_l$. 

From the proof of the fact that $\cc{M}$ is 2-filtered, we have morphisms $\vcenter{\xymatrix@R=0pc{\{(\ff{r}_l,\varphi_l)\}_{l=1,...,m}\ar[rd]^{\quad \quad (u,c,\{\theta_l\}_{l=1,...,m})} & \\ & \{(\ff{t}_l,\epsilon_l)\}_{l=1,...,m}\\ \{(\ff{s}_l,id)\}_{l=1,...,m}\ar[ru]_{\quad \quad (v,id,\{\eta_l\}_{l=1,...,m})}}}$. It is straightforward to check that $\{(\ff{r}_l,\varphi_l)\}_{l=1,...,m}\stackrel{(u,c,\{\theta_l\}_{l=1,...,m})}\longrightarrow \{(\ff{t}_l,\epsilon_l)\}_{l=1,...,m}$ proves CF2. 
\end{itemize}

Finally, take $\D_{{(\ff{r}_l,\varphi_l)}_{l=1,...,k}}$ with $\ff{X}_i \mr{\ff{r}_l} \ff{Y}^l_j$ given by the following:

\noindent Send $x$ to $\ff{X}_{i}$, $\Delta'$ to $\D'_j$ and the morphisms that link both parts to the corresponding $\ff{r}_l$'s.

\end{proof}

\begin{corollary}\label{reindexingparadiagramas}
Let $\Delta\mr{\D}\Pro{C}$ be a finite diagram with commutation relations and no loops in $2$-$\cc{P}ro(\cc{C})$. Then there exists a cofinite and filtered poset with a unique initial object $\ff{J}$ and a diagram $\Delta \mr{\D'} \cc{H}om(\ff{J}^{op},\cc{C})$ equivalent to $\D$ up to isomorphism. 

Equivalently every object $\D\in \cc{H}om(\Delta,\Pro{C})$ have a lifting to $\cc{H}om(\ff{J}^{op},\cc{C})$ up to equivalence for some cofinite and filtered poset with a unique initial object $\ff{J}$.

$$\xymatrix{ & \cc{H}om(\ff{J}^{op},\cc{C}) \ar[d]^{inc}_>>>>{\simeq \quad \;\;} \\ 
             \Delta \ar[ru]^{\f'} \ar[r]_{\f} & \Pro{C}}$$
\end{corollary}

\begin{proof}
 It follows from \ref{diagrama de diagramas} as \ref{mardesictrickparaflechas} follows from \ref{reindexingparamorfismos}.
\end{proof}

\begin{corollary}\label{reindexingparadiagramasenprop}
\ref{reindexingparadiagramas} also holds in $\Prop{C}$.
\end{corollary}

\begin{proof}
 It follows from \ref{reindexingparadiagramas} as \ref{mardesictrickparaflechasenprop} follows from \ref{mardesictrickparaflechas}.
\end{proof}

%
\pagebreak
\begin{center}
{\bf Resumen en castellano de la secci\'on \ref{Mtrick}}
\end{center}

En esta secci\'on probamos ciertas propiedades de reindexaci\'on para las 2-categor\'ias $\Pro{C}$ y $\Prop{C}$ que ser\'an usadas para probar que son ``closed 2-bmodel \mbox{2-categories}'' (\ref{2-closed}) as\'i como Edwards-Hastings lo hacen para $\ff{Pro}(\ff{C})$ en \cite{EH} en el caso 1-dimensional. Las propiedades de reindexaci\'on para $\ff{Pro}(\ff{C})$ pueden hallarse en \cite{AM} o \cite{G2}.

El primer resultado es una versi\'on 2-categ\'orica de un resultado debido a Deligne \cite[Expose I, 8.1.6]{G2} que es clave en el caso 1-dimensional en el desarrollo de la estructura de modelos de la categor\'ia 
$\mathsf{Pro(C)}$ \cite{EH}. 
El enunciado 1-dimensional establece que todo pro-objeto es isomorfo a uno indexado por un poset cofinito y filtrante. Nuestra versi\'on establece que todo 2-pro-objeto es equivalente a uno indexado por un poset cofinito y filtrante. El segundo resultado establece que todo morfismo de 2-pro-objetos puede ser levantado salvo equivalencia a un morfismo entre 2-pro-objetos indexados por un poset cofinito y filtrante. Esto es un caso particular de un tercer resultado que establece que todo diagrama finito en $2$-$\cc{P}ro(\cc{C})$ puede ser levantado salvo equivalencia a un diagrama finito de 2-pro-objetos indexados por un poset cofinito y filtrante. Es clave para estos resultados la noci\'on de pseudo-funtor 2-cofinal dada en la secci\'on \ref{prelims}. Toda esta secci\'on ser\'a usada para probar el teorema central de la secci\'on \ref{2-modelos}.

\pagebreak

\section{Closed 2-model 2-categories}\label{Definitions and basic lemmas}

In this section we introduce original notions of closed 2-model  and closed 2-bmodel 2-category and state some lemmas and propositions that we are going to use later. Our notion is stronger than Pronk's ``fibration structures'' (\cite{Pronk}) since it is a 2-dimensional version of the full Quillen's axioms for closed model structures. It also differs in the important fact that we do not assume the choice of a privileged global factorization given in a pseudo-functorial way but stipulates, as Quillen does, only the existence of factorizations for each arrow. 

Most of the results of this section are generalizations to the context of 2-categories of well known statements about closed model categories. For definitions and results in the 1-dimensional case, check for example \cite{Q1} or \cite{GJ}.

\subsection{Definitions and basic lemmas}

\begin{definition}\label{lifting}
Let $\cc{C}$ be a 2-category and $\ff{A}\mr{\ff{i}}\ff{X}$, $\ff{Y}\mr{\ff{p}}\ff{B}$ two morphisms in $\cc{C}$. We say that the pair $(\ff{i},\ff{p})$ has the \emph{lifting property} (or equivalently that $\ff{i}$ has the left lifting property with respect to $\ff{p}$ or equivalently that $\ff{p}$ has the right lifting property with respect to $\ff{i}$) if for each diagram in $\cc{C}$ of the form

\begin{equation}\label{cuadradoallenar}
 \xymatrix@C=1.5pc@R=3pc{\ff{A} \ar[rr]^{\ff{a}} \ar[d]_{\ff{i}} &\ar@{}[d]|{\cong \; \Downarrow \; \gamma}& \ff{Y} \ar[d]^{\ff{p}} \\
            \ff{X} \ar[rr]_{\ff{b}} && \ff{B}}
\end{equation}

       
\noindent there exist a morphism $\ff{f}$ and invertible 2-cells $\lambda$, $\rho$ as in the following diagram

$$\xymatrix@C=1.5pc@R=1.5pc{\ff{A} \ar[rr]^{\ff{a}} \ar[dd]_{\ff{i}}  \ar@{}[dr]|{\cong \; \Downarrow \; \lambda}&  & \ff{Y} \ar[dd]^{\ff{p}}\\
                  & \ar@{}[dr]|{ \cong \; \Downarrow \; \rho} \\ 
                  \ff{X} \ar[rr]_{\ff{b}} \ar@{-->}[rruu]|{\comw{M^M} \ff{f} \comw{M^M} } && \ff{B}}$$
                  
\noindent such that $ \quad 
\vcenter{\xymatrix@C=-0pc{ \p \deq && & \fa \op{\lambda} \\ \p && \f && \ii \deq \\ & \fb \cl{\rho} &&& \ii}}
\vcenter{\xymatrix@C=-0pc{ \quad = \quad }}
\vcenter{\xymatrix@C=-0pc{ \p \dl & \dc{\gamma} & \dr \fa \\ \fb && \ii}}$.

In this case, we say that $(\ff{f},\lambda,\rho)$ is a filler for diagram \eqref{cuadradoallenar}.
\end{definition}

\begin{definition}\label{2-closed}
  We say that a 2-category $\cc{C}$ is a \emph{closed 2-model 2-category (respectively a closed 2-bmodel 2-category)} if it is equipped with three classes of morphisms called \emph{fibrations}, \emph{cofibrations} and \emph{weak equivalences} satisfying the following properties:
  
   \item[2-M0:] $\cc{C}$ is closed under finite weighted pseudo-limits and pseudo-colimits of pseudo-functors $\F:\cc{P} \mr{} \cc{C}$ with finite weights $\ff{W}: \cc{P} \mr{} \cc{C}at$ (see \cite{K2}).

\noindent (Respectively:

  \item[2-M0b:] $\cc{C}$ is closed under finite weighted bi-limits and bi-colimits of pseudo-functors \mbox{$\F:\cc{P} \mr{} \cc{C}$} with finite weights $\ff{W}: \cc{P} \mr{} \cc{C}at$ (see \cite{K2}).)
   
\noindent To simplify, by finite we mean that $ \cc{P}$ is finite and $\ff{W}(\ff{P})$ is finite for all $\ff{P} \in \cc{P}$. 
   \item[2-M2:] 
 Every morphism $\ff{f}\in \cc{C}$ can be factored up to isomorphism as $\ff{f}\cong \ff{p} \ff{i}$ with $\ff{i}$ a cofibration which is also a weak equivalence and $\ff{p}$ a fibration or $\ff{i}$ a cofibration and $\ff{p}$ a fibration which is also a weak equivalence.
 
  \item[2-M5:] Given a diagram of the form
 
 $$\xymatrix{\ff{X} \ar[rr]^{h} \ar[dr]_{\ff{f}} &
             \ar@{}[d]|{ \cong} &
             \ff{Y} \\
             &
             \ff{Z} \ar[ru]_{\ff{g}}}$$

\noindent If two of the three $\ff{f}$, $\ff{g}$, $\ff{h}$ are weak equivalences, then so is the third one. Every isomorphism is a weak equivalence.              

\item[2-M6a):] A morphism $\ff{p}\in \cc{C}$ is a fibration iff the pair $(\ff{i},\ff{p})$ has the lifting property for every $\ff{i}$ that is both a cofibration and a weak equivalence. 
 \item[2-M6b):] A morphism $\ff{i}\in \cc{C}$ is a cofibration iff the pair $(\ff{i},\ff{p})$ has the lifting property for every $\ff{p}$ that is both a fibration and a weak equivalence. 
 \item[2-M6c):] A morphism $\ff{f}\in \cc{C}$ is a weak equivalence iff it can be factored up to isomorphism as $\ff{f}\cong \ff{u}\ff{v}$ where $\ff{u}$ has the right lifting property with respect to all cofibrations and $\ff{v}$ has the left lifting property with respect to all fibrations.
\end{definition}

For some of the proofs of section \ref{2-modelos} we are going to assume that our 2-category $\cc{C}$ satisfies the following ``2-niceness conditions'':

\vspace{2ex}

\emph{2-N1 Every cofibration is a bi-pushout of a cofibration between cofibrant objects.}

\emph{2-N2 Every fibration is a bi-pullback of a fibration between fibrant objects.}

\emph{2-N3 At least one of the following is satisfied:}

\emph{\ \ \  2-N3a) Every object is cofibrant.
\ \ \  2-N3b) Every object is fibrant.}

\begin{remark}\label{N4}
A fourth niceness condition is considered in \cite{EH} in the 1-dimensional case: \emph{N4: There exist functorial cylinder objects.} Though it is not mentioned, we think that this condition is only needed in the proof of the analogous of \ref{lema5'.2}, but we also believe that this condition is not necessary as we proved in the 2-dimensional case that locally pseudo-functorial cylinder objects can be chosen (see \ref{cilindros funtoriales}) which is enough to prove \ref{lema5'.2}. Clearly, from the proof of \ref{cilindros funtoriales} we see that it follows from Quillen's axioms that is possible to choose locally functorial cylinder objects, result that we have not found in the literature. 
\end{remark}

\begin{remark}
Any closed 2-model 2-category is a closed 2-bmodel 2-category. Note also that the two notions differ only in the first axiom. \cqd 
\end{remark}

\begin{remark}\label{limites que alcanzan}
To check axiom 2-M0b, it is enough to check the existence of bi-equalizers, finite bi-products (binary plus bi-1) and bi-cotensors with a finite category (see \cite{Street}, \cite{K2}, \cite{Canevalli}). 
\cqd
\end{remark}

\begin{remark}
If $\cc{C}$ is a closed 2-bmodel 2-category, in particular $\cc{C}$ has finite bi-limits (that is finite conical weighted bi-limits) indexed by a poset, and more in particular bi-pullbacks and bi-1. \cqd
\end{remark}

\begin{lemma}\label{lema2.2}
Let $\cc{C}$ be a 2-category with three classes of morphisms satisfying 2-M6a), 2-M6b) and 2-M6c). Then a morphism $\ff{p}\in \cc{C}$ is both a fibration (respectively cofibration) and a weak equivalence iff it has the right lifting property (respectively left lifting property) with respect to all cofibrations (respectively fibrations). 
\end{lemma}

\begin{proof} We will prove the case where $\ff{p}$ is a fibration. The other case is similar and we omit it.

$\Rightarrow)$ Let $\ff{Y}\mr{\ff{p}}\ff{B}\in \cc{C}$ be a morphism that is both a fibration and a weak equivalence and $\ff{A}\mr{\ff{i}}\ff{X}\in \cc{C}$ a cofibration and suppose that we have the following situation:

$$\xymatrix@C=1.5pc@R=3pc{\ff{A} \ar[rr]^{\ff{a}} \ar[d]_{\ff{i}} &\ar@{}[d]|{\cong \; \Downarrow \; \gamma}& \ff{Y} \ar[d]^{\ff{p}} \\
            \ff{X} \ar[rr]_{\ff{b}} && \ff{B}}$$
            
Since $\cc{C}$ satisfies axiom 2-M6c) and $\ff{p}$ is a weak equivalence, there exist morphisms $\ff{u}$, $\ff{v}$ and an invertible 2-cell $\ff{p}\Mr{\epsilon} \ff{u}\ff{v}$ such that $\ff{u}$ has the right lifting property with respect to all cofibrations and $\ff{v}$ has the left lifting property with respect to all fibrations. Then there exist fillers $(\ff{f}_1,\lambda_1,\rho_1)$, $(\ff{f}_2,\lambda_2,\rho_2)$ for the following diagrams

$$\vcenter{\xymatrix@R=1.5pc{\ff{A} \ar[r]^{\ff{a}} \ar[dd]_{\ff{i}}  \ar@{}[ddr]|{\cong \; \Downarrow \; \gamma} & \ff{Y} \ar@{}[ddr]|{\stackrel[\scriptstyle \cong]{\scriptstyle \epsilon}{\Rightarrow}} \ar[r]^{\ff{v}} \ar[dd]^{\ff{p}} & \ff{Z} \ar@/^3ex/[ddl]^{\ff{u}} \\
                            \\
                            \ff{X} \ar[r]_{\ff{b}} &  \ff{B} &}} 
\vcenter{\xymatrix{=}}
\vcenter{\xymatrix@C=1.5pc@R=1.5pc{\ff{A} \ar[r]^{\ff{a}} \ar[dd]_{\ff{i}}  \ar@{}[ddrr]|{\cong \; \Downarrow \; \gamma \circ \epsilon^{-1}\ff{a}} & \ff{Y} \ar[r]^{\ff{v}} & \ff{Z} \ar[dd]^{\ff{u}} \\
                            \\
                            \ff{X} \ar[rr]_{\ff{b}} & & \ff{B}}}
\vcenter{\xymatrix{ \quad \quad \quad }}                           
\vcenter{\xymatrix@C=3.5pc@R=3.5pc{\ff{Y} \ar[r]^{id_\ff{Y}} \ar[d]_{\ff{v}}  \ar@{}[dr]|{\cong \; \Downarrow \; \epsilon}&  \ff{Y} \ar[d]^{\ff{p}} \\
                  \ff{Z} \ar[r]_{\ff{u}} & \ff{B}}}$$

Let's check that $(\ff{f}_2 \ff{f}_1, 
\vcenter{\xymatrix@C=-0.2pc@R=1pc{ && \fa \opunodos{= \quad} \\ & id_{\Y} \op{\lambda_2} &&& \fa \deq \\ \f_2 \deq && \fv \dl & \dc{\lambda_1} & \dr \fa \\ \f_2 && \f_1 && \ii}}
, 
\vcenter{\xymatrix@C=-0.2pc@R=1pc{ \p && \f_2 && \f_1 \deq \\ & \fu \cl{\rho_2} &&& \f_1 \\ &&& \fb \cldosuno{\rho_1} }}
)$ is the filler that we were looking for:

$$\vcenter{\xymatrix@C=-0pc{ \p \deq &&&& \fa \opunodos{= \quad } \\ 
\p \deq &&& id_{\Y} \op{\lambda_2} &&& \fa \deq \\
\p \deq && \f_2 \deq && \fv \dl & \dc{\lambda_1} & \dr \fa \\
\p && \f_2 && \f_1 \deq && \ii \deq \\
& \fu \cl{\rho_2} &&& \f_1 && \ii \deq \\
&& \fb \clunodos{\rho_1} &&&& \ii}}
\vcenter{\xymatrix@C=-0pc{ \quad = \quad }}
\vcenter{\xymatrix@C=-0pc{ \p \deq &&& \fa \op{=} \\
\p \dl & \dc{\eps} & \dr id_{\Y} && \fa \deq \\
\fu && \fv && \fa \deq \\
& \p \cl{\eps\inv} \dl & \dcr{\gamma} && \dr \fa \\
& \fb &&& \ii }}
\vcenter{\xymatrix@C=-0pc{ \quad = \quad }}
\vcenter{\xymatrix@C=-0pc{ \p \dl & \dc{\gamma} & \dr \fa \\ \fb && \ii}}$$

The first equality holds by elevators calculus plus the fact that $(\ff{f}_1,\lambda_1,\rho_1)$ and $(\ff{f}_2,\lambda_2,\rho_2)$ are fillers for the corresponding diagrams. 

$\Leftarrow)$ Since $\cc{C}$ satisfies axioms 2-M6a) and 2-M6c), it is clear that $\ff{p}$ is a fibration and it is also a weak equivalence because it can be factored as $\ff{p}=\ff{p} id_\ff{Y}$. 

\end{proof}

%
%

%

\begin{proposition}\label{sobran axiomas}
Let $\cc{C}$ be a 2-category with three classes of morphisms satisfying axioms 2-M6a), 2-M6b) and 2-M6c). Then the following hold:

\begin{itemize}
 \item[2-M1:] Given $\ff{i}$ a cofibration and $\ff{p}$ a fibration, if one of them is a weak equivalence, then the pair $(\ff{i},\ff{p})$ has the lifting property.
 \item[2-M3b:] Fibrations (respectively cofibrations) are closed under composition and bi-pullbacks (respectively bi-pushouts). Every isomorphism is a fibration and a cofibration.
 
\noindent In particular: 
 
 \item[2-M3:] Fibrations (respectively cofibrations) are closed under composition and pseudo-pullbacks (respectively pseudo-pushouts). Every isomorphism is a fibration and a cofibration.
 
 \item[2-M4b:] If $\ff{f}\in \cc{C}$ is the bi-pullback (respectively bi-pushout) of a fibration (respectively cofibration) which is also a weak equivalence, then $\ff{f}$ is a weak equivalence.
 
 \noindent In particular: 
 
 \item[2-M4:] If $\ff{f}\in \cc{C}$ is the pseudo-pullback (respectively pseudo-pushout) of a fibration (respectively cofibration) which is also a weak equivalence, then $\ff{f}$ is a weak equivalence.
 
 \item[2-M7:] Fibrations, cofibrations and weak equivalences are closed under isomorphisms, i.e. if there is an invertible 2-cell $\ff{f}\Mr{}\ff{g}$ and $\ff{f}$ is a fibration (respectively a cofibration or a weak equivalence), then $\ff{g}$ is also a fibration (respectively a cofibration or a weak equivalence).
 
\end{itemize}

\end{proposition}

\begin{proof}
 2-M1: is clear by \ref{lema2.2}.
 
 2-M3b: We are going to prove the case of fibrations, the case of cofibrations is similar and we leave it to the reader.

 \begin{itemize}
  \item[-] Suppose that $\ff{p}$ and $\ff{q}$ are two fibrations in $\cc{C}$. By axiom 2-M6a), to prove that $\ff{q}\ff{p}$ is a fibration, we only need to check that it has the right lifting property with respect to all morphisms that are both cofibrations and weak equivalences. So let $\ff{i}$ be a cofibration which is also a weak equivalence and suppose that we have a diagram of the form 

 $$\xymatrix@R=3pc@C=1.5pc{\ff{A} \ar[rr]^{\ff{a}} \ar[d]_{\ff{i}} &\ar@{}[d]|{\cong \; \Downarrow \; \gamma}& \ff{Y} \ar[d]^{\ff{q}\ff{p}} \\
            \ff{X} \ar[rr]_{\ff{b}} && \ff{B}}$$
            
\noindent Since $\ff{q}$ is a fibration, there exists a filler $(\ff{f}_0,\lambda_0, \rho_0)$ for the following diagram

$$\xymatrix@C=1.5pc@R=1.5pc{\ff{A} \ar[r]^{\ff{a}} \ar[dd]_{\ff{i}}  \ar@{}[ddrr]|{\cong \; \Downarrow \; \gamma } & \ff{Y} \ar[r]^{\ff{p}} & \ff{Z} \ar[dd]^{\ff{q}} \\
                            \\
                            \ff{X} \ar[rr]_{\ff{b}} & & \ff{B}} $$
 
\noindent Now, since $\ff{p}$ is also a fibration, there exists a filler $(\ff{f}_1,\lambda_1,\rho_1)$ for the following diagram

 $$\xymatrix@R=3pc@C=1.5pc{\ff{A} \ar[rr]^{\ff{a}} \ar[d]_{\ff{i}} &\ar@{}[d]|{\cong \; \Downarrow \; \lambda_0}& \ff{Y} \ar[d]^{\ff{p}} \\
            \ff{X} \ar[rr]_{\ff{f}_0} && \ff{Z}}$$

\noindent Let's check that $(\ff{f}_1,\lambda_1, 
\vcenter{\xymatrix@C=-0.2pc@R=1pc{ \q \deq && \p && \f_1 \\ \q &&& \f_0 \cl{\rho_1} \\ && \fb \cldosuno{\rho_0} }}
)$ is the filler that we were looking for:


$$\vcenter{\xymatrix@C=-0pc{ \q \deq && \p \deq &&& \fa \op{\lambda_1} \\
\q \deq && \p && \f_1 && \ii \deq \\
\q &&& \f_0 \cl{\rho_1} &&& \ii \deq \\
& \fb \clunodos{\rho_0} &&&&& \ii }}
\vcenter{\xymatrix@C=-0pc{ \quad = \quad }}
\vcenter{\xymatrix@C=-0pc{ \q \deq && \p \dl & \dc{\lambda_0} & \dr \fa \\ \q && \f_0 && \ii \deq \\ & \fb \cl{\rho_0} &&& \ii   }}
\vcenter{\xymatrix@C=-0pc{ \quad = \quad }}
\vcenter{\xymatrix@C=-0pc{ \q \ardr && \p \dc{\gamma} && \fa \ardl \\ & \fb && \ii  }}$$

\item[-] Now suppose that $\ff{p}$ is a fibration and we have a bi-pullback 

\begin{equation}\label{ppb}
  \xymatrix@R=3pc@C=1.5pc{\ff{P} \ar[rr]^{\pi_0} \ar[d]_{\pi_1}  &\ar@{}[d]|{\cong \; \Downarrow \; \alpha }& \ff{Y} \ar[d]^{\ff{p}} \\
            \ff{Z} \ar[rr]_{\ff{f}} && \ff{B}}
\end{equation}
            
\noindent We need to prove that $\pi_1$ has the right lifting property with respect to all morphisms that are both cofibrations and weak equivalences. So let $\ff{i}$ be a cofibration which is also a weak equivalence and suppose that we have a diagram of the form 
 
 $$\xymatrix@R=3pc@C=1.5pc{\ff{A} \ar[rr]^{\ff{a}} \ar[d]_{\ff{i}} &\ar@{}[d]|{\cong \; \Downarrow \; \gamma}& \ff{P} \ar[d]^{\pi_1} \\
            \ff{X} \ar[rr]_{\ff{b}} && \ff{Z}}$$             
            
\noindent Since $\ff{p}$ is a fibration, there exists a filler $(\ff{f}_0,\lambda_0,\rho_0)$ for the following diagram

 $$
 \vcenter{\xymatrix@R=3pc@C=3pc{\ff{A} \ar[r]^{\fa} \ar[d]_{\ii} \ar@{}[dr]|{\cong \; \Downarrow \; \gamma} & \ff{P} \ar[d]^{\pi_1} \ar[r]^{\pi_0} \ar@{}[dr]|{\; \; \cong \; \Downarrow \; \alpha} & \ff{Y} \ar[d]^{\p} \\
 \X \ar[r]_{\fb} & \Z \ar[r]_{\f} & \B}}
 \vcenter{\xymatrix@C=1pc{=}}
 \vcenter{\xymatrix@R=3pc@C=1.5pc{\ff{A} \ar[rr]^{\pi_0\ff{a}} \ar[d]_{\ff{i}} &\ar@{}[d]|{\cong \; \Downarrow \; \ff{f}\gamma \circ \alpha \ff{a}}& \ff{Y} \ar[d]^{\ff{p}} \\
            \ff{X} \ar[rr]_{\ff{f}\ff{b}} && \ff{B}}}$$   

\noindent Since \eqref{ppb} is a bi-pullback, there exists a morphism $\ff{X}\mr{\ff{g}}\ff{P}$ and invertible \mbox{2-cells} $\pi_0 \ff{g} \Mr{\beta_0} \ff{f}_0$, $\pi_1 \ff{g} \Mr{\beta_1} \ff{b}$ satisfying the following equality: 


$$\vcenter{\xymatrix@C=-0pc{ \p \dl & \dc{\alpha} & \dr \pi_0 && \g \deq \\ \f \deq && \pi_1 && \g \\ \f &&& \fb \cl{\beta_1} }}
\vcenter{\xymatrix@C=-0pc{ \quad = \quad }}
\vcenter{\xymatrix@C=-0pc{ \p \deq && \pi_0 && \g \\ \p \dl & \dcr{\rho_0} && \dr \f_0 \cl{\beta_0} \\ \f &&& \fb  }}$$


\noindent It is straightforward to check that $(\ff{g},\lambda, \beta_1)$ is the filler that we were looking for, where $\lambda$ is such that $\pi_0 \lambda= 
\vcenter{\xymatrix@C=-0pc{& \pi_0 \dl & \dcr{\lambda_0} && \fa \dr \\ & \f_0 \op{\beta_0\inv} &&& \ii \deq \\ \pi_0 && \g && \ii  }}
$ and $\pi_1 \lambda= 
\vcenter{\xymatrix@C=-0pc{& \pi_1 \dl & \dcr{\gamma} && \fa \dr \\ & \fb \op{\beta_1\inv} &&& \ii \deq \\ \pi_1 && \g && \ii    }}
$.

\item[-] To conclude with axiom 2-M3b, suppose that $\ff{f}$ is an isomorphism, $\ff{i}$ is a morphism that is both a cofibration and a weak equivalence and we have a diagram of the form 

$$\xymatrix@R=3pc@C=1.5pc{\ff{A} \ar[rr]^{\ff{a}} \ar[d]_{\ff{i}} &\ar@{}[d]|{\cong \; \Downarrow \; \gamma}& \ff{Y} \ar[d]^{\ff{f}} \\
            \ff{X} \ar[rr]_{\ff{b}} && \ff{B}}$$
            
Since $\ff{f}$ is an isomorphism, there exists $\ff{B}\mr{\ff{g}}\ff{Y}$ such that $\ff{f}\ff{g}=id_{\ff{B}}$ and $\ff{g}\ff{f}=id_{\ff{Y}}$. It is clear that $(\ff{g}\ff{b},\gamma,id)$ is the filler that we were looking for.

 \end{itemize}

2-M4b: Suppose that $\ff{p}$ is both a fibration and a weak equivalence and we have a bi-pullback 

\begin{equation}\label{ppb2}
 \xymatrix@R=3pc@C=1.5pc{\ff{P} \ar[rr]^{\pi_0} \ar[d]_{\pi_1} 
            & \ar@{}[d]|{\cong \; \Downarrow \; \alpha}
            & \ff{Y} \ar[d]^{\ff{p}}\\
            \ff{Z} \ar[rr]_{\ff{f}} 
            & & \ff{B}}
\end{equation}

\noindent Since we have already proved axiom 2-M3b, we know that $\pi_1$ is a fibration. By \ref{lema2.2}, we have to check that $\pi_1$ has the right lifting property with respect to all cofibrations. The proof follows as the proof of axiom 2-M3b.

%
%
%
%

2-M7: Suppose that $\ff{f}$ is a fibration (the case of a cofibration is similar and we leave it to the reader) and there is an isomorphism $\ff{f}\Mr{\alpha}\ff{g}$. We want to check that $\ff{g}$ has the right lifting property with respect to all morphisms that are both cofibrations and weak equivalences. So, suppose that $\ff{i}$ is a morphism that is both a cofibration and a weak equivalence and we have a diagram of the form 

$$\xymatrix@R=3pc@C=1.5pc{\ff{A} \ar[rr]^{\ff{a}} \ar[d]_{\ff{i}} &\ar@{}[d]|{\cong \; \Downarrow \; \gamma}& \ff{Y} \ar[d]^{\ff{g}} \\
            \ff{X} \ar[rr]_{\ff{b}} && \ff{B}}$$
            
\noindent Since $\ff{f}$ is a fibration, there is a filler $(\ff{f}_0,\lambda_0,\rho_0)$ for the following diagram

$$\vcenter{\xymatrix@R=3pc@C=1.5pc{\ff{A} \ar[rr]^{\ff{a}} \ar[d]_{\ff{i}} &\ar@{}[d]|{\cong \; \Downarrow \; \gamma \quad}& \ff{Y} \ar@<1.5ex>[d]^{\ff{f}} \ar@<-1.5ex>[d]_{\ff{g}} \ar@{}[d]|{\stackrel[\scriptstyle \cong]{\scriptstyle \alpha}{\Leftarrow}} \\
            \ff{X} \ar[rr]_{\ff{b}} && \ff{B}}}
\vcenter{\xymatrix@C=1pc{=}}
\vcenter{\xymatrix@R=3pc@C=1.5pc{\ff{A} \ar[rr]^{\ff{a}} \ar[d]_{\ff{i}} &\ar@{}[d]|{\cong \; \Downarrow \; \gamma\circ \alpha \ff{a}}& \ff{Y} \ar[d]^{\ff{f}} \\
            \ff{X} \ar[rr]_{\ff{b}} && \ff{B}}}$$
            
\noindent It is straightforward to check that $(\ff{f}_0,\lambda_0, 
\vcenter{\xymatrix@C=-0.2pc@R=1pc{ \g \dcellb{\alpha\inv} && \f_0 \deq \\ \f && \f_0 \\ & \fb \cl{\rho_0} }}
)$ is the filler that we were looking for.

To conclude, suppose that $\ff{f}$ is a weak equivalence and there is an isomorphism $\ff{f}\Mr{\alpha}\ff{g}$. Then $\ff{g}\cong \ff{f}\cong \ff{u}\ff{v}$ as in 2-M6c) and so is also a weak equivalence.
\end{proof}

Although we do not use it in this work, we set (for the record) the following definition:

\begin{definition}
 We say that a 2-category $\cc{C}$ is a \emph{2-model 2-category} (respectively a \emph{2-bmodel 2-category}) if it is equipped with three classes of morphisms called fibrations, cofibrations and weak equivalences satisfying 2-M0 (respectively 2-M0b), 2-M1, 2-M2, 2-M3 (respectively 2-M3b), 2-M4 (respectively 2-M4b), 2-M5 and 2-M7.
\end{definition}

\begin{corollary}[of \ref{sobran axiomas}]\label{axiomasnecesariosysuficientes}
\comw{A}

\begin{enumerate}
  \item If $\cc{C}$ is a closed 2-bmodel 2-category, then $\cc{C}$ is a 2-bmodel 2-category.
  
  \item If $\cc{C}$ is a closed 2-model 2-category, then $\cc{C}$ is a 2-model 2-category. $\hfill \square$
 \end{enumerate}

\end{corollary}

\begin{remark}
By axiom 2-M7, axiom 2-M5 can be replaced in each definition by its following weaker version

\begin{center}
\textsl{2-M5w: Given two composable morphisms $\ff{f}$, $\ff{g}\in\cc{C}$, if two of the three $\ff{f}$, $\ff{g}$, $\ff{g}\ff{f}$ are weak equivalences, then so is the third one. Every isomorphism is a weak equivalence.\cqd} 
\end{center}

\end{remark}

The following is the 2-dimensional version of ``the retract argument'' \cite[7.2.2]{H}. 

\begin{proposition}\label{retractargument}
Let $\cc{C}$ be a 2-category and let $\ff{X}\mr{\ff{f}}\ff{Y}\in \cc{C}$. 

\begin{enumerate}
 \item If $\ff{f}$ is factorized as $\ff{f}\cong \ff{p}\ff{i}$ and the pair $(\ff{f},\ff{p})$ has the lifting property, then $\ff{f}$ is a retract of $\ff{i}$ in $\cc{M}aps_p(\cc{C})$. 
 
 \item If $\ff{f}$ is factorized as $\ff{f}\cong \ff{p}\ff{i}$ and the pair $(\ff{i},\ff{f})$ has the lifting property, then $\ff{f}$ is a retract of $\ff{p}$ in $\cc{M}aps_p(\cc{C})$.
\end{enumerate}

\end{proposition}

\begin{proof}
 
\begin{enumerate}
 \item Let $(\ff{g},\lambda,\rho)$ be a filler for the following diagram
 
 $$\xymatrix@C=1.5pc@R=3pc{\ff{X} \ar[rr]^{\ff{i}} \ar[d]_{\ff{f}} &\ar@{}[d]|{\cong \; \Downarrow \; \gamma}& \ff{Z} \ar[d]^{\ff{p}} \\
            \ff{Y} \ar[rr]_{id_{\ff{Y}}} && \ff{Y}}$$
            
 Then $\ff{f}$ is a retract of $\ff{i}$ via $(id_{\ff{X}},\ff{g},\lambda,id_{\ff{X}},\ff{p},\gamma^{-1},id_{id_{\ff{X}}},\rho)$.
 
 \item The proof is similar to the previous one and we leave it to the reader.
\end{enumerate}

\end{proof}

\begin{proposition}\label{retractospreservanlifting}
 Let $\cc{C}$ be a 2-category, $\ff{p}'$ a retract of $\ff{p}$ and $\ff{i}'$ a retract of $\ff{i}$. If the pair $(\ff{i},\ff{p})$ has the lifting property, then the pair  $(\ff{i}',\ff{p}')$ also does.
\end{proposition}

\begin{proof}
We know that $\ff{p}'$ is a retract of $\ff{p}$ via $(\theta_0,\theta_1,\theta_m,\eta_0,\eta_1,\eta_m,\mu_0,\mu_1)$ and $\ff{i}'$ is a retract of $\ff{i}$ via $(\theta'_0,\theta'_1,\theta'_m,\eta'_0,\eta'_1,\eta'_m,\mu'_0,\mu'_1)$. Suppose that we have a diagram of the form

\begin{equation}\label{cuadradoconlosprima}
\xymatrix@C=1.5pc@R=3pc{\ff{A}' \ar[rr]^{\ff{a}} \ar[d]_{\ff{i}'}&\ar@{}[d]|{\cong \; \Downarrow \; \gamma}& \ff{Y}' \ar[d]^{\ff{p}'} \\
            \ff{X}' \ar[rr]_{\ff{b}} && \ff{B}'} 
\end{equation}

By hypothesis, there is a filler $(\ff{f},\lambda,\rho)$ for diagram

$$\xymatrix@C=1.5pc@R=3pc{\ff{A} \ar[rr]^{\theta_0 \ff{a} \eta'_0} \ar[d]_{\ff{i}}&\ar@{}[d]|{\cong \; \Downarrow \; \gamma'}& \ff{Y} \ar[d]^{\ff{p}} \\
            \ff{X} \ar[rr]_{\theta_1 \ff{b} \eta'_1} && \ff{B}} $$

\noindent where $\gamma= 
\vcenter{\xymatrix@C=0pc@R=1pc{ \p \dl & \dc{\theta_m} & \dr \theta_0 && \fa \deq && \eta_0' \deq \\
\theta_1 \deq && \p' \dl & \dc{\gamma} & \dr \fa \dr && \eta_0' \deq \\
\theta_1 \deq && \fb \deq && \ii' \dl & \dc{\eta'_m} & \dr \eta_0' \\
\theta_1 && \fb && \eta_1' && \ii }}$.

\noindent It can be checked that $(\eta_0\ff{f}\theta'_1, 
\vcenter{\xymatrix@C=-0.2pc@R=1.5pc{ &&&& \fa \optrestres{=} \\
& id_{\Y'} \opb{{\mu_0}\inv} &&& \fa \deq &&& id_{\A'} \opb{{\mu'_0}\inv} \\
\eta_0 \deq && \theta_0 \ardr && \fa \dc{\lambda} && \eta'_0 \ardl && \theta'_0 \deq \\
\eta_0 \deq &&& \f \deq && \ii \dl & \dcr{\theta_m} && \theta'_0 \dr \\
\eta_0 &&& \f && \theta'_1 &&& \ii'}}
, 
\vcenter{\xymatrix@C=-0.2pc@R=1.5pc{ \p' \dl & \dcr{\eta_m} && \eta_0 \dr && \f \deq &&& \theta_1' \deq \\
\eta_1 \deq &&& \p \ardl & \dc{\rho} & \f \ardr &&& \theta_1' \deq \\
\eta_1 && \theta_1 && \fb \deq && \eta_1' && \theta_1' \\
& id_{\B'} \cl{\mu_1} &&& \fb &&& id_{\X'} \cl{\mu_1'} \\
&&&& \fb \cltrestres{=} }}
)$ is the filler that we were looking for.
\end{proof}

\begin{proposition}\label{retractospreservan}
Let $\cc{C}$ be a 2-category with three classes of morphisms satisfying axioms 2-M6a), 2-M6b) and 2-M6c). Then fibrations, cofibrations, morphisms that are both fibrations and weak equivalences and morphisms that are both cofibrations and weak equivalences are closed under the formation of retracts.
\end{proposition}

\begin{proof}

It is straightforward from \ref{retractospreservanlifting} plus \ref{lema2.2}.
\end{proof}

 We believe the previous statement is also true for weak equivalences but its proof is a deeper result. Classically (in the 1-dimensional case, for closed model categories) it follows immediately from Quillen's theorem ``$\gamma(\ff{f}) \; invertible \;\; if\!f \;\; \ff{f} \; is \; a \; weak \; equivalence$'' (where $\gamma$ is the universal functor inverting weak equivalences). We expect to finish a 2-dimensional version of this theorem in future work.

\subsection{Locally functorial factorizations and cylinder objects}

\begin{definition}\label{locally functorial}
A factorization up to invertible 2-cell for arrows in a 2-category $\cc{C}$ is said to be locally pseudo-functorial (on the sequel we will just say functorial) if given $\f \mr{\theta} \f' \mr{\eta} \f'' \in \cc{H}om_p(\ff{2},\cc{C})$ as in the following diagram
 
 $$ 
 \xymatrix@C=3pc@R=3pc{ \X\ar[r]^{\theta_0} \ar[d]_\f \ar@{}[dr]|{\cong \; \Downarrow \; \theta_m} & \X' \ar[d]_{\f'} \ar[r]^{\eta_0} \ar@{}[dr]|{\cong \; \Downarrow \; \eta_m} & \X'' \ar[d]^{\f''} \\
 \Y \ar[r]_{\theta_1} & \Y' \ar[r]_{\eta_1} & \Y''}
 $$
 
\noindent there are suitable factorizations
 
 $$ 
 \xymatrix{\X \ar[rr]^\f \ar[rd]_\ii & \ar@{}[d]|{\cong \; \Downarrow \; \alpha } & \Y \\
 & \Z \ar[ru]_\p }, \quad
 \xymatrix{\X' \ar[rr]^{\f'} \ar[rd]_{\ii'} & \ar@{}[d]|{\cong \; \Downarrow \; \alpha' } & \Y' \\
 & \Z' \ar[ru]_{\p'} }, \quad
 \xymatrix{\X'' \ar[rr]^{\f''} \ar[rd]_{\ii''} & \ar@{}[d]|{\cong \; \Downarrow \; \alpha'' } & \Y'' \\
 & \Z'' \ar[ru]_{\p''} }$$
  
 \noindent such that there exist morphisms 
 $\Z \mr{\ff{g}} \Z' \mr{\ff{h}} \Z''$ and invertible 2-cells $\vcenter{\xymatrix@C=-0pc{\ii' \dl & \dc{\beta} & \theta_0 \dr \\ \ff{g} & & \ii}}$, $\vcenter{\xymatrix@C=-0pc{\ii' \dl & \dc{\gamma} & \eta_0 \dr \\ \ff{h} & & \ii'}}$, $\vcenter{\xymatrix@C=-0pc{\p' \dl & \dc{\mu} & \ff{g} \dr \\ \theta_1 & & \p}}$, $\vcenter{\xymatrix@C=-0pc{\p'' \dl & \dc{\epsilon} & \ff{h} \dr \\ \eta_1 & & \p'}}$ fitting in the following diagram:
 
 $$ 
 \xymatrix{\X \ar[dd]_>>>>{\f} \ar[dr]^{\ii} \ar[rr]^{\theta_0} && \X' \ar[dd]_>>>>{\f'} \ar[dr]^{\ii'} \ar[rr]^{\eta_0} && \X'' \ar[dd]_>>>>{\f''} \ar[dr]^{\ii''} \\
 & \Z \ar[rr]|{\quad} \ar@{}[rr]^>>>>{\g} \ar[dl]^{\p} && \Z' \ar[dl]^{\p'} \ar[rr]|{\quad} \ar@{}[rr]^>>>>{\h} && \Z'' \ar[dl]^{\p''} \\
 \Y \ar[rr]_{\theta_1} && \Y' \ar[rr]_{\eta_1} && \Y''}
 $$
 
 \noindent and satisfying the following equations:
 
 $$ 
\vcenter{\xymatrix@C=-0pc{& \f' \op{\alpha'} &&& \theta_0 \deq \\ \p' \deq && \ii' \dl & \dc{\beta} & \dr \theta_0 \\ \p' \dl & \dc{\mu} & \dr \g && \ii \deq \\ \theta_1 && \p && \ii  }}
\vcenter{\xymatrix@C=-0pc{ \quad = \quad }}
\vcenter{\xymatrix@C=-0pc{\f' \dl & \dcr{\theta_m} && \dr \theta_0 \\ \theta_1 \deq &&& \f \op{\alpha} \\ \theta_1 && \p && \ii }}
 \hbox{ and } \quad 
\vcenter{\xymatrix@C=-0pc{& \f'' \op{\alpha''} &&& \eta_0 \deq \\ \p'' \deq && \ii'' \dl & \dc{\gamma} & \dr \eta_0 \\ \p'' \dl & \dc{\eps} & \dr \h && \ii' \deq \\ \eta_1 && \p' && \ii'  }}
\vcenter{\xymatrix@C=-0pc{ \quad = \quad }}
\vcenter{\xymatrix@C=-0pc{\f'' \dl & \dcr{\eta_m} && \dr \eta_0 \\ \eta_1 \deq &&& \f' \op{\alpha'} \\ \eta_1 && \p' && \ii' }} $$
%
\end{definition}

\begin{proposition}\label{factorizacionfuntorial}
If $\cc{C}$ is a closed 2-bmodel 2-category, then the factorization of axiom 2-M2 is locally functorial. 

%
 
 \end{proposition}

\begin{proof}
 In the situation of \ref{locally functorial}, by axiom 2-M2, we have a factorization $\ff{f} \Mr{\alpha \cong} \ff{p} \ff{i}$ where $\ff{i}$ is a cofibration and $\ff{p}$ is both a fibration that is also a weak equivalence. Consider the following diagram 
 
 $$ 
 \xymatrix{ \X \ar[rrr]^{\theta_0} \ar[dd]_{\ii} \ar@{}[rrrd]|{\cong \; \Downarrow \; \delta} \ar@{}[rrrdd]|{bi-p.o.} &&& \X' \ar[rrr]^{\eta_0} \ar[dd]_{\lambda_0} \ar@{}[rrrd]|{\cong \; \Downarrow \; \delta'} \ar@{}[rrrdd]|{bi-p.o.} &&& \X'' \ar[dd]^{\lambda_0'} \\ &&&&&& \\
 \Z \ar[rrr]^{\lambda_1} \ar[dd]_{\p} \ar@{}[rrrdd]|{\cong \; \Downarrow \; \beta_1} &&& \fP \ar[rrr]^{\lambda_1'} \ar@{-->}[dd]_{\fa} \ar@{-->}[rd]^{\tilde{\ii}} & \ar@{}[rd]|{\cong \; \Downarrow \; \lambda} && \Q \ar@{-->}[dd]^{\fb} \ar@{-->}[dl]_{\tilde{\ii}'} \\
 &&& \ar@{}[r]|>>>>>>{\cong \; \Rightarrow \; \tilde{\alpha}} & \Z' \ar@{-->}[dl]^{\p'} \ar@{-->}[r]^{\ff{h}} \ar@{}[rd]|{\cong \; \Downarrow \; \rho} & \Z'' \ar@{-->}[dr]_{\p''} \ar@{}[r]|>>>>>>{\cong \; \Rightarrow \; \tilde{\alpha}'}  & \\
 \Y \ar[rrr]_{\theta_1} &&& \Y' \ar[rrr]_{\eta_1} &&& \Y'' }
 $$
 
 \noindent Both bi-pushouts exist by axiom 2-M0b. $\ff{P}\mr{\ff{a}} \ff{Y}'$ together with invertible 2-cells $ 
 \vcenter{\xymatrix@C=-0pc{ \fa && \lambda_0 \\ & \f' \cl{\beta_0} }}
 $, $ 
 \vcenter{\xymatrix@C=-0pc{ \fa \dl & \dc{\beta_1} & \dr \lambda_1 \\ \theta_1 && \p}}
 $ such that $ 
 \vcenter{\xymatrix@C=-0pc@R=1pc{ \fa && \lambda_0 && \theta_0 \deq \\ 
 & \f' \cl{\beta_0} \dl & \dcr{\theta_m} && \theta_0 \dr \\
 & \theta_1 \deq &&& \f \op{\alpha} \\
 & \theta_1 && \p && \ii }}
\vcenter{\xymatrix@C=-0pc{ \quad = \quad }}
\vcenter{\xymatrix@C=-0pc@R=1pc{ \fa \deq && \lambda_0 \dl & \dc{\delta} & \dr \theta_0 \\
\fa \dl & \dc{\beta_1} & \dr \lambda_1 && \ii \deq \\
\theta_1 && \p && \ii}}
 $ exist by universal property of $\ff{P}$. By axiom 2-M2, we have a factorization $\ff{a} \Mr{\tilde{\alpha}} \ff{p}' \tilde{\ff{i}}$ with $\tilde{\ff{i}}$ a cofibration and $\ff{p}'$ a fibration and a weak equivalence. $\ff{Q}\mr{\ff{b}} \ff{Y}''$ together with invertible 2-cells $ 
 \vcenter{\xymatrix@C=-0pc{ \fb && \lambda_0' \\ & \f'' \cl{\beta_0'} }}
 $, $ 
 \vcenter{\xymatrix@C=-0pc{ & \fb \ardl & \dc{\beta_1'} & \lambda_1' \ardr \\ \eta_1 && \p' && \tilde{\ii} }}
 $ such that $ 
  \vcenter{\xymatrix@C=-0pc@R=1pc{ \fb && \lambda_0' &&& \eta_0 \deq \\ 
  & \f'' \cl{\beta_0'} \dl && \dc{\eta_m} && \dr \eta_0 \\
  & \eta_1 \deq &&&& \f' \opunodos{\!\!\!\!\!\beta_0\inv} \\
  & \eta_1 \deq &&& \fa \op{\tilde{\alpha}} &&& \lambda_0 \deq \\
  & \eta_1 && \p' && \tilde{\ii} && \lambda_0 }}
\vcenter{\xymatrix@C=-0pc{ \quad = \quad }}
\vcenter{\xymatrix@C=-0pc@R=1pc{ & \fb \deq && \lambda_0' \dl & \dcr{\delta'} && \dr \eta_0 \\
 & \fb \ardl & \dc{\beta_1'} & \lambda_1' \ardr &&& \lambda_0 \deq \\ 
 \eta_1 && \p' && \tilde{\ii} && \lambda_0 }}
 $ exist by universal property of $\ff{Q}$. By axiom 2-M2, we have a factorization $\ff{b} \Mr{\tilde{\alpha'}} \ff{p}'' \tilde{\ff{i}'}$ with $\tilde{\ff{i}'}$ a cofibration and $\ff{p}''$ a fibration and a weak equivalence. Finally, $(\ff{h}, \lambda, \rho)$ is a filler given by axiom 2-M6. Thus we have the following equality: 
 
 $$ 
 \vcenter{\xymatrix@C=-0pc{ \p'' \deq && \tilde{\ii}' \dl & \dc{\lambda} & \dr \lambda_1' \\
 \p'' \dl & \dc{\rho} & \dr \ff{h} && \tilde{\ii} \deq \\
 \eta_1 && \p' && \tilde{\ii}  }}
\vcenter{\xymatrix@C=-0pc{ \quad = \quad }}
\vcenter{\xymatrix@C=-0pc{ \p'' && \tilde{\ii}' && \lambda_1' \deq \\
& \fb \clb{\tilde{\alpha}'{}\inv} \ardl & \dcr{\beta_1'} && \lambda_1' \ardr \\
\eta_1 && \p' &&& \tilde{\ii} }}
 $$
 
Take $\ff{i}'=\tilde{\ff{i}}\lambda_0$ and $\ff{i}''=\tilde{\ff{i}'}\lambda'_0$ which are cofibrations because they are compositions of cofibrations ($\lambda_0$ and $\lambda'_0$ are cofibrations by axiom 2-M3). It is straightforward to check that $ 
\vcenter{\xymatrix@C=-0pc@R=1.2pc{& \f \op{\alpha} \\ \p && \ii }}
$, $ 
\vcenter{\xymatrix@C=-0pc@R=1.2pc{ && \f' \opunodos{\!\!\!\!\beta_0\inv} \\ & \fa \op{\tilde{\alpha}} &&& \lambda_0 \deq \\ \p' && \tilde{\ii} && \lambda_0 }}
$, $ 
\vcenter{\xymatrix@C=-0pc@R=1.2pc{ && \; \f'' \; \opunodos{\!\!\!\!\!\!\beta_0'{}\inv} & \; \\ & \fb \op{\tilde{\alpha}'} &&& \lambda_0' \deq \\ \p'' && \tilde{\ii}' && \lambda_0' }}
$, 
$\ff{g}=\tilde{\ii} \lambda_1$, 
$\h$, 
$\beta=\vcenter{\xymatrix@C=-0pc@R=1.2pc{ \tilde{\ii} \deq && \lambda_0 \dl & \dc{\delta} & \dr \theta_0 \\ \tilde{\ii} && \lambda_1 && \ii }}$, 
$\gamma=\vcenter{\xymatrix@C=-0pc@R=1.2pc{ \tilde{\ii}' \deq && \lambda_0' \dl & \dc{\delta'} & \dr \eta_0 \\
\tilde{\ii}' \dl & \dc{\lambda} & \dr \lambda_1' && \lambda_0 \deq \\
\ff{h} && \tilde{\ii} && \lambda_0 }}$, 
$\mu=\vcenter{\xymatrix@C=-0pc@R=1.2pc{ \p' && \tilde{\ii} && \lambda_1 \deq \\
& \fa \cl{\tilde{\alpha}\inv} \dl & \dcr{\beta_1} && \dr \lambda_1 \\
& \theta_1 &&& \p }}$ 
and $\epsilon=\vcenter{\xymatrix@C=-0pc@R=1.2pc{ \p'' \dl & \dc{\rho'} & \dr \ff{h} \\ \eta_1 && \p' }}$ 
satisfy the desired property.
\end{proof}

\begin{remark}\label{factorizacionfuntorialretractos}
In the situation of \ref{locally functorial}, if $\ff{f}=\ff{f}''$ and $\theta$, $\eta$ are part of a retraction from $\ff{f}$ to $\ff{f}'$, then factorizations for $\ff{f}$ and $\ff{f'}$ given by \ref{factorizacionfuntorial} can be chosen in such a way that $\ff{i}$ is a retract of $\ff{i}'$ and $\ff{p}$ is a retract of $\ff{p}'$. \cqd
\end{remark}

\begin{definition} Let $\cc{C}$ be a closed 2-bmodel 2-category and $\ff{X}$ an object of $\cc{C}$. 
 \begin{enumerate}
  
  \item We say that $\ff{X}$ is a fibrant object if the only morphism $\ff{X}\mr{}*$ is a fibration.
  
  \item We say that $\ff{X}$ is a cofibrant object if the only morphism $0\mr{}\ff{X}$ is a cofibration.
                   
 \end{enumerate}

\end{definition}

\begin{remark}
Note that $0$ and $*$ are denoting the bi-initial and the bi-terminal object respectively given by axiom M0b. More explicitly, $0$ satisfies that for each $\X\in \cc{C}$, there exists a morphism $0\mr{}\X\in \cc{C}$ up to unique invertible 2-cell. And $*$ satisfies that for each $\X\in \cc{C}$, there exists a morphism $\X\mr{}* \in \cc{C}$ up to unique invertible 2-cell.

In the previous definition the abuse of saying ``the only morphism'' is justified by axiom 2-M7. \cqd
\end{remark}

\begin{definition}\label{cilindros}
Let $\cc{C}$ be a closed 2-bmodel 2-category and $\ff{X}$ an object of $\cc{C}$. 

\begin{enumerate} 

  \item A cylinder object for $\ff{X}$ consists of a diagram 
  
  \begin{equation}\label{cilindro}
\xymatrix@C=3pc{\ff{X}\amalg \ff{X} \ar[d]_{\ff{i}^{\ff{X}}} \ar@/^2ex/[rd]^{\nabla^{\ff{X}}} & \\
              \widetilde{\ff{X}} \ar[r]_{\sigma^{\ff{X}}} \ar@{}[u]_{\;\; \cong \; \Uparrow \; \gamma^{\ff{X}}} 
              & \ff{X} }   
  \end{equation}
  
  \noindent where $\nabla^{\ff{X}}=\left(\vcenter{\xymatrix@R=.0pc{id_{\ff{X}} \\ id_{\ff{X}} }} \right)$, $\ff{i}^{\ff{X}}=\left(\vcenter{\xymatrix@R=-.2pc{\ff{i}^{\ff{X}}_0 \\ \ff{i}^{\ff{X}}_1 }} \right)$ is a cofibration and $\sigma^{\ff{X}}$ is a weak equivalence. By $\ff{X}\amalg \ff{X}$ we denote the bi-coproduct.

\item A path object $\ff{X}$ consists of a diagram 

$$\xymatrix@C=3pc{& \widehat{\ff{X}} \ar@{}[d]_{\cong \; \Downarrow \; \gamma^{\ff{X}} \;\;} \ar[d]^{\ff{p}^{\ff{X}}} \\
            \ff{X} \ar@/^2ex/[ru]^{\ff{s}^{\X}}  \ar[r]_{\Delta^{\ff{X}}} 
            & \ff{X}\times \ff{X}}$$
            
\noindent where $\Delta^{\ff{X}}=(id_{\ff{X}}, id_{\ff{X}})$, $\ff{p}^{\ff{X}}=(\ff{p}^{\ff{X}}_0,\ff{p}^{\ff{X}}_1)$ is a fibration and $\ff{s}^{\ff{X}}$ is a weak equivalence. By $\ff{X}\times \ff{X}$ we denote the bi-product.

\end{enumerate}
\end{definition}

The following corollary says that in a closed 2-bmodel 2-category there are locally functorial cylinder objects (c.f. \ref{N4}).

\begin{corollary}[of \ref{factorizacionfuntorial}]\label{cilindros funtoriales}
Let $\cc{C}$ be a closed 2-bmodel 2-category. Given \mbox{$\X \mr{\f} \X' \mr{\f'} \X'' \in \cc{C}$}. There are suitable cylinder objects for $\X$, $\X'$ and $\X''$

$$\xymatrix@C=3pc{\ff{X}\amalg \ff{X} \ar[d]_{\ff{i}^{\ff{X}}} \ar@/^2ex/[rd]^{  \nabla^{\ff{X}}} & \\
              \widetilde{\ff{X}} \ar[r]_{\sigma^{\ff{X}}} \ar@{}[u]_{\;\; \cong \; \Uparrow \; \gamma^{\ff{X}}} 
              & \ff{X} }   \quad
\xymatrix@C=3pc{\ff{X}'\amalg \ff{X}'\ar[d]_{\ff{i}^{\ff{X}'}} \ar@/^2ex/[rd]^{ \nabla^{\ff{X}'} } & \\
              \widetilde{\ff{X}'} \ar[r]_{\sigma^{\ff{X}'}} \ar@{}[u]_{\;\; \cong \; \Uparrow \; \gamma^{\ff{X}'}} 
              & \ff{X}' }\quad
\xymatrix@C=3pc{\ff{X}''\amalg \ff{X}''\ar[d]_{\ff{i}^{\ff{X}''}} \ar@/^2ex/[rd]^{ \nabla^{\ff{X}''} } & \\
              \widetilde{\ff{X}''} \ar[r]_{\sigma^{\ff{X}''}} \ar@{}[u]_{\;\; \cong \; \Uparrow \; \gamma^{\ff{X}''}} 
              & \ff{X}'' }$$
  
 \noindent morphisms $\widetilde{\X} \mr{\ff{g}} \widetilde{\X}' \mr{\ff{h}} \widetilde{\X}''$ and invertible 2-cells $\vcenter{\xymatrix@C=-0pc{\ff{i}^{\ff{X}'} \dl & \dc{\beta} & \f\amalg \f \dr \\ \ff{g} & & \ff{i}^{\ff{X}}}}$, $\vcenter{\xymatrix@C=-0pc{\ff{i}^{\ff{X}''} \dl & \dc{\gamma} & \f'\amalg \f' \dr \\ \ff{h} & & \ff{i}^{\ff{X}'} }}$, $\vcenter{\xymatrix@C=-0pc{\sigma^{\ff{X}'} \dl & \dc{\mu} & \ff{g} \dr \\ \f & & \sigma^{\ff{X}}}}$, $\vcenter{\xymatrix@C=-0pc{\sigma^{\ff{X}''} \dl & \dc{\epsilon} & \ff{h} \dr \\ \f' & & \sigma^{\ff{X}'}}}$ fitting in the following diagram:
 
 $$ 
 \xymatrix{\X \amalg \X \ar[dd]_>>>>>{\nabla^{\X}} \ar[dr]^{i^{\X}} \ar[rr]^{\f \amalg \f } && \X' \amalg \X' \ar[dd]_>>>>>{\nabla^{\X'}} \ar[dr]^{i^{\X'}} \ar[rr]^{\f' \amalg \f' } && \X'' \amalg \X'' \ar[dd]_>>>>>{\nabla^{\X''}} \ar[dr]^{i^{\X''}} \\
 & \widetilde{\X} \ar[rr]|{\quad} \ar@{}[rr]^>>>>>>>>{\g} \ar[dl]^<<<<{\sigma^{X}} && \widetilde{\X'} \ar[dl]^<<<<{\sigma^{X'}} \ar[rr]|{\quad} \ar@{}[rr]^>>>>>>>>{\h} && \widetilde{\X''} \ar[dl]^<<<<{\sigma^{X''}} \\
 \X \ar[rr]_{\f} && \X' \ar[rr]_{\f'} && \X''}
 $$
 
 \noindent satisfying the following equations:
 
 $$ 
 \vcenter{\xymatrix@C=-.2pc{& \nabla^{\X'} \opbb{(\gamma^{\X'}){}\inv} &&& \f \amalg \f \deq \\ \sigma^{\X'} \deq && \ii^{\X'} \dl & \dc{\beta} & \dr \f \amalg \f \\ \sigma^{\X'} \dl & \dc{\mu} & \dr \g && \ii^{\X} \deq \\ \f && \sigma^{\X} && \ii^{\X}  }}
\vcenter{\xymatrix@C=-.2pc{ = }}
\vcenter{\xymatrix@C=-.2pc{\nabla^{\X'} \dl & \dcr{\cong} && \dr \f \amalg \f \\ \f \deq &&& \nabla^{\X} \opbb{(\gamma^{\X'}){}\inv} \\ \f && \sigma^{\X} && \ii^{\X} }}
 \hbox{ and } 
 \vcenter{\xymatrix@C=-.2pc{& \nabla^{\X''} \opbymedio{(\gamma^{\X''}){}\inv} &&& \f' \amalg \f' \deq \\ \sigma^{\X''} \deq && \ii^{\X''} \dl & \dc{\gamma} & \dr \f' \amalg \f' \\ \sigma^{\X''} \dl & \dc{\eps} & \dr \h && \ii^{\X'} \deq \\ \f' && \sigma^{\X'} && \ii^{\X'}  }}
\vcenter{\xymatrix@C=-.2pc{ = }}
\vcenter{\xymatrix@C=-.2pc{\nabla^{\X''} \dl & \dcr{\cong} && \dr \f' \amalg \f' \\ \f' \deq &&& \nabla^{\X'} \opbymedio{(\gamma^{\X'}){}\inv} \\ \f' && \sigma^{\X'} && \ii^{\X'} }}
 $$ \cqd
 
\end{corollary}

\begin{proposition}\label{2.3.5 de EH}
 Let $\cc{C}$ be a closed 2-bmodel 2-category and $\ff{X}\mr{\ff{f}}\ff{Y}\in \cc{C}$ a cofibration. There exist suitable cylinder objects for $\ff{X}$ and $\ff{Y}$ such that $\ff{f}$ induces a morphism that is both a cofibration and a weak equivalence
 
 $$\ff{Y} \bigtriangleup \widetilde{\X} \mr{\ff{k}'_{\ff{f}}} \widetilde{\ff{Y}}$$
 
 \noindent and a cofibration 
 
 $$\ff{Y} \bigtriangleup \widetilde{\X} \bigtriangledown \Y \mr{\ff{k}_{\ff{f}}} \widetilde{\ff{Y}}$$

\noindent where $\ff{Y} \bigtriangleup \widetilde{\X}$ is the following bi-pushout 

$$\xymatrix@R=3pc@C=1.5pc{\X \ar[rr]^{\ii_0^{\X}} \ar[d]_{\f} & \ar@{}[d]|<<<<<<<{\cong \; \Downarrow \; \delta_{\f}'} \ar@{}[d]|>>>>>>>{bi-p.o} & \widetilde{\X} \ar[d]^{\fa_{\f}'}  \\
\Y \ar[rr]_{\fb_{\f}'} && \Y \bigtriangleup \widetilde{\X} & }$$


\noindent and $\ff{Y} \bigtriangleup \widetilde{\X} \bigtriangledown \Y$ is the following bi-pushout

$$\xymatrix@R=3pc@C=1.5pc{\X \amalg \X \ar[rr]^{\ii^{\X}} \ar[d]_{\f \amalg \f}  & \ar@{}[d]|<<<<<<<{\cong \; \Downarrow \; \delta_{\f}'} \ar@{}[d]|>>>>>>>{bi-p.o} & \widetilde{\X} \ar[d]^{\fa_{\f}} & \\
\Y \amalg \Y \ar[rr]_{\fb_{\f}} && \Y \bigtriangleup \widetilde{\X} \bigtriangledown \Y \\
}$$ 


 Plus, if $\ff{f}$ is both a cofibration and a weak equivalence, so is $\ff{k}_{\f}$.

\vspace{1ex}

For a geometric intuition, one can think of $\ff{Y} \bigtriangleup \widetilde{\X}$ as the cylinder of $\ff{f}$ which is obtained by pasting a copy of $\Y$ on the bottom part of the cylinder of $\X$ via $\ff{f}$. And $\ff{Y} \bigtriangleup \widetilde{\X} \bigtriangledown \Y$ can be seen as the double cylinder of $\ff{f}$, which is obtained by pasting a copy of $\Y$ on the top and another one on the bottom of the cylinder of $\X$ via $\ff{f}$.
 \end{proposition}

\begin{proof}
By axiom 2-M2, we can choose a cylinder object for $\ff{X}$ with $\sigma_{\ff{X}}$ a fibration

  \begin{equation}
\xymatrix@C=3pc{\ff{X}\amalg \ff{X} \ar[d]_{\ff{i}^{\ff{X}}} \ar@/^2ex/[rd]^{\nabla^{\ff{X}}} & \\
              \widetilde{\ff{X}} \ar[r]_{\sigma^{\ff{X}}} \ar@{}[u]_{\;\; \cong \; \Uparrow \; \gamma^{\ff{X}}} 
              & \ff{X} }   
  \end{equation}
            
Now consider the following diagram

\begin{equation}\label{D24}
\xymatrix{&& \\ \X \amalg \X \ar[rr]^{\ii^{\X}} \ar@/^7ex/[rrrr]^{\nabla^{\X}} \ar[dd]_{\f \amalg \f} \ar@{}[drr]|{\cong \; \Downarrow \; \delta_{\f}} \ar@{}[ddrr]|{bi-p.o.} && \widetilde{\X} \ar@{}[u]|{\cong \; \Uparrow \; \gamma^{\X}} \ar[rr]^{\sigma^{\X}} \ar[dd]^{\fa_{\f}} \ar@{}[ddrr]|{\cong \; \Uparrow \; \beta_{\f}} && \X \ar[dd]^{\f} \\ && \\
\Y \amalg \Y \ar[rr]^{\fb_{\f}} \ar@/_16ex/[rrrr]^{\nabla^{\Y}} && \Y \bigtriangleup \widetilde{\X} \bigtriangledown \Y \ar@{}[d]|>{\cong \; \Downarrow \; \gamma_{\f}} \ar@{-->}[rd]_{\fk_{\f}} \ar@{-->}[rr]^{\nabla_{\f}} & \ar@{}[d]|{\cong \; \Downarrow \; \theta_{\f}} & \widetilde{\Y} \\
&&& \Y \ar@{-->}[ru]|{\sigma^{\Y}} }
\end{equation}

The upper left bi-pushout exists by axiom 2-M0b. $\nabla_{\f}$ and invertible 2-cells $ 
\vcenter{\xymatrix@C=-0pc{ \nabla_{\f} \dl & \dc{\beta_{\f}} & \dr \fa_{\f} \\ \f && \sigma^{\X} }}
$, $ 
\vcenter{\xymatrix@C=-0pc{ \nabla_{\f} && \fb_{\f} \\ & \nabla^{\Y} \cl{\gamma_{\f}} }}
$ such that $ 
\vcenter{\xymatrix@C=-0pc@R=1pc{ \nabla_{\f} \dl & \dc{\beta_{\f}} & \dr \fa_{\f} && \ii\tX \deq \\
\f \deq && \sigma\tX && \ii\tX \\
\f \dl & \dcr{\cong} && \dr \nabla\tX \cl{\gamma\tX} \\
\nabla\tY &&& \f\amalg\f}}
\vcenter{\xymatrix@C=-0pc{ \quad = \quad }}
\vcenter{\xymatrix@C=-0pc@R=1pc{\nabla_{\f} \deq && \fa_{\f} \dl & \dc{\delta_{\f}} & \dr \ii\tX \\
\nabla_{\f} && \fb_{\f} && \f\amalg\f \deq \\
& \nabla\tY \cl{\gamma_{\f}} &&& \f\amalg\f}}
$ are given by the universal property of $\ff{Y} \bigtriangleup \widetilde{\X} \bigtriangledown \Y$.


By axiom 2-M2, we can factorize $\nabla_{\ff{f}}$ with $\sigma_{\ff{Y}}$ a fibration that is also a weak equivalence and $\ff{k}_{\ff{f}}$ a cofibration. 

Consider the following cylinder object for $\Y$:

  \begin{equation}
\xymatrix@C=3pc{\ff{Y}\amalg \ff{Y} \ar[d]_{\ff{k}_{\ff{f}}\ff{b}_{\f}} \ar@/^2ex/[rd]^{\nabla^{\ff{Y}}} & \\
              \widetilde{\ff{Y}} \ar[r]_{\sigma^{\ff{Y}}} \ar@{}[u]_{\;\; \cong \; \Uparrow \; \gamma^{\Y}} 
              & \ff{Y} }   
  \end{equation}

\noindent where $\gamma^{\Y}= \gamma_{\ff{f}}\ff{b}_f \circ \theta_{\f}^{-1}$  
Then, $\ff{k}_{\ff{f}}$ is the wanted cofibration.

To construct $\ff{Y} \bigtriangleup \widetilde{\X}$ and $\ff{k}'_{\ff{f}}$, consider the following diagram:

$$\xymatrix@R=3pc{\X \ar[rr]^{\ii_0^{\X}} \ar[d]_{\f} & \ar@{}[d]|<<<<<<<{\cong \; \Downarrow \; \delta_{\f}'} \ar@{}[d]|>>>>>>>{bi-p.o} & \widetilde{\X} \ar[d]^{\fa_{\f}'} \ar@/^4ex/[drd]^{\fa_{\f}} \\
\Y \ar[rr]_{\fb_{\f}'} \ar@/_4ex/[drrr]_{\fb_{\f} \lambda_0} && \Y \bigtriangleup \widetilde{\X} \ar@{-->}[rd]|{\tilde{\fk}_f} \ar@{}[d]|{\cong \; \Downarrow \; \mu_{\f}} \ar@{}[r]|{\cong \; \Uparrow \; \alpha_{\f}}   & \\
            & & & \Y \bigtriangleup \widetilde{\X} \bigtriangledown \Y  }$$
            
The upper left bi-pushout exists by axiom 2-M0b and then, by universal property, we have $\tilde{\ff{k}}_{\ff{f}}$ and invertible 2-cells $ 
\vcenter{\xymatrix@C=-0pc{\tilde{\fk}_{\f} && \fa_{\f}' \\ & \fa_{\f} \cl{\alpha_{\f}} }}
$, $ 
\vcenter{\xymatrix@C=-0pc{\tilde{\fk}_{\f} \dl & \dc{\mu_{\f}} & \dr \fb_{\f}' \\ \fb_{\f} && \lambda_0 }}
$ such that $ 
\vcenter{\xymatrix@C=-0pc@R=1pc{\tilde{\fk}_{\f} && \fa_{\f}' && \ii_0^{\X} \deq \\ & \fa_{\f} \cl{\alpha_{\f}} \ardl & \dcr{(\delta_{\f})_0} && \ii_0^{\X} \ardr \\ \fb_{\f} && \lambda_0 &&& \f }}
\vcenter{\xymatrix@C=-0pc@R=1pc{ \quad = \quad }}
\vcenter{\xymatrix@C=-0pc@R=1pc{\tilde{\fk}_{\f} \deq && \fa_{\f}' \dl & \dc{\delta_{\f}'} & \dr \ii_0^{\X} \\
\tilde{\fk}_{\f} \dl & \dc{\mu_{\f}} & \dr \fb_{\f}' && \f \deq \\
\fb_{\f} && \lambda_0 && \f}}
$.

\vspace{1ex}

Take $\ff{k}'_{\ff{f}}=\ff{k}_{\ff{f}}\tilde{\ff{k}}_{\ff{f}}$.

Recall that the upper left square of \eqref{D24} is a bi-pushout. Then, it can be easily checked that the following diagram is also a bi-pushout:

$$\vcenter{\xymatrix@C=3pc@R=3pc{\ff{X} \ar[r]^{\ff{i}^{\ff{X}}_1} \ar[d]_{\ff{f}} & \widetilde{\X} \ar@{}[d]|{\cong \; \Downarrow \; (\delta_{\f})_1} \ar[r]^{\ff{a}'_{\f}} \ar[dr]_{\ff{a}_{\f}}^{\cong \; \Downarrow \; \alpha_{\f}} &   \ff{Y} \bigtriangleup \widetilde{\X} \ar[d]^{\tilde{\ff{k}}_{\ff{f}}} \\
            \ff{Y} \ar[r]_{\lambda_1} & \Y \amalg \Y \ar[r]_{\ff{b}_{\f}} & \ff{Y} \bigtriangleup \widetilde{\X} \bigtriangledown \Y }}
\vcenter{\xymatrix@C=1pc{=}}
\vcenter{\xymatrix@C=1.5pc@R=3pc{\ff{X} \ar[rr]^{\ff{a}'_{\f} \ff{i}^{\ff{X}}_1} \ar[d]_{\ff{f}} &\ar@{}[d]|{\quad \;\;\; \cong \; \Downarrow \; (\delta_{\f})_1 \circ \alpha_{\f} \ff{i}^{\X}_1}& \ff{Y} \bigtriangleup \widetilde{\X} \ar[d]^{\tilde{\ff{k}}_{\ff{f}}} \\
            \ff{Y} \ar[rr]_{\ff{b}_{\f} \lambda_1} && \ff{Y} \bigtriangleup \widetilde{\X} \bigtriangledown \Y }} $$

Thus, since $\ff{f}$ is a cofibration it follows from axiom 2-M3b that $\tilde{\ff{k}}_{\ff{f}}$ is also a cofibration. We can conclude then that $\ff{k}'_{\ff{f}}$  is a cofibration.

It only remains to check that $\ff{k}'_{\ff{f}}$ is a weak equivalence (then, in case $\ff{f}$ is a weak equivalence, $\tilde{\ff{k}}_{\ff{f}}$ is also a weak equivalence by axiom 2-M3b and so, by axiom 2-M5, $\ff{k}_{\ff{f}}$ is a weak equivalence):

Also by axiom 2-M5, it is enough to check that $\ff{b}'_{\f}$ is a weak equivalence because $\sigma_{\ff{Y}}$ is a weak equivalence and $\sigma_{\ff{Y}}\ff{k}'_{\ff{f}}\ff{b}'_{\f}\cong id_{\ff{Y}}$. To check that, consider the following diagram


\begin{equation}\label{D31}
\xymatrix@C=3pc@R=3pc{\ff{X} \ar@/^5ex/[rr]^{id_{\ff{X}}}_{\cong \; \Uparrow \; \gamma_0^{\X}} \ar[r]^{\ff{i}_0^{\X}} \ar[d]_{\ff{f} } \ar@{}[dr]|{\cong \; \Downarrow \; \delta_{\f}'}
            & \widetilde{\ff{X}} \ar[r]^{\sigma^{\ff{X}}} \ar[d]^{\fa_{\f}'} \ar@{}[dr]|{\cong \; \Uparrow \; \epsilon_{\f}}
            & \ff{X} \ar[d]^{\ff{f}}  \\
            \ff{Y} \ar[r]_{\fb_{\f}'} \ar@/_5ex/[rr]_{id_{\ff{Y}}}^{\cong \; \Downarrow \; \nu_{\f}}
            & \Y \bigtriangleup \widetilde{\X} \ar@{-->}[r]_{\ff{h}} 
            & \ff{Y}}
\end{equation}

By universal property of $\ff{Y} \bigtriangleup \widetilde{\X}$, there exist $\ff{h}$ and invertible 2-cells $ 
\vcenter{\xymatrix@C=-0pc{ \h \dl & \dc{\epsilon_{\f}} & \dr \fa_{\f}' \\ \f && \sigma^{\X} }}
$, $ 
\vcenter{\xymatrix@C=-0pc{ \h && \fb'_{\f} \\ & id_{\Y} \cl{\nu_{\f}} }}
$ such that $ 
\vcenter{\xymatrix@C=-0pc@R=1pc{\h \dl & \dc{\epsilon_{\f}} & \dr \fa_{\f}' && \ii_0^{\X} \deq \\
\f \deq && \sigma^{\X} && \ii_0^{\X} \\
\f \dl & \dcr{=} && \dr id_{\X} \cl{\gamma_0^{\X}} \\
id_{\Y} &&& \f}}
\vcenter{\xymatrix@C=-0pc{ \quad = \quad }}
\vcenter{\xymatrix@C=-0pc@R=1pc{ \h \deq && \fa_{\f}' \dl & \dc{\delta_{\f}'} & \dr \ii_0^{\X} \\
\h && \fb_{\f}' && \f \deq \\
& id_{\Y} \cl{\nu_{\f}} &&& \f}}
$.  

Again, by axiom 2-M5, it is enough to check that $\ff{h}$ is a weak equivalence and this is the case because $\sigma^{\X}$ is a weak equivalence and the right square of diagram \eqref{D31} is a bi-pushout since both the left square and the outside square are bi-pushouts.

\end{proof}

\begin{remark}\label{kretracto}
 If $\X \mr{\ff{f}} \Y$ is a retract of $\X' \mr{\ff{f'}} \Y'$  via $(\theta_0,\theta_1,\theta_m,\eta_0,\eta_1,\eta_m,\mu_0,\mu_1)$ and the constructions of the previous proposition are performed for both morphisms, then $\ff{k}_{\f}$ is a retract of $\ff{k}_{\f'}$ and $\ff{k}'_{\f}$ is a retract of $\ff{k}'_{\f'}$.
\end{remark}

\begin{proof}
We give a sketch of the proof, leaving the details to the reader.
First observe that $\nabla^{\X}$ is a retract of $\nabla^{\X'}$. Then, from \ref{factorizacionfuntorialretractos} plus \ref{cilindros funtoriales}, one can choose ``retract cylinder objects'' for $\X$ and $\X'$ as in the following diagram

$$ 
 \xymatrix{\X \amalg \X \ar[dd]_>>>>>{\nabla^{\X}} \ar[dr]^{\ii^{\X}} \ar[rr]^{\f \amalg \f } && \X' \amalg \X' \ar[dd]_>>>>>{\nabla^{\X'}} \ar[dr]^{\ii^{\X'}} \ar[rr]^{\f' \amalg \f' } && \X \amalg \X \ar[dd]_>>>>>{\nabla^{\X}} \ar[dr]^{\ii^{\X}} \\
 & \widetilde{\X} \ar[rr]|{\quad} \ar@{}[rr]^>>>>>>>>{\g} \ar[dl]^<<<<{\sigma^{X}} && \widetilde{\X'} \ar[dl]^<<<<{\sigma^{X'}} \ar[rr]|{\quad} \ar@{}[rr]^>>>>>>>>{\h} && \widetilde{\X} \ar[dl]^<<<<{\sigma^{X}} \\
 \X \ar[rr]_{\f} && \X' \ar[rr]_{\f'} && \X''}
 $$

By functoriality of the bi-pushout, one can construct a retraction \mbox{$\ff{Y} \bigtriangleup \widetilde{\X} \bigtriangledown \Y \mr{\tilde{\g}} \ff{Y'} \bigtriangleup \widetilde{\X'} \bigtriangledown \Y' \mr{\tilde{\h}} \ff{Y} \bigtriangleup \widetilde{\X} \bigtriangledown \Y $} and it can be checked that this is part of a retraction from $\nabla_{\ff{f}}$ to $\nabla_{\ff{f}'}$. Then, by \ref{factorizacionfuntorialretractos}, one can factorize $\nabla_{\ff{f}}$ and $\nabla_{\ff{f}'}$ in such way that $\ff{k}_{\ff{f}}$ is a retract of $\ff{k}_{\ff{f}'}$

$$ 
\xymatrix@C=3pc@R=3pc{\Y \bigtriangleup \widetilde{\X} \bigtriangledown \Y \ar[r]^{\tilde{\g}} \ar[d]_{\fk_{\f}} & \Y' \bigtriangleup \widetilde{\X'} \bigtriangledown \Y' \ar[r]^{\tilde{\h}} \ar[d]_{\fk_{\f'}} & \Y \bigtriangleup \widetilde{\X} \bigtriangledown \Y \ar[d]_{\fk_{\f}} \\
\widetilde{\Y} \ar[d]_{\sigma^{\Y}} \ar[r]^{\tilde{\tilde{\g}}} & \widetilde{\Y}' \ar[d]_{\sigma^{\Y'}} \ar[r]^{\tilde{\tilde{\h}}} & \widetilde{\Y} \ar[d]_{\sigma^{\Y}} \\
\Y \ar[r]_{\theta_1} & \Y' \ar[r]_{\eta_1} & \Y}
$$

Similar arguments can be used to prove that $\ff{k}'_{\ff{f}}$ is a retract of $\ff{k}'_{\ff{f}'}$. 
\end{proof}

\subsection{Some transfer properties}

\begin{lemma}\label{lemaadjuntos}
 Let $\ff{F}: \vcenter{\xymatrix{\cc{C} \ar@<.75ex>[r] \ar@<-.75ex>@{<-}[r] & \cc{D}}}:\ff{G}$ be pseudo-functors such that $\ff{F}\dashv_b \ff{G}$ via $\epsilon$, $\eta$ (see \ref{equivalencia de adjuntos con mi def}) and let $\ff{A}\mr{\ff{i}}\ff{X}\in \cc{C}$ and $\ff{Y}\mr{\ff{p}}\ff{B}\in \cc{D}$. Then the pair $(\ff{F}\ff{i},\ff{p})$ has the lifting property iff the pair $(\ff{i},\ff{G}\ff{p})$ does. 
\end{lemma}

\begin{proof}
$\Rightarrow)$ Suppose that we have a diagram of the form

$$\xymatrix@C=1.5pc@R=3pc{\ff{A} \ar[rr]^{\ff{a}} \ar[d]_{\ff{i}} &\ar@{}[d]|{\cong \; \Downarrow \; \gamma}& \ff{G}\ff{Y} \ar[d]^{\ff{G}\ff{p}} \\
            \ff{X} \ar[rr]_{\ff{b}} && \ff{G}\ff{B}}$$
            
By hypothesis, we have a filler $(\ff{f},\lambda,\rho)$ for the following diagram:

$$\xymatrix@C=3pc@R=3pc{\ff{F}\ff{A} \ar[r]^{\ff{F}\ff{a}} \ar[d]_{\ff{F}\ff{i}} & \ff{F}\ff{G}\ff{Y} \ar[r]^{\epsilon_{\ff{Y}}} \ar@{}[d]|{\cong \; \Downarrow \; \gamma'} & \ff{Y} \ar[d]^{\ff{p}} \\
            \ff{F}\ff{X} \ar[r]_{\ff{F}\ff{b}} & \ff{F}\ff{G}\ff{B} \ar[r]_{\epsilon_{\ff{B}}} & \ff{B}}$$

\noindent where $\gamma'= 
\vcenter{\xymatrix@C=-0.2pc{\p \dl & \dc{\eps_{\p}} & \dr \eps_{\Y} && \F\fa \deq \\
\eps_{\B} \deq && \F\G\p && \F\fa \\
\eps_{\B} \deq &&& \F(\G\p\fa) \cl{\alpha^{\F}_{\G\p,\fa}} \dcell{\F\gamma} \\
\eps_{\B} \deq &&& \F(\fb\ii) \op{\alpha^{\F}_{\fb,\ii}} \\
\eps_{\B} && \F\fb && \F\ii}}$.            

\noindent It is straightforward to check that $(\ff{G}\ff{f} \eta_{\ff{X}}, 
\vcenter{\xymatrix@C=-0.2pc@R=1pc{ && \fa \opunodos{= \quad} \\ 
& id_{\G\Y} \op{=} &&& \fa \deq \\
\G\eps_{\Y} \deq && \eta_{\G\Y} \dl & \dc{\eta_{\ff{a}}} & \dr \fa \\
\G\eps_{\Y} && \G\F\fa && \eta_{\A} \deq \\
& \!\!\!\! \G(\eps_{\Y} \F\fa) \!\!\!\! \cl{\alpha^{\G}_{\eps_{\Y},\F\fa}} \dcellb{\G\lambda}  &&& \eta_{\A} \deq \\
& \G(\f\F\ii) \op{\alpha^{\G}_{\f,\F\ii}} &&& \eta_{\A} \deq \\
\G\f \deq && \G\F\ii \dl & \dc{\eta_{\ff{i}}} & \dr \eta_{\A} \\
\G\f && \eta_{\X} && \F\ii}}
,  
\vcenter{\xymatrix@C=-0.2pc@R=1pc{ 
\G\p && \G\f && \eta_{\X} \deq \\
& \G(\p\f) \cl{\alpha^{\G}_{\p,\f}} \dcellb{\G\rho} &&& \eta_{\X} \deq \\
& \!\!\!\! \G(\eps_{\B}\F\fb) \!\!\!\! \op{\alpha^{\G}_{\eps_{\B},\F\fb}} &&& \eta_{\X} \deq \\
\G\eps_{\B} \deq && \G\F\fb \dl & \dc{\eta_{\ff{b}}} & \dr \eta_{\X} \\
\G\eps_{\B} && \eta_{\G\B} && \fb \deq \\
& id_{\G\B} \cl{=} &&& \fb \\
&& \fb \clunodos{=} }}
)$ is the filler that we were looking for.            

$\Leftarrow)$ The proof is similar to the previous one.
\end{proof}

\begin{lemma}\label{lemaadjuntos2}
 Let $\ff{F}: \vcenter{\xymatrix{\cc{C} \ar@<.75ex>[r] \ar@<-.75ex>@{<-}[r] & \cc{D}}}:\ff{G}$ be pseudo-functors between closed 2-bmodel 2-categories such that $\ff{F}\dashv_b \ff{G}$. The following properties hold:
 
 \begin{enumerate}
  \item $\ff{F}$ preserves cofibrations iff $\ff{G}$ preserves morphisms that are both fibrations and weak equivalences.  
  
  \item $\ff{F}$ preserves morphisms that are both cofibrations and weak equivalences iff $\ff{G}$ preserves fibrations.
 \end{enumerate}

\end{lemma}

\begin{proof}
 
\begin{enumerate}
 \item $\Rightarrow)$ Let $\ff{p}$ be a morphism in $\cc{D}$ that is both a fibration and a weak equivalence. By \ref{lema2.2}, we only need to check that $\ff{G}\ff{p}$ has the right lifting property with respect to all cofibrations. So let $\ff{i}$ be a cofibration in $\cc{C}$, then $\ff{F}\ff{i}$ is a cofibration in $\cc{D}$ and so it has the left lifting property with respect to all morphisms that are both fibrations and weak equivalence. Then the pair $(\ff{F}\ff{i},\ff{p})$ has the lifting property and, by \ref{lemaadjuntos}, so does the pair $(\ff{i},\ff{G}\ff{p})$.
 
 $\Leftarrow)$ The proof is similar.
 
 \item The proof is similar.
\end{enumerate}

\end{proof}

\begin{proposition}\label{pseudo-equivalencerespetaaxiomas}
Let $\cc{C}\mr{\ff{F}}\cc{D}$ be a 2-functor that is a retract pseudo-equivalence of \mbox{2-categories} (\ref{pequivalencia}). If $\cc{D}$ is a closed 2-bmodel 2-category and we define the corresponding structure in $\cc{C}$, then $\cc{C}$ is also a closed 2-bmodel 2-category. 
\end{proposition}

\begin{proof} Let $\ff{G}$ be a pseudo-quasi-inverse to $\ff{F}$ and $\ff{F}\ff{G}\Mr{\alpha}id_{\cc{D}}$ an equivalence as in \ref{pequivalencia} such that $\G\F=id_{\cc{C}}$.

2-M0b: The proof is straightforward.

2-M2: Let $\ff{f}\in \cc{C}$. Since $\cc{D}$ satisfies 2-M2, $\ff{F}\ff{f}$ can be factorized as $\ff{F}\ff{f}\cong \tilde{\ff{p}}\tilde{\ff{i}}$ where $\tilde{\ff{p}}$ is a fibration and $\tilde{\ff{i}}$ is both a cofibration and a weak equivalence. Let $\ff{p}=\ff{G}\tilde{\ff{p}}$ and $\ff{i}=\ff{G}\tilde{\ff{i}}$. Then, it can be easily checked that $\ff{F}\ff{p}=\ff{F}\ff{G}\tilde{\ff{p}}$ is a retract of $\tilde{\ff{p}}$. Then, by \ref{retractospreservan}, $\ff{F}\ff{p}$ is a fibration in $\cc{D}$ and so $\ff{p}$ is a fibration in $\cc{C}$. With a similar argument, one can check that $\ff{i}$ is both a cofibration and a weak equivalence in $\cc{C}$. Plus $\ff{f}=\ff{G}\ff{F}\ff{f}\cong \ff{G}\tilde{\ff{p}}\tilde{\ff{i}}\cong \ff{G}\tilde{\ff{p}}\ff{G}\tilde{\ff{i}}=\ff{p}\ff{i}$ as we wanted to prove. The case in which $\ff{p}$ is a weak equivalence is similar and we leave it to the reader.

2-M5w: Let $\ff{f}$ and $\ff{g}$ be two composable arrows in $\cc{C}$ such that two out of the three $\ff{f}$, $\ff{g}$ and $\ff{g}\ff{f}$ are weak equivalences. Then two out of the three $\ff{F}\ff{f}$, $\ff{F}\ff{g}$ and $\ff{F}\ff{g}\ff{f}=\ff{F}\ff{g}\ff{F}\ff{f}$ are weak equivalences and so is the third because 2-M5w is satisfied in $\cc{D}$. But this implies that $\ff{f}$, $\ff{g}$ and $\ff{g}\ff{f}$ are all weak equivalences as we wanted to prove.

2-M6a): $\Rightarrow)$ Let $\ff{p}$ be a fibration and $\ff{i}$ a cofibration which is also a weak equivalence in $\cc{C}$ and suppose that we have a diagram of the form:

$$\xymatrix@C=1.5pc@R=3pc{\ff{A} \ar[rr]^{\ff{a}} \ar[d]_{\ff{i}} &\ar@{}[d]|{\cong \; \Downarrow \; \gamma}& \ff{Y} \ar[d]^{\ff{p}} \\
            \ff{X} \ar[rr]_{\ff{b}} && \ff{B}}$$
            
\noindent Since $\ff{F}\ff{p}$ is a fibration and $\ff{F}\ff{i}$ is both a cofibration and a weak equivalence in $\cc{D}$, there exists a filler $(\ff{f}_0,\lambda_0,\rho_0)$ for the following diagram

$$\xymatrix@C=1.5pc@R=3pc{\ff{F}\ff{A} \ar[rr]^{\ff{F}\ff{a}} \ar[d]_{\ff{F}\ff{i}} &\ar@{}[d]|{\cong \; \Downarrow \; \ff{F}\gamma}& \ff{F}\ff{Y} \ar[d]^{\ff{F}\ff{p}} \\
            \ff{F}\ff{X} \ar[rr]_{\ff{F}\ff{b}} && \ff{F}\ff{B}}$$
                  
\noindent It is straightforward to check that $(\ff{G}\ff{f}_0, 
\vcenter{\xymatrix@C=-0.2pc@R=1.2pc{ & \fa \dcell{=} \\
& \G\F\fa \dcellb{\G\lambda_0} \\
& \G(\f_0\F\ii) \opbb{(\alpha^{\G}_{\f_0,\F\ii})\inv} \\
\G\f_0 \deq && \G\F\ii \dcell{=} \\
\G\f_0 && \ii}}
, 
\vcenter{\xymatrix@C=-0.2pc@R=1.2pc{ 
\p \dcell{=} && \G\f_0 \deq \\
\G\F\p && \G\f_0 \\
& \G(\F\p\f_0) \cl{\alpha^{\G}_{\F\p,\f_0}} \dcellb{\G\rho_0} \\
& \G\F\fb \dcell{=} \\
& \fb }}
)$ is the filler that we were looking for.    

$\Leftarrow)$ Let $\ff{Y}\mr{\ff{p}}\ff{B}\in \cc{C}$ such that it has the right lifting property with respect to all morphisms that are both a cofibration and a weak equivalence. To check that it is a fibration, we need to check that $\ff{F}\ff{p}$ is a fibration in $\cc{D}$ where we have axiom 2-M6a). So suppose that we have a morphism $\ff{i}\in \cc{D}$ which is both a cofibration and a weak equivalence and a diagram of the form:

$$\xymatrix@C=1.5pc@R=3pc{\ff{A} \ar[rr]^{\ff{a}} \ar[d]_{\ff{i}} &\ar@{}[d]|{\cong \; \Downarrow \; \gamma}& \ff{F}\ff{Y} \ar[d]^{\ff{F}\ff{p}} \\
            \ff{X} \ar[rr]_{\ff{b}} && \ff{F}\ff{B}}$$
            
\noindent Then we have a diagram in $\cc{C}$ as follows:

\begin{equation}\label{aplicandoG}
\xymatrix@C=1.5pc@R=3pc{\ff{G}\ff{A} \ar[rr]^{\ff{G}\ff{a}} \ar[d]_{\ff{G}\ff{i}} &\ar@{}[d]|{\cong \; \Downarrow \; \gamma'}& \ff{Y} \ar[d]^{\ff{p}} \\
            \ff{G}\ff{X} \ar[rr]_{\ff{G}\ff{b}} && \ff{B}}
\end{equation}  

\noindent where $\gamma'= 
\vcenter{\xymatrix@C=-0pc@R=1.2pc{ 
\p \dcell{=} && \G\fa \deq \\
\G\F\p && \G\fa \\
& \G(\F\p\fa) \cl{\alpha^{\G}_{\F\p,\fa}} \dcellb{\G\gamma} \\
& \G(\fb \ii) \opbymedio{(\alpha^{\G}_{\fb,\ii})\inv} \\
\G\fb && \G\ii }}
$. 

We are going to prove that $\ff{G}\ff{i}$ is both a cofibration and a weak equivalence: In order to do that, by \ref{lema2.2}, we only need to check that $\ff{F}\ff{G}\ff{i}$ has the left lifting property with respect to all fibrations. So suppose that we have a fibration $\ff{q}$ and a diagram of the form

$$\xymatrix@C=1.5pc@R=3pc{\ff{F}\ff{G}\ff{A} \ar[rr]^{\tilde{\ff{a}}} \ar[d]_{\ff{F}\ff{G}\ff{i}} &\ar@{}[d]|{\cong \; \Downarrow \; \tilde{\gamma}}& \tilde{\ff{Y}} \ar[d]^{\ff{q}} \\
            \ff{F}\ff{G}\ff{X} \ar[rr]_{\tilde{\ff{b}}} && \tilde{\ff{B}}}$$ 
            
\noindent Since $\ff{i}$ is both a cofibration and a weak equivalence and $\ff{q}$ is a fibration, the following diagram admits a filler $(\ff{f}_1,\lambda_1,\rho_1)$

$$\xymatrix@C=3pc@R=3pc{\ff{A} \ar[r]^{\overline{\alpha_{\ff{A}}}} \ar[d]_{\ff{i}} & \ff{F}\ff{G}\ff{A} \ar[r]^{\tilde{\ff{a}}} \ar@{}[d]|{\cong \; \Downarrow \; \tilde{\gamma'}} & \tilde{\ff{Y}} \ar[d]^{\ff{q}} \\
            \ff{X} \ar[r]_{\overline{\alpha_{\ff{X}}}} & \ff{F}\ff{G}\ff{X} \ar[r]_{\tilde{\ff{b}}} & \tilde{\ff{B}}}$$ 

\noindent where $\tilde{\gamma'}= 
\vcenter{\xymatrix@C=-0pc@R=1.2pc{ \q \dl && \dc{\tilde{\gamma}} && \dr \tilde{\fa} &&&& \overline{\alpha_{\A}} \deq \\
\tilde{\fb} \deq &&&& \F\G\ii \opdosdos{\cong} &&&& \overline{\alpha_{\A}} \deq \\
\tilde{\fb} \deq && \overline{\alpha_{\X}} \deq && \alpha_{\X} \dl & \dc{\alpha_i\inv} & \dr \F\G\ii && \overline{\alpha_{\A}} \deq \\
\tilde{\fb} \deq && \overline{\alpha_{\X}} \deq && \ii && \alpha_{\A} && \overline{\alpha_{\A}} \\
\tilde{\fb} && \overline{\alpha_{\X}} && && \ii \cldosdos{\cong} }}
$

It is straightforward to check that $(\ff{f}_1\alpha_{\ff{X}}, 
\vcenter{\xymatrix@C=-0.2pc@R=1pc{ && \tilde{\fa} \opdosdos{\cong} \\
\tilde{\fa} \dl & \dc{\lambda_1} & \dr \overline{\alpha_{\ff{A}}} && \alpha_{\A} \deq \\
\f_1 \deq && \ii \dl & \dc{\alpha_{\ii}} & \dr \alpha_{\A} \\
\f_1 && \alpha_{\X} && \F\G\ii}}
,  
\vcenter{\xymatrix@C=-0.2pc@R=1pc{ \q \dl & \dc{\rho_1} & \dr \f_1 && \alpha_{\X} \deq \\
\tilde{\fb} && \overline{\alpha_{\ff{X}}} && \alpha_{\X} \\
&& \tilde{\fb} \cldosdos{\cong} }}
)$ is the filler that we were looking for. 

Then, by hypothesis, diagram \eqref{aplicandoG} admits a filler $(\ff{f}_0,\lambda_0,\rho_0)$. 

\noindent It is straightforward to check that $$(\alpha_{\ff{F}\ff{Y}} \ff{F}\ff{f}_0 \overline{\alpha_{\ff{X}}}, 
\vcenter{\xymatrix@C=-0.2pc@R=1pc{&&&&& \fa \opcincocinco{\cong} \\
\fa \dl && \dc{\alpha_{\fa}} && \dr \alpha_{\A} &&&&&& \overline{\alpha_{\ff{A}}} \deq \\
\alpha_{\F\Y} \deq &&&& \F\G\fa \opdosdos{\F\lambda_0} &&&&&& \overline{\alpha_{\ff{A}}} \deq \\
\alpha_{\F\Y} \deq && \F\f_0 \deq &&&& \F\G\ii \opdosdos{\cong} &&&& \overline{\alpha_{\ff{A}}} \deq \\
\alpha_{\F\Y} \deq && \F\f_0 \deq && \overline{\alpha_{\ff{X}}} \deq && \alpha_{\X} \dl & \dc{\alpha_{\ii}\inv} & \dr \F\G\ii && \overline{\alpha_{\ff{A}}} \deq \\
\alpha_{\F\Y} \ardrr && \F\f_0 && \overline{\alpha_{\ff{X}}} & \dc{\cong} & \ii && \alpha_{\A} && \overline{\alpha_{\ff{A}}} \ardll \\
&& \alpha_{\F\Y} && \F\f_0 && \overline{\alpha_{\ff{X}}} && \ii }}
, 
\vcenter{\xymatrix@C=-0.2pc@R=1pc{ \F\p \dl & \dc{\alpha_{\F\p}} & \dr \alpha_{\F\Y} && \F\f_0 \deq && \overline{\alpha_{\ff{X}}} \deq \\
\alpha_{\F\B} \deq && \F\p && \F\f_0 && \overline{\alpha_{\ff{X}}} \deq \\
\alpha_{\F\B} \dl & \dcr{\alpha_{\fb}\inv} && \dr \F\G\fb \cl{\F\rho_0} &&& \overline{\alpha_{\ff{X}}} \deq \\
\fb &&& \alpha_{\X} &&& \overline{\alpha_{\ff{X}}} \\
&&& \fb \cltrestres{\cong} }}
)$$ is the filler that we were looking for.

2-M6b): The proof of this axiom is similar to the previous one and we leave it to the reader.

2-M6c): $\Rightarrow)$ Let $\ff{f} \in \cc{C}$ be a weak equivalence. Then $\ff{F}\ff{f}\in \cc{D}$ is a weak equivalence and therefore we can factorize it as $\ff{F}\ff{f}\cong \ff{u}\ff{v}$ where $\ff{u}$ has the right lifting property with respect to all cofibrations and $\ff{v}$ has the left lifting property with respect to all fibrations. Consider $\tilde{\ff{u}}=\ff{G}\ff{u}$ and $\tilde{\ff{v}}=\ff{G}\ff{v}$. Then $\ff{F}\tilde{\ff{u}}=\ff{F}\ff{G}\ff{u}$ is a retract of $\ff{u}$. Then, by \ref{retractospreservan}, $\tilde{\ff{u}}$ has the right lifting property with respect to all cofibrations.

%
%
%
%

By a similar argument, we can prove that $\tilde{\ff{v}}$ has the left lifting property with respect to all fibrations. 

So we factorized $\ff{f}$ as we wanted because $\ff{f}=\ff{G}\ff{F}\ff{f}\cong \ff{G}\ff{u}\ff{v} \cong \ff{G}\ff{u}\ff{G}\ff{v}=\tilde{\ff{u}}\tilde{\ff{v}}$.

$\Leftarrow)$ Let $\ff{f}\cong \ff{u}\ff{v} \in \cc{C}$ with $\ff{u}$ having the right lifting property with respect to all cofibrations and $\ff{v}$ having the left lifting property with respect to all fibrations. We want to check that $\ff{F}\ff{f}$ is a weak equivalence in $\cc{D}$: since $\ff{F}\ff{f} \cong \ff{F}\ff{u} \ff{F}\ff{v}$, we only need to check that $\ff{F}\ff{u}$ has the right lifting property with respect to all cofibrations and $\ff{F}\ff{v}$ has the left lifting property with respect to all fibrations. This can be checked by working as before. 
\end{proof}

\pagebreak
\begin{center}
{\bf Resumen en castellano de la secci\'on \ref{Definitions and basic lemmas}}
\end{center}

En esta secci\'on introducimos las nociones in\'editas de ``closed 2-model 2-category'' y ``closed 2-bmodel 2-category'' y enunciamos y demostramos algunos lemas y proposiciones que usaremos m\'as adelante. Nuestra noci\'on es m\'as fuerte que las ``fibration structures'' de Pronk (\cite{Pronk}) pues es una versi\'on 2-dimensional de los axiomas de Quillen completos para ``closed model categories''. Tambi\'en difiere en el hecho importante de que no asumimos la elecci\'on de una factorizaci\'on global privilegiada dada de forma pseudo-funtorial sino que estipulamos, como Quillen, solo la existencia de factorizaciones para cada flecha. 

La mayor\'ia de los resultados de esta secci\'on son generalizaciones al contexto de 2-categor\'ias de enunciados bien conocidos de la teor\'ia de ``closed model categories''.  
Para ver las definiciones y resultados en el caso 1-dimensional, se puede consultar por ejemplo \cite{Q1} o \cite{GJ}.

\pagebreak 

\section{Closed 2-bmodel structure for $\Pro{C}$}\label{2-modelos}

In this section, we give $\Pro{C}$ a closed 2-bmodel structure provided that $\cc{C}$ has one. This section is inspired in the proof given in \cite{EH} of the fact that $\ff{Pro}(\ff{C})$ is a closed model category in the 1-dimensional case. The proof in our context turned out to be more complicated due to the fact that diagrams doesn't strictly commute but only commute up to an invertible 2-cell. This is the reason why we were forced to work with pseudo-functors and pseudo-natural transformations even though objects and morphisms in $\Pro{C}$ are 2-functors and 2-natural transformations. We proceed in three steps. First, in \ref{closed 2-model structure for CJ} we define a closed 2-bmodel structure for the 2-category $\pCJ$ (see \ref{ccHom}) out of a closed 2-bmodel structure for $\cc{C}$, where $\ff{J}$ is a cofinite and filtered poset with a unique initial object. Second, in \ref{closed 2-bmodel structure in Prop} we use the closed 2-bmodel structure in $\pCJ$ to define such an structure in the 2-category 
$\Prop{C}$. Finally, we transfer this structure into $\Pro{C}$ using that this 2-category is retract pseudo-equivalent to $\Prop{C}$ \mbox{(see \ref{proppseudoeqapro}).}   


\subsection{Closed 2-bmodel structure for $\pCJ$}\label{closed 2-model structure for CJ}

The aim of this subsection is to prove that given a closed 2-bmodel 2-category $\cc{C}$ and a cofinite and filtered poset $\ff{J}$ with a unique initial object, the 2-category $p\cc{H}om_p(\ff{J}^{op},\cc{C})$ (see \ref{ccHom}) is a closed 2-bmodel 2-category. The proof is inspired in the 1-dimensional case treated in \cite{EH}. For the 2-categorical setting, things become more complicated. So is that we were forced to work with pseudo-functors and pseudo-natural transformations instead of 2-functors and 2-natural transformations because of the non-strict commutativity of diagrams. One would think (and we did for a while) that 2-functors and pseudo-natural transformations should be enough but they are not. The reason for taking pseudo-functors evidences itself in the proof of axiom 2-M2 where $\ff{Z}$ turns out to be a pseudo-functor that is not necessarily a 2-functor even if all the others are.  

All along this subsection, $\ff{J}$ will be a cofinite and filtered poset with a unique initial object $0$ and $\cc{C}$ will be a closed 2-bmodel 2-category. 
We comment that in \cite{EH} $\ff{J}$ is not supposed to have a unique initial object, which for us is an essential requirement, also in the 1-dimensional case.

\begin{sinnadastandard} {\bf Notation.} Since there are at most one morphism between any pair of objects of $\ff{J}$, we will write $\alpha_{k,l,j}$ instead of $\alpha_{k\leq l,l\leq j}$ (see \ref{pseudo-functor}). Also, we will use the subindex notation for the evaluation of 2-functors.
\end{sinnadastandard}

\begin{definition}\label{estructuraenCJ}
 
We define fibrations, cofibrations and weak equivalences in $\pCJ$ as follows:

\begin{itemize}
 \item A morphism $\ff{f}\in \pCJ$ is a cofibration if the morphism $\ff{f}_j$ is a cofibration in $\cc{C}\ \forall \ j\in \ff{J}$. We say ``pointwise cofibration''.
 \item A morphism $\ff{f}\in \pCJ$ is a weak equivalence if the morphism $\ff{f}_j$ is a weak equivalence in $\cc{C}\ \forall \ j\in \ff{J}$. We say ``pointwise weak equivalence''.
 \item A morphism $\ff{f}\in \pCJ$ is a fibration if it has the right lifting property with respect to all the morphisms that are both cofibrations and weak equivalences.
 \end{itemize}

\end{definition}

\begin{lemma}\label{lema0.2}
\comw{A}
\begin{enumerate}
 \item The 2-functor constant diagram from $\cc{C}$ to $\pCJ$ preserves cofibrations, fibrations and weak equivalences. 
 \item The pseudo-functor inverse bi-limit from $\pCJ$ to $\cc{C}$ preserves fibrations and morphisms that are both fibrations and weak equivalences.
\end{enumerate}

\end{lemma}

\begin{proof}

\begin{enumerate}
 \item
It is clear that the 2-functor constant diagram preserves cofibrations and weak equivalences. We will check now that it also preserves fibrations:

Let $\ff{C}\mr{\ff{p}}\ff{D}\in \cc{C}$ be a fibration. We need to check that if we embed $\ff{p}$ in $\pCJ$, then it has the right lifting property with respect to all morphisms that are both cofibrations and weak equivalences. So let \mbox{$\ff{A}\mr{\ff{i}}\ff{X}\in \pCJ$} be a cofibration which is also a weak equivalence and suppose that we have the following situation:

$$\xymatrix@C=1.5pc@R=3pc{\ff{A} \ar[rr]^{\ff{a}} \ar[d]_{\ff{i}} &\ar@{}[d]|{\cong \; \Downarrow \; \gamma}& \ff{C} \ar[d]^{\ff{p}} \\
            \ff{X} \ar[rr]_{\ff{b}} && \ff{D}}$$
            
\noindent We are going to define the filler $(\ff{f},\lambda,\rho)$:

\noindent $\cc{C}$ is a closed 2-bmodel 2-category, so, since $\ff{p}$ is a fibration and $\ff{i}_0$ is both a cofibration and a weak equivalence, there exists a filler $(\ff{f}_0,\lambda_0,\rho_0)$ for the following diagram

\begin{equation}\label{sub0}
\xymatrix@C=1.5pc@R=3pc{\ff{A}_0 \ar[rr]^{\ff{a}_0} \ar[d]_{\ff{i}_0} &\ar@{}[d]|{\cong \; \Downarrow \; \gamma_0}& \ff{C} \ar[d]^{\ff{p}} \\
            \ff{X}_0 \ar[rr]_{\ff{b}_0} && \ff{D}} 
\end{equation}

For $j\neq 0\in \ff{J}$, take $\ff{f}_j=\ff{f}_0 \ff{X}_{0< j}$, $\ff{f}_{0< j}=id_{f_j}$, $\ff{f}_{id_j}=\ff{f}_j \alpha^\ff{X}_j$, 
\mbox{$\ff{f}_{k<j}=\vcenter{\xymatrix@C=-.2pc@R=1pc{&& \ff{f}_j \opdosuno{=} &&&& \\
\ff{f}_0 \deq &&& \ff{X}_{0<j} \opbb{\alpha^\ff{X}_{0,k,j}{}^{-1}} && \\
\ff{f}_0 && \ff{X}_{0<k} && \ff{X}_{k<j} \deq \\
& \; \ff{f}_k \cl{=} \; &&& \ff{X}_{k<j}}}$,}
\mbox{$\lambda_j= \vcenter{\xymatrix@C=-.2pc@R=1pc{&& \ff{a}_j \opunodos{\ff{a}_{0<j} \quad} \\
& \ff{a}_0 \op{\lambda_0} &&& \ff{A}_{0<j} \deq \\
\ff{f}_0 \deq && \ff{i}_0 \dl & \dc{\ff{i}_{0<j}^{-1}} & \ff{A}_{0<j} \dr \\
\ff{f}_0 && \ff{X}_{0<j} && \ff{i}_j \deq \\
& \ff{f}_j \cl{=} &&& \ff{i}_j }}$} and 
\mbox{$\rho_j=\vcenter{\xymatrix@C=-.2pc@R=1pc{\ff{p} \deq &&& \ff{f}_j \op{=} \\
\ff{p} && \ff{f}_0 && \ff{X}_{0<j} \deq \\
& \ff{b}_0 \cl{\rho_0} &&& \ff{X}_{0<j} \\
&& \ff{b}_j \clunodos{\; \ff{b}_{0<j}^{-1}} }}$.}
                                   
\noindent Let's check that $\ff{f}$ defined this way is a pseudo-natural transformation: PNO is straightforward and PN2 is vacuous since $\ff{J}$ doesn't have any non trivial 2-cells. For axiom PN1, we need to verify that $\forall \ k<l<j \in \ff{J}$ the following equality holds:


$$\vcenter{\xymatrix@C=-0pc{ id_{\C} \deq && id_{\C} \dl & \dc{\f_{l < j}} & \dr \f_j \\
				id_{\C} \dl & \dc{\f_{k < l}} & \dr \f_l && X_{l < j} \deq \\
				\f_k \deq && \X_{k < l} && \X_{l < j} \\
				\f_k &&& \X_{k < j} \cl{\alpha^\X_{k,l,j}}  }}
\vcenter{\xymatrix@C=-.5pc{ = }} 
\vcenter{\xymatrix@C=-0pc{ id_{\C} && id_{\C} && \f_j \deq \\
			      & id_{\C} \cl{=}  \dl & \dcr{\f_{k < j}} && \dr \f_j \\ 
			      & \f_k &&& \X_{k < j} }}$$

But, if $0<k<l<j$,


$$\vcenter{\xymatrix@C=-0pc{ id_{\C} \deq && id_{\C} \dl & \dc{\f_{l < j}} & \dr \f_j \\
				id_{\C} \dl & \dc{\f_{k < l}} & \dr \f_l && X_{l < j} \deq \\
				\f_k \deq && \X_{k < l} && \X_{l < j} \\
				\f_k &&& \X_{k < j} \cl{\alpha^\X_{k,l,j}}  }}
\vcenter{\xymatrix@C=-.5pc{ = }} 
\vcenter{\xymatrix@C=-0pc{ id_{\C} \deq && id_{\C} \dl & \dcr{=} && \dr \f_j \\
id_{\C} \deq && \f_0 \deq &&& \X_{0<j} \opb{\alpha^\X_{0,l,j}{}^{-1}}   \\
id_{\C} \deq && \f_0 && \X_{0<l} && \X_{l < j} \deq \\
id_{\C} \dl & \dcr{\f_{k<l}} && \dr \f_l \cl{=} &&& \X_{l < j} \deq \\
\f_k \deq &&& \X_{k<l} \ar@{-}[dr] & \dc{\alpha^\X_{k,l,j}} && \ar@{-}[dll] \X_{l < j}  \\
\f_k &&&& \X_{k<j}  }}
\vcenter{\xymatrix@C=-0pc{ = }}$$
\vspace{1ex}
$$\vcenter{\xymatrix@C=-0pc{ = }}
\vcenter{\xymatrix@C=-0pc{ id_{\C} \deq && id_{\C} \dl & \dcr{=} && \dr \f_j \\
id_{\C} \deq && \f_0 \deq &&& \X_{0<j}  \opb{\alpha^\X_{0,k,j}{}^{-1}} \\
id_{\C} \ar@{-}[drr] && \f_0 \dc{=}  &&  \ar@{-}[dll] \X_{0<k} && \X_{k<j} \deq \\
&& \f_k &&&& \X_{k<j}  }}
\vcenter{\xymatrix@C=-0pc{ =  }}
\vcenter{\xymatrix@C=-0pc{id_{\C} && id_{\C} && \f_j \deq \\
			      & id_{\C} \cl{=}  \dl & \dcr{\f_{k < j}} && \dr \f_j \\ 
			      & \f_k &&& \X_{k < j} }}$$

\noindent where the second equality is satisfied by inductive hypothesis and the third one is due to the fact that $\ff{X}$ is a pseudo-functor (cases where $k=0$ or there are equalities are straightforward).

To check that $\lambda$ is a modification, we need to verify that $\forall \ k<j \in \ff{J}$ the following equality holds: 


$$\vcenter{\xymatrix@C=-0pc{ & id_{\C} \dl & \dcr{\fa_{k<j}} && \dr \fa_j \\ 
& \fa_k \op{\lambda_k} &&& \A_{k<j} \deq \\
\f_k && \ii_k && \A_{k<j}    }}
\vcenter{\xymatrix@C=-0pc{ =  }}
\vcenter{\xymatrix@C=-0pc{  id_{\C} \deq  &&& \fa_j \op{\lambda_j} \\
id_{\C}   \dl & \dc{\f_{k<j}} & \dr \f_j && \ii_j \deq \\
\f_k \deq && \X_{k<j} \dl & \dc{\ii_{k<j}} & \dr \ii_j \\
\f_k && \ii_k && \A_{k<j}  }}$$

But


$$\vcenter{\xymatrix@C=-.5pc{ & id_{\C} \dl & \dcr{\fa_{k<j}} && \dr \fa_j \\ 
& \fa_k \op{\lambda_k} &&& \A_{k<j} \deq \\
\f_k && \ii_k && \A_{k<j}    }}
\vcenter{\xymatrix@C=-.5pc{\quad = \quad }}
\vcenter{\xymatrix@C=-.5pc{ &&& id_{\C} \dl & \dcr{\fa_{k<j}} && \dr \fa_j \\
&&& \fa_k \ar@{-}[dll] \ar@{-}[dr] \ar@{}[d]|{\fa_{0<k}} &&& \A_{k<j} \deq \\
& \fa_0 \op{\lambda_0} &&& \A_{0<k} \deq && \A_{k<j} \deq \\
\f_0 \deq && \ii_0 \dl & \dc{\ii^{-1}_{0<k}} & \dr \A_{0<k} && \A_{k<j} \deq \\
\f_0 && \X_{0<k} && \ii_k \deq && \A_{k<j} \deq \\
& \f_k \cl{=} &&& \ii_k && \A_{k<j} }}
\vcenter{\xymatrix@C=-.5pc{\quad = \quad }}
\vcenter{\xymatrix@C=-.5pc{ & id_{\C} \dl && \dc{\fa_{0<j}} && \dr \fa_j \\
& \fa_0 \deq &&&& \A_{0<j} \opbb{\alpha^\A_{0,k,j}{}^{-1}} \\
& \fa_0 \op{\lambda_0} &&& \A_{0<k} \deq && \A_{k<j} \deq \\
\f_0 \deq && \ii_0 \dl & \dc{\ii^{-1}_{0<k}} & \dr \A_{0<k} && \A_{k<j} \deq \\
\f_0 && \X_{0<k} && \ii_k \deq && \A_{k<j} \deq \\
& \f_k \cl{=} &&& \ii_k && \A_{k<j} }}
\vcenter{\xymatrix@C=-.5pc{ =  }}$$

$$\vcenter{\xymatrix@C=-.5pc{ =  }}
\vcenter{\xymatrix@C=-.5pc{ & id_{\C} \dl && \dc{\fa_{0<j}} && \dr \fa_j \\
& \fa_0 \op{\lambda_0} &&&& \A_{0<j} \deq \\
\f_0 \deq && \ii_0 \deq &&& \A_{0<j} \opbb{\alpha^\A_{0,k,j}{}^{-1}} \\
\f_0 \deq && \ii_0 \dl & \dc{\ii^{-1}_{0<k}} & \dr \A_{0<k} && \A_{k<j} \deq \\
\f_0 && \X_{0<k} && \ii_k \deq && \A_{k<j} \deq \\
& \f_k \cl{=} &&& \ii_k && \A_{k<j} }}
\vcenter{\xymatrix@C=-.5pc{ \quad = \quad }}
\vcenter{\xymatrix@C=-.5pc{ && id_{\C} \dl && \dc{\fa_{0<j}} && \dr \fa_j \\
&& \fa_0 \ar@{-}[dll] \dc{\lambda_0} \ar@{-}[dr] &&&& \A_{0<j} \deq \\
\f_0 \deq &&& \ii_0 \dl & \dcr{\ii^{-1}_{0<j}} && \dr \A_{0<j} \\
\f_0 \deq &&& \X_{0<j} \opbb{\alpha^\X_{0,k,j}{}^{-1}} &&& \ii_j \deq \\
\f_0 \deq && \X_{0<k} \deq &&  \X_{k<j} \dl & \dc{\ii_{k<j}} & \dr \ii_j \\
\f_0  && \X_{0<k} && \ii_k \deq && \A_{k<j} \deq \\
& \f_k \cl{=} &&& \ii_k && \A_{k<j}  }}
\vcenter{\xymatrix@C=-.5pc{ =  }}$$

$$\vcenter{\xymatrix@C=-.5pc{ =  }}
\vcenter{\xymatrix@C=-.5pc{ & id_{\C} \dl & \dcr{\fa_{0<j}} && \dr \fa_j \\
& \fa_0 \op{\lambda_0} &&& \A_{0<j} \deq \\
\f_0 \deq && \ii_0 \dl & \dc{\ii^{-1}_{0<j}}   & \dr \A_{0<j} \\
\f_0 \dl & \dc{\f_{k<j}} & \dr \X_{0<j} && \ii_j \deq \\
\f_k \deq && \X_{k<j} \dl & \dc{\ii_{k<j}} & \dr \ii_j \\
\f_k && \ii_k && \A_{k<j} }}
\vcenter{\xymatrix@C=-.5pc{ \quad = \quad }}
\vcenter{\xymatrix@C=-.5pc{  id_{\C} \deq  &&& \fa_j \op{\lambda_j} \\
id_{\C}   \dl & \dc{\f_{k<j}} & \dr \f_j && \ii_j \deq \\
\f_k \deq && \X_{k<j} \dl & \dc{\ii_{k<j}} & \dr \ii_j \\
\f_k && \ii_k && \A_{k<j}  }}$$

\noindent where the second equality is due to the fact that $\ff{a}$ is a pseudo-natural transformation, the third one is due to elevators calculus, the fourth one is due to the fact that $\ff{i}$ is a pseudo-natural transformation, the fifth one is due to elevators calculus plus the definition of $\ff{f}_{k<j}$ and the last one is due to the definition of $\lambda_j$. Again, simpler cases are omitted.

To check that $\rho$ is a modification, we need to verify that $\forall \ k<j \in \ff{J}$ the following equality holds: 


$$\vcenter{\xymatrix@C=.5pc{ \p \deq &&& \f_j \op{\f\kj} \\
\p && \f_k && \X\kj \deq \\
& \fb_k \cl{\rho_k} &&& \X\kj }}
\vcenter{\xymatrix@C=.5pc{\quad = \quad }}
\vcenter{\xymatrix@C=.5pc{ \p && \f_j \\
& \fb_j \cl{\rho_j} \op{\fb\kj} \\
\fb_k && \X\kj }}$$

But 


$$\vcenter{\xymatrix@C=-0pc{ \p \deq &&& \f_j \op{\; \f\kj} \\
\p && \f_k && \X\kj \deq \\
& \fb_k \cl{\rho_k} &&& \X\kj }}
\vcenter{\xymatrix@C=-.5pc{ \quad =  \quad }}
\vcenter{\xymatrix@C=-0pc{ \p \deq &&& \f_j \ar@{-}[dl] \ar@{-}[drr] \ar@{}[dr]|{=} \\
\p \deq && \f_0 \deq &&& \X\oj \opbb{\alpha^\X_{0,k,j}{}^{-1}} \\
\p && \f_0 && \X\ok \deq && \X\kj \deq \\
& \fb_0 \cl{\rho_0} \ar@{-}[dr] & \ar@{}[d]|{\; \fb^{-1}\ok} && \X\ok \ar@{-}[dll] && \X\kj \deq \\
&& \fb_k &&&& \X\kj }}
\vcenter{\xymatrix@C=-.5pc{ =  }}$$

$$\vcenter{\xymatrix@C=-0pc{ =  }}
\vcenter{\xymatrix@C=-0pc{ \p \deq &&& \f_j \op{=} \\
\p && \f_0 && \X\oj \deq \\
& \fb_0 \deq \cl{\rho_0} &&& \X\oj \ar@{-}[dl] \ar@{}[d]|{\alpha^\X_{0,k,j}{}^{-1}} \ar@{-}[drr]  \\
& \fb_0 && \X\ok &&& \X\kj \deq \\
&& \fb_k \cl{\fb^{-1}\ok} &&&& \X\kj }}
\vcenter{\xymatrix@C=-0pc{ =  }}
\vcenter{\xymatrix@C=-0pc{ \p \deq &&& \f_j \op{=} \\
\p && \f_0 && \X\oj \deq \\
& \fb_0 \cl{\rho_0} \ar@{-}[dr] & \dc{\; \fb^{-1}\oj} && \ar@{-}[dll] \X\oj \\
&& \fb_j \ar@{-}[dl] \ar@{-}[drr] \dc{\; \fb\kj} \\
& \fb_k &&& \X\kj }}
\vcenter{\xymatrix@C=-0pc{ =  }}
\vcenter{\xymatrix@C=-0pc{ \p && \f_j \\
& \fb_j \cl{\rho_j} \op{\ \fb\kj} \\
\fb_k && \X\kj }}$$

\noindent where the second equality is due to elevators calculus and the third one is due to the fact that $\ff{b}$ is a pseudo-natural transformation. Simpler cases are omitted.

Finally, let's check that $(\ff{f},\lambda,\rho)$ is the filler that we were looking for:


$$\vcenter{\xymatrix@C=-.5pc{ \p \deq &&& \fa_j \op{\lambda_j} \\
\p && \f_j && \ii_j \deq \\
& \fb_j \clmediob{\rho_j} &&& \ii_j}}
\vcenter{\xymatrix@C=-.5pc{ =  }}
\vcenter{\xymatrix@C=-.5pc{ \p \deq &&&& \fa_j \opunodos{\; \ \! \fa\oj \quad} \\
\p \deq &&& \fa_0 \op{\lambda_0} &&& \A\oj \deq \\
\p \deq && \f_0 \deq && \ii_0 \dl & \dc{\ii^{-1}\oj} & \dr \A\oj \\
\p && \f_0 && \X\oj \deq && \ii_j \deq \\
& \fb_0 \clmediob{\rho_0} & && \X\oj && \ii_j \deq \\
&& \fb_j \clunodos{\ \ \fb^{-1}\oj} &&&& \ii_j }}
\vcenter{\xymatrix@C=-.5pc{ =  }}
\vcenter{\xymatrix@C=-.5pc{ \p \deq &&&& \fa_j \fa_j \opunodos{\; \ \! \fa\oj \quad} \\
\p \deq &&& \fa_0 \op{\lambda_0} &&& \A\oj \deq \\
\p && \f_0 && \ii_0 \deq && \A\oj \deq \\
& \fb_0 \clmediob{\rho_0} \deq & && \ii_0 \dl & \dc{\ii^{-1}\oj} & \dr \A\oj  \\
& \fb_0 & && \X\oj && \ii_j \deq \\
&& \fb_j \clunodos{\ \ \fb^{-1}\oj} &&&& \ii_j   }}
\vcenter{\xymatrix@C=-.5pc{ =  }}
\vcenter{\xymatrix@C=-.5pc{ \p \deq &&& \fa_j \opb{\fa\oj} \\
\p \dl & \dc{\gamma_0} & \dr \fa_0 && \A\oj \deq \\
\fb_0 \deq && \ii_0 \dl & \dc{\ii^{-1}\oj} & \dr \A\oj \\
\fb_0 && \X\oj && \ii_j \deq \\
& \fb_j \clb{\fb^{-1}\oj} &&& \ii_j   }}
\vcenter{\xymatrix@C=-.5pc{ =  }}
\vcenter{\xymatrix@C=-.5pc{ \p \dl & \dc{\gamma_j} & \dr \fa_j \\ \fb_j && \ii_j   }}$$

\noindent The second equality is due to elevators calculus, the third one is due to the fact that $(\ff{f}_0,\lambda_0,\rho_0)$ is a filler for diagram \eqref{sub0} and the last one is due to the fact that $\gamma$ is a modification. 

\item By \ref{equivalencia de adjuntos fiore}, it is straightforward that the 2-functor constant diagram is left bi-adjoint to the pseudo-functor inverse bi-limit. Then, by \ref{lemaadjuntos2} and the previous item, we have what we wanted.

\end{enumerate}
\end{proof}

The following characterization of fibrations, similar but stronger than pointwise fibrations, is key to manipulate fibrations and prove that \ref{estructuraenCJ} determines a closed 2-bmolel structure (Theorem \ref{CJdeclosed 2-modelos}).

\begin{lemma}\label{lema1.2}
A morphism $\ff{Y}\mr{\ff{p}}\ff{B}\in \pCJ$ is a fibration iff $\forall \ j\in \ff{J}$ the morphism $\ff{q}_j$ of the following diagram is a fibration in $\cc{C}$:

$$\xymatrix@C=3pc{\ff{Y}_j   \ar@/^4ex/[rrd]^{\ff{a}_{\ff{Y}}^j} \ar@{-->}[rd]|{\comw{M^M} \ff{q}_j \comw{M^M} } \ar@/_4ex/[rdd]_{\ff{p}_j} & \\
            \ar@{}[r]|{\cong \; \Downarrow \; \beta_1^j}  & \ff{P}_j \ar@{}[u]|{\cong \; \Uparrow \; \beta_0^j} \ar[r]^{\pi_0^j} \ar[d]_{\pi_1^j} \ar@{}[rd]^{bipb \quad\quad\quad\quad\quad\quad\quad\quad\quad\quad}_{\quad\quad\quad\quad\quad\quad\quad\quad\quad\quad \cong \; \Downarrow \; \alpha_j } & \biLim{k<j}{\ff{Y}_k} \ar[d]^{\biLim{k<j}{\ff{p}_k}}\\
            & \ff{B}_j \ar[r]_{\ff{a}_{\ff{B}}^j} & \biLim{k<j}{\ff{B}_k}}$$
            
\noindent where $\biLim{k<j}{\ff{p}_k}$ is induced by the pseudo-cone $\{\ff{p}_k \pi_k\}_{k<j},\ $ \mbox{$\Bigg\{ 
\vcenter{\xymatrix@C=-0.2pc@R=1pc{ \B\kl \dl & \dc{\p\kl} & \dr \p_l && \pi_l \deq \\   \p_k \deq && \Y\kl && \pi_l \\   \p_k &&& \pi_k \clb{\pi\kl \ } }}
\Bigg\}_{k<l<j}\cup \Bigg\{ 
\vcenter{\xymatrix@C=-0.2pc@R=1pc{ \dl & \B_{id_k} \dc{(\alpha_k\tB)\inv} & \drbis & \p_k \deq && \pi_k \deq \\ & id_{\B_k} \ardr && \p_k \dc{=} && \pi_k \ardl \\ & & \p_k && \pi_k }}
\Bigg\}_{k<j}$,} 
$\ff{a}_{\ff{Y}}^j$ is induced by the pseudo-cone $\{\ff{Y}_{k<j}\}_{k<j},\ \Bigg\{ 
\vcenter{\xymatrix@C=-0.2pc@R=1pc{ \Y\kl && \Y\lj \\ & \Y\kj \cl{\Y_{k,l,j}} }}
\Bigg\}_{k<l<j}\cup \Bigg\{ 
\vcenter{\xymatrix@C=-0.2pc@R=1pc{\dl & \Y_{id_k} \dc{(\alpha_k\tY)\inv} & \dr && \Y\kj \deq \\  & id_{\Y_k} &&& \Y\kj \\ &  && \Y\kj \cldosuno{=} }}
\Bigg\}_{k<j}$ , 
$\ff{a}_{\ff{B}}^j$ is induced by the pseudo-cone $\{\ff{B}_{k<j}\}_{k<j},\ \Bigg\{ 
\vcenter{\xymatrix@C=-0.2pc@R=1pc{ \B\kl && \B\lj \\ & \B\kj \cl{\B_{k,l,j}} }}
\Bigg\}_{k<l<j}\cup \Bigg\{ 
\vcenter{\xymatrix@C=-0.2pc@R=1pc{\dl & \B_{id_k} \dc{(\alpha_k\tB)\inv} & \dr && \B\kj \deq \\  & id_{\B_k} &&& \B\kj \\  & && \B\kj \cldosuno{=} }}
\Bigg\}_{k<j}$ 
and so we have:

\begin{enumerate}
 
 \item[a)] invertible 2-cells $\pi_k \biLim{k<j}{\ff{p}_k} \Mr{\mu_k} \ff{p}_k \pi_k \ \forall \ k<j$ such that 
 $$\vcenter{\xymatrix@C=-0pc{ \B\kl && \pi_l && \biLim{k<j}{\p_k} \deq \\
 & \pi_k \cl{\pi\kl} \dl & \dcr{\mu_k} && \dr \biLim{k<j}{\p_k} \\
 & \p_k &&& \pi_k }}
\vcenter{\xymatrix@C=-0pc{ \quad = \quad }}
\vcenter{\xymatrix@C=-0pc{ \B\kl \deq && \pi_l \dl & \dc{\mu_l} & \dr \biLim{k<j}{\p_k} \\
\B\kl \dl & \dc{\p\kl} & \dr \p_l && \pi_l \deq \\
\p_k \deq && \Y\kl && \pi_l \\
\p_k &&& \pi_k \cl{\pi\kl} }}
\vcenter{\xymatrix@C=-0pc{\quad \forall \ k<l<j}}  $$
 
  \item[b)] invertible 2-cells $\pi_k \ff{a}_{\ff{Y}}^j \Mr{\eta_k} \ff{Y}_{k<j} \ \forall \ k<j$ such that 
 $$\vcenter{\xymatrix@C=-0pc{ \Y\kl && \pi_l && \fa_{\Y}^j \deq \\
 & \pi_k \cl{\pi\kl}  &&& \fa_{\Y}^j \\
 && \Y\kj \clunodos{\eta_k} }}
\vcenter{\xymatrix@C=-0pc{ \quad = \quad }}
\vcenter{\xymatrix@C=-0pc{ \Y\kl \deq && \pi_l && \fa_{\Y}^j \\
\Y\kl &&& \Y\lj \cl{\eta_l} \\
&& \Y\kj \cldosuno{\alpha\tY_{k,l,j}} }}
\vcenter{\xymatrix@C=-0pc{\quad \forall \ k<l<j}}  $$
 
 \item[c)] invertible 2-cells $\pi_k \ff{a}_{\ff{B}}^j \Mr{\epsilon_k} \ff{B}_{k<j} \ \forall \ k<j$ such that 
 $$\vcenter{\xymatrix@C=-0pc{ \B\kl && \pi_l && \fa_{\B}^j \deq \\
 & \pi_k \cl{\pi\kl}  &&& \fa_{\B}^j \\
 && \B\kj \clunodos{\eps_k} }}
\vcenter{\xymatrix@C=-0pc{ \quad = \quad }}
\vcenter{\xymatrix@C=-0pc{ \B\kl \deq && \pi_l && \fa_{\B}^j \\
\B\kl &&& \B\lj \cl{\eps_l} \\
&& \B\kj \cldosuno{\alpha\tB_{k,l,j}} }}
\vcenter{\xymatrix@C=-0pc{\quad \forall \ k<l<j}}  $$
  
 \item[d)] 
 $$\vcenter{\xymatrix@C=-0pc{ \pi_k \deq && \biLim{k<j}{\p_k} \dl \dcr{\quad \alpha_j} && \dr \pi_0^j && \q_j \deq \\
 \pi_k \deq && \fa_{\B}^j \deq && \pi_1^j && \q_j \\
 \pi_k && \fa_{\B}^j &&& \p_j \cl{\beta_1^j}  }}
\vcenter{\xymatrix@C=-0pc{ \quad = \quad }}
\vcenter{\xymatrix@C=-0pc{ & \pi_k \dl & \dc{\mu_k} & \dr \biLim{k<j}{\p_k} && \pi_0^j && \q_j \\
& \p_k \deq && \pi_k &&& \fa_{\Y}^j \cl{\beta_0^j} \\
& \p_k \dl && \dc{\p\kj\inv} && \dr \Y\kj \cldosuno{\eta_k} \\
& \B\kj \op{\eps_k\inv} &&&& \p_j \deq \\
\pi_k && \fa_{\B}^j &&& \p_j}}
\vcenter{\xymatrix@C=-0pc{\forall \ k<j}}  $$

 \end{enumerate}
\end{lemma}

\begin{proof}
$\Rightarrow)$ Since $\cc{C}$ is a closed 2-bmodel 2-category, it is enough to check that $\forall \ j\in \ff{J}$, $\ff{q}_j$ has the right lifting property with respect to all morphisms that are both cofibrations and weak equivalences. So, let $\ff{A}\mr{\ff{i}}\ff{X}\in \cc{C}$ be a cofibration which is also a weak equivalence and suppose that we have the following situation:

$$\xymatrix@C=1.5pc@R=3pc{\ff{A} \ar[rr]^{\ff{a}} \ar[d]_{\ff{i}} &\ar@{}[d]|{\cong \; \Downarrow \; \gamma}& \ff{Y}_j \ar[d]^{\ff{q}_j} \\
            \ff{X} \ar[rr]_{\ff{b}} && \ff{P}_j}$$
            
\noindent Let's define $\ff{K}$, $\ff{L}$, $\ff{K}\mr{\ff{i}'}\ff{L}$, $\ff{K}\mr{\ff{a}'}\ff{Y}$ and $\ff{L}\mr{\ff{b}'}\ff{B}$ in $\pCJ$ by 
                        $\ff{K}_k=\begin{cases}
                         \ff{X} &  \hbox{if } k<j \\
                         \ff{A} &  \hbox{if } k=j \\
                         0 &  \hbox{otherwise} 
                        \end{cases}$,
                        $\ff{K}_{id_k}=id_{\ff{K}_k}$,
                        $\vcenter{\xymatrix{\\ \\ \ff{K}_{k<l} = \\ \\ \\ }} 
                        \begin{cases}
                                       id_{\ff{X}} &  \hbox{if } k<l<j \\ 
                                       \ff{i} &  \hbox{if } k<j \hbox{ and } l=j \\
                                      0\mr{}\ff{X} &  \hbox{if } k<j \hbox{ and } l\nleq j\\
                                      0\mr{}\ff{A} &  \hbox{if } k=j \hbox{ and } l\nleq j\\
                                      id_0 & \hbox{otherwise } \\
                                      \end{cases}$,
                         $\ff{L}_k=\begin{cases}
                              \ff{X} & \hbox{if } k\leq j\\
                              0 & \hbox{otherwise}
                             \end{cases}$,
                             $\ff{L}_{id_k}=id_{\ff{L}_k}$,
                        $\vcenter{\xymatrix{\\ \ff{L}_{k<l}= \\ \\ }} \begin{cases}
                                       id_{\ff{X}} &  \hbox{if } k<l\leq j \\ 
                                       0\mr{}\ff{X} &  \hbox{if } k\leq j \hbox{ and } l\nleq j\\
                                      id_0 & \hbox{otherwise } \\
                                      \end{cases}$,
                        $\ff{i}'_k=\begin{cases}
                               id_{\ff{X}} & \hbox{if } k<j \\
                               \ff{i} & \hbox{if } k=j \\
                               id_0 & \hbox{otherwise}
                               \end{cases}$,
                        $\ff{a}'_k=\begin{cases}
                                   \pi_k \pi_0^j \ff{b} & \hbox{if } k<j \\
                                   \ff{a} & \hbox{if } k=j \\
                                   0\mr{} \ff{Y}_k & \hbox{otherwise}
                                  \end{cases}$,
                        $\ff{a}'_{id_k}=(\alpha\tY_k)^{-1}\ff{a}'_k$,
                        \mbox{$\vcenter{\xymatrix{\\ \ff{a}'_{k<l}= \\ \\ }} 
                        \begin{cases}
                                       \pi_{k<l}\pi_0^j \ff{b} &  \hbox{if } k<l<j \\ 
                                       \pi_k \pi_0^j\gamma \circ \pi_k (\beta_0^j)^{-1} \ff{a} \circ \eta_k^{-1} \ff{a} &  \hbox{if } k<j \hbox{ and } l=j\\
                                      id_{0\mr{}\ff{Y}_k} & \hbox{otherwise } \\
                                      \end{cases}$,}
                        $\ff{b}'_k=\begin{cases}
                                  \ff{B}_{k\leq j} \pi_1^j \ff{b} & \hbox{if } k\leq j \\
                                  0\mr{} \ff{B}_k & \hbox{otherwise}
                                 \end{cases}$,
                                 \vspace{2ex}
                        $\ff{b}'_{id_k}=(\alpha\tB_k)^{-1}\ff{b}'_k$ and
                        $\ff{b}'_{k<l}=\begin{cases}
                                       \alpha\tB_{k,l,j}\pi_1^j \ff{b} &  \hbox{if } k<l<j \\ 
                                       \ff{B}_{k<j} (\alpha\tB_j)^{-1} \pi_1^j \ff{b} &  \hbox{if } k<j \hbox{ and } l=j\\
                                      id_{0\mr{}\ff{B}_k} & \hbox{otherwise } \\
                                      \end{cases}$.
\vspace{2ex}
It is straightforward to check that $\ff{K}$ and $\ff{L}$ are 2-functors, $\ff{i}'$ is a 2-natural transformation and $\ff{a}'$, $\ff{b}'$ are pseudo-natural transformations.
                                 
\noindent Then, since $\ff{p}$ is a fibration and $\ff{i}'$ is both a cofibration and a weak equivalence, there exists a filler $(\ff{f},\lambda,\rho)$ for the following diagram

\begin{equation}\label{diagramaconlosprima}
\xymatrix@C=1.5pc@R=3pc{\ff{K} \ar[rr]^{\ff{a}'} \ar[d]_{\ff{i}'} &\ar@{}[d]|{\cong \; \Downarrow \; \tilde{\gamma}}& \ff{Y} \ar[d]^{\ff{p}} \\
            \ff{L} \ar[rr]_{\ff{b}'} && \ff{B}} 
\end{equation}

\noindent where $\tilde{\gamma}_k=\begin{cases}
                                   \vcenter{\xymatrix@C=-0.2pc@R=0.5pc{\p_k \dl & \dc{\mu_k\inv} & \dr \pi_k && \pi_0^j \deq && \fb \deq \\
                                   \pi_k \deq && \biLim{k<j}{\p_k} \dl \dcr{\quad \alpha_j} && \dr \pi_0^j && \fb \deq \\
                                   \pi_k && \fa_{\B}^j && \pi_1^j \deq && \fb \deq \\
                                   & \B\kj \cl{\eps_k} &&& \pi_1^j && \fb }}
                                   & \mbox{ if } k<j \\
                                   \\
                                   \vcenter{\xymatrix@C=0.1pc@R=0.5pc{ &&& \p_j \opdosuno{\! \! \! \! \! (\beta_1^j)\inv} &&& \fa \deq \\
                                   & \pi_1^j \deq &&& \q_j \dl & \dc{\gamma} & \dr \fa \\
                                   & \pi_1^j \op{=} &&& \fb \deq && \ii \deq \\
                                   id_{\B_j} \dcell{\alpha\tB_j} && \pi_1^j \deq && \fb \deq && \ii \deq \\
                                   \B_{id_j} && \pi_1^j && \fb && \ii }}
                                   & \mbox{ if } k=j \\
                                   id_{0\mr{}\ff{B}_k} & \mbox{ otherwise} 
                                  \end{cases}$.

Let's check that $(\ff{f}_j,\lambda_j,\rho')$ (where $\rho'$ is such that $\pi_1\rho'= 
\vcenter{\xymatrix@C=-0.2pc@R=0.5pc{& \pi_1^j && \q_j && \f_j \deq \\
& & \p_j \cl{\beta_1^j} &&& \f_j \\
& && \fb'_j \clunodos{\rho_j} \opdosdos{=} \\
\dl & \B_{id_j} \dc{(\alpha\tB_j)\inv} & \drbis & \pi_1^j \deq && \fb \deq \\
& id_{\B_j} && \pi_1^j && \fb}}$ 
and 
\mbox{$\pi_k\pi_0\rho'= 
\vcenter{\xymatrix@C=-0.2pc@R=0.5pc{\pi_k \deq && \pi_0^j && \q_j & \; \; & \f_j \deq \\
\pi_k &&& \fa_{\Y}^j \cl{\beta_0^j} &&& \f_j \deq \\
&& \Y\kj \cldosuno{\eta_k} &&&& \f_j \\
&& && \f_k \cldosdos{\f\kj} \opdosdos{\lambda_k\inv} \\
&& \pi_k && \pi_0^j && \fb}}$)} is the filler that we were looking for: To do that, it is enough to check that $\pi_k \pi_0^j (\rho' \ff{i} \circ \ff{q}_j \lambda_j)=\pi_k \pi_0^j \gamma\ \forall \ k<j$ and $\pi_1^j (\rho' \ff{i} \circ \ff{q}_j \lambda_j)=\pi_1^j \gamma$:  

By elevators calculus plus the fact that $\lambda$ is a modification, we have the following equality: 


$$\vcenter{\xymatrix@C=-0pc{ \pi_k \deq && \pi_0^j \deq && \q_j \deq &&& \fa \op{\lambda_j} \\
\pi_k \deq && \pi_0^j && \q_j & \; \; & \f_j \deq && \ii \deq \\
\pi_k &&& \fa_{\Y}^j \cl{\beta_0^j} &&& \f_j \deq && \ii \deq  \\
&& \Y\kj \cldosuno{\eta_k} &&&& \f_j && \ii \deq \\
&& && \f_k \cldosdos{\f\kj} \opdosdos{\lambda_k\inv} &&&& \ii \deq \\
&& \pi_k && \pi_0^j && \fb && \ii}}
\vcenter{\xymatrix@C=-0pc{ \quad = \quad }}
\vcenter{\xymatrix@C=-0pc{ \pi_k \deq && \pi_0^j \deq && \q_j \dl & \dc{\gamma} & \dr \fa \\
\pi_k && \pi_0^j && \fb && \ii}}$$

And, by elevators calculus plus the fact that $(\ff{f},\lambda,\rho)$ is a filler for diagram  \eqref{diagramaconlosprima}, we have the following equality:


$$\vcenter{\xymatrix@C=-0pc{ \pi_1^j \deq && \q_j \deq &&& \fa \op{\lambda_j} \\
\pi_1^j && \q_j && \f_j \deq && \ii \deq \\
& \p_j \cl{\beta_1^j} &&& \f_j && \ii \deq \\
&& \fb'_j \clunodos{\rho_j} \opdosdos{=} && && \ii \deq \\
\B_{id_j} \dcellbymedio{(\alpha\tB_j)\inv} && \pi_1^j \deq && \fb \deq && \ii \deq \\
id_{\B_j} && \pi_1^j && \fb && \ii}}
\vcenter{\xymatrix@C=-0pc{ \quad = \quad }}
\vcenter{\xymatrix@C=-0pc{ \pi_1^j && \q_j &&& \fa \deq \\
& \p_j \cl{\beta_1^j} \ardl && \dc{\tilde{\gamma}_j} && \fa \ardr \\
\B_{id_j} \dcellbymedio{(\alpha\tB_j)\inv} && \pi_1^j \deq && \fb \deq && \ii \deq \\
id_{\B_j} && \pi_1^j && \fb && \ii}}
\vcenter{\xymatrix@C=-0pc{ \quad = \quad }}
\vcenter{\xymatrix@C=-0pc{ \pi_1^j \deq && \q_j \dl & \dc{\gamma} & \dr \fa \\
\pi_1^j && \fb && \ii}}$$

$\Leftarrow)$ Let $\ff{A}\mr{\ff{i}}\ff{X}\in \pCJ$ be a cofibration that is also a weak equivalence and suppose that we have the following situation: 

$$\xymatrix@C=1.5pc@R=3pc{\ff{A} \ar[rr]^{\ff{a}} \ar[d]_{\ff{i}} &\ar@{}[d]|{\cong \; \Downarrow \; \gamma}& \ff{Y} \ar[d]^{\ff{p}} \\
            \ff{X} \ar[rr]_{\ff{b}} && \ff{B}}$$
            
\noindent We are going to construct the filler $(\ff{f},\lambda,\rho)$ inductively:

\noindent $\ff{p}_0\cong \ff{q}_0$ and therefore is a fibration. So, since $\cc{C}$ is a closed 2-bmodel 2-category, there exists a filler $(\ff{f}_0,\lambda_0,\rho_0)$ for the following diagram

$$\xymatrix@C=1.5pc@R=3pc{\ff{A}_0 \ar[rr]^{\ff{a}_0} \ar[d]_{\ff{i}_0} &\ar@{}[d]|{\cong \; \Downarrow \; \gamma_0}& \ff{Y}_0 \ar[d]^{\ff{p}_0} \\
            \ff{X}_0 \ar[rr]_{\ff{b}_0} && \ff{B}_0}$$

\noindent Suppose that we have already constructed $(\ff{f}_k,\lambda_k,\rho_k)\ \forall k<j$. Then, since $\ff{q}_j$ is a fibration and $\ff{i}_j$ is both a cofibration and a weak equivalence, there exists a filler $(\ff{f}_j,\lambda_j,\tilde{\rho_j})$ for the following diagram

\begin{equation}\label{diagramacongammatilde}
\xymatrix@C=1.5pc@R=3pc{\ff{A}_j \ar[rr]^{\ff{a}_j} \ar[d]_{\ff{i}_j} &\ar@{}[d]|{\cong \; \Downarrow \; \tilde{\gamma}_j}& \ff{Y}_j \ar[d]^{\ff{q}_j} \\
            \ff{X}_j \ar[rr]_{\ff{c}_j} && \ff{P}_j} 
\end{equation}
                   
\noindent where $\ff{c}_j$ is given by diagram \eqref{diagramaC} and $\tilde{\gamma}_j$ is such that \mbox{$\pi_k \pi_0^j \tilde{\gamma}_j= 
\vcenter{\xymatrix@C=-0pc@R=0.5pc{ \pi_k \deq && \pi_0^j && \q_j & \; \; & \fa_j \deq \\
\pi_k &&& \fa_{\Y}^j \cl{\beta_0^j} &&& \fa_j \deq \\
&& \Y\kj \cldosuno{\eta_k} \dl && \dc{\fa\kj} && \dr \fa_j \\
&& \fa_k \opdosuno{\lambda_k} &&&& \A\kj \deq \\
\f_k \deq &&& \ii_k \dl & \dcr{\ii\kj\inv} && \dr \A\kj \\
\f_k \dl && \dc{\nu_k\inv} & \dr \X\kj &&& \ii_j \deq \\
\pi_k \deq &&& \fa^j_{\X,\varphi} \op{\! (\theta_0^j)\inv} &&& \ii_j \deq \\
\pi_k && \pi_0^j && \fc_j && \ii_j}}
$} $\quad$
and $\quad$
$\pi_1^j \tilde{\gamma}_j= 
\vcenter{\xymatrix@C=0.2pc@R=0.5pc{ \pi_1^j && \q_j && \fa_j \deq \\
& \p_j \cl{\beta_1^j} \dl & \dcr{\gamma_j} && \dr \fa_j \\
& \fb_j \opb{(\theta_1^j)\inv} &&& \ii_j \deq \\
\pi_1^j && \fc_j && \ii_j}}
$:

\begin{equation} \label{diagramaC}
\xymatrix@C=3pc{\ff{X}_j   \ar@/^4ex/[rrd]^{\ff{a}_{\ff{X},\ff{Y}}^j} \ar@{-->}[rd]|{\comw{M^M} \ff{c}_j \comw{M^M} } \ar@/_4ex/[rdd]_{\ff{b}_j} & \\
            \ar@{}[r]|{\cong \; \Downarrow \; \theta_1^j} & \ff{P}_j \ar@{}[u]|{\cong \; \Uparrow \; \theta_0^j} \ar[r]^{\pi_0^j} \ar[d]_{\pi_1^j} \ar@{}[rd]^{bipb \quad\quad\quad\quad\quad\quad\quad\quad\quad\quad}_{\quad\quad\quad\quad\quad\quad\quad\quad\quad\quad \cong \; \Downarrow \; \alpha_j }
            & \biLim{k<j}{\ff{Y}_k} \ar[d]^{\biLim{k<j}{\ff{p}_k}}\\
            & \ff{B}_j \ar[r]_{\ff{a}_{\ff{B}}^j} & \biLim{k<j}{\ff{B}_k}}
\end{equation}

\noindent where $\ff{a}_{\ff{X},\ff{Y}}^j$ is induced by the pseudo-cone $\{\ff{f}_k\ff{X}_{k<j}\}_{k<j},\ $ \mbox{$\Bigg\{ 
\vcenter{\xymatrix@C=-0.2pc@R=1pc{ \Y\kl \dl & \dc{\f\kl} & \dr \f_l && \X\lj \deq \\   \f_k \deq && \X\kj && \X\lj \\   \f_k &&& \X\kj \cl{\alpha\tX_{k,l,j}} }}
\Bigg\}_{k<l<j}\cup \Bigg\{ 
\vcenter{\xymatrix@C=-0.2pc@R=1pc{ \dl & \Y_{id_k} \dc{(\alpha\tY_k)\inv} & \drbis & \f_k \deq && \X\kj \deq \\  & id_{\Y_k} \ardr && \f_k \dc{=} && \X\kj \ardl \\ & & \f_k && \X\kj }}
\Bigg\}_{k<j}$} and so we have:

\begin{enumerate}
 
 \item[e)] invertible 2-cells $\pi_k \ff{a}_{\ff{X},\ff{Y}}^j \Mr{\nu_k} \ff{f}_k\ff{X}_{k<j} \ \forall \ k<j$ such that $ 
 \vcenter{\xymatrix@C=-0pc{ \Y\kl && \pi_l && \fa\sXY^j \deq \\
 & \pi_k \cl{\pi\kl} \dl & \dcr{\nu_k} && \dr \fa\sXY^j \\
 & \f_k &&& \X\kj}}
\vcenter{\xymatrix@C=-0pc{ \quad = \quad }}
\vcenter{\xymatrix@C=-0pc{ \Y\kl \deq && \pi_l \dl & \dc{\nu_l} & \dr \fa\sXY^j \\
\Y\kl \dl & \dc{\f\kl} & \dr \f_l && \X\lj \deq \\
\f_k \deq && \X\kl && \X\lj \\
\f_k &&& \X\kj \cl{\alpha\tX_{k,l,j}} }}
 \vcenter{\xymatrix@C=-0pc{ \ \forall \ k<l<j}}$
 
 \item[f)] $ 
 \vcenter{\xymatrix@C=-0pc{ \pi_k \deq && \biLim{k<j}{\p_k} \dl \dcr{\quad \alpha_j} && \dr \pi_0^j && \fc_j \deq \\
 \pi_k \deq && \fa_{\B}^j \deq && \pi_1^j && \fc_j \\
 \pi_k && \fa_{\B}^j &&& \fb_j \cl{\theta_1^j}  }}
\vcenter{\xymatrix@C=-0pc{ \quad = \quad }}
\vcenter{\xymatrix@C=-0pc{ \pi_k \dl & \dc{\mu_k} & \dr \biLim{k<j}{\p_k} && \pi_0^j && \fc_j \\
\p_k \deq && \pi_k \dl & \dcr{\nu_k} && \dr \fa\sXY^j \cl{\theta_0^j} \\
\p_k && \f_k &&& \X\kj \deq \\
& \fb_k \cl{\rho_k} \dl & \dcr{\fb\kj\inv} &&& \dr \X\kj \\
& \B\kj \op{\eps_k\inv} &&&& \fb_j \deq \\
\pi_k && \fa_{\B}^j &&& \fb_j}}
 \vcenter{\xymatrix@C=-0pc{ \ \forall \ k<j}}$
\end{enumerate}

Take $\ff{f}_{k<j}= 
\vcenter{\xymatrix@C=-0pc@R=1pc{ && \Y\kj \opdosuno{\eta_k\inv} &&&& \f_j \deq \\
\pi_k \deq &&& \fa_{\Y}^j \opb{(\beta_0^j)\inv \ } &&& \f_j \deq \\
\pi_k \deq && \pi_0^j \deq && \q_j && \f_j \\
\pi_k \deq && \pi_0^j &&& \fc_j \cl{\tilde{\rho}_j} \\
\pi_k \dl & \dcr{\nu_k} && \dr \fa\sXY^j \clunodos{\theta_0^j} \\
\f_k &&& \X\kj}}
$, $\ff{f}_{id_j}= 
\vcenter{\xymatrix@C=-0pc@R=1pc{ \dl & \Y_{id_j} \dc{(\alpha\tY_j)\inv} & \drmediobis & \f_j \deq \\
& id_{\Y_j} \dl & \dc{=} & \dr \f_j \\
& \f_j \deq && id_{\X_j} \dcellb{\alpha\tX_j} \\
& \f_j && \X_{id_j} }}
$ and \mbox{$\rho_j= 
\vcenter{\xymatrix@C=-0pc@R=1pc{ & \; \p_j \; \opbymedio{(\beta_1^j)\inv} &&& \f_j \deq \\
\pi_1^j \deq && \q_j && \f_j \\
\pi_1^j &&& \fc_j \cl{\tilde{\rho}_j} \\
&& \fb_j \cldosuno{\theta_1^j} }}
$.}

Now we are going to check that $\ff{f}$ constructed this way is a pseudo-natural transformation: PNO is satisfied by construction and PN2 is vacuous since there are no 2-cells in $\ff{J}$. To check axiom PN1, we need to check that the following equality holds $\forall \ k<l<j$:


$$\vcenter{\xymatrix@C=-0pc{ \Y\kl \deq && \Y\lj \dl & \dc{\f\lj} & \dr \f_j \\
\Y\kl \dl & \dc{\f\kl} & \dr \f_l && \X\lj \deq \\
\f_k \deq && \X\kl && \X\lj \\
\f_k &&& \X\kj \cl{\alpha\tX_{k,l,j}} }}
\vcenter{\xymatrix@C=-0pc{ \quad = \quad }}
\vcenter{\xymatrix@C=-0pc{ \Y\kl && \Y\lj && \f_j \deq \\
& \Y\kj \cl{\alpha\tY_{k,l,j}} \dl & \dcr{\f\kj} && \dr \f_j \\
& \f_k &&& \X\kj}}$$

\noindent But


$$\vcenter{\xymatrix@C=-0pc{ \Y\kl \deq && \Y\lj \dl & \dc{\f\lj} & \dr \f_j \\
\Y\kl \dl & \dc{\f\kl} & \dr \f_l && \X\lj \deq \\
\f_k \deq && \X\kl && \X\lj \\
\f_k &&& \X\kj \cl{\alpha\tX_{k,l,j}} }}
\vcenter{\xymatrix@C=-0pc{ \quad = \quad }}
\vcenter{\xymatrix@C=-0pc{ \Y\kl \deq &&&& \Y\lj \opdosuno{\eta_l\inv} &&&& \f_j \deq \\
\Y\kl \deq && \pi_l \deq &&& \fa_{\Y}^j \opb{(\beta_0^j)\inv} &&& \f_j \deq \\
\Y\kl \deq && \pi_l \deq && \pi_0^j \deq && \q_j && \f_j \\
\Y\kl \deq && \pi_l \deq && \pi_0^j &&& \fc_j \cl{\tilde{\rho}_j} \\
\Y\kl \deq && \pi_l \dl & \dcr{\nu_l} && \dr \fa\sXY^j \clunodos{\theta_0^j} \\
\Y\kl \dl & \dc{\f\kl} & \dr \f_l &&& \X\lj \deq \\
\f_k \deq && \X\kl &&& \X\lj \\
\f_k &&& \X\kj \clunodos{\alpha\tX_{k,l,j}}  }}
\vcenter{\xymatrix@C=-0pc{ \quad = \quad }}$$

$$\vcenter{\xymatrix@C=-0pc{ \quad = \quad }}
\vcenter{\xymatrix@C=-0pc{ \Y\kl \deq &&&& \Y\lj \opdosuno{\eta_l\inv} &&&& \f_j \deq \\
\Y\kl \deq && \pi_l \deq &&& \fa_{\Y}^j \opb{(\beta_0^j)\inv} &&& \f_j \deq \\
\Y\kl \deq && \pi_l \deq && \pi_0^j \deq && \q_j && \f_j \\
\Y\kl \deq && \pi_l \deq && \pi_0^j &&& \fc_j \cl{\tilde{\rho}_j} \\
\Y\kl && \pi_l &&& \fa\sXY^j \clunodos{\theta_0^j} \deq \\ 
& \pi_k \cl{\pi\kl} \dl && \dc{\nu_k} && \dr \fa\sXY^j \\
& \f_k &&&& \X\kj}}
\vcenter{\xymatrix@C=-0pc{ \quad = \quad }}
\vcenter{\xymatrix@C=-0pc{ \Y\kl &&& \Y\lj &&& \f_j \deq \\
&& \Y\kj \cldosuno{\alpha\tY_{k,l,j}} \opdosuno{\eta_k\inv} &&&& \f_j \deq \\
\pi_k \deq &&& \fa_{\Y}^j \opb{(\beta_0^j)\inv} &&& \f_j \deq  \\
\pi_k \deq && \pi_0^j \deq && \q_j && \f_j \\
\pi_k \deq && \pi_0^j &&& \fc_j \cl{\tilde{\rho}_j} \\
\pi_k \dl && \dc{\nu_k} && \dr \fa\sXY^j \cldosuno{\theta_0^j} \\
\f_k &&&& \X\kj}}
\vcenter{\xymatrix@C=-0pc{ \quad = \quad }}$$
\vspace{1ex}
$$\vcenter{\xymatrix@C=-0pc{ \quad = \quad }}
\vcenter{\xymatrix@C=-0pc{ \Y\kl && \Y\lj && \f_j \deq \\
& \Y\kj \cl{\alpha\tY_{k,l,j}} \dl & \dcr{\f\kj} && \dr \f_j \\
& \f_k &&& \X\kj}}$$

\noindent where the second equality is due to $e)$ and the third one is due to elevators calculus plus $b)$.

To check that $\lambda$ is a modification, we need to verify that the following equality holds $\forall \ k<j$:


$$\vcenter{\xymatrix@C=-0pc{ & \Y\kj \dl & \dcr{\fa\kj} && \dr \fa_j \\
& \fa_k \op{\lambda_k} &&& \A\kj \deq \\
\f_k && \ii_k && \A\kj}}
\vcenter{\xymatrix@C=-0pc{ \quad = \quad }}
\vcenter{\xymatrix@C=-0pc{ \Y\kj \deq &&& \fa_j \op{\lambda_j} \\
\Y\kj \dl & \dc{\f\kj} & \dr \f_j && \ii_j \deq \\
\f_k \deq && \X\kj \dl & \dc{\ii\kj} & \dr \ii_j \\
\f_k && \ii_k && \A\kj}}$$

\noindent But


$$\vcenter{\xymatrix@C=-0pc{ \Y\kj \deq &&& \fa_j \op{\lambda_j} \\
\Y\kj \dl & \dc{\f\kj} & \dr \f_j && \ii_j \deq \\
\f_k \deq && \X\kj \dl & \dc{\ii\kj} & \dr \ii_j \\
\f_k && \ii_k && \A\kj}}
\vcenter{\xymatrix@C=-0pc{ \quad = \quad }}
\vcenter{\xymatrix@C=-0pc{ && \Y\kj \deq &&&&& \fa_j \op{\lambda_j} \\
&& \Y\kj \opdosuno{\eta_k\inv} &&&& \f_j \deq && \ii_j \deq \\
\pi_k \deq &&& \fa_{\Y}^j \opb{(\beta_0^j)\inv} &&& \f_j \deq && \ii_j \deq \\
\pi_k \deq && \pi_0^j \deq && \q_j && \f_j && \ii_j \deq \\
\pi_k \deq && \pi_0^j &&& \fc_j \cl{\tilde{\rho}_j} &&& \ii_j \deq \\
\pi_k \dl & \dcr{\nu_k} && \dr \fa\sXY^j \clunodos{\theta_0^j} &&&&& \ii_j \deq \\
\f_k \deq &&& \X\kj \dl && \dcr{\ii\kj} &&& \dr \ii_j \\
\f_k &&& \ii_k &&&&& \A\kj }}
\vcenter{\xymatrix@C=-0pc{ \quad = \quad }}$$

$$\vcenter{\xymatrix@C=-0pc{ \quad = \quad }}
\vcenter{\xymatrix@C=-0pc{ && \Y\kj \opdosuno{\eta_k\inv} &&&& \fa_j \deq \\
\pi_k \deq &&& \fa_{\Y}^j \opb{(\beta_0^j)\inv} &&& \fa_j \deq \\
\pi_k \deq && \pi_0^j \deq && \q_j \dl & \dc{\tilde{\gamma}_j} & \dr \fa_j \\
\pi_k \deq && \pi_0^j && \fc_j && \ii_j \deq \\
\pi_k \dl & \dcr{\nu_k} && \dr \fa\sXY^j \cl{\theta_0^j} &&& \ii_j \deq \\
\f_k \deq &&& \X\kj \dl & \dcr{\ii\kj} && \dr \ii_j \\
\f_k &&& \ii_k &&& \A\kj}}
\vcenter{\xymatrix@C=-0pc{ \quad = \quad }}
\vcenter{\xymatrix@C=-0pc{ & \Y\kj \dl & \dcr{\fa\kj} && \dr \fa_j \\
& \fa_k \op{\lambda_k} &&& \A\kj \deq \\
\f_k && \ii_k && \A\kj}}$$

\noindent where the second equality is due to elevators calculus plus the fact that $(\ff{f}_j,\lambda_j,\tilde{\rho}_j)$ is a filler for diagram \eqref{diagramacongammatilde} and the last one is due to the definition of $\tilde{\gamma}_j$ plus elevators calculus.

To check that $\rho$ is a modification, we need to verify that the following equality holds $\forall \ k<j$:


$$\vcenter{\xymatrix@C=-0pc{ \B\kj \dl &\dc{\p\kj} & \dr \p_j && \f_j \deq \\
\p_k \deq && \Y\kj \dl & \dc{\f\kj} & \dr \f_j \\
\p_k && \f_k && \X\kj \deq \\
& \fb_k \cl{\rho_k} &&& \X\kj}}
\vcenter{\xymatrix@C=-0pc{ \quad = \quad }}
\vcenter{\xymatrix@C=-0pc{ \B\kj \deq && \rho_j && \f_j \\
\B\kj \dl & \dcr{\fb\kj} && \dr \fb_j \cl{\rho_j} \\
\fb_k &&& \X\kj}}$$

\noindent But


$$\vcenter{\xymatrix@C=-0pc{ \B\kj \dl &\dc{\p\kj} & \dr \p_j && \f_j \deq \\
\p_k \deq && \Y\kj \dl & \dc{\f\kj} & \dr \f_j \\
\p_k && \f_k && \X\kj \deq \\
& \fb_k \cl{\rho_k} &&& \X\kj}}
\vcenter{\xymatrix@C=-0pc{ \quad = \quad }}
\vcenter{\xymatrix@C=-0pc{ \B\kj \dl && \dc{\p\kj} && \dr \p_j &&&& \f_j \deq \\
\p_k \deq &&&& \Y\kj \opdosuno{\eta_k\inv} &&&& \f_j \deq \\
\p_k \deq && \pi_k \deq &&& \fa_{\Y}^j \op{\! \! (\beta_0^j)\inv} &&& \f_j \deq \\
\p_k \deq && \pi_k \deq && \pi_0^j \deq && \q_j && \f_j \\
\p_k \deq && \pi_k \deq && \pi_0^j &&& \fc_j \cl{\tilde{\rho}_j} \\
\p_k \deq && \pi_k \dl & \dcr{\nu_k} && \dr \A\sXY^j \clunodos{\theta_0^j} \\
\p_k && \f_k &&& \X\kj \deq \\
& \fb_k \cl{\rho_k} &&&& \X\kj }}
\vcenter{\xymatrix@C=-0pc{ \quad = \quad }}$$

$$\vcenter{\xymatrix@C=-0pc{ \quad = \quad }}
\vcenter{\xymatrix@C=-0pc{ \B\kj \dl && \dc{\p\kj} && \dr \p_j &&&& \f_j \deq \\
\p_k \deq &&&& \Y\kj \opdosuno{\eta_k\inv} &&&& \f_j \deq \\
\p_k \deq && \pi_k \deq &&& \fa_{\Y}^j \op{\! \! (\beta_0^j)\inv} &&& \f_j \deq \\
\p_k \deq && \pi_k \deq && \pi_0^j \deq && \q_j && \f_j \\
\p_k \dl & \dc{\mu_k\inv} & \dr \pi_k && \pi_0^j \deq &&& \fc_j \cl{\tilde{\rho}_j} \deq \\
\pi_k \deq && \biLim{k<j}{\p_k} \dl & \dc{\alpha_j} & \dr \pi_0^j &&& \fc_j \deq \\
\pi_k && \fa_{\B}^j && \pi_1^j &&& \fc_j \\
& \B\kj \cl{\eps_k} \dl && \dc{\fb\kj} && \dr \fb_j \clunodos{\theta_1^j} \\
& \fb_k &&&& \X\kj }}
\vcenter{\xymatrix@C=-0pc{ \quad = \quad }}
\vcenter{\xymatrix@C=-0pc{ \B\kj \deq && \rho_j && \f_j \\
\B\kj \dl & \dcr{\fb\kj} && \dr \fb_j \cl{\rho_j} \\
\fb_k &&& \X\kj}}$$

\noindent where the second equality is due to item $f)$ and the last one is due to item $d)$.

Finally, let's check that $(\ff{f},\lambda,\rho)$ constructed this way is the filler that we were looking for:


$$\vcenter{\xymatrix@C=-0pc{ \p_j \deq &&& \fa_j \op{\lambda_j} \\   \p_j && \f_j && \ii_j \deq \\   & \fb_j \cl{\rho_j} &&& \ii_j }}
\vcenter{\xymatrix@C=-0pc{ \quad = \quad }}
\vcenter{\xymatrix@C=-0pc{ & \p_j \deq &&&& \fa_j \op{\lambda_j} \\
& \p_j \opb{(\beta_1^j)\inv} &&& \f_j \deq && \ii_j \deq \\
\pi_1^j \deq && \q_j && \f_j && \ii_j \deq  \\
\pi_1^j &&& \fc_j \cl{\tilde{\rho}_j} &&& \ii_j \deq \\
& \fb_j \clunodos{\theta_1^j} &&&&& \ii_j}}
\vcenter{\xymatrix@C=-0pc{ \quad = \quad }}
\vcenter{\xymatrix@C=-0pc{ & \p_j \opb{(\beta_1^j)\inv} &&& \fa_j \deq \\    
\pi_1^j \deq && \q_j \dl & \dc{\tilde{\gamma}_j} & \dr \fa_j \\
\pi_q^j && \fc_j && \ii_j \deq \\
& \fb_j \cl{\theta_1^j} &&& \ii_j}}
\vcenter{\xymatrix@C=-0pc{ \quad = \quad }}
\vcenter{\xymatrix@C=-0pc{ \p_j \dl & \dc{\gamma_j} & \dr \fa_j \\   \fb_j && \ii_j}}$$

\noindent where the second equality is due to elevators calculus plus the fact that $(\ff{f}_j,\lambda_j,\tilde{\rho}_j)$ is a filler for diagram \eqref{diagramacongammatilde} and the last one is due to the definition of $\tilde{\gamma}_j$.
\end{proof}

\begin{lemma}\label{lema3.2}
If $\ff{Y}\mr{\ff{p}}\ff{B}\in \pCJ$ is a fibration, then $\ff{p}_j$ is a fibration in $\cc{C} \ \forall \ j\in\ff{J}$. 
\end{lemma}

\begin{proof}
$\ff{p}_0\cong \ff{q}_0$ which is a fibration by \ref{lema1.2}.

If $j\in \ff{J}$ is not the initial object, consider $\ff{p}$ as an object in \mbox{$p\cc{H}om_p(\left\{k\in \ff{J}\ |\ k<j\right\}^{op},\cc{C})$.} Since $\ff{p}\in \pCJ$ is a fibration, by \ref{lema1.2}, \mbox{$\ff{p}\in p\cc{H}om_p(\left\{k\in \ff{J}\ |\ k<j\right\}^{op},\cc{C})$} is a fibration and then, by \ref{lema0.2} $\biLim{k<j}{\ff{p}_k}\in \cc{C}$ is a fibration. Then, since $\cc{C}$ is a closed \mbox{2-bmodel} 2-category, $\pi_1^j$ is a fibration. We also know that $\ff{q}_j$ is a fibration by \ref{lema1.2}. Then, $\ff{p}_j \cong \pi_1^j \ff{q}_j\in \cc{C}$ is also a fibration.
\end{proof}

\begin{lemma}\label{lema4.2}
A morphism $\ff{Y}\mr{\ff{p}}\ff{B}\in \pCJ$ is both a fibration and a weak equivalence iff $\ff{q}_j$ is both a fibration and a weak equivalence in $\cc{C} \ \forall j\in \ff{J}$ where $\ff{q}_j$ is defined as in \ref{lema1.2}. 
\end{lemma}

\begin{proof}
$\Rightarrow)$ By \ref{lema1.2}, it only remains to check that $\ff{q}_j$ is a weak equivalence $\forall \ j\in \ff{J}$. We are going to prove this inductively:

$\ff{q}_0 \cong \ff{p}_0$ and therefore is a weak equivalence.

Suppose that $\ff{q}_k$ is a weak equivalence $\forall \ k<m$. Let's check that, in that case, $\biLim{k<m}{\ff{p}_k}$ is both a fibration and a weak equivalence: 

By \ref{lema2.2}, it is enough to check that it has the right lifting property with respect to all cofibrations. So let $\ff{A}\mr{\ff{i}}\ff{X} \in \cc{C}$ be a cofibration and suppose that we have the following situation:

$$\xymatrix@C=1.5pc@R=3pc{\ff{A} \ar[rr]^{\ff{a}} \ar[d]_{\ff{i}} &\ar@{}[d]|{\cong \; \Downarrow \; \gamma}& \biLim{k<m}{\ff{Y}_k} \ar[d]^{\biLim{k<m}{\ff{p}_k}} \\
            \ff{X} \ar[rr]_{\ff{b}} && \biLim{k<m}{\ff{B}_k}}$$
            
We are going to define a pseudo-cone $\left\{\ff{X}\mr{\ff{f}_k}\ff{Y}_k\right\}_{k<m},\ \left\{\ff{Y}_{k<j}\ff{f}_j\Mr{\ff{f}_{k<j}}\ff{f}_k\right\}_{k<j<m}$ and invertible morphisms of pseudo-cones $\left\{\pi_k \ff{a}\Mr{\tilde{\lambda_k}} \ff{f}_k \ff{i}\right\}_{k<m}$, $\left\{\ff{p}_k \ff{f}_k \Mr{\tilde{\rho_k}}\pi_k \ff{b} \right\}_{k<m}$ as follows:

For the initial object, use \ref{lema2.2} to construct a filler $(\ff{f}_0, \tilde{\lambda_0}, \tilde{\rho_0})$ for the following diagram:

$$\xymatrix@C=1.5pc@R=3pc{\ff{A} \ar[rr]^{\pi_0\ff{a}} \ar[d]_{\ff{i}} &\ar@{}[d]|{\cong \; \Downarrow \; \pi_0\gamma} & \ff{Y}_0 \ar[d]^{\ff{p}_0} \\
            \ff{X} \ar[rr]_{\pi_0\ff{b}} && \ff{B}_0}$$

If $j$ is not the initial object, suppose that we have already defined $\ff{f}_{k}$, $\tilde{\lambda_k}$, $\tilde{\rho_k}\ \forall \ k<j$ and consider the following diagram:

$$\xymatrix@C=3pc{\ff{X}  \ar@/^4ex/[rrd]^{\ff{b}_{\ff{X},\ff{Y}}^j} \ar@{-->}[rd]|{\comw{M^M} \ff{c}_j \comw{M^M} } \ar@/_4ex/[rdd]_{\pi_j \ff{b}}& \\
            \ar@{}[r]|{\cong \; \Downarrow \; \theta_1^j} & \ff{P}_j \ar@{}[u]|{\cong \; \Uparrow \; \theta_0^j} \ar[r]^{\pi_0^j} \ar[d]_{\pi_1^j} \ar[d]_{\pi_1^j} \ar@{}[rd]^{bipb \quad\quad\quad\quad\quad\quad\quad\quad\quad\quad}_{\quad\quad\quad\quad\quad\quad\quad\quad\quad\quad \cong \; \Downarrow \; \alpha_j } & \biLim{k<j}{\ff{Y}_k} \ar[d]^{\biLim{k<j}{\ff{p}_k}}\\
            & \ff{B}_j \ar[r]_{\ff{a}_{\ff{B}}^j} & \biLim{k<j}{\ff{B}_k}}$$

\noindent where $\biLim{k<j}{\ff{p}_k}$ and
$\ff{a}_{\ff{B}}^j$ 
are defined as in \ref{lema1.2} and $\ff{b}_{\ff{X},\ff{Y}}^j$ is induced by the pseudo-cone $\left\{\ff{f}_k\right\}_{k<j}$, $\left\{\ff{f}_{k<l}\right\}_{k<l<j}$. Then we have invertible 2-cells $\pi_k \ff{b}_{\ff{X},\ff{Y}}^j \Mr{\nu_k} \ff{f}_k \ \forall \ k<j$ such that 

\begin{equation}\label{A30}
 \vcenter{\xymatrix@C=-0pc{ \Y\kl && \pi_l && \fb\sXY^j \deq \\ & \pi_k \cl{\! \! \pi\kl} &&& \fb\sXY^j \\   && \f_k \clunodos{\nu_k} }}
\vcenter{\xymatrix@C=-0pc{ \quad = \quad }}
\vcenter{\xymatrix@C=-0pc{ \Y\kl \deq && \pi_l && \fb\sXY^j \\  \Y\kl &&& \f_l \cl{\nu_l} \\ && \f_k \cldosuno{\! \! \f\kl}  }}
 \vcenter{\xymatrix@C=-0pc{ \quad \forall \ k<l<j}}
\end{equation}
 
\noindent and we also have the following equality

\begin{equation}\label{A32}
 \vcenter{\xymatrix@C=-0pc{ \pi_k \deq && \biLim{k<j}{\p_k} \dl & \dc{\! \! \! \! \! \alpha_j} & \dr \pi_0^j && \fc_j \deq \\   \pi_k && \fa_{\B}^j && \pi_1^j && \fc_j }}
 \vcenter{\xymatrix@C=-0pc{ \quad = \quad }}
 \vcenter{\xymatrix@C=-0pc{ && \pi_k \dl & \dc{\; \; \mu_k} & \dr \biLim{k<j}{\p_k} && \pi_0^j && \fc_j \\
 && \p_k \deq && \pi_k &&& \fb\sXY^j \cl{\theta_0^j} \\
 && \p_k \dl && \dc{\tilde{\rho}_k} && \dr \f_k \cldosuno{\nu_k} \\
 && \pi_k \opunodos{\pi\kj\inv} &&&& \fb \deq \\
 & \B\kj \op{\eps_k\inv} &&& \pi_1^j \dl & \dc{\! \! \! \! \! (\theta_1^j)\inv} & \dr \fb \\
 \pi_k && \fa_{\B}^j && \pi_1^j && \fc_j}}
 \vcenter{\xymatrix@C=-0pc{ \quad \forall \ k<j }}
\end{equation}

%
%
%
%
%
%
%

Then there exists a filler $(\ff{f}_j,\tilde{\lambda}_j,\rho'_j)$ for the following diagram

\begin{equation}\label{diagramaconbeta}
\xymatrix@C=1.5pc@R=3pc{\ff{A} \ar[rr]^{\pi_j \ff{a}} \ar[d]_{\ff{i}} &\ar@{}[d]|{\cong \; \Downarrow \; \tilde{\gamma}_j}& \ff{Y}_j \ar[d]^{\ff{q}_j} \\
            \ff{X} \ar[rr]_{\ff{c}_j} && \ff{P}_j} 
\end{equation}

\noindent where $\ff{q}_j$ corresponds to a diagram as the one in \ref{lema1.2}, $\beta_j$ is given by the formulas \mbox{$\pi^j_1 \tilde{\gamma}_j= 
\vcenter{\xymatrix@C=-0.2pc{\pi_1^j && \q_j && \pi_j \deq && \fa \deq \\
& \p_j \cl{\beta_1^j} \dl & \dcr{\mu_j\inv} && \dr \pi_j && \fa \deq \\
& \pi_j \deq &&& \biLim{k<m}{\p_k} \dl & \dc{\gamma} & \dr \fa \\
& \pi_j \dl & \dcr{\theta_1^j} && \dr \fb && \ii \deq \\
& \pi_1^j &&& \fc_j && \ii}}
$} 
$\quad$ and \mbox{$\quad$
$\pi_k \pi^j_0 \tilde{\gamma}_j= 
\vcenter{\xymatrix@C=-0.2pc{ \pi_k \deq && \pi_0^j && \q_j && \pi_j \deq && \fa \deq \\
\pi_k &&& \fa_{\Y}^j \cl{\beta_0^j} &&& \pi_j \deq && \fa \deq \\
&& \Y\kj \cldosuno{\eta_k} &&&& \pi_j && \fa \deq \\
&&&& \pi_k \cldosdos{\pi\kj} \dl && \dc{\tilde{\lambda}_k} && \dr \fa \\
&&&& \f_k \opdosuno{\nu_k\inv} &&&& \ii \deq \\
&& \pi_k \deq &&& \fb\sXY^j \opb{(\theta_0^j)\inv} &&& \ii \deq \\
&& \pi_k && \pi_0^j && \fc_j && \ii}}
\vcenter{\xymatrix@C=-0pc{ \ \forall \ k<j}}$.}
\vspace{2ex}
Take \mbox{$\tilde{\rho_j}= 
\vcenter{\xymatrix@C=.2pc{& \p_j \opb{(\beta_1^j)\inv} &&& \f_j \deq \\ \pi_1^j \deq && \q_j && \f_j \\ \pi_1^j \dl & \dcr{\theta_1^j} && \dr \fc_j \cl{\rho'_j} \\ \pi_j &&& \fb}}
$, 
$\ff{f}_{id_j}=(\alpha\tY_j)^{-1} \ff{f}_j$} $\quad$ and $\quad$ \mbox{$\ff{f}_{k<j}= 
\vcenter{\xymatrix@C=-0.2pc{\\ && \Y\kj \opdosuno{\eta_k\inv} &&&& \f_j \deq \\
\pi_k \deq &&& \fa_{\Y}^j \opb{(\beta_0^j)\inv} &&& \f_j \deq \\
\pi_k \deq && \pi_0^j \deq && \q_j && \f_j \\
\pi_k \deq && \pi_0^j &&& \fc_j \cl{\rho'_j} \\
\pi_k &&&& \fb\sXY^j \cldosuno{\theta_0^j} \\
&& \f_k \cldosdos{\nu_k} \\ }}
$.}

To verify that $\left\{\ff{X}\mr{\ff{f}_k}\ff{Y}_k\right\}_{k<m},\ \left\{\ff{Y}_{k<j}\ff{f}_j\Mr{\ff{f}_{k<j}}\ff{f}_k\right\}_{k<j<m}$ is a pseudo-cone, observe that axiom PC0 is satisfied by definition and axiom PC2 is vacuous because there are no \mbox{2-cells} in $\ff{J}$. So, we only need to check that axiom PC1 holds. To do that, we need to prove that the following equality holds $\forall \ k<l<j<m$:


$$\vcenter{\xymatrix@C=-0pc{\Y\kl \deq && \Y\lj && \f_j \\   \Y\kl &&& \f_l \cl{\f\lj} \\ && \f_k \cldosuno{\f\kl} }}
\vcenter{\xymatrix@C=-0pc{ \quad = \quad }}
\vcenter{\xymatrix@C=-0pc{ \Y\kl && \Y\lj && \f_j \deq \\  & \Y\kj \cl{\alpha\tY_{k,l,j}} &&& \f_j \\   && \f_k \clunodos{\f\kj} }}$$

\noindent But


$$\vcenter{\xymatrix@C=-0pc{\Y\kl \deq && \Y\lj && \f_j \\   \Y\kl &&& \f_l \cl{\f\lj} \\ && \f_k \cldosuno{\f\kl} }}
\vcenter{\xymatrix@C=-0pc{ \quad = \quad }}
\vcenter{\xymatrix@C=-0pc{\Y\kl \deq &&&& \Y\lj \opdosuno{\eta_l\inv} &&&& \f_j \deq \\
\Y\kl \deq && \pi_l \deq &&& \fa_{\Y}^j \opb{(\beta_0^j)\inv} &&& \f_j \deq \\
\Y\kl \deq && \pi_l \deq && \pi_0^j \deq && \q_j && \f_j \\
\Y\kl \deq && \pi_l \deq && \pi_0^j &&& \fc_j \cl{\rho'_j} \\
\Y\kl \deq && \pi_l &&&& \fb\sXY^j \cldosuno{\theta_0^j} \\
\Y\kl &&&& \f_l \cldosdos{\nu_l} \\
&& \f_k \cldosdos{\f\kl} }}
\vcenter{\xymatrix@C=-0pc{ \quad = \quad }}$$

$$\vcenter{\xymatrix@C=-0pc{ \quad = \quad }}
\vcenter{\xymatrix@C=-0pc{\Y\kl \deq &&&& \Y\lj \opdosuno{\eta_l\inv} &&&& \f_j \deq \\
\Y\kl \deq && \pi_l \deq &&& \fa_{\Y}^j \opb{(\beta_0^j)\inv} &&& \f_j \deq \\
\Y\kl \deq && \pi_l \deq && \pi_0^j \deq && \q_j && \f_j \\
\Y\kl \deq && \pi_l \deq && \pi_0^j &&& \fc_j \cl{\rho'_j} \\
\Y\kl && \pi_l &&&& \fb\sXY^j \cldosuno{\theta_0^j} \deq \\
& \pi_k \cl{\pi\kl} &&&&& \fb\sXY^j \\
& &&& \f_k \cltresdos{\nu_k} }}
\vcenter{\xymatrix@C=-0pc{ \quad = \quad }}
\vcenter{\xymatrix@C=-0pc{ \Y\kl && \Y\lj && \f_j \deq \\  & \Y\kj \cl{\alpha\tY_{k,l,j}} &&& \f_j \\   && \f_k \clunodos{\f\kj} }}$$

\noindent where the second equality is due to the definition of $\ff{b}^j_{\ff{X},\ff{Y}}$ and the last one to the definition $\ff{a}^j_{\ff{Y}}$.

To check that $\Bigg\{\pi_k \ff{a}\Mr{\tilde{\lambda_k}} \ff{f}_k \ff{i}\Bigg\}_{k<m}$ is a morphism of pseudo-cones, we need to verify that the following equality holds $\forall \ k<j<m$:


$$\vcenter{\xymatrix@C=-0pc{\Y\kj && \pi_j && \fa \deq \\   & \pi_k \cl{\pi\kj} \dl & \dcr{\tilde{\lambda}_k} && \dr \fa \\   & \f_k &&& \ii }}
\vcenter{\xymatrix@C=-0pc{ \quad = \quad }}
\vcenter{\xymatrix@C=-0pc{ \Y\kj \deq && \pi_j \dl & \dc{\tilde{\lambda}_j} & \dr \fa \\   \Y\kj && \f_j && \ii \deq \\   & \f_k \cl{\f\kj} &&& \ii  }}$$

\noindent But 


$$\vcenter{\xymatrix@C=-0pc{ \Y\kj \deq && \pi_j \dl & \dc{\tilde{\lambda}_j} & \dr \fa \\   \Y\kj && \f_j && \ii \deq \\   & \f_k \cl{\f\kj} &&& \ii  }}
\vcenter{\xymatrix@C=-0pc{ \quad = \quad }}
\vcenter{\xymatrix@C=-0pc{ && \Y\kj \deq &&&& \pi_j \dl & \dc{\tilde{\lambda}_j} & \dr \fa \\ 
&& \Y\kj \opdosuno{\eta_k\inv} &&&& \f_j \deq && \ii \deq \\
\pi_k \deq &&& \fa_{\Y}^j \opb{(\beta_0^j)\inv} &&& \f_j \deq && \ii \deq \\
\pi_k \deq && \pi_0^j \deq && \q_j && \f_j && \ii \deq \\
\pi_k \deq && \pi_0^j &&& \fc_j \cl{\rho'_j} &&& \ii \deq \\
\pi_k &&&& \fb\sXY^j \cldosuno{\theta_0^j} &&&& \ii \deq \\
&& \f_k \cldosdos{\nu_k} &&&&&& \ii}}
\vcenter{\xymatrix@C=-0pc{ \quad = \quad }}$$

$$\vcenter{\xymatrix@C=-0pc{ \quad = \quad }}
\vcenter{\xymatrix@C=-0pc{ && \Y\kj \opdosuno{\eta_k\inv} &&&& \pi_j \deq && \fa \deq \\
\pi_k \deq &&& \fa_{\Y} \opb{(\beta_0^j)\inv} &&& \pi_j \deq && \fa \deq  \\
\pi_k \deq && \pi_0^j \deq && \q_j \ardr && \pi_j \dc{\tilde{\gamma}_j} && \fa \ardl \\
\pi_k \deq && \pi_0^j &&& \fc_j && \ii \deq \\
\pi_k &&&& \fb\sXY^j \cldosuno{\theta_0^j} &&& \ii \deq \\
&& \f_k \cldosdos{\nu_k} &&&&& \ii}}
\vcenter{\xymatrix@C=-0pc{ \quad = \quad }}
\vcenter{\xymatrix@C=-0pc{\Y\kj && \pi_j && \fa \deq \\   & \pi_k \cl{\pi\kj} \dl & \dcr{\tilde{\lambda}_k} && \dr \fa \\   & \f_k &&& \ii }}$$

\noindent where the first equality is due to the definition of $\ff{f}_{k<j}$, the second one is due to elevators calculus plus the fact that $(\ff{f}_j,\tilde{\lambda}_j,\rho'_j)$ is a filler for diagram \eqref{diagramaconbeta} and the last one is due to the definition of $\tilde{\gamma}_j$.

To check that $\Bigg\{\ff{p}_k \ff{f}_k \Mr{\tilde{\rho_k}}\pi_k \ff{b} \Bigg\}_{\stackrel[\comw{a}]{}{k<m}}$ is a morphism of pseudo-cones, we need to verify that the following equality holds $\forall \ k<j<m$:


$$\vcenter{\xymatrix@C=-0pc{ \B\kj \dl & \dc{\p\kj} & \p_j \dr && \f_j \deq \\
\p_k \deq && \Y\kj && \f_j \\
\p_k \dl & \dcr{\tilde{\rho}_k} && \dr \f_k \clb{\f\kj} \\
\pi_k &&& \fb}}
\vcenter{\xymatrix@C=-0pc{ \quad = \quad }}
\vcenter{\xymatrix@C=-0pc{ \B\kj \deq && \p_j \dl & \dc{\tilde{\rho}_j} & \dr \f_j \\
\B\kj && \pi_j && \fb \deq \\
& \pi_k \clb{\pi\kj} &&& \fb}}$$

\noindent But


$$\vcenter{\xymatrix@C=-0pc{ \B\kj \dl & \dc{\p\kj} & \p_j \dr && \f_j \deq \\
\p_k \deq && \Y\kj && \f_j \\
\p_k \dl & \dcr{\tilde{\rho}_k} && \dr \f_k \cl{\f\kj} \\
\pi_k &&& \fb}}
\vcenter{\xymatrix@C=-0pc{ \quad = \quad }}
 \vcenter{\xymatrix@C=-0pc{\B\kj \dl && \dc{\p\kj} && \dr \p_j &&&& \f_j \deq \\
 \p_k \deq &&&& \Y\kj \opdosuno{\eta_k\inv} &&&& \f_j \deq \\
 \p_k \deq && \pi_k \deq &&& \fa_{\Y}^j \opb{(\beta_0^j)\inv} &&& \f_j \deq \\
 \p_k \deq && \pi_k \deq && \pi_0^j \deq && \q_j && \f_j \\
 \p_k \deq && \pi_k \deq && \pi_0^j &&& \fc_j \cl{\rho'_j} \\
 \p_k \deq && \pi_k &&&& \fb\sXY^j \cldosuno{\theta_0^j} \\
 \p_k \dl && \dc{\tilde{\rho}_k} && \dr \f_k \cldosdos{\nu_k} \\
 \pi_k &&&& \fb }}
\vcenter{\xymatrix@C=-0pc{ \quad = \quad }}$$

$$\vcenter{\xymatrix@C=-0pc{  = \quad }}
\vcenter{\xymatrix@C=-0pc{ \B\kj \dl && \dc{\p\kj} && \dr \p_j &&&& \f_j \deq \\
\p_k \deq &&&& \Y\kj \opdosuno{\eta_k\inv} &&&& \f_j \deq \\
\p_k \deq && \pi_k \deq &&& \fa_{\Y}^j \opb{(\beta_0^j)\inv} &&& \f_j \deq \\
\p_k \deq && \pi_k \deq && \pi_0^j \deq && \q_j && \f_j \\
\p_k \deq && \pi_k \deq && \pi_0^j &&& \fc_j \cl{\rho'_j} \\
\p_k \dl & \dc{\mu_k\inv} & \dr \pi_k &&& \fb\sXY^j \clunodos{\theta_0^j} \opunodos{\!\!\!\!\!\!\!\!\!\!\!\!(\theta_0^j)\inv} \\
\pi_k \deq && \biLim{k<j}{\p_k} \dl & \dc{\alpha_j} & \dr \pi_0^j &&& \fc_j \deq \\
\pi_k && \fa_{\B}^j && \pi_1^j \dl & \dcr{\theta_1^j} && \dr \fc_j \\
& \B\kj \cl{\eps_k} &&& \pi_j &&& \fb \deq \\
&& \pi_k \clunodos{\pi\kj} &&&&& \fb }}
\vcenter{\xymatrix@C=-0pc{  = \quad }}
\vcenter{\xymatrix@C=-0pc{ \B\kj \deq &&& \p_j \opbymedio{(\beta_1^j)\inv} &&& \f_j \deq \\
\B\kj \deq && \pi_1^j \deq && \q_j && \f_j \\
\B\kj \deq && \pi_1^j \dl & \dcr{\theta_1^j} && \dr \fc_j \cl{\rho'_j} \\
\B\kj && \pi_j &&& \fb \deq \\
& \pi_k \cl{\pi\kj} &&&& \fb}}
\vcenter{\xymatrix@C=-0pc{  = }}$$

$$\vcenter{\xymatrix@C=-0pc{ \quad = \quad }}
\vcenter{\xymatrix@C=-0pc{ \B\kj \deq && \p_j \dl & \dc{\tilde{\rho}_j} & \dr \f_j \\
\B\kj && \pi_j && \fb \deq \\
& \pi_k \cl{\pi\kj} &&& \fb}}$$

\noindent where the first equality is due to the definition of $\ff{f}_{k<j}$, the second one is due to \eqref{A32}, the third one is due to item $d)$ from \ref{lema1.2} and the last one is due to the definition of $\tilde{\rho}_j$.

Then, by the universal property of $\biLim{k<m}{\ff{Y}_k}$, there exist a morphism $\ff{X}\mr{\ff{f}}\biLim{k<m}{\ff{Y}_k}$ and invertible 2-cells $\pi_k \ff{f} \Mr{\delta_k} \ff{f}_k \ \forall \ k<m$ such that $ 
\vcenter{\xymatrix@C=-0pc{ \Y\kj && \pi_j && \f \deq \\   & \pi_k \cl{\pi\kj} &&& \f \\   && \f_k \clunodos{\delta_k} }}
\vcenter{\xymatrix@C=-0pc{ \quad = \quad }}
\vcenter{\xymatrix@C=-0pc{ \Y\kj \deq && \pi_j && \f \\   \Y\kj &&& \f_j \cl{\delta_j} \\    && \f_k \cldosuno{\f\kj}  }}
\forall \ k<j<m$; and there also exist invertible 2-cells $\ff{a}\Mr{\lambda}\ff{f}\ff{i}$, $\ff{b}\Mr{\rho}\biLim{k<m}{p_k} \ff{f}$ such that $\pi_k \lambda= 
\vcenter{\xymatrix@C=-0pc{ & \pi_k \dl & \dcr{\tilde{\lambda}_k} && \dr \fa \\  & \f_k \op{\delta_k\inv} &&& \ii \deq \\   \pi_k && \f && \ii }}
$ and \mbox{$\pi_k \rho= 
\vcenter{\xymatrix@C=-0pc{\pi_k \dl & \dc{\mu_k} & \dr \biLim{k<m}{\p_k} && \f \deq \\
\p_k \deq && \pi_k && \f \\
\p_k \dl & \dcr{\tilde{\rho}_k} && \dr \f_k \cl{\delta_k} \\
\pi_k &&& \fb}}
\ \forall \ k<m$.} To check that $(\ff{f},\lambda,\rho)$ is the filler that we were looking for, it is enough to check that the following equality holds $\forall \ k<m$:  


$$\vcenter{\xymatrix@C=-0pc{\pi_k \deq && \biLim{k<m}{\p_k} \deq &&& \fa \op{\lambda} \\
\pi_k \deq && \biLim{k<m}{\p_k} && \f && \ii \deq \\
\pi_k &&& \fb \cl{\rho} &&& \ii}}
\vcenter{\xymatrix@C=-0pc{ \quad = \quad }}
\vcenter{\xymatrix@C=-0pc{ \pi_k \deq && \biLim{k<m}{\p_k} \dl & \dc{\gamma} & \dr \fa \\   \pi_k && \fb && \ii }}$$

\vspace{5ex}

\noindent But

\vspace{5ex}


$$\vcenter{\xymatrix@C=-0pc{\pi_k \deq && \biLim{k<m}{\p_k} \deq &&& \fa \op{\lambda} \\
\pi_k \deq && \biLim{k<m}{\p_k} && \f && \ii \deq \\
\pi_k &&& \fb \cl{\rho} &&& \ii}}
\vcenter{\xymatrix@C=-0pc{ \quad = \quad }}
\vcenter{\xymatrix@C=-0pc{ \pi_k \deq && \biLim{k<m}{\p_k} \deq &&& \fa \op{\lambda} \\
\pi_k \dl & \dc{\mu_k} & \dr \biLim{k<m}{\p_k} && \f \deq && \ii \deq \\
\p_k \deq && \pi_k && \f && \ii \deq \\
\p_k \dl & \dcr{\tilde{\rho}_k} && \dr \f_k \cl{\delta_k} &&& \ii \deq \\
\pi_k &&& \fb &&& \ii}}
\vcenter{\xymatrix@C=-0pc{ \quad = \quad }}
\vcenter{\xymatrix@C=-0pc{ \pi_k \dl & \dc{\mu_k} & \dr \biLim{k<m}{\p_k} && \fa \deq \\
\p_k \deq && \pi_k \dl & \dc{\tilde{\lambda}_l} & \dr \fa \\
\p_k \dl & \dc{\tilde{\rho}_k} & \dr \f_k && \ii \deq \\
\pi_k && \fb && \ii}}
\vcenter{\xymatrix@C=-0pc{ =  }}$$

$$\vcenter{\xymatrix@C=-0pc{ \quad = \quad }}
\vcenter{\xymatrix@C=0.3pc{ & \pi_k \dl & \dcr{\mu_k} && \dr \biLim{k<m}{\p_k} && \fa \deq \\
& \p_k \deq &&& \pi_k \dl & \dc{\tilde{\lambda}_l} & \dr \fa \\
& \p_k \opb{(\beta_1^j)\inv} &&& \f_k \deq && \ii \deq \\
\pi_q^k \deq && \q_k && \f_k && \ii \deq \\
\pi_1^k \dl & \dcr{\theta_1^k} && \dr \fc_k \cl{\rho'_k} &&& \ii \deq \\
\pi_k &&& \fb &&& \ii}}
\vcenter{\xymatrix@C=-0pc{ \quad = \quad }}
\vcenter{\xymatrix@C=0.3pc{  & \pi_k \dl & \dcr{\mu_k} && \dr \biLim{k<m}{\p_k} && \fa \deq \\  
& \p_k \opb{(\beta_1^j)\inv} &&& \pi_k \deq && \fa \deq \\
\pi_1^k \deq && \q_k \ardr && \pi_k \dc{\tilde{\gamma}_k} && \fa \ardl \\
\pi_q^k &&& \fc_k && \ii \deq \\
& \fb \clunodos{\theta_1^k} &&&& \ii}}
\vcenter{\xymatrix@C=-0pc{  =  }}$$

$$\vcenter{\xymatrix@C=-0pc{ \quad = \quad }}
\vcenter{\xymatrix@C=-0pc{ \pi_k \deq && \biLim{k<m}{\p_k} \dl & \dc{\gamma} & \dr \fa \\   \pi_k && \fb && \ii }}$$

\noindent where the first equality is due to the definition of $\rho$, the second one is due to elevators calculus plus the definition of $\lambda$, the third one is due to the definition of $\tilde{\rho}_k$, the fourth one is due to elevators calculus plus the fact that $(\ff{f}_k,\tilde{\lambda}_k,\rho'_k)$ is a filler for diagram \eqref{diagramaconbeta} and the last one is due to the definition of $\tilde{\gamma}_k$.

Finally, since $\cc{C}$ is a closed 2-bmodel 2-category, $\ff{P}_j \mr{\pi_1^j} \ff{B}_j$ is both a fibration and a weak equivalence. Also, by definition of weak equivalences, $\ff{p}_j$ is a weak equivalence. So, $\ff{q}_j$ is also a weak equivalence by axiom 2-M5.

$\Leftarrow)$ By \ref{lema1.2}, it is clear that $\ff{p}$ is a fibration. Let's check inductively that $\ff{p}_j$ is a weak equivalence: If $j=0$, $\ff{p}_0\cong \ff{q}_0$ and so is a weak equivalence. Now suppose that $\ff{p}_k$ is a weak equivalence $\forall \ k<j$ and consider $\ff{p}$ as an object of \mbox{$p\cc{H}om_p(\left\{k\in \ff{J}\ |\ k<j\right\}^{op},\cc{C})$.} Since $\ff{p}\in \pCJ$ is a fibration, by \ref{lema1.2}, $\ff{p}\in p\cc{H}om_p(\left\{k\in \ff{J}\ |\ k<j\right\}^{op},\cc{C})$ is also a fibration. Plus, we know that $\ff{p}$ is a weak equivalence as an object of \mbox{$p\cc{H}om_p(\left\{k\in \ff{J}\ |\ k<j\right\}^{op},\cc{C})$} and then, by \ref{lema0.2} $\biLim{k<j}{\ff{p}_k}\in \cc{C}$ is both a fibration and a weak equivalence. Then, since $\cc{C}$ is a closed 2-bmodel 2-category, $\pi_1^j$ is a weak equivalence. We also know that $\ff{q}_j$ is a weak equivalence and so $\ff{p}_j\cong \pi_1^j \ff{q}_j\in \cc{C}$ is a weak equivalence as we wanted to prove.  
\end{proof} 

\begin{theorem}\label{CJdeclosed 2-modelos}
$\pCJ$ with the structure provided in \ref{estructuraenCJ} is a closed \mbox{2-bmodel} 2-category.
\end{theorem}

\begin{proof} $ $

{\bfseries Axiom 2-M0b}: It is clear from \ref{limites que alcanzan} since bi-limits, bi-colimits, bi-tensors and bi-cotensors in $\pCJ$ are computed pointwise and therefore exist (see \ref{pointwise en phomp} and \ref{tensorptoapto}). \cqd

\vspace{1ex}
 
{\bfseries Axiom 2-M2}: We will first do the case where $\ff{p}$ is a fibration and $\ff{i}$ is both a cofibration and a weak equivalence:
 
 It is enough to check that we can factor $\ff{p}_j \ff{i}_j \Mr{\gamma_j \cong} \ff{f}_j \ \forall \ j\in \ff{J}$, $\ff{i}=\left\{\ff{i}_j\right\}_{j\in \ff{J}}$ and $\ff{p}=\left\{\ff{p}_j\right\}_{j\in \ff{J}}$ are pseudo-natural transformations, $\ff{i}_j$ is both a cofibration and a weak equivalence $\forall \ j\in \ff{J}$, $\ff{q}_j$ (associated to $\ff{p}$ as in \ref{lema1.2}) is a fibration $\forall \ j\in \ff{J}$ and $\gamma=\left\{\gamma_j\right\}_{j\in \ff{J}}$ is a modification. We are going to do this by induction in $j$:
 
 If $j$ is the initial object of $\ff{J}$, since $\cc{C}$ is a closed 2-bmodel 2-category, we can factorize $\ff{f}_0$ as $\ff{p}_0 \ff{i}_0 \Mr{\gamma_0 \cong} \ff{f}_0$ where $\ff{i}_0$ is both a cofibration and a weak equivalence and $\ff{p}_0$ is a fibration. As $\ff{q}_0\cong \ff{p}_0$, $\ff{p}_0$ is a fibration. 
 
 Now, suppose that we have already defined $\ff{p}_k$, $\ff{i}_k$ and $\gamma_k \ \forall \ k<j$ and let's define $\ff{p}_j$, $\ff{i}_j$ and $\gamma_j$:
 
 Consider the following diagram in $\cc{C}$:
 
\begin{equation}\label{diagramaconQ}
 \xymatrix@C=3pc{\ff{X}_j  \ar[rr]^{\ff{a}_{\ff{X}}^j} \ar@{-->}[rd]^{\ff{a}_j} \ar@/_4ex/[rdd]_{\ff{f}_j} & &  \biLim{k<j}{\ff{X}_k} \ar[d]^{\biLim{k<j}{\ff{i}_k}}\\
            & \ff{Q}_j  \ar@{}[l]|{\cong \; \Downarrow \; \theta_1^j} \ar@{}[ur]|{\cong \; \Uparrow \; \theta_0^j} \ar[r]^{\ff{h}_0^j} \ar[d]_{\ff{h}_1^j} \ar[d]_{\pi_1^j} \ar@{}[rd]^{bipb \quad\quad\quad\quad\quad\quad\quad\quad\quad\quad}_{\quad\quad\quad\quad\quad\quad\quad\quad\quad\quad \cong \; \Downarrow \; \delta_j} 
            & \biLim{k<j}{\ff{Z}_k} \ar[d]^{\biLim{k<j}{\ff{p}_k}}\\
            & \ff{Y}_j \ar[r]_{\ff{a}_{\ff{Y}}^j} & \biLim{k<j}{\ff{Y}_k}}
\end{equation}
            
\noindent where $\biLim{k<j}{\ff{p}_k}$ is induced by the pseudo-cone $\left\{\ff{p}_k \pi_k\right\}_{k<j},\  \Bigg\{ 
\vcenter{\xymatrix@C=-0.2pc@R=1pc{ \Y\kl \dl & \dc{\p\kl} & \dr \p_l && \pi_l \deq \\   \p_k \deq && \Z\kl && \pi_l \\   \p_k &&& \pi_k \cl{\pi\kl \ } }}
\Bigg\}_{k<l<j}\cup \Bigg\{ 
\vcenter{\xymatrix@C=-0.2pc@R=1pc{ \dl & \Y_{id_k} \dc{(\alpha\tY_k)\inv} & \drbis & \p_k \deq && \pi_k \deq \\  & id_{\Y_k} \ardr && \p_k \dc{=} && \pi_k \ardl \\ && \p_k && \pi_k}}
\Bigg\}_{k<j}$, 
$\biLim{k<j}{\ff{i}_k}$ is induced by the pseudo-cone $\left\{\ff{i}_k \pi_k\right\}_{k<j},\ $  \mbox{$\Bigg\{ 
\vcenter{\xymatrix@C=-0.2pc@R=1pc{ \Z\kl \dl & \dc{\ii\kl} & \dr \ii_l && \pi_l \deq \\   \ii_k \deq && \X\kl && \pi_l \\   \ii_k &&& \pi_k \cl{\pi\kl \ } }}
\Bigg\}_{k<l<j}\cup \Bigg\{ 
\vcenter{\xymatrix@C=-0.2pc@R=1pc{ \dl & \Z_{id_k} \dc{(\alpha\tZ_k)\inv} & \drbis & \ii_k \deq && \pi_k \deq \\  & id_{\Z_k} \ardr && \ii_k \dc{=} && \pi_k \ardl \\  && \ii_k && \pi_k}}
\Bigg\}_{k<j}$,} 
$\ff{a}_{\ff{X}}^j$ is induced by the pseudo-cone $\left\{\ff{X}_{k<j}\right\}_{k<j},\ $ \mbox{$\Bigg\{ 
\vcenter{\xymatrix@C=-0.2pc@R=1pc{ \X\kl && \X\lj \\ & \X\kj \cl{\alpha\tX_{k,l,j}} }}
\Bigg\}_{k<l<j}\cup \Bigg\{ 
\vcenter{\xymatrix@C=-0.2pc@R=1pc{\dl & \X_{id_k} \dc{(\alpha\tX_k)\inv} & \dr && \X\kj \deq \\  & id_{\X_k} &&& \X\kj \\   &&& \X\kj \cldosuno{=} }}
\Bigg\}_{k<j}$,} 
$\ff{a}_{\ff{Y}}^j$ is induced by the pseudo-cone $\left\{\ff{Y}_{k<j}\right\}_{k<j},\ \Bigg\{ 
\vcenter{\xymatrix@C=-0.2pc@R=1pc{ \Y\kl && \Y\lj \\ & \Y\kj \cl{\alpha\tY_{k,l,j}} }}
\Bigg\}_{k<l<j}\cup \Bigg\{ 
\vcenter{\xymatrix@C=-0.2pc@R=1pc{\dl & \Y_{id_k} \dc{(\alpha\tY_k)\inv} & \dr && \Y\kj \deq \\  & id_{\Y_k} &&& \Y\kj \\  &&& \Y\kj \cldosuno{=} }}
\Bigg\}_{k<j}$, 
and so we have:

\begin{enumerate}
 
 \item[a)] invertible 2-cells $\pi_k \biLim{k<j}{\ff{p}_k} \Mr{\mu_k} \ff{p}_k \pi_k \ \forall \ k<j$ such that $ 
 \vcenter{\xymatrix@C=-0pc{\Y\kl && \pi_l && \biLim{k<j}{\p_k} \deq \\
 & \pi_k \cl{\pi\kl} \dl & \dcr{\mu_k} && \dr \biLim{k<j}{\p_k} \\
 & \p_k &&& \pi_k }}
\vcenter{\xymatrix@C=-0pc{ \quad = \quad }}
\vcenter{\xymatrix@C=-0pc{ \Y\kl \deq && \pi_l \dl & \dc{\mu_l} & \dr \biLim{k<j}{\p_k} \\
\Y\kl \dl & \dc{\p\kl} & \dr \p_l && \pi_l \deq \\
\p_k \deq && \Z\kl && \pi_l \\
\p_k &&& \pi_k \cl{\pi\kl}}}
 \ \forall \ k<l<j$
 
 \item[b)] invertible 2-cells $\pi_k \biLim{k<j}{\ff{i}_k}\Mr{\epsilon_k} \ff{i}_k \pi_k \ \forall \ k<j$ such that $ 
 \vcenter{\xymatrix@C=-0pc{\Z\kl && \pi_l && \biLim{k<j}{\ii_k} \deq \\
 & \pi_k \cl{\pi\kl} \dl & \dcr{\eps_k} && \dr \biLim{k<j}{\ii_k} \\
 & \ii_k &&& \pi_k }}
\vcenter{\xymatrix@C=-0pc{ \quad = \quad }}
\vcenter{\xymatrix@C=-0pc{ \Z\kl \deq && \pi_l \dl & \dc{\eps_l} & \dr \biLim{k<j}{\ii_k} \\
\Z\kl \dl & \dc{\ii\kl} & \dr \ii_l && \pi_l \deq \\
\ii_k \deq && \X\kl && \pi_l \\
\ii_k &&& \pi_k \cl{\pi\kl}}}
 \ \forall \ k<l<j$
 
 \item[c)] invertible 2-cells $\pi_k \ff{a}_{\ff{X}}^j \Mr{\varphi_k} \ff{X}_{k<j} \ \forall \ k<j$ such that $ 
 \vcenter{\xymatrix@C=-0pc{\X\kl && \pi_l && \fa_{\X}^j \deq \\   & \pi_k \cl{\pi\kl} &&& \fa_{\X}^j \\  && \X\kj \clunodos{\varphi_k} }}
\vcenter{\xymatrix@C=-0pc{ \quad = \quad }}
\vcenter{\xymatrix@C=-0pc{\X\kl \deq && \pi_l && \fa_{\X}^j \\   \X\kl &&& \X\lj \cl{\varphi_l} \\   && \X\kj \cldosuno{\alpha\tX_{k,l,j}}  }}
 \ \forall \ k<l<j$
 
 \item[d)] invertible 2-cells $\pi_k \ff{a}_{\ff{Y}}^j \Mr{\psi_k} \ff{Y}_{k<j} \ \forall \ k<j$ such that $ 
 \vcenter{\xymatrix@C=-0pc{\Y\kl && \pi_l && \fa_{\Y}^j \deq \\   & \pi_k \cl{\pi\kl} &&& \fa_{\Y}^j \\  && \Y\kj \clunodos{\psi_k} }}
\vcenter{\xymatrix@C=-0pc{ \quad = \quad }}
\vcenter{\xymatrix@C=-0pc{\Y\kl \deq && \pi_l && \fa_{\Y}^j \\   \Y\kl &&& \Y\lj \cl{\psi_l} \\   && \Y\kj \cldosuno{\alpha\tY_{k,l,j}}  }}
 \ \forall \ k<l<j$
 
 \item[e)] \mbox{$ 
 \vcenter{\xymatrix@C=-0.2pc{ \pi_k \deq && \biLim{k<j}{\p_k} \dl & \dc{\delta_j} & \dr \h_0^j && \fa_j \deq \\
 \pi_k \deq && \fa_{\Y}^j \deq && \h_1^j && \fa_j \\
 \pi_k && \fa_{\Y}^j &&& \f_j \cl{\theta_1^j}  }}
\vcenter{\xymatrix@C=-0pc{ \quad = \quad }}
\vcenter{\xymatrix@C=-0.2pc{ \pi_k \dl & \dc{\mu_k} & \dr \biLim{k<j}{\p_k} && \h_0^j \dl & \dc{\theta_0^j} & \dr \fa_j \\
\p_k \deq && \pi_k \dl & \dc{\eps_k} & \dr \biLim{k<j}{\ii_k} && \fa_{\X}^j \deq \\
\p_k && \ii_k && \pi_k && \fa_{\X}^j \\
& \f_k \cl{\gamma_k} \dl && \dc{\f\kj\inv} && \dr \X\kj \cl{\varphi_k} \\
& \Y\kj \op{\psi_k\inv} &&&& \f_j \deq \\
\pi_k && \fa_{\Y}^j &&& \f_j}}
 \ \forall \ k<j$}
\end{enumerate}

Since $\cc{C}$ is a closed 2-bmodel 2-category, $\ff{a}_j$ can be factored as $\ff{q}'_j \ff{i}_j \Mr{\gamma'_j \cong} \ff{a}_j$ where $\ff{i}_j$ is both a cofibration and a weak equivalence and $\ff{q}'_j$ is a fibration. Consider $\ff{p}_j=\ff{h}_1^j \ff{q}'_j$ and $\gamma_j=\ff{h}_1^j \gamma'_j$. Let's check that this data satisfies the desired properties:

It can be easily checked that $\ff{Z}$ is a pseudo-functor ($\ff{Z}_{k<j}=\pi_k \ff{h}_0^j \ff{q}'_j$, $\alpha\tZ_j=id_{\ff{Z}_j}$, $\alpha\tZ_{k,l,j}= 
\vcenter{\xymatrix@C=-0pc{\Z\kl \deq &&&& \Z\lj \opdosdos{=} \\
\Z\kl && \pi_l && \h_0^l \deq && \q'_j \deq \\
& \pi_k \cl{\pi\kl} &&& \h_0^j && \q'_j \\
&&&& \Z\kj \cltresdos{=} }}$). 

We are now going to prove that $\ff{i}$ is pseudo-natural:

\noindent We define $\ff{i}_{k<j}= 
\vcenter{\xymatrix@C=-0pc{ && \Z\kj \opdosdos{=} &&&& \ii_j \deq \\
\pi_k \deq && \h_0^j \deq && \q'_j && \ii_j \\
\pi_k \deq && \h_0^j \dl & \dcr{\theta_0^j} && \dr \fa_j \cl{\gamma'_j} \\
\pi_k \dl & \dc{\eps_k} & \dr \biLim{k<j}{\ii_k} &&& \fa_{\X}^j \deq \\
\ii_k \deq && \pi_k &&& \fa_{\X}^j \\
\ii_k &&&& \X\kj \cldosuno{\varphi_k} }}$ 
$\quad$ and $\quad$ 
$\ff{i}_{id_j}= 
\vcenter{\xymatrix@C=-0pc{ \dl & \Z_{id_j} \dc{(\alpha\tZ_j)\inv} & \drmediobis & \ii_j \deq \\  & id_{\Z_j} \dl & \dc{=} & \dr \ii_j \\  & \ii_j \deq && id_{\X_j} \dcell{\alpha\tX_j} \\   & \ii_j && \X_{id_j}}}
$. \mbox{PN0 is satisfied} by definition and PN2 is vacuous because there are no 2-cells in $\ff{J}$, so we only need to check PN1: consider $k<l<j$, we want to check that the following equality holds:

$$\vcenter{\xymatrix@C=-.5pc{\ff{Z}_{k<l} \deq 
                             && \ff{Z}_{l<j} \dl 
                             & \dc{\ff{i}_{l<j}} 
                             & \ff{i}_j \dr \\
                             \ff{Z}_{k<l} \dl 
                             & \dc{\ff{i}_{k<l}} 
                             & \ff{i}_l \dr 
                             && \ff{X}_{l<j} \deq \\
                             \ff{i}_k \deq 
                             && \ff{X}_{k<l} 
                             & 
                             & \ff{X}_{l<j} \\
                             \ff{i}_k 
                             &&& \ff{X}_{k<j} \clb{\alpha\tX_{k,l,j}}  &&   }}
\vcenter{\xymatrix@C=-.4pc{\quad = \quad \quad }}
\vcenter{\xymatrix@C=-.5pc{\ff{Z}_{k<l} 
                           & 
                           & \ff{Z}_{l<j} 
                           &&& \ff{i}_{j} \deq \\
                           & \ff{Z}_{k<j} \clb{\alpha\tZ_{k,l,j}} \dl 
                           & \dc{\ff{i}_{k<j}} &  
                           && \ff{i}_{j} \dr \\
                           & \ff{i}_{k} 
                           &&&& \ff{X}_{k<j}}}$$

\noindent But 

$$\vcenter{\xymatrix@C=-.5pc{\ff{Z}_{k<l} \deq && \ff{Z}_{l<j} \dl & \dc{\ff{i}_{l<j}} & \ff{i}_j \dr \\
                             \ff{Z}_{k<l} \dl & \dc{\ff{i}_{k<l}} & \ff{i}_l \dr && \ff{X}_{l<j} \deq \\
                             \ff{i}_k \deq && \ff{X}_{k<l} && \ff{X}_{l<j} \\
                             \ff{i}_k &&& \ff{X}_{k<j} \clb{\alpha\tX_{k,l,j}} }}
\vcenter{\xymatrix@C=-.4pc{\quad = \quad }}
\vcenter{\xymatrix@C=-0.3pc@R=1.5pc{\pi_k \deq && \ff{h}_0^l \deq && \ff{q}'_l \deq && \pi_l \deq && \ff{h}_0^j \deq && \ff{q}'_j &  & \ff{i}_j \\
                          \pi_k \deq && \ff{h}_0^l \deq & \; & \ff{q}'_l \deq && \pi_l \deq && \ff{h}_0^j \dl & \dcr{\theta_0^j} && \dr \ff{a}_j \cl{\gamma'_j} \\
                          \pi_k \deq && \ff{h}_0^l \deq && \ff{q}'_l \deq && \pi_l \dl & \dc{\epsilon_l} & \biLim{k<j}{i_k} \dr &&& \ff{a}_{\ff{X}}^j \deq \\
                          \pi_k \deq && \ff{h}_0^l \deq && \ff{q}'_l \deq && \ff{i}_l \deq && \pi_l &&& \ff{a}_{\ff{X}}^j \\
                          \pi_k \ardr && \ff{h}_0 & \dc{\ff{i}_{k<l}} & \ff{q}'_l && \ff{i}_l \ardl &&& \ff{X}_{l<j} \deq \clunodos{\varphi_l} \\
                          & \ff{i}_k \deq &&&& \ff{X}_{k<l} \ardrr && \dc{\alpha^\ff{X}_{k,l,j}} && \ff{X}_{l<j} \ardll \\
                          & \ff{i}_k &&&&&& \ff{X}_{k<j} }}
\vcenter{\xymatrix@C=-.4pc{ = \quad }}$$

$$
\vcenter{\xymatrix@C=-.3pc@R=1.5pc{\pi_k \deq && \ff{h}_0^l \deq && \ff{q}'_l \deq && \pi_l \deq && \ff{h}_0^j \deq && \ff{q}'_j && \ff{i}_j \\
                          \pi_k \deq && \ff{h}_0^l \deq && \ff{q}'_l \deq && \pi_l \deq && \ff{h}_0^j \dl & \dcr{\theta_0^j} && \dr \ff{a}_j \cl{\gamma'_j} \\
                          \pi_k \deq && \ff{h}_0^l \deq && \ff{q}'_l \deq && \pi_l \dl & \dc{\epsilon_l} & \biLim{k<j}{i_k} \dr &&& \ff{a}_{\ff{X}}^j \deq \\
                          \pi_k \ardr && \ff{h}_0^l & \dc{\ff{i}_{k<l}} & \ff{q}'_l && \ff{i}_l \ardl && \pi_l \deq &&& \ff{a}_{\ff{X}}^j \deq \\
                          & \ff{i}_k \deq &&&& \ff{X}_{k<l} &&& \pi_l &&& \ff{a}_{\ff{X}}^j \deq \\
                          & \ff{i}_k \deq &&&&& \pi_k \clunodos{\pi_{k<l}} &&&&& \ff{a}_{\ff{X}}^j \\
                          & \ff{i}_k &&&&&&& \ff{X}_{k<j} \cldostres{\varphi_k} }}
\vcenter{\xymatrix@C=-.4pc{\quad = \quad }}
\vcenter{\xymatrix@C=-.3pc@R=1.5pc{\pi_k \deq && \ff{h}_0^l \deq && \ff{q}'_l \deq && \pi_l \deq && \ff{h}_0^j \deq && \ff{q}'_j && \ff{i}_j \\
				    \pi_k \deq && \ff{h}_0^l \deq && \ff{q}'_l \deq && \pi_l \deq && \ff{h}_0^j \dl & \dcr{\theta_0^j} && \dr \ff{a}_j \cl{\gamma'_j} \\
				    \pi_k && \ff{h}_0^l && \ff{q}'_l && \pi_l && \biLim{k<j}{i_k} \deq &&& \ff{a}_{\ff{X}}^j \deq \\
				    &&& \pi_k \cltrestres{\pi_{k<l}} \dl && \dcr{\epsilon_k} &&& \biLim{k<j}{i_k} \dr &&& \ff{a}_{\ff{X}}^j \deq \\
				    &&& \ff{i}_k \deq &&&&& \pi_k &&& \ff{a}_{\ff{X}}^j \\
				    &&& \ff{i}_k &&&&&&& \ff{X}_{k<j} \cldosuno{\varphi_k} }}$$
                          
$$
\vcenter{\xymatrix@C=-.4pc{\quad = \quad }}
\vcenter{\xymatrix@C=-.2pc{\pi_k \ardrr
                           & \ff{h}_0^l 
                           & \dc{\pi_{k<l}}
                           & \ff{q}'_l 
                           & \pi_l \ardll
                           && \ff{h}_0^j \deq 
                           & \ff{q}'_j \deq
                           && \ff{i}_j \deq \\
                           && \pi_k \deq 
                           &&&& \ff{h}_0^j \deq 
                           & \ff{q}'_j \ardr
                           & \dc{\gamma'_j}
                           & \ff{i}_j \ardl \\
                           && \pi_k \deq 
                           &&&& \ff{h}_0^j \dl 
                           & \dc{\theta_0^j} 
                           & \ff{a}_j \dr  
                           & \\
                           && \pi_k \dl
                           && \dc{\epsilon_k} 
                           && \biLim{k<j}{\ff{i}_k} \dr 
                           && \ff{a}_{\ff{X}}^j \deq
                           & \\
                           && \ff{i}_k \deq
                           &&&& \pi_k \ardr
                           & \dc{\varphi_k}
                           & \ff{a}_{\ff{X}}^j \ardl
                           & \\
                           && \ff{i}_k 
                           &&&&& \ff{X}_{k<j}
                           && }}
\vcenter{\xymatrix@C=-.4pc{\quad = \quad }}
\vcenter{\xymatrix@C=-.5pc{\ff{Z}_{k<l} \ardr
                           & \dc{\alpha\tZ_{k,l,j}} 
                           & \ff{Z}_{l<j} \ardl 
                           &&& \ff{i}_{j} \deq \\
                           & \ff{Z}_{k<j} \dl 
                           & \dcr{\ff{i}_{k<j}} & 
                           && \ff{i}_{j} \dr \\
                           & \ff{i}_{k} 
                           &&&& \ff{X}_{k<j}}}$$

\noindent where the second equality is due to the elevators calculus plus c), the third one is due to b) and the fourth one is due to elevators calculus again.

Now we are going to prove that $\ff{p}$ is pseudo-natural:

\noindent We define $\ff{p}_{k<j}= 
\vcenter{\xymatrix@C=-0.1pc@R=1pc{& \Y\kj \op{\psi_k\inv} &&&& \p_j \op{=} \\
\pi_k \deq && \fa_{\Y}^j \dl & \dc{\delta_j\inv} & \dr \h_1^j && \q'_j \deq \\
\pi_k \dl & \dc{\mu_k} & \dr \biLim{k<j}{\p_k} && \h_0^j \deq && \q'_j \deq \\
\p_k \deq && \pi_k && \h_0^j && \q'_j \\
\p_k &&&& \Z\kj \cldosdos{=} }}
$ 
$\quad$ and $\quad$ 
$\ff{p}_{id_j}= 
\vcenter{\xymatrix@C=0pc{ \Y_{id_j} \dcellbymedio{(\alpha\tY_j)\inv} && \p_j \deq \\   id_{\Y_j} \dl & \dc{=} & \dr \p_j \\   \p_j \deq && id_{\Z_j} \dcell{\alpha\tZ_j} \\   \p_j && \Z_{id_j} }}
$. PN0 is satisfied by definition and PN2 is vacuous because there are no 2-cells in $\ff{J}$, so we only need to check PN1: consider $k<l<j$, we want to check that the following equality holds:

$$\vcenter{\xymatrix@C=-.4pc{\ff{Y}_{k<l} \deq 
                             && \ff{Y}_{l<j} \dl 
                             & \dc{\ff{p}_{l<j}} 
                             & \ff{i}_j \dr \\
                             \ff{Y}_{k<l} \dl 
                             & \dc{\ff{p}_{k<l}} 
                             & \ff{p}_l \dr 
                             && \ff{Z}_{l<j} \deq \\
                             \ff{p}_k \deq 
                             && \ff{Z}_{k<l} \ardr
                             & \dc{\alpha\tZ_{k,l,j}} 
                             & \ff{Z}_{l<j} \ardl \\
                             \ff{p}_k 
                             &&& \ff{Z}_{k<j} &&}}
\vcenter{\xymatrix@C=-.4pc{\quad = \quad \quad }}
\vcenter{\xymatrix@C=-.4pc{\ff{Y}_{k<l} \ardr
                           & \dc{\alpha\tY_{k,l,j}} 
                           & \ff{Y}_{l<j} \ardl 
                           &&& \ff{p}_{j} \deq \\
                           & \ff{Y}_{k<j} \dl 
                           & \dc{\ff{p}_{k<j}} & 
                           && \ff{p}_{j} \dr \\
                           & \ff{p}_{k} 
                           &&&& \ff{Z}_{k<j}}}$$

\noindent But

$$\vcenter{\xymatrix@C=-.4pc{\ff{Y}_{k<l} \deq 
                             && \ff{Y}_{l<j} \dl 
                             & \dc{\ff{p}_{l<j}} 
                             & \ff{i}_j \dr \\
                             \ff{Y}_{k<l} \dl 
                             & \dc{\ff{p}_{k<l}} 
                             & \ff{p}_l \dr 
                             && \ff{Z}_{l<j} \deq \\
                             \ff{p}_k \deq 
                             && \ff{Z}_{k<l} \ardr
                             & \dc{\alpha\tZ_{k,l,j}} 
                             & \ff{Z}_{l<j} \ardl \\
                             \ff{p}_k 
                             &&& \ff{Z}_{k<j} &&}}
\vcenter{\xymatrix@C=-.4pc{\quad = \quad \quad }}
\vcenter{\xymatrix@C=-.4pc{\ff{Y}_{k<l} \deq
                           &&& \ff{Y}_{l<j} \op{\psi_l^{-1}} 
                           &&& \ff{h}_1^j \deq
                           &&& \ff{q}'_j \deq
                           \\
                           \ff{Y}_{k<l} \deq 
                           && \pi_l \deq 
                           && \ff{a}_{\ff{Y}}^j \dl 
                           & \dc{ \delta_j^{-1} \quad} 
                           & \ff{h}_1^j \dr
                           &&& \ff{q}'_j \deq
                           \\
                           \ff{Y}_{k<l} \deq 
                           && \pi_l \dl
                           & \dc{\mu_l} 
                           & \biLim{k<j}{\ff{p}_k} \dr
                           && \ff{h}_0^j \deq 
                           &&& \ff{q}'_j \deq 
                           \\
                           \ff{Y}_{k<l} \dl 
                           & \dc{\ff{p}_{k<l}} 
                           & \ff{p}_l \dr 
                           && \pi_l \deq
                           && \ff{h}_0^j \deq
                           &&& \ff{q}'_j \deq
                           \\
                           \ff{p}_k \deq 
                           && \ff{Z}_{k<l} \ardr
                           & \dc{\pi_{k<l}} 
                           & \pi_l \ardl
                           && \ff{h}_0^j \deq
                           &&& \ff{q}'_j \deq
                           \\
                           \ff{p}_k \deq 
                           &&& \pi_k \ardr
                           & \dc{=} & & \ff{h}_0^j 
                           &&& \ff{q}'_j \ardlllll
                           \\
                           \ff{p}_k 
                           &&&& \ff{Z}_{k<j} }}
\vcenter{\xymatrix@C=-.4pc{\quad = \quad \quad }}$$

$$
\vcenter{\xymatrix@C=-.4pc{\ff{Y}_{k<l} \deq
                           &&&&& \ff{Y}_{l<j} \op{\; \psi_l^{-1}} 
                           &&& \ff{h}_1^j \deq
                           && \ff{q}'_j \deq
                           \\
                           \ff{Y}_{k<l} \deq 
                           &&&& \pi_l \deq 
                           && \ff{a}_{\ff{Y}}^j \dl 
                           & \dc{ \delta_j^{-1}} 
                           & \ff{h}_1^j \dr
                           && \ff{q}'_j \deq
                           \\
                           \ff{Y}_{k<l} \ardr 
                           && \dc{\pi_{k<l}}
                           && \pi_l \ardl
                           && \biLim{k<j}{\ff{p}_k} \deq
                           && \ff{h}_0^j \deq 
                           && \ff{q}'_j \deq 
                           \\
                           && \pi_k \dl
                           && 
                           & \dc{\mu_k} & \biLim{k<j}{\ff{p}_k} \dr
                           && \ff{h}_0^j \deq
                           && \ff{q}'_j \deq
                           \\
                           && \ff{p}_k \deq 
                           &&&& \pi_k \ardrr
                           && \ff{h}_0^j \dc{=}
                           && \ff{q}'_j \ardll
                           \\
                           && \ff{p}_k
                           &&&&&& \ff{Z}_{k<j}
                           &&
                           }}
\vcenter{\xymatrix@C=-.4pc{\quad = \quad \quad }}
\vcenter{\xymatrix@C=-.4pc{\ff{Y}_{k<l} \ardr
                             & \dc{\alpha\tY_{k,l,j}}
                             & \ff{Y}_{l<j} \ardl
                             && \ff{h}_1^j \deq
                             && \ff{q}'_j \deq 
                             \\
                             & \ff{Y}_{k<j} \op{\psi_k^{-1}}
                             &&& \ff{h}_1^j \deq 
                             && \ff{q}'_j \deq
                             \\
                             \pi_k \deq 
                             && \ff{a}_{\ff{Y}}^j \dl 
                             & \dc{ \delta_j^{-1}}
                             & \ff{h}_1^j \dr
                             && \ff{q}'_j \deq
                             \\
                             \pi_k \dl
                             & \dc{\mu_k}
                             & \biLim{k<j}{\ff{p}_k} \dr
                             && \ff{h}_0^j \deq
                             && \ff{q}'_j \deq
                             \\
                             \ff{p}_k \deq 
                             && \pi_k \ardrr
                             && \ff{h}_0^j \dc{=} 
                             && \ff{q}'_j \ardll
                             \\
                             \ff{p}_k 
                             &&&& \ff{Z}_{k<j} 
                             && }}$$
                             
$$\vcenter{\xymatrix@C=-.4pc{\quad = \quad \quad }}
\vcenter{\xymatrix@C=-.4pc{\ff{Y}_{k<l} \ardr
                           & \dc{\alpha\tY_{k,l,j}} 
                           & \ff{Y}_{l<j} \ardl
                           &&& \ff{p}_{j} \deq \\
                           & \ff{Y}_{k<j} \dl 
                           & \dc{\ff{p}_{k<j}} & 
                           && \ff{p}_{j} \dr \\
                           & \ff{p}_{k} 
                           &&&& \ff{Z}_{k<j}}}$$                           

\noindent where the second equality is due to $a)$ and the third one is due to elevators calculus plus $d)$.

Now we are going to prove that $\gamma$ is a modification: Consider $k<j$, we want to check that the following equality holds:

$$\vcenter{\xymatrix@C=-.3pc{\ff{Y}_{k<j} \dl
                             & \dc{\ff{p}_{k<j}}
                             & \ff{p}_j \dr
                             && \ff{i}_j \deq
                             \\
                             \ff{p}_k \deq
                             && \ff{Z}_{k<j} \dl
                             & \dc{\ff{i}_{k<j}}
                             & \ff{i}_j \dr
                             \\
                             \ff{p}_k \ardr
                             & \dc{\gamma_k}
                             & \ff{i}_k \ardl
                             && \ff{X}_{k<j} \deq
                             \\
                             & \ff{f}_k 
                             &&& \ff{X}_{k<j}}}
\vcenter{\xymatrix@C=-.4pc{\quad = \quad \quad }}
\vcenter{\xymatrix@C=-.3pc{\ff{Y}_{k<j} \deq
                           & \ff{p}_j \ardr
                           & \dc{\gamma_j} 
                           & \ff{i}_j \ardl
                           \\
                           \ff{Y}_{k<j} \dl
                           & \dc{\ff{f}_{k<j}}
                           & \ff{f}_j \dr
                           &\\
                           \ff{f}_k 
                           && \ff{X}_{k<j}
                           &}}$$

\noindent But

$$\vcenter{\xymatrix@C=-.5pc{\ff{Y}_{k<j} \dl
                             & \dc{\ff{p}_{k<j}}
                             & \ff{p}_j \dr
                             && \ff{i}_j \deq
                             \\
                             \ff{p}_k \deq
                             && \ff{Z}_{k<j} \dl
                             & \dc{\ff{i}_{k<j}}
                             & \ff{i}_j \dr
                             \\
                             \ff{p}_k \ardr
                             & \dc{\gamma_k}
                             & \ff{i}_k \ardl
                             && \ff{X}_{k<j} \deq
                             \\
                             & \ff{f}_k 
                             &&& \ff{X}_{k<j}}}
\vcenter{\xymatrix@C=-.4pc{\quad = \quad \quad }}
\vcenter{\xymatrix@C=-.5pc@R=1pc{& \ff{Y}_{k<j} \deq
                           &&&&&& \ff{p}_j \op{=} 
                           &&&&& \ff{i}_j \deq
                           \\
                           &\ff{Y}_{k<j} \op{\psi_k^{-1}}
                           &&&&& \ff{h}_1^j \deq
                           && \ff{q}'_j \deq
                           &&&& \ff{i}_j \deq
                           \\
                           \pi_k \deq 
                           && \ff{a}_{\ff{X}}^j \dl 
                           && \dc{ \delta_j^{-1}} 
                           && \ff{h}_1^j \dr
                           && \ff{q}'_j \deq
                           &&&& \ff{i}_j  \deq
                           \\
                           \pi_k \dl
                           & \dc{\mu_k}
                           & \biLim{k<j}{\ff{p}_k} \dr
                           &&&& \ff{h}_0^j \deq
                           && \ff{q}'_j \deq
                           &&&& \ff{i}_j \deq
                           \\
                           \ff{p}_k \deq 
                           && \pi_k \deq 
                           &&&& \ff{h}_0^j \deq 
                           && \ff{q}'_j  \ardr
                           &&  \dc{\gamma'_j} 
                           && \ff{i}_j \ardl
                           \\
                           \ff{p}_k \deq 
                           &&  \pi_k \deq
                           &&&& \ff{h}_0^j \dl
                           && \dc{\theta_0^j}
                           && \ff{a}_j \dr
                           &&
                           \\
                           \ff{p}_k \deq
                           && \pi_k \dl
                           && \dc{\epsilon_k}
                           && \biLim{k<j}{\ff{i}_k} \dr
                           &&&& \ff{a}_{\ff{X}}^j \deq
                           &&
                           \\
                           \ff{p}_k \deq
                           && \ff{i}_k \deq 
                           &&&& \pi_k \ardrr
                           && \dc{\varphi_k}
                           && \ff{a}_{\ff{X}}^j \ardl
                           &&
                           \\
                           \ff{p}_k \ardr
                           & \dc{\gamma_k}
                           & \ff{i}_k \ardl
                           &&&&&& \ff{X}_{k<j} \deq 
                           &&&& 
                           \\
                           & \ff{f}_k 
                           &&&&&&& \ff{X}_{k<j}
                           &&&&
                           }}
\vcenter{\xymatrix@C=-.4pc{\quad = \quad }}$$

$$\vcenter{\xymatrix@C=-.5pc@R=1.5pc{& \ff{Y}_{k<j} \deq &&&& \ff{p}_j \op{=} &&& \ff{i}_j \deq \\
                             & \ff{Y}_{k<j} \deq &&& \ff{h}_1^j \deq && \ff{q}'_j && \ff{i}_j \\
                             & \ff{Y}_{k<j} \op{\psi_k^{-1}} &&& \ff{h}_1^j \deq &&& \ff{a}_j \clb{\gamma'_j} \deq \\
                             \pi_k \deq && \ff{a}_{\ff{Y}}^j \deq & \; & \ff{h}_1^j &&& \ff{a}_j \\
                             \pi_k && \ff{a}_{\ff{Y}}^j &&& \ff{f}_j \clunodos{\theta_1^j} \deq \\
                             & \ff{Y}_{k<j} \cl{\psi_k} \dl && \dc{\ff{f}_{k<j}} && \ff{f}_j \dr \\
                             & \ff{f}_k &&&& \ff{X}_{k<j} }}
\vcenter{\xymatrix@C=-.4pc{\quad = \quad }}
\vcenter{\xymatrix@C=-.4pc{\ff{Y}_{k<j} \deq 
                             &&& \ff{p}_j \op{=} 
                             &&&&& \ff{i}_j \deq 
                             \\
                             \ff{Y}_{k<j} \deq
                             && \ff{h}_1^j \deq
                             && \ff{q}'_j \ardr 
                             && \dc{\gamma'_j}
                             && \ff{i}_j \ardl
                             \\
                             \ff{Y}_{k<j} \deq
                             && \ff{h}_1^j \ardrr
                             && \dc{\theta_1^j}
                             && \ff{a}_j \ardll
                             &&
                             \\
                             \ff{Y}_{k<j} \dl
                             && \dc{\ff{f}_{k<j}} 
                             && \ff{f}_j \dr
                             &&&&
                             \\
                             \ff{f}_k 
                             &&&& \ff{X}_{k<j}
                             &&&&}}
\vcenter{\xymatrix@C=-.4pc{\quad = \quad }}
\vcenter{\xymatrix@C=-.4pc{\ff{Y}_{k<j} \deq
                           & \ff{p}_j \ardr
                           & \dc{\gamma_j} 
                           & \ff{i}_j \ardl
                           \\
                           \ff{Y}_{k<j} \dl
                           & \dc{\ff{f}_{k<j}}
                           & \ff{f}_j \dr
                           &\\
                           \ff{f}_k 
                           && \ff{X}_{k<j}
                           &}}$$

\noindent where the second equality is due to elevators calculus plus $e)$ and the third one is due to elevators calculus.

$\ff{i}_j$ is both a cofibration and a weak equivalence by construction.

It can be easily checked that the $\ff{q}_j$ associated to $\ff{p}$ is $\ff{q}'_j$ and so is a fibration.

Now we will focus on the case where $\ff{p}$ is both a fibration and a weak equivalence and $\ff{i}$ is a cofibration:

It is enough to check that we can factor $\ff{p}_j \ff{i}_j \Mr{\gamma_j \cong} \ff{f}_j \ \forall \ j\in \ff{J}$, $\ff{i}=\left\{\ff{i}_j\right\}_{j\in \ff{J}}$ and $\ff{p}=\left\{\ff{p}_j\right\}_{j\in \ff{J}}$ are pseudo-natural transformations, $\ff{i}_j$ is a cofibration $\forall \ j\in \ff{J}$, $\ff{q}_j$ (associated to $\ff{p}$ as in \ref{lema1.2}) is a fibration $\forall \ j\in \ff{J}$, $\ff{p}_j$ is a weak equivalence $\forall \ j\in \ff{J}$ and $\gamma=\left\{\gamma_j\right\}_{j\in \ff{J}}$ is a modification. We are going to do this by induction in $j$:
 
 If $j$ is the initial object of $\ff{J}$, since $\cc{C}$ is a closed 2-bmodel 2-category, we have $\ff{p}_0 \ff{i}_0 \Mr{\gamma_0 \cong} \ff{f}_0$ where $\ff{i}_0$ is a cofibration and $\ff{p}_0$ is a both a fibration and a weak equivalence. As $\ff{q}_0\cong\ff{p}_0$, $\ff{q}_0$ is a fibration. 
 
 Now, suppose that we have already defined $\ff{p}_k$, $\ff{i}_k$ and $\gamma_k \ \forall \ k<j$. In order to define $\ff{p}_j$, $\ff{i}_j$ and $\gamma_j$, consider diagram \eqref{diagramaconQ} as before:
%
 
 Since $\cc{C}$ is a closed 2-bmodel 2-category, $\ff{a}_j$ may be factored as $\ff{q}'_j \ff{i}_j \Mr{\gamma'_j \cong} \ff{a}_j$ where $\ff{i}_j$ is a cofibration and $\ff{q}'_j$ is both a fibration and a weak equivalence. Consider $\ff{p}_j=\ff{h}_1^j \ff{q}'_j$ and $\gamma_j= 
 \vcenter{\xymatrix@C=-0pc{ & \p_j \op{=} &&& \ii_j \deq \\
 \h_1^j \deq && \q'_j && \ii_j \\
 \h_1^j &&& \fa_j \cl{\gamma'_j} \\
 && \f_j \cldosuno{\theta_1^j} }} 
 $. Let's check that this data satisfies the desired properties:

One can check as before that $\ff{Z}$ is a pseudo-functor, $\ff{i}$ is pseudo-natural, $\ff{p}$ is pseudo-natural and $\gamma$ is a modification.

$\ff{i}_j$ is a cofibration by construction.

As before, $\ff{q}'_j$ is the $\ff{q}_j$ associated to $\ff{p}$ and it is a fibration.

It only remains to check that $\ff{p}_j$ is a weak equivalence: Since $\ff{p}$ is a fibration, by \ref{lema3.2}, $\ff{p}_j$ is a fibration $\forall \ j\in \ff{J}$. Then $\ff{p}_k$ is both a fibration and a weak equivalence $\forall \ k<j$ and so, by \ref{lema0.2}, $\biLim{k<j}{\ff{p}_k}$ is both a fibration and a weak equivalence. Therefore, since 2-M4 is satisfied in $\cc{C}$, $\ff{h}_1^j$ is a weak equivalence. Then $\ff{p}_j$ is a weak equivalence by axiom 2-M5. 

\vspace{1ex}

{\bfseries Axiom 2-M5}: It follows from the fact that weak equivalences in $\pCJ$ are defined pointwise and $\cc{C}$ is a closed 2-bmodel 2-category.

\vspace{1ex}

{\bfseries Axiom 2-M6a)}: It is tautological from the definition of fibrations in $\pCJ$. 

\vspace{1ex}

{\bfseries Axiom 2-M6b)}: $\Rightarrow)$ Suppose that we have a cofibration $\ff{i}$, a fibration $\ff{p}$ which is also a weak equivalence and they fit in a diagram 
 
 $$\xymatrix@C=1.5pc@R=3pc{\ff{A} \ar[rr]^{\ff{a}} \ar[d]_{\ff{i}} &\ar@{}[d]|{\cong \; \Downarrow \; \gamma}& \ff{Y} \ar[d]^{\ff{p}} \\
            \ff{X} \ar[rr]_{\ff{b}} && \ff{B}}$$

The filler can be constructed inductively exactly as in the proof of \ref{lema1.2} $\Leftarrow$.

$\Leftarrow)$ Suppose that $\ff{A}\mr{\ff{i}}\ff{X}$ has the left lifting property with respect to all morphisms that are both fibrations and weak equivalences. We have to check that $\ff{i}_j$ is a cofibration $\forall \ j\in \ff{J}$ but, since $\cc{C}$ is a closed 2-bmodel 2-category, it is enough to check that $\ff{i}_j$ has the left lifting property with respect to all morphisms that are both fibrations and weak equivalences. So take a morphism $\ff{Y}\mr{\tilde{\ff{p}}}\ff{B}\in \cc{C}$ which is both a fibration and a weak equivalence and suppose that we have a diagram of the form

 $$\xymatrix@C=1.5pc@R=3pc{\ff{A}_j \ar[rr]^{\tilde{\ff{a}}} \ar[d]_{\ff{i}_j} &\ar@{}[d]|{\cong \; \Downarrow \; \tilde{\gamma}}& \ff{Y} \ar[d]^{\tilde{\ff{p}}} \\
            \ff{X}_j \ar[rr]_{\tilde{\ff{b}}} && \ff{B}}$$
            
Let's define $\ff{E}$, $\ff{D}$, $\ff{E}\mr{\ff{p}}\ff{D}$, $\ff{A}\mr{\ff{a}}\ff{E}$, $\ff{X}\mr{\ff{b}}\ff{D}\in \pCJ$ by: 

\mbox{$\ff{E}_k=\begin{cases}\ff{Y} & \hbox{if } k\geq j \\
                               * & \hbox{otherwise}
                  \end{cases}$,}
             \mbox{$\ff{E}_{id_k}=id_{\ff{E}_k}$,}
             \mbox{$\vcenter{\xymatrix{\\ \ff{E}_{k<l} = \\ \\ }}\begin{cases}
                                       id_{\ff{Y}} &  \hbox{if } k\geq j \\ 
                                       \ff{Y}\mr{}* &  \hbox{if } k\ngeq j \hbox{ and } l\geq j\\
                                      id_* & \hbox{otherwise } \\
                                      \end{cases}$,}
             \mbox{$\ff{D}_k=\begin{cases}\ff{B} & \hbox{if } k\geq j \\
                               * & \hbox{otherwise}
                  \end{cases}$,}
             \mbox{$\ff{D}_{id_k}=id_{\ff{D}_k}$,}
             \mbox{$\ff{D}_{k<l}= \begin{cases}
                                       id_{\ff{B}} &  \hbox{if } k\geq j \\ 
                                       \ff{B}\mr{}* &  \hbox{if } k\ngeq j \hbox{ and } l\geq j\\
                                      id_* & \hbox{otherwise } \\
                                      \end{cases}$,}
            \mbox{$\ff{p}_k=\begin{cases}\tilde{\ff{p}} & \hbox{if } k\geq j \\
                               id_* & \hbox{otherwise}
                  \end{cases}$,}
            \mbox{$\vcenter{\xymatrix{\\ \ff{p}_{id_k}= \\ \\ }} \begin{cases}
                                       id_{\tilde{\ff{p}}} &  \hbox{if } k\geq j \\ 
                                       id_{id_*} & \hbox{otherwise } \\
                                      \end{cases}$,}
            \mbox{$\ff{p}_{k<l}=\begin{cases}
                                       id_{\tilde{\ff{p}}} &  \hbox{if } k\geq j \\ 
                                       id_{\ff{Y}\mr{}*} & \hbox{if } k\ngeq j \hbox{ and } l\geq j \\
                                       id_{id_*} & \hbox{otherwise } \\
                                      \end{cases}$,}  
             \mbox{$\ff{a}_k=\begin{cases}\tilde{\ff{a}} \ff{A}_{j\leq k} & \hbox{if } k\geq j \\
                               \ff{A}_k\mr{}* & \hbox{otherwise}
                  \end{cases}$,}
             \mbox{$\ff{a}_{id_k}=\begin{cases}
                                       \tilde{\ff{a}} \ff{A}_{j\leq k} \alpha\tA_k &  \hbox{if } k\geq j \\ 
                                       id_{\ff{A}_k \mr{} *} & \hbox{otherwise } \\
                                      \end{cases}$,}
             \mbox{$\vcenter{\xymatrix{\\ \ff{a}_{k<l}= \\ \\ }}\begin{cases}
                                       \tilde{\ff{a}} (\alpha\tA_{j,k,l})^{-1} &  \hbox{if } k\geq j \\ 
                                       id_{\ff{A}_l\mr{}*} & \hbox{otherwise}
                                       \end{cases}$,}     
             \mbox{$\ff{b}_k=\begin{cases}\tilde{\ff{b}}\ff{X}_{j\leq k} & \hbox{if } k\geq j \\
                               \ff{X}_k\mr{}* & \hbox{otherwise}
                  \end{cases}$} 
             \mbox{$\ff{b}_{id_k}=\begin{cases}
                                        \tilde{\ff{b}} \ff{X}_{j\leq k} \alpha\tX_k &  \hbox{if } k\geq j \\ 
                                       id_{\ff{X}_k \mr{} *} & \hbox{otherwise } \\
                                      \end{cases}$} \linebreak
	    and
             \mbox{$\ff{b}_{k<l}=\begin{cases}
                                       \tilde{\ff{b}} (\alpha\tX_{j,k,l})^{-1} &  \hbox{if } k\geq j \\ 
                                       id_{\ff{X}_l\mr{}*} & \hbox{otherwise}
                                      \end{cases}$.}
                  
It is straightforward to check that $\ff{E}$ and $\ff{D}$ are 2-functors and that $\ff{p}$, $\ff{a}$ and $\ff{b}$ are pseudo-natural transformations.                  
                                   
By using \ref{lema4.2}, $\ff{p}$ is both a fibration and a weak equivalence and then there exists a filler $(\ff{f},\lambda,\rho)$ for the following diagram

\begin{equation}\label{diagramaconEyD}
\xymatrix@C=1.5pc@R=3pc{\ff{A} \ar[rr]^{\ff{a}} \ar[d]_{\ff{i}} &\ar@{}[d]|{\cong \; \Downarrow \; \gamma}& \ff{E} \ar[d]^{\ff{p}} \\
            \ff{X} \ar[rr]_{\ff{b}} && \ff{D}} 
\end{equation}

\noindent where $\gamma_k=\begin{cases}
                                        \vcenter{\xymatrix@C=0pc{ \tilde{\p} \dl & \dc{\tilde{\gamma}} & \dr \tilde{\fa} && \A_{j \leq k} \deq \\
                                        \tilde{\fb} \deq && \ii_j \dl & \dc{\ii_{j \leq k}\inv} & \dr \A_{j \leq k} \\
                                        \tilde{\fb} && \X_{j \leq k} && \ii_k }}                                        
                                        &  \hbox{if } k\geq j \\ 
                                        id_{\ff{A}_k\mr{}*}& \hbox{otherwise } \\
                                      \end{cases}$.
\vspace{1ex}

Consider $\tilde{\lambda}= 
\vcenter{\xymatrix@C=0pc{ & \tilde{\fa} \op{=} \\  \tilde{\fa} \deq && id_{\A_j} \dcellb{\alpha\tA_j} \\  \tilde{\fa} \dl & \dc{\lambda_j} & \dr \A_{id_j} \\   \f_j && \ii_j }}
$ and $\tilde{\rho}= 
\vcenter{\xymatrix@C=0pc{ \tilde{\p} \dl & \dc{\rho_j} & \dr \f_j \\ \tilde{\fb} \deq && \X_{id_j} \dcellbymedio{(\alpha\tX_j)\inv} \\ \tilde{\fb} && id_{\X_j} \\  & \tilde{\fb} \cl{=} }}
$

Let's check that $(\ff{f}_j,\lambda_j,\rho_j)$ is the filler that we were looking for: 
 

$$\vcenter{\xymatrix@C=-0pc{ \tilde{\p} \deq &&& \tilde{\fa} \op{=} \\
\tilde{\p} \deq && \tilde{\fa} \deq && id_{\A_j} \dcellb{\alpha\tA_j} \\
\tilde{\p} \deq && \tilde{\fa} \dl & \dc{\lambda_j} & \dr \A_{id_j} \\
\tilde{\p} \dl & \dc{\rho_j} & \dr \f_j && \ii_j \deq \\
\tilde{\fb} \deq && \X_{id_j} \dcellbymedio{(\alpha\tX_j)\inv} && \ii_j \deq \\
\tilde{\fb} && id_{\X_j} && \ii_j \deq \\
& \tilde{\fb} \cl{=} &&& \ii_j}}
\vcenter{\xymatrix@C=-0pc{ \quad = \quad }}
\vcenter{\xymatrix@C=-0pc{ \tilde{\p} \deq &&& \tilde{\fa} \op{=} \\
\tilde{\p} \deq && \tilde{\fa} \deq && id_{\A_j} \dcellb{\alpha\tA_j}\\
\tilde{\p} \dl && \tilde{\fa} \dc{\gamma_j} && \dr \A_{id_j} \\
\tilde{\fb} \deq && \X_{id_j} \dcellbymedio{(\alpha\tX_j)\inv} && \ii_j \deq \\
\tilde{\fb} && id_{\X_j} && \ii_j \deq \\
& \tilde{\fb} \cl{=} &&& \ii_j}}
\vcenter{\xymatrix@C=-0pc{ \quad = \quad }}
\vcenter{\xymatrix@C=-0pc{ \tilde{\p} \dl & \dc{\tilde{\gamma}} & \dr \tilde{\fa} \\ \tilde{\fb} && \ii_j }}$$

\noindent where the first equality is due to the fact that $(\ff{f},\lambda,\rho)$ is a filler for diagram \eqref{diagramaconEyD} and the last one is due to elevators calculus plus the definition of $\gamma_j$ and the fact that $\ff{i}$ is a pseudo-natural transformation.

\vspace{1ex}

{\bfseries Axiom 2-M6c)}: $\Rightarrow)$ Let $\ff{f}$ be a weak equivalence. By axiom 2-M2, $\ff{f}$ can be factored as $\ff{f}\cong \ff{u}\ff{v}$ where $\ff{u}$ is both a fibration and a weak equivalence and $\ff{v}$ is a cofibration. Since $\ff{f}_j$ is a weak equivalence $\forall \ j\in \ff{J}$ and $\cc{C}$ is a closed 2-bmodel \mbox{2-category}, $\ff{u}_j\ff{v}_j$ is a weak equivalence $\forall \ j\in \ff{J}$ and so $\ff{u}\ff{v}$ is a weak equivalence. Then, by axiom 2-M5, $\ff{v}$ is also a weak equivalence. This plus axioms 2-M6a) and 2-M6b) conclude the proof. 
 
 $\Leftarrow)$ By an argument similar to the one used in the proof of $\Rightarrow)$, it is enough to check that $\ff{u}\ff{v}$ is a weak equivalence. And to do that, it is enough to check that $\ff{u}$ and $\ff{v}$ are both weak equivalences. We are going to do the proof for $\ff{v}$ (the proof for $\ff{u}$ is analogous but easier): By definition, we want to check that $\ff{v}_j$ is a weak equivalence $\forall \ j\in \ff{J}$ and, by \ref{lema2.2} and the fact that $\cc{C}$ is a closed 2-bmodel \mbox{2-category}, it is enough to check that it has the left lifting property with respect to all fibrations. So suppose that we have a diagram of the form
 
 $$\xymatrix@C=1.5pc@R=3pc{\ff{X}_j \ar@{}[drr]|{\cong \; \Downarrow  \; \gamma } \ar[rr]^{\ff{a}} \ar[d]_{\ff{v}_j} 
			   & & \ff{A} \ar[d]^{\ff{p}} \\
			    \ff{Y}_j \ar[rr]_{\ff{b}} & & \ff{B} }$$
\noindent where $\ff{p}$ is a fibration.			   

The proof follows by an argument exactly as the one used in the proof of axiom 2-M6b) $\Leftarrow$.

\end{proof}

\begin{corollary}
 Let $\cc{C}$ be a closed 2-bmodel 2-category. Then $\cc{H}om_p(\hat{\ff{J}}^{op}, \cc{C})$ is a closed 2-bmodel 2-category.
\end{corollary}

\begin{proof}
 It follows immediately from \ref{CJdeclosed 2-modelos} and \ref{Asombrero}.
\end{proof}

It is worth mention that the proof given in this subsection can be easily adapted to the case of closed 2-model 2-categories giving the following also interesting result:

\begin{theorem}
$\pCJ$ with the structure provided in \ref{estructuraenCJ} is a closed \mbox{2-model} 2-category if $\cc{C}$ is. \cqd
\end{theorem}
\subsection{Closed 2-bmodel structure in $\Pro{C}$}\label{closed 2-bmodel structure in Prop}

In order to prove that $\Pro{C}$ is a closed 2-bmodel 2-category, we are going to give first a closed 2-bmodel structure to its retract pseudo-equivalent 2-category $\Prop{C}$ (see \ref{proppseudoeqapro}). As we have already said, we were forced to work with pseudo-natural transformations instead of 2-natural transformations due to the non-strict commutativity of diagrams.
We start with the finite completeness and finite cocompleteness aspects of the 2-category $\Prop{C}$.



\begin{proposition}\label{Bp}
Let $\ff{J}$ be a filtered category and $\cc{C}$ a 2-category with finite bi-colimits (respectively bi-limits, bi-tensors and bi-cotensors). Consider $\cc{J}=\hat{\ff{J}}$ (see \ref{Asombrero}). Then the inclusion 2-functor $\cc{H}om_p(\cc{J}^{op},\cc{C})\mr{inc}\Prop{C}$ preserves finite bi-colimits (respectively bi-limits, bi-tensors and bi-cotensors). 
\end{proposition}

\begin{proof}
We are going to prove the assertion for bi-colimits and bi-tensors, other cases are dual to these ones.

\begin{itemize}
 \item[-] Let $\{\ff{D}_{\alpha}\}_{\alpha \in \Gamma}$ be a finite diagram in $\cc{H}om_p(\cc{J}^{op},\cc{C})$. We have to check that $inc(\bicoLim{\alpha \in \Gamma}{\ff{D}_\alpha})$ is the bi-colimit $\bicoLim{\alpha \in \Gamma}{inc(\ff{D}_\alpha)}$ in $\Prop{C}$ i.e. that \mbox{$\forall \ \ff{Y}=\{\ff{Y}_i\}_{i\in \cc{I}}\in \Prop{C}$}, 
 $$\Prop{C}(inc(\bicoLim{\alpha \in \Gamma}{\ff{D}_\alpha}),\ff{Y}) \simeq \biLim{\alpha \in \Gamma}{\Prop{C}(inc(\ff{D}_\alpha),\ff{Y})} \mbox{(see \ref{colimits}):}$$

We do as follows: 
 
$$\Prop{C}(inc(\bicoLim{\alpha \in \Gamma}{\ff{D}_\alpha}),\ff{Y}) \stackrel{I}{\simeq}
\Lim{i\in \cc{I}}{\coLim{j\in \cc{J}}{\cc{C}((\bicoLim{\alpha \in \Gamma}{\ff{D}_\alpha})(j),\ff{Y}_i)}} \stackrel{II}{\simeq}$$
$$\Lim{i\in \cc{I}}{\coLim{j\in \cc{J}}{\cc{C}(\bicoLim{\alpha \in \Gamma}{\ff{D}_\alpha j},\ff{Y}_i)}} \stackrel{III}{\simeq}
\Lim{i\in \cc{I}}{\coLim{j\in \cc{J}}{\biLim{\alpha \in \Gamma}{\cc{C}(\ff{D}_\alpha j,\ff{Y}_i)}}} \stackrel{IV}{\simeq} $$
$$\Lim{i\in \cc{I}}{\biLim{\alpha\in \Gamma}{\coLim{j \in \cc{J}}{\cc{C}(\ff{D}_\alpha j,\ff{Y}_i)}}} \stackrel{V}{\simeq}
\biLim{\alpha \in \Gamma}{\Lim{i\in \cc{I}}{\coLim{j \in \cc{J}}{\cc{C}(\ff{D}_\alpha j,\ff{Y}_i)}}} \stackrel{VI}{\simeq}$$
$$ \biLim{\alpha \in \Gamma}{\Prop{C}(inc(\ff{D}_\alpha),\ff{Y})}$$

\noindent where $I$ is due to \ref{formula en prop}, $II$ is due to the fact that bi-colimits in $\cc{H}om_p(\cc{J}^{op},\cc{C})$ are computed pointwise (see \ref{pointwise en sombrero}), $III$ holds by \ref{colimits}, $IV$ is true because the 2-filtered bi-colimit and the finite bi-limit of categories can be replaced by equivalent pseudo-colimit and pseudo-limit, and these commute (\cite{Canevalli}), $V$ holds because bi-limits are associative up to equivalence and $VI$ is due to \ref{formula en prop} again.

\item[-] Let $\E$ be a finite category and $\F\in \cc{H}om_p(\cc{J}^{op},\cc{C})$. We denote for simplicity \mbox{$\tilde{\otimes}_{\cc{H}} = \tilde{\otimes}_{\cc{H}om_p(\cc{J}^{op},\cc{C})}$,} $\tilde{\otimes}_{\cc{P}} = \tilde{\otimes}_{\Prop{C}}$. 
We have to check that $inc(\E \tilde{\otimes}_{\cc{H}} \F )$ is the bi-tensor $\E \tilde{\otimes}_{\cc{P}} inc(\F)$, i.e. that \mbox{$\forall \ \ff{Y}=\{\ff{Y}_i\}_{i\in \cc{I}}\in \Prop{C}$}, 
$$\Prop{C}(inc(\E \tilde{\otimes}_{\cc{H}} \F ),\Y)\simeq \cc{C}at(\E,\Prop{C}(inc(\F),\Y)) \mbox{(see \ref{bi-tensor}):}$$

We do as follows:

$$\Prop{C}(inc(\E \tilde{\otimes}_{\cc{H}} \F),\Y)\stackrel{I}{\simeq}
\Lim{i\in \cc{I}}{\coLim{j\in \cc{J}}{\cc{C}((\E \tilde{\otimes}_{\cc{H}} \F)(j),\Y_i)}} \stackrel{II}{\simeq}$$
$$\Lim{i\in \cc{I}}{\coLim{j\in \cc{J}}{\cc{C}(\E \tilde{\otimes}_{\cc{C}} \F j,\Y_i)}} \stackrel{III}{\simeq}
\Lim{i\in \cc{I}}{\coLim{j\in \cc{J}}{\cc{C}at(\E ,\cc{C}(\F j,\Y_i))}} \stackrel{IV}{\simeq}$$
$$\Lim{i\in \cc{I}}{\cc{C}at(\E ,\coLim{j\in \cc{J}}{\cc{C}(\F j,\Y_i)})} \stackrel{V}{\simeq}
\cc{C}at(\E ,\Lim{i\in \cc{I}}{\coLim{j\in \cc{J}}{\cc{C}(\F j,\Y_i)}}) \stackrel{VI}{\simeq}$$
$$\cc{C}at(\E ,\Prop{C}(inc(\F),\Y))$$

\noindent where $I$ is due to \ref{formula en prop}, $II$ is due to the fact that bi-tensors in $\cc{H}om_p(\cc{J}^{op},\cc{C})$ are computed pointwise (see \ref{pointwise en sombrero}, \ref{tensorptoapto}), $III$ holds by definition \ref{bi-tensor}, $IV$ is true by \cite[2.4]{DS}, $V$ is due to definition \ref{colimits} and $VI$ holds by \ref{formula en prop} again.
\end{itemize}
\end{proof}

\begin{proposition} \label{2M0b}
  If $\cc{C}$ has finite weighted bi-limits and bi-colimits of pseudo-functors $\F : \cc{P} \mr{} \cc{C}$ with finite weights $\ff{W}: \cc{P} \mr{} \cc{C}at$ (see \cite{K2}), then so does $\Prop{C}$ (to simplify, by finite we mean that $\cc{P}$ is finite and $\ff{W}(\ff{P})$ is finite for all $\ff{P} \in \cc{P}$).
\end{proposition}

\begin{proof}
 We are going to prove only the case of bi-colimits. The case of bilimits is analogous and we leave it to the reader. We are going to check that $\Prop{C}$ has bi-colimits of pseudo-functors indexed by a finite category $\Delta$ with no loops. As a particular case, we will have bi-coequalizers, binary bi-coproducts and $0$ which, by \ref{limites que alcanzan}, is enough to prove the statement in the proposition. 

Let $\Delta \mr{\D} \Prop{C}$ be a pseudo-functor. Then, by \ref{reindexingparadiagramasenprop}, we have \mbox{$\Delta \mr{\D'} \cc{H}om_p(\ff{J}^{op},\cc{C})$} equivalent to $\D$ in $\Prop{C}$ as in the following diagram with $\ff{J}$ a cofinite and filtered poset with a unique initial object:

$$\xymatrix{ & \cc{H}om_p(\ff{J}^{op},\cc{C}) \ar[d]^{inc}_>>>>{\simeq \quad \;\;} \\ 
             \Delta \ar[ru]^{\D'} \ar[r]_{\D} & \Prop{C}}$$

If we apply the construction of \ref{Asombrero} to $\ff{J}$, we obtain $\Delta \mr{\D''} \cc{H}om_p(\hat{\ff{J}}^{op},\cc{C})$ which by \ref{reindexing} is equivalent to $\D$ in $\Prop{C}$.

It suffices to show that the bi-colimit of $inc(\D'')$ exists in $\Prop{C}$ and this follows from \ref{Bp} plus the fact that $\cc{C}$ is a closed 2-bmodel 2-category.

To conclude the proof, we have to check that $\Prop{C}$ has bi-tensors $\E \tilde{\otimes}_{\Prop{C}} \X$ with $\E$ a finite category:

By \ref{reindexingparaobjetos}, we obtain a 2-pro-object $\X'$ equivalent to $\X$ and indexed by a cofinite and filtered poset with a unique initial object $\ff{J}$. Then, by \ref{Asombrero} plus \ref{reindexing}, we can construct a 2-pro-object $\X''$ equivalent to $\X$ and indexed by $\cc{J}=\hat{\ff{J}}$. Then, by \ref{Bp}, there exists the bi-tensor $\E \tilde{\otimes}_{\Prop{C}} \X''$ and so there exists the bi-tensor $\E \tilde{\otimes}_{\Prop{C}} \X$. 
%
%
\end{proof}

In all what it follows we assume that $\cc{C}$ is a closed 2-bmodel 2-category.

\begin{definition}\label{estructuraenpropC}
 We define strong fibrations, strong cofibrations, strong trivial fibrations, strong trivial cofibrations in $\Prop{C}$ as the image of fibrations, cofibrations, fibrations that are also weak equivalences and cofibrations that are also weak equivalences respectively in some $\cc{H}om_p(\hat{\ff{J}}^{op},\cc{C})$ with $\ff{J}$ a cofinite, filtered poset with a unique initial object.

\begin{itemize}
 \item A morphism $\ff{f}\in \Prop{C}$ is a cofibration if is the retract in $\cc{H}om_p(\ff{2},\Prop{C})$ of a strong cofibration. 
 \item A morphism $\ff{f}\in \Prop{C}$ is a fibration if is the retract in $\cc{H}om_p(\ff{2},\Prop{C})$ of a strong fibration. 
 \item A morphism $\ff{f}\in \Prop{C}$ is a trivial cofibration if is the retract in $\cc{H}om_p(\ff{2},\Prop{C})$ of a strong trivial cofibration. 
 \item A morphism $\ff{f}\in \Prop{C}$ is a trivial fibration if is the retract in $\cc{H}om_p(\ff{2},\Prop{C})$ of a strong trivial fibration.
 \item A morphism $\ff{f}\in \Prop{C}$ is a weak equivalence if it can be factored up to isomorphism as $\ff{f}\cong \ff{p}\ff{i}$ where $\ff{p}$ is a trivial fibration and $\ff{i}$ is a trivial cofibration.
 \end{itemize}
\end{definition}

All the rest of this section is devoted  to prove the following theorem:
\begin{theorem}\label{Propde2modelos}
If $\cc{C}$ is a closed 2-bmodel 2-category, then 
$\Prop{C}$ with the structure given in \ref{estructuraenpropC} is a closed 2-bmodel \mbox{2-category.} \cqd  
\end{theorem}
But before we state and prove the theorem we wanted in the first place:
\begin{theorem}\label{Proesde2-bmodelos}
If $\cc{C}$ is a closed 2-bmodel 2-category, then $\Pro{C}$ is a closed \mbox{2-bmodel} 2-category. 
\end{theorem}
\begin{proof}
Recall that the inclusion $\Pro{C}\mr{}\Prop{C}$ is a retract pseudo-equivalence (see \ref{proppseudoeqapro}). Then, the result follows immediately from \ref{pseudo-equivalencerespetaaxiomas} and  \ref{Propde2modelos}.
\end{proof}
\vspace{2ex}

\emph{Proof of theorem \ref{Propde2modelos}}

{\bfseries Axiom 2-M0b}: It holds by proposition \ref{2M0b}.

\vspace{1ex}

{\bfseries Axiom 2-M2}: We are going to give the proof for the case where $\ff{p}$ is both a fibration and a weak equivalence and $\ff{i}$ is a cofibration. The other case is analogous and we leave it to the reader.
 
\noindent Let $\ff{X}\mr{\ff{f}} \ff{Y}\in \Prop{C}$. By \ref{mardesictrickparaflechasenprop}, we have a diagram of the form 

$$\xymatrix@R=1.5pc@C=1.5pc{\ff{X} \ar@{}[ddrr]|{\cong \; \Downarrow \; \gamma} \ar[rr]^{\ff{f}} \ar[dd]_{\ff{a}} 
			   & & \ff{Y} \ar[dd]^{\ff{b}} \\
			   & 
			    \\
			   \ff{X}' \ar[rr]_{\ff{f}'}
			   & & \ff{Y}' }$$

\noindent where $\ff{a}$ and $\ff{b}$ are equivalences with quasi inverses $\overline{\ff{a}}$ and $\overline{\ff{b}}$ and \mbox{$\ff{X}'\mr{\ff{f}'}\ff{Y}'\in \CJ$} for some $\ff{J}$ cofinite, filtered poset with a unique initial object.

Consider $\ff{J}^{op} \mr{\ff{T}} \widehat{\ff{J}^{op}}=\hat{\ff{J}}^{op}$ as in \ref{Asombrero}. Then there are 2-functors \mbox{$\widehat{\ff{X}'}$, $\widehat{\ff{Y}'}: \hat{\ff{J}}^{op}\mr{} \cc{C}$} such that $\widehat{\ff{X}'} \ff{T}=\ff{X}'$ and $\widehat{\ff{Y}'} \ff{T}=\ff{Y}'$. Then, by \ref{Asombrero} plus \ref{reindexing}, $\widehat{\ff{X}'}$ and $\widehat{\ff{Y}'}$ are equivalent to $\ff{X}'$ and $\ff{Y}'$ respectively in $\Prop{C}$ via some equivalences $\ff{a}'$, $\ff{b}'$ with quasi inverses $\overline{\ff{a}'}$, $\overline{\ff{b}'}$. Then we have the following diagram in $\Prop{C}$:

$$\vcenter{\xymatrix@R=1.5pc@C=1pc{\ff{X} \ar@{}[rrrd]|{\cong \; \Downarrow \; \gamma } \ar[rrr]^{\ff{f}} \ar[d]_{\ff{a}} &&& \ff{Y} \ar[d]^{\ff{b}} \\
			   \ff{X}' \ar@{}[rrrd]|{= \; \Downarrow \; id}  \ar[d]_{\ff{a}'} \ar[rrr]|{\comw{a} \ff{f}' \comw{a} } &&& \ff{Y'} \ar[d]^{\ff{b}'} \\
			   \widehat{\ff{X}'} \ar[r]_{\overline{\ff{a}'}} & \X' \ar[r]_{\ff{f}'} & \Y' \ar[r]_{\ff{b}'} & \widehat{\ff{Y}'} }}
\vcenter{\xymatrix{\quad = \quad }}
\vcenter{\xymatrix@R=1.5pc@C=1.5pc{\ff{X} \ar@{}[rrdd]|{\cong \; \Downarrow \; \ff{b}'\gamma \circ \ff{b}'\ff{f}'} \ar[rr]^{\ff{f}} \ar[dd]_{\ff{a}'\ff{a}} 
			   & & \ff{Y} \ar[dd]^{\ff{b}'\ff{b}} \\
			   & 
			    \\
			   \widehat{\ff{X}'} \ar[rr]_{\ff{b}'\ff{f}'\overline{\ff{a}'}}
			   & & \widehat{\ff{Y}'} }}$$

\noindent where the bottom row of this diagram belongs to $\cc{H}om_p(\hat{\ff{J}}^{op},\cc{C})$.

Since 
$\cc{H}om_p(\hat{\ff{J}}^{op},\cc{C})$ is 
a closed 2-bmodel 2-category, $\ff{b}'\ff{f}'\overline{\ff{a}'}$ can be factored as $ \ff{b}'\ff{f}'\overline{\ff{a}'} \Mr{\alpha \cong} \ff{p}'\ff{i}'$ where $\ff{p}'$ is both a fibration and a weak equivalence and $\ff{i}'$ is a cofibration in $\cc{H}om_p(\hat{\ff{J}}^{op},\cc{C})$. Consider $\ff{i}=\ff{i}'\ff{a}'\ff{a}$ and $\ff{p}=\overline{\ff{b}}\overline{\ff{b}'}\ff{p}'$. Then \mbox{$\ff{f}\cong \overline{\ff{b}}\overline{\ff{b}'}\ff{b}'\ff{b}\ff{f}\cong \overline{\ff{b}}\overline{\ff{b}'}\ff{b}'\ff{f}'\overline{\ff{a}'}\ff{a}'\ff{a}\cong \overline{\ff{b}}\overline{\ff{b}'}\ff{p}'\ff{i}'\ff{a}'\ff{a}=\ff{p}\ff{i}$.} It can be easily checked that $\ff{p}$ is a retract of $\ff{p}'$ and so, by definition of the structure in $\Prop{C}$, it is a fibration and $\ff{i}$ is a retract of $\ff{i}'$ and so it is a cofibration. It only remains to check that $\ff{p}$ is a weak equivalence: We 
know that $\ff{p}'$ is a weak 
equivalence in $\cc{H}om_p(\hat{\ff{J}}^{op},\cc{C})$, then $\ff{p}'\cong \ff{u}\ff{v}$ where $\ff{u}$ is both a fibration and a weak equivalence and $\ff{v}$ is both a cofibration and a weak equivalence in $\cc{H}om_p(\hat{\ff{J}}^{op},\cc{C})$ and so, $\ff{p}\cong\overline{\ff{b}}\overline{\ff{b}'}\ff{u}\ff{v}$. It can be checked that $\overline{\ff{b}}\overline{\ff{b}'}\ff{u}$ is a retract of $\ff{u}$ and $\ff{v}$ is a retract of $\ff{v}$ which concludes the proof of axiom 2-M2.

\vspace{1ex}

In order to prove axioms 2-M5 and 2-M6, we state and prove some previous lemmas:

\begin{lemma}\label{lema1'.2}
Given a diagram in $\Prop{C}$ of the form 

$$\xymatrix@R=3pc@C=1.5pc{\ff{A} \ar[rr]^{\ff{a}} \ar[d]_{\ff{i}} &\ar@{}[d]|{\cong \; \Downarrow \; \gamma}& \ff{Y} \ar[d]^{\ff{p}} \\
            \ff{X} \ar[rr]_{\ff{b}} && \ff{B}}$$

\noindent where $\ff{p}$ is a fibration and $\ff{i}$ is a trivial cofibration (or $\ff{p}$ is a trivial fibration and $\ff{i}$ is a cofibration), there exists a filler $(\ff{f},\lambda,\rho)$ for that diagram.

\end{lemma}

\begin{proof} We will focus on the case where $\ff{i}$ is a trivial cofibration and $\ff{p}$ is a fibration. The other one is analogous and we omit it. Also, since the lifting property is preserved by the formation of retracts (see \ref{retractospreservanlifting}), it is enough to check this lemma for $\ff{i}$ a strong trivial cofibration in some $\cc{H}om_p(\hat{\ff{K}}^{op},\cc{C})$ and $\ff{p}$ a strong trivial fibration in some $\cc{H}om_p(\hat{\ff{J}}^{op},\cc{C})$:

We are going to define an order preserving morphism $\hat{\ff{J}}^{op} \mr{\F} \hat{\ff{K}}^{op}$ and we are going to construct a diagram 

\begin{equation}\label{diagramalema1'.2}
\xymatrix@R=3pc@C=1.5pc{\ff{A}' \ar[rr]^{\tilde{\ff{a}}} \ar[d]_{\ff{i}'} &\ar@{}[d]|{\cong \; \Downarrow \; \tilde{\gamma}}& \ff{Y} \ar[d]^{\ff{p}} \\
            \ff{X}' \ar[rr]_{\tilde{\ff{b}}} && \ff{B}}
\xymatrix@R=.5pc{\\ \quad \in \cc{H}om_p(\hat{\ff{J}}^{op},\cc{C})}
\end{equation}

\noindent where $\ff{A}'=\ff{A}\F$, $\ff{X}'=\ff{X}\F$, $(\tilde{\ff{a}}_j,\epsilon_j)$ represents $\ff{a}$ and $(\tilde{\ff{b}}_j,\mu_j)$ represents $\ff{b} \ \forall j\in \ff{J}$ for some invertible 2-cells $\epsilon_j$, $\mu_j$:

For $j=0$, consider $\tilde{k_0}\in \ff{K}$ and morphisms $\ff{A}_{\tilde{k_0}}\mr{\ff{a}'_0}\ff{Y}_0$, $\ff{X}_{\tilde{k_0}}\mr{\ff{b}'_0}\ff{B}_0$ and appropriate invertible 2-cells $\epsilon_0$, $\mu_0$ such that $(\ff{a}'_0,\epsilon_0)$ represents $\ff{a}$ and $(\ff{b}'_0,\mu_0)$ represents $\ff{b}$ (see \ref{idrepresentap}). Consider $\ff{A}_{\tilde{k_0}} \mrpair{\ff{p}_0 \ff{a}'_0}{\ff{b}'_0 \ff{i}_{\tilde{k_0}}} \ff{B}_0$  and $\alpha_0$ given by the following composition

$$\vcenter{\xymatrix@C=-0pc{\ff{p}_0 \deq & & \ff{a}'_0 \dl & \dc{\epsilon_0} & \pi_{\tilde{k_0}} \dr \\
                              \ff{p}_0 \dl & \dc{=} & \pi_{0} \dr & & \ff{a} \deq \\
                              \pi_0 \deq & & \ff{p} \dl & \dc{\gamma} & \ff{a} \dr\\
                              \pi_0 \dl & \dc{\mu_0^{-1}} & \ff{b} \dr & & \ff{i} \deq \\
                              \ff{b}'_0 \deq & & \pi_{\tilde{k_0}} \dl & \dc{=} & \ff{i} \dr \\
                              \ff{b}'_0 & &\ff{i}_{\tilde{k_0}} && \pi_{\tilde{k_0}} }}$$
                              
\noindent Then, by \ref{lema1p}, there exists $\tilde{k_0}\leq k_0 \in \ff{K}$ and an invertible 2-cell $\ff{A}_{k_0}\cellrd{\ff{p}_0 \ff{a}'_0 \ff{A}_{\tilde{k_0}\leq k_0}}{\theta_0}{\ff{b}'_0 \ff{i}_{\tilde{k_0}} \ff{A}_{\tilde{k_0}\leq k_0}} \ff{B}_0 \in \cc{C}$ such that    


$$\vcenter{\xymatrix@C=-0pc{ \p_0 \dl && \fa'_0 \dc{\theta_0} && \dr \A\kko && \pi\sko \deq \\
\fb'_0 \deq && \ii\skto \deq && \A\kko && \pi\sko \\
\fb'_0 && \ii\skto &&& \pi\skto \clb{\pi\kko} }}
\vcenter{\xymatrix@C=-0pc{ \quad = \quad }}
\vcenter{\xymatrix@C=-0pc{ \p_0 \deq && \fa'_0 \deq && \A\kko && \pi\sko \\
\p_0 \dl && \fa'_0  & \dcr{\alpha_0} && \dr \pi\skto \clb{\pi\kko} \\
\fb'_0 && \ii\skto &&& \pi\skto }}$$

Take $\F(0)=k_0$, $\ff{i}'_0=\ff{i}_{k_0}$, $\tilde{\ff{a}}_0=\ff{a}'_0 \ff{A}_{\tilde{k_0}\leq k_0}$, $\tilde{\ff{b}}_0=\ff{b}'_0 \ff{X}_{\tilde{k_0}\leq k_0}$, $\tilde{\gamma}_0= 
\vcenter{\xymatrix@C=-0.2pc@R=1pc{ \p_0 \deq &&& \tilde{\fa}_0 \op{=} \\
\p_0 \dl && \fa'_0 \dc{\theta_0} && \dr \A\kko \\
\fb'_0 \deq && \ii\skto \dl & \dc{\ii\kko} & \dr \A\kko \\
\fb'_0 && \X\kko && \ii\sko \dcell{=} \\
& \tilde{\fb}_0 \cl{=} &&& \ii'_0}}
$ and we redefine $\epsilon_0$ and $\mu_0$ so that $(\tilde{\ff{a}}_0,\epsilon_0)$ represents $\ff{a}$ and $(\tilde{\ff{b}}_0,\mu_0)$ represents $\ff{b}$.

For $j\neq 0$, suppose that we have already defined all the data $\forall \ j'<j$. By using \ref{idrepresentap}, consider $\tilde{k_j}\in \ff{K}$, morphisms $\ff{A}_{\tilde{k_j}}\mr{\ff{a}'_j}\ff{Y}_j$, $\ff{X}_{\tilde{k_j}}\mr{\ff{b}'_j}\ff{B}_j$ and appropriate invertible 2-cells $\epsilon_j$, $\mu_j$ such that $(\ff{a}'_j,\epsilon_j)$ represents $\ff{a}$, $(\ff{b}'_j,\mu_j)$ represents $\ff{b}$ and $\tilde{k_j}\geq \F(j') \ \forall \ j'<j$ (this can be done because $\ff{K}$ is cofinite and filtered). 

Suppose $\{j'| j'<j\}=\{0,j_0,...,j_n\}$. By applying \ref{lema1p} to $\tilde{\ff{a}}_0$ and $\ff{a}'_j$ we obtain \mbox{$\tilde{\tilde{k_j}}\geq \tilde{k_j}, \F(0)$} and an invertible 2-cell $\ff{A}_{\tilde{\tilde{k_j}}}\cellrd{\ff{Y}_{0<j} \ff{a}'_j \ff{A}_{\tilde{k_j}\leq \tilde{\tilde{k_j}}}}{\mu}{\tilde{\ff{b}}_0 \ff{A}_{\F(0) \leq \tilde{\tilde{k_j}}}} \ff{Y}_0$ such that 


$$\vcenter{\xymatrix@C=-0pc{ \Y\oj \ardr && \fa'_j \dc{\mu} && \A\kkj \ardl && \pi\skttj \deq \\ 
& \tilde{\fa}_0 \deq && \A\vokj &&& \pi\skttj \\
& \tilde{\fa}_0 \dl && \dc{\eps_0} & \dr \pi\svo \clunodos{\pi\vokj} \\
& \pi_0 &&& \fa }}
\vcenter{\xymatrix@C=-0pc{ \quad = \quad }}
\vcenter{\xymatrix@C=-0pc{ \Y\oj \ddeq && \fa'_j \deq && \A\kkj && \pi\skttj \\
&& \fa'_j \dl && \dc{\eps_j} & \dr \pi\sktj \cl{\pi\kkj \ \ } \\
\Y\oj && \pi_j &&& \fa \deq \\
& \pi_0 \cl{\pi\oj} &&&& \fa}}$$

We rename $\ff{a}'_j=\ff{a}'_j \ff{A}_{\tilde{k_j}\leq \tilde{\tilde{k_j}}}$ and $\epsilon_j$ the corresponding invertible 2-cell. We repeat the procedure with each $j_i$ from $j_0$ to $j_n$ instead of $0$. Then we do the same thing for $\ff{b}$ and consider a new $k_j$ above both the obtained ones. Then we repeat for $j$ the procedure we did in the beginning for $0$ and we take $\ff{i}'_j=\ff{i}_{\F(j)}$ and $\tilde{\gamma}_j= 
\vcenter{\xymatrix@C=-0pc{ & \p_j \ardl & \dc{\theta_j} & \tilde{\fa}_j \ardr \\
\fb'_j \deq && \ii\sktj \dl & \dc{\ii\kvj} & \dr \A\kvj \\
\fb'_j && \X\kvj && \ii\svj \deq \\
& \tilde{\fb}_j \cl{=} &&& \ii'_j}}
$.
 
Since $\F$ is an order preserving morphism, we define it in morphisms and 2-cells in the obvious way.

It is straightforward to check that $\ff{A}'$ and $\ff{X}'$ are 2-functors. It also can be easily checked that $\tilde{\ff{a}}$, $\tilde{\ff{b}}$ and $\ff{i}'$ are pseudo-natural transformations.   

Then we have a filler $(\ff{f}',\lambda',\rho')$ for diagram \eqref{diagramalema1'.2}. 

Since $\{\ff{X}\mr{\pi_{\F(j)}} \ff{X}'_j\}_{j\in \ff{J}}$, $\Bigg\{ 
\vcenter{\xymatrix@C=-0.2pc@R1pc{ \X'_{j'<j} \dcell{=} && \pi\svj \deq \\
\X_{\F(j')<\F(j)} && \pi\svj \\
& \pi\svjj \clrightb{\pi_{\F(j')<\F(j)}} }}
\Bigg\}_{j'<j}$ is a pseudo-cone, there exists a morphism $\ff{X}\mr{\ff{g}}\ff{X}'\in \Prop{C}$ such that $\pi_j \ff{g}=\pi_{\F(j)} \ \forall \ j\in \ff{J}$ and \mbox{$\pi_{j'<j} \ff{g}=\pi_{\F(j')< \F(j)} \ \forall \ j'<j\in \ff{J}$}. In a similar way, we can define a morphism \mbox{$\ff{A}\mr{\ff{f}}\ff{A}'\in \Prop{C}$} such that  $\pi_j \ff{f}=\pi_{\F(j)} \ \forall \ j\in \ff{J}$ and $\pi_{j'<j} \ff{f}=\pi_{\F(j')< \F(j)} \ \forall \ j'<j\in \ff{J}$.

We also have invertible 2-cells $\ff{A}\cellrd{\ff{a}}{\alpha}{\tilde{\ff{a}} \ff{f}} \ff{Y}$, $\ff{X}\cellrd{\tilde{\ff{b}} \ff{g}}{\beta}{\ff{b}} \ff{Y}$ induced by the morphisms of pseudo-cones $\Bigg\{ 
\vcenter{\xymatrix@C=-0.2pc@R.5pc{ \pi_j \dl & \dcr{\eps_j\inv} && \fa \dr \\  \tilde{a}_j \deq &&& \pi\svj \op{=} \\ \tilde{a}_j \dl & \dc{=} & \pi_j \dr && \f \deq \\ \pi_j && \tilde{\ff{a}} && \f }}
\Bigg\}_{j\in \ff{J}}$, $\Bigg\{ 
\vcenter{\xymatrix@C=-0.2pc@R.5pc{ \pi_j \dl & \dc{=} & \dr \tilde{b} && \g \deq \\  \tilde{b}_j \deq && \pi_j && \g \\  \tilde{b}_j \dl & \dcr{\mu_j} && \dr \pi\svj \cl{=} \\  \pi_j &&& \fb}}
\Bigg\}_{j\in \ff{J}}$ respectively.

Note that $\ff{i}' \ff{f}=\ff{g} \ff{i}$.

Then,  $\;\;(\ff{f}' \ff{g}, 
\vcenter{\xymatrix@C=-0.2pc@R=1pc{ &&&  \fa \opdosuno{\alpha \;\;} \\  & \tilde{\fa} \op{\lambda'} &&& \f \deq \\  \f' \deq && \ii' \dl & \dc{=} & \dr \f \\  \f' && \g && \ii }}
, 
\vcenter{\xymatrix@C=-0.2pc@R=1.3pc{ \p && \f' && \g \deq \\ & \tilde{\fb} \clmediob{\rho'} &&& \g \\ && \fb \clunodos{\beta} }}
)\;\;$ is the filler that we were looking for.

\end{proof}

\begin{lemma}\label{lema2'.2}
\comw{A}

\begin{enumerate}
 \item A morphism $\ff{i}\in \Prop{C}$ is a trivial cofibration iff is both a cofibration and a weak equivalence.
 \item A morphism $\ff{p}\in \Prop{C}$ is a trivial fibration iff is both a fibration and a weak equivalence.
\end{enumerate}
\end{lemma}

\begin{proof}\comw{aaaa}

\begin{enumerate}
  \item $\Rightarrow)$ Let $\ff{i}\in \Prop{C}$ be a trivial cofibration. Since strong trivial cofibrations are cofibrations, $\ff{i}$ is a cofibration by definition. Plus, we can factorize $\ff{i}\cong id \ff{i}$ and so $\ff{i}$ is a weak equivalence.
  
  $\Leftarrow)$ Let $\ff{i}\in \Prop{C}$ be a morphism that is both a cofibration and a weak equivalence. Then we can factorize $\ff{i}\stackrel{\gamma}{\cong}\ff{p}\ff{j}$ where $\ff{p}$ is a trivial fibration and $\ff{j}$ is a trivial cofibration. It is enough to check that $\ff{i}$ is a retract of $\ff{j}$ but this is true by \ref{retractargument} plus \ref{lema1'.2}.
%
%
%

  \item The proof is analogous to the previous one and we leave it to the reader. 
 \end{enumerate}
\end{proof}

{\bfseries Axiom 2-M6a)}: $\Rightarrow)$ If $\ff{p}$ is a fibration and $\ff{i}$ is both a cofibration and a weak equivalence, then, by \ref{lema2'.2}, $\ff{i}$ is a trivial cofibration and, by \ref{lema1'.2}, the pair $(\ff{i},\ff{p})$ has the lifting property. 

$\Leftarrow)$ Suppose that we have a morphism $\ff{X}\mr{\ff{p}}\ff{Y}$ that has the right lifting property with respect to all morphisms that are both a cofibration and a weak equivalence. By 2-M2, $\ff{p}$ can be factorized as $\ff{p}\stackrel{\gamma}{\cong}\ff{q}\ff{i}$ where $\ff{q}$ is a fibration and $\ff{i}$ is both a cofibration and a weak equivalence. Then the pair $(\ff{i},\ff{p})$ has the lifting property and so, by \ref{retractargument}, $\ff{p}$ is a retract of $\ff{q}$. Then $\ff{p}$ is a fibration.

%
%
%

{\bfseries Axiom 2-M6b)}: The proof of this axiom is analogous to the previous one and we leave it to the reader.

\vspace{2ex}
\begin{lemma}\label{lema3'.2}
\comw{A}

\begin{enumerate}
  \item A morphism $\ff{i}\in \Prop{C}$ is a trivial cofibration iff for every fibration $\ff{p}$, the pair $(\ff{p},\ff{i})$ has the lifting property.
 \item A morphism $\ff{p}\in \Prop{C}$ is a trivial fibration iff for every cofibration $\ff{i}$, the pair $(\ff{p},\ff{i})$ has the lifting property.
\end{enumerate} 
\end{lemma}

\begin{proof}\comw{aaa}
 \begin{enumerate}
  \item $\Rightarrow)$ It is immediate from \ref{lema1'.2}.
  
  $\Leftarrow)$ By axiom 2-M2 plus \ref{lema2'.2}, we can factorize $\ff{i}\stackrel{\gamma}{\cong} \ff{u}\ff{v}$ where $\ff{u}$ is a fibration and $\ff{v}$ is a trivial cofibration. By \ref{retractargument}, $\ff{i}$ is a retract of $\ff{v}$ which concludes the proof. 
  
  \item The proof is analogous to the previous one and we leave it to the reader. 
 \end{enumerate}
\end{proof}

{\bfseries Axiom 2-M6c)}: It follows immediately from \ref{lema3'.2} and the definition of weak equivalences in $\Prop{C}$.

\begin{lemma}\label{lema4'.2}
 \comw{A}
 
 \begin{enumerate}
 \item If a morphism $\ff{i}\in \Prop{C}$ has the left lifting property with respect to all fibrations between fibrant objects, then $\ff{i}$ is a trivial cofibration.
 \item If a morphism $\ff{p}\in \Prop{C}$ has the right lifting property with respect to all cofibrations between cofibrant objects, then $\ff{p}$ is a trivial fibration.
\end{enumerate} 
\end{lemma}

\begin{proof}\comw{aaa}
 \begin{enumerate}
  \item By \ref{lema3'.2}, it is enough to check that the set \mbox{$L=\{\ff{p}\ |\ \mbox{ the pair } (\ff{i},\ff{p}) \mbox{ has the lifting property }\}$} contains all fibrations. In order to do that, first, we are going to prove some properties about this set:
  
  \begin{enumerate}
   \item $L$ is closed under bi-pullbacks: The proof is the same as the one in \ref{sobran axiomas} of the fact that fibrations are closed under bi-pushouts.
%
%
%
%
%
   
   \item Given an inverse bi-limit pseudo-cone $\{\ff{Y}\mr{\pi_j}\ff{E}(j)\}_{j\in \ff{J}}$ where $\{\ff{E}(j)\}_{j\in \ff{J}}$ is an inverse pseudo-system in $\Prop{C}$ with $\ff{J}$ a cofinite filtered poset with a unique initial object $0$ such that for each \mbox{$j\neq 0\in \ff{J}$}, the morphism $\ff{E}(j)\mr{\ff{e}_j} \biLim{k<j}{\ff{E}(k)}$ induced by the pseudo-cone \mbox{$\{\ff{E}(j)\mr{\ff{E}(k<j)} \ff{E}(k)\}_{k<j}$}, $\Bigg\{
   \vcenter{\xymatrix@R=1pc@C=-1.7pc{ \ff{E}(k<l) && \ff{E}(l<j) \\ & \ff{E}(k<j) \cl{{\alpha^\ff{E}_{k,l,j}}} }}
   \Bigg\}_{k<l<j}$ belongs to $L$, then the induced morphism \mbox{$\ff{Y} \mr{\pi_0} E(0)$} also belongs to $L$:
   
   Suppose that we have a diagram  as follows in $\Prop{C}$: 
   
   $$\xymatrix@R=1.5pc@C=1.5pc{\ff{A} \ar@{}[ddrr]|{\cong \; \Downarrow \; \gamma } \ar[rr]^{\ff{a}} \ar[dd]_{\ff{i}} 
			   & & \ff{Y} \ar[dd]^{\pi_0} \\
			   & 
			   \\
			   \ff{X} \ar[rr]_{\ff{b}}
			   & & \ff{E}(0) }$$
			   
We are going to define inductively a pseudo-cone $\{\ff{X}\mr{\ff{f}_j}\ff{E}(j)\}_{j\in \ff{J}}, \{\ff{f}_{k<j}\}_{k<j}$ and an isomorphism of pseudo-cones $\{\pi_j \ff{a} \Mr{\alpha_j} \ff{f}_j \ff{i}\}_{j\in \ff{J}}$: $\ff{f}_0=\ff{b}$, $\alpha_0=\gamma$. Suppose that we have already defined $\ff{f}_k$, $\alpha_k \ \forall \ k<j$:

\noindent From definition of $\ff{e}_j$, we have invertible 2-cells $\pi_k \ff{e}_j \Mr{\beta_k} \ff{E}(k<j) \ \forall \ k<j$ such that $ 
\vcenter{\xymatrix@C=-1pc{ \E\pkl && \pi_l && \e_j \deq \\ & \pi_k \cl{\pi_{k<l}} &&& \e_j \\ && \E\pkj \clunodos{\beta_k} }}
\vcenter{\xymatrix@C=-0pc{ \quad = \quad }}
\vcenter{\xymatrix@C=-1pc{ \E\pkl \deq && \pi_l && \e_j \\ \E\pkl &&& \E\plj \cl{\beta_l} \\ && \E\pkj \cldosuno{\alpha^\E_{k,l,j}} }}
\ \forall \ k<l<j$. 

\noindent Since $\ff{e}_j \in L$, there exists a filler $(\ff{f}_j,\tilde{\lambda_j},\tilde{\rho_j})$ for the following diagram:

$$\xymatrix@R=1.5pc@C=1.5pc{\ff{A} \ar@{}[rrdd]|{\cong \; \Downarrow \; \tilde{\gamma} } \ar[rr]^{\pi_j \ff{a}} \ar[dd]_{\ff{i}} 
			   & & \ff{E}(j) \ar[dd]^{\ff{e}_j} \\
			   & 
			    \\
			   \ff{X} \ar[rr]_{\tilde{\ff{b}}}
			   & & \biLim{k<j}{\ff{E}(k)} }$$

\noindent where $\tilde{b}$ is induced by the pseudo-cone $\{\ff{f}_k\}_{k<j}$, $\{\ff{f}_{k<l}\}_{k<l<j}$ and so we have invertible 2-cells $\pi_k \tilde{\ff{b}} \Mr{\delta_k} \ff{f}_k \ \forall \ k<j$ such that $ 
\vcenter{\xymatrix@C=-0pc@R=1pc{ \E\pkl && \pi_l && \tilde{\fb} \deq \\ & \pi_k \cl{\pi_{k<l} \;\;} &&& \tilde{\fb} \\ && \f_k \clunodos{\delta_k} }}
\vcenter{\xymatrix@C=-0pc{ \quad = \quad }}
\vcenter{\xymatrix@C=-0pc@R=1pc{ \E\pkl \deq && \pi_l && \tilde{\fb} \\ \E\pkl &&& \f_l \cl{\delta_l} \\ && \f_k \cldosuno{\f_{k<l} \quad} }}
\ \forall \ k<l<j$. And $\tilde{\gamma}$ is such that $$\pi_k \gamma= 
\vcenter{\xymatrix@C=-0pc@R=1pc{ \pi_k && \e_j && \pi_j \deq && \fa \deq \\
& \E\pkj \cl{\beta_k} &&& \pi_j && \fa \deq \\
&& \pi_k \clunodos{\pi\kj} \dl && \dc{\alpha_k} && \dr \fa \\
&& \f_k \opmediob{\delta_k\inv} &&&& \ii \deq \\
& \pi_k && \tilde{\fb} &&& \ii}}
\ \forall k<j.$$

\noindent Take $\ff{f}_{k<j}= 
\vcenter{\xymatrix@C=-0pc@R=1pc{ & \E\pkj \op{\beta_k\inv} &&& \f_j \deq \\   \pi_k \deq && \e_j && \f_j \\   \pi_k &&& \tilde{\fb} \cl{\tilde{\rho}_j} \\   & \f_k \clunodos{\delta_k} }}
$ and $\alpha_j=\tilde{\lambda_j}$. It can be easily checked that this data defines a pseudo-cone and an isomorphism of pseudo-cones as we wanted.
                                                                                                       
This pseudo-cone induces a morphism $\ff{X}\mr{\ff{f}}\ff{Y} \in \Prop{C}$ and so we have invertible 2-cells $\pi_j \ff{f} \Mr{\mu_j} \ff{f}_j \ \forall \ j\in \ff{J}$ such that $ 
\vcenter{\xymatrix@C=-0pc{ \E\pkl && \pi_l && \f \deq \\ & \pi_k \cl{\pi_{k<l} \;} &&& \f \\ && \f_k \clunodos{\mu_k} }}
\vcenter{\xymatrix@C=-0pc{ \quad = \quad }}
\vcenter{\xymatrix@C=-0pc{ \E\pkl \deq && \pi_l && \f \\ \E\pkl &&& \f_l \cl{\mu_l} \\ && \f_k \cldosuno{\f_{k<l} \quad} }}
\ \forall \ k<l<j$. Let's check that $(\ff{f}, \lambda , \rho)$ is the filler that we were looking for, where $\lambda$ is such that $\pi_j \lambda= 
\vcenter{\xymatrix@C=-0pc{ & \pi_j \dl & \dcr{\alpha_j} && \dr \fa \\   & \f_j \opb{\mu_j\inv} &&& \ii \deq \\   \pi_j && \f && \ii }}
\ \forall \ j\in \ff{J}$ and $\rho=\mu_0$: 


$$\vcenter{\xymatrix@C=-0pc{\pi_0 \deq &&& \fa \op{\lambda} \\   \pi_0 && \f && \ii \deq \\  & \fb \cl{\mu_0} &&& \ii }}
\vcenter{\xymatrix@C=-0pc{ \quad = \quad }}
\vcenter{\xymatrix@C=-0pc{ & \pi_0 \dl & \dcr{\gamma} && \dr \fa \\  & \f_0 \opb{\mu_0\inv} &&& \ii \deq \\   \pi_0 && \f && \ii \deq \\   & \fb \cl{\mu_0} &&& \ii }}
\vcenter{\xymatrix@C=-0pc{ \quad = \quad }}
\vcenter{\xymatrix@C=-0pc{ \pi_0 \dl & \dc{\gamma} & \dr \fa \\   \fb && \ii}}$$
   
   \item $L$ is closed under the formation of retracts: It follows immediately from \ref{retractospreservanlifting}.
%
%
%
%
%
%
\end{enumerate}

Now we are going to prove that $L$ contains all fibrations in $\Prop{C}$:

Let $\ff{p}$ be a fibration in $\Prop{C}$. If $\ff{p}$ is a fibration between fibrant objects of $\cc{C}$, by \ref{lema0.2}, $\ff{p}$ is a fibration between fibrant objects in $\pCJ$ and $\hat{\ff{p}}$ is a fibration between fibrant objects in $\cc{H}om_p(\hat{\ff{J}}^{op},\cc{C})$. But $\hat{\ff{p}}$ is $\ff{p}$ seen as a morphism in $\cc{H}om_p(\hat{\ff{J}}^{op},\cc{C})$ and so $\ff{p}$ is a fibration between fibrant objects in $\Prop{C}$. Then, by hypothesis, $\ff{p}\in L$.


Since $L$ is closed under bi-pullbacks and $\cc{C}$ satisfies axiom 2-N2, $L$ contains every fibration in $\cc{C}$.

Let $\ff{Y}\mr{\ff{p}}\ff{B}\in \cc{H}om_p(\hat{\ff{J}}^{op},\cc{C})$ be a fibration for some cofinite and filtered poset $\ff{J}$ with a unique initial object. 

We are going to construct a system $\{\ff{E}(j)\}_{j\in \ff{J}}$ satisfying the hypothesis of $b)$:

Take $\ff{E}(0)=\ff{B}$ and for each $j\neq 0$, let $\ff{E}(j)$ be the following bi-pullback:

$$\xymatrix@C=1.5pc@R=3pc{\ff{E}(j) \ar[rr]^{\h_0^j} \ar[d]_{\h_1^j}
                            & \ar@{}[d]|{\cong \; \Downarrow \; \delta_j } 
                            &\ff{Y}_j \ar[d]^{\ff{p}_j} \\
                            \ff{B} \ar[rr]_{\pi^{\ff{B}}_j} 
                            && \ff{B}_j}$$

\noindent $\ff{E}(0<j)=\h_1^j$ and if $0\neq k<j$, $\ff{E}(k<j)$ is given by the following diagram

$$\xymatrix@R=3pc@C=3.5pc{\ff{E}(j)\ar@/^4ex/[rrd]^{\ff{Y}_{k<j}\h_0^j} \ar@{-->}[rd]|{\ff{E}(k<j)} \ar@/_4ex/[rdd]_{\h_1^j}  & \ar@{}[d]|{\cong \; \Uparrow \; \beta_0^{k,j}}\\
            \ar@{}[r]|{\cong \; \Downarrow \; \beta_1^{k,j}} & \ff{E}(k) \ar[r]^{\h_0^k} \ar[d]_{\h_1^k} \ar@{}[rd]|{\cong \; \Downarrow \; \delta_k} & \ff{Y}_k \ar[d]^{\ff{p}_k}\\
            & \ff{B} \ar[r]_{\pi^{\ff{B}}_k} & \ff{B}_k}$$

\noindent and so $ 
\vcenter{\xymatrix@C=-0pc@R=1.3pc{ \p_k \dl & \dc{\delta_k} & \dr \h_0^k && \E\pkj \deq \\  \pi_k^{\B} \deq && \h_1^k && \E\pkj \\  \pi_k^{\B} &&& \h_1^j \cl{\beta_1\tkj} }}
\vcenter{\xymatrix@C=-0pc{ \quad = \quad }}
\vcenter{\xymatrix@C=-0pc@R=1pc{ \p_k \deq && \h_0^k \dl & \dc{\beta_0\tkj} & \dr \E\pkj \\
\p_k \dl & \dc{\p\kj\inv} & \dr \Y\kj && \h_0^j \deq \\
\B\kj \deq && \p_j \dl & \dc{\delta_j} & \dr \h_0^j \\
\B\kj && \pi_j^{\B} && \h_1^j \deq \\
& \pi_k^{\B} \cl{\pi\kj} &&& \h_1^j}}
$. It can be checked that $\ff{E}$ is a pseudo-functor.

Let's check that $\ff{Y}$ satisfies the universal property of $\biLim{j\in \ff{J}}{\ff{E}(j)}$ with projection $\h_0=\ff{p}$: 

For $j\neq 0$, take $\ff{Y} \mr{\pi_j} \ff{E}(j)$ as in the following diagram:

$$\xymatrix@R=3pc@C=3.5pc{\ff{Y} \ar@/^4ex/[rrd]^{\pi^{\ff{Y}}_{j}} \ar@{-->}[rd]|{\comw{M^M} \pi_j \comw{M^M}} \ar@/_4ex/[rdd]_{\ff{p}}  & \ar@{}[d]|{\cong \; \Uparrow \; \theta_0^{j}} \\
            \ar@{}[r]|{\cong \;\Downarrow \;\theta_1^{j}} & \ff{E}(j) \ar[r]^{\h_0^j} \ar[d]_{\h_1^j} \ar@{}[rd]|{\cong \; \Downarrow \; \delta_j } & \ff{Y}_j \ar[d]^{\ff{p}_j}\\
            & \ff{B} \ar[r]_{\pi^{\ff{B}}_j} & \ff{B}_j}$$

\noindent and so we have 

$$\vcenter{\xymatrix@C=-0pc{ \p_j \dl & \dc{\delta_j} & \dr \h_0^j && \pi_j \deq \\  \pi_j\tB \deq && \h_1^j && \pi_j \\  \pi_j\tB &&& \p \cl{\theta_1^j} }}
\vcenter{\xymatrix@C=-0pc{ \quad = \quad }}
\vcenter{\xymatrix@C=-0pc{ \p_j \deq && \h_0^j && \pi_j \\  \p_j \dl & \dcr{=} && \dr \pi_j\tY \cl{\theta_0^j} \\  \pi_j\tB &&& \p}}$$

Then take $\pi_{0<j}=\theta_1^j$ and if $0\neq k<j$, $\pi_{k<j}$ such that

$\h_0^k \pi_{k<j}=\vcenter{\xymatrix@C=-0pc@R=1pc{\h_0^k \dl & \dc{\beta_0\tkj} & \dr \E\pkj && \pi_j \deq \\
\Y\kj \deq && \h_0^j && \pi_j \\
\Y\kj &&& \pi_j\tY \cl{\theta_0^j} \\
&& \pi_k\tY \cldosuno{\pi\kj\tY} \opdosuno{(\theta_0^k)\inv} \\ 
\h_0^k &&& \pi_k}}$
and $\quad$ 
$\h_1^k \pi_{k<j}=\vcenter{\xymatrix@C=-0pc@R=1.3pc{\h_1^k && \E\pkj && \pi_j \deq \\ 
& \h_1^j \cl{\; \; \beta_1\tkj} &&& \pi_j \\ 
&& \p \clunodos{\theta_1^j} \opunodos{(\theta_1^k)\inv \quad \quad } \\ 
& \h_1^k &&& \pi_k }}$

It can be checked that this data defines a pseudo-cone. 

Now suppose that we have another pseudo-cone $\{\ff{Z}\mr{\varphi_j} \ff{E}(j)\}_{j\in \ff{J}}$, \mbox{$\{\ff{E}(k<j) \varphi_j \Mr{\varphi_{k<j}} \varphi_k \}_{k<j}$}. Consider the pseudo-cone $\{\ff{Z}\mr{\h_0^j\varphi_j} \ff{Y}_j\}_{j\in \ff{J}}$, $\Bigg\{ 
\vcenter{\xymatrix@C=-0.2pc@R=.5pc{\Y\kj \dl & \dc{\;\; (\beta_0\tkj)\inv} & \dr \h_0^j && \varphi_j \deq \\
\h_0^k \deq && \E\pkj && \varphi_j \\
\h_0^k &&& \varphi_k \clb{\varphi\kj} }}
\Bigg\}_{k<j}$. Then there exists a morphism $\ff{Z}\mr{\varphi} \ff{Y} \in \Prop{C}$ such that $\pi^{\ff{Y}}_j\varphi=\h_0^j \varphi_j$ and $ 
\vcenter{\xymatrix@C=-0pc@R=1pc{ \Y\kj && \pi_j\tY && \varphi \deq \\  & \pi_k\tY \clb{\pi\kj\tY} &&& \varphi}}
\vcenter{\xymatrix@C=-0pc{ \quad = \quad }}
\vcenter{\xymatrix@C=-0pc@R=1pc{ \Y\kj \deq && \pi_j\tY \dl & \dc{=} & \dr \varphi \\
\Y\kj \dl & \dc{(\beta_0\tkj)\inv} & \dr \h_0^j && \varphi_j \deq \\
\h_0^k \deq && \E\pkj && \varphi_j \\
\h_0^k \dl & \dcr{=} && \dr \varphi_k \clb{\varphi\kj} \\
\pi_k\tY &&& \varphi}}
$.
 
It can be checked that $\{\pi_j \varphi \Mr{\rho_j} \varphi_j\}_{j\in \ff{J}}$ is an isomorphism of pseudo-cones where $\rho_j$ is such that $\h_0^j \rho_j= 
\vcenter{\xymatrix@C=-0pc@R=1pc{ \h_0^j && \pi_j && \varphi \deq \\  & \pi_j \cl{\theta_0^j} \dl & \dcr{=} && \dr \varphi \\ & \h_0^j &&& \varphi_j }}
\quad $ and 

$\pi^{\ff{B}}_k \h_1^j \rho_j= 
\vcenter{\xymatrix@C=-0pc@R=1pc{ \pi_k\tB \deq && \h_1^j && \pi_j && \varphi \deq \\
\pi_k\tB \dl & \dcr{=} && \dr \p \cl{\theta_1^j} &&& \varphi \deq \\
\p_k \deq &&& \pi_k\tY \dl & \dcr{=} && \dr \varphi \\
\p_k \dl & \dcr{\delta_k} && \dr \h_0^k &&& \varphi_k \deq \\
\pi_k\tB \deq &&& \h_1^k \deq &&& \varphi_k \op{\varphi\kjj\inv} \\
\pi_k\tB \deq &&& \h_1^k && \E\pkjj && \varphi\sj \deq \\
\pi_k\tB \deq &&&& \h_1\tj \cl{\:\: \beta_1\tkjj} \opbymedio{\;\; (\beta_1\tjjj)\inv} &&& \varphi\sj \deq \\
\pi_k\tB \deq &&& \h_1^j \deq && \E\pjjj && \varphi\sj \\
\pi_k\tB &&& \h_1^j &&& \varphi_j \cl{\varphi_{j<j'}} }}
\forall \ k\in \ff{J}$ where $j'$ is such that \mbox{$k,j\leq j'$.}

By similar arguments, one can check the full and faithfulness of the equivalence.

Now we are going to check that $\ff{E}(j)\mr{\ff{e}_j}\biLim{k<j}{\ff{E}(k)}\in L$:

\noindent Suppose that we have a diagram of the form

\begin{equation}\label{diagramaej}
\xymatrix@R=1.5pc@C=1.5pc{\ff{A} \ar@{}[ddrr]|{\cong \; \Downarrow \; \gamma } \ar[rr]^{\ff{a}} \ar[dd]_{\ff{i}} 
			   & & \ff{E}(j) \ar[dd]^{\ff{e}_j} \\
			   & 
			    \\
			   \ff{X} \ar[rr]_{\ff{b}}
			   & & \biLim{k<j}{\ff{E}(k)}}
\end{equation}
			   
\noindent It can be checked that the following square is a bi-pullback

$$\xymatrix@R=1.5pc@C=1.5pc{\biLim{k<j}{\ff{E}(k)} \ar@{}[rrdd]|{\cong \; \Downarrow \; \nu_j  } \ar[rr]^{\ff{a}_{\ff{E},\ff{Y}}^j} \ar[dd]_{\pi^{\ff{E}}_0} 
			   & & \biLim{k<j}{\ff{Y}_k} \ar[dd]^{\biLim{k<j}{\ff{p}_k}} \\
			   & 
			   \\
			   \ff{B} \ar[rr]_{\ff{b}_{\ff{B}}^j}
			   & & \biLim{k<j}{\ff{B}_k}}$$
			   
\noindent where $\biLim{k<j}{\ff{p}_k}$ is defined as in \ref{lema1.2},
$\ff{b}_{\ff{B}}^j$ is induced by the pseudo-cone $\{\ff{B}\mr{\pi^{\ff{B}}_k} \ff{B}_k\}_{k<j}$, $\{ \pi^{\ff{B}}_{k<l}\}_{k<l<j}$ and so we have an isomorphism of pseudo-cones $\{\pi^{\ff{B}}_k  \ff{b}_{\ff{B}}^j \Mr{\beta'_k} \pi^{\ff{B}}_k\}_{k<j}$; and $\ff{a}_{\ff{E},\ff{Y}}^j$ is induced by the pseudo-cone $\{\biLim{k<j}{\ff{E}(k)} \mr{\h_0^k \pi_k} \ff{Y}_k\}_{k<j}$, $\Bigg\{ 
\vcenter{\xymatrix@C=-0pc{ \Y\kj \dl & \dc{(\beta_0\tkl)\inv} & \dr \h_0^l && \pi_l \deq \\  
\h_0^k \deq && \E\pkl && \pi_l \\
\h_0^k &&& \pi_k \cl{\pi\kl}  }}
\Bigg\}_{k<l<j}$ and so we have an isomorphism of pseudo-cones $\{\pi^{\ff{Y}}_k \ff{a}_{\ff{E},\ff{Y}}^j \Mr{\theta_k} \h_0^k \pi_k\}_{k<j}$.

Consider the following diagram 

$$\xymatrix@C=3pc{\ff{X} \ar@{-->}[rd]^{\ff{h}} \ar[d]_{\ff{b}}  \ar[r]^{\ff{b}} & \biLim{k<j}{\E(k)} \ar@/^4ex/[rd]^{\ff{a}_{\ff{E},\ff{Y}}^j }  \ar@{}[d]^{\cong \; \Uparrow \; \nu_0^{j}}\\
            \biLim{k<j}{\E(k)} \ar@{}[r]_{\cong \; \Downarrow \; \nu_1^{j}} \ar[d]_{\pi_0\tE} & \ff{P}_j \ar[r]^{\pi_0^j} \ar[d]_{\pi_1^j} 
            \ar@{}[rd]^{bipb \quad\quad\quad\quad}_{\quad\quad\quad\quad\quad \cong \; \Downarrow \; \alpha_j } 
             & \biLim{k<j}{\ff{Y}_k} \ar[d]^{\biLim{k<j}{\ff{p}_k}}\\
            \B \ar[r]_{\pi^{\ff{B}}_j} & \ff{B}_j \ar[r]_{\ff{a}_{\ff{B}}^j} & \biLim{k<j}{\ff{B}_k}}$$
            
\noindent where $\ff{a}_{\ff{B}}^j$ is defined as in \ref{lema1.2}
and the following equality holds

$$\vcenter{\xymatrix@C=-0pc{ \biLim{k<j}{\p_k} \ardr &&  \dcr{\alpha_j} &&& \pi_0^j \ardl && \h \deq \\
& \fa\sB^j \deq &&& \pi^j_1 \ardl & \dcr{\nu_1^j} && \h \ardr \\
& \fa\sB^j && \pi_j\tB && \pi^{\E}_0 &&& \fb}}
\vcenter{\xymatrix@C=-0pc{ \quad = \quad }}
\vcenter{\xymatrix@C=-0pc{& \biLim{k<j}{\p_k} \deq && \pi_0^j \dl & \dc{\nu_0^j} & \dr \h \\
& \biLim{k<j}{\p_k} \ardl && \fa\sEY^j \dc{\eps} && \fb \ardr \\
\fa\sB^j && \pi_j^{\B} && \pi^{\E}_0 && \fb}}$$

\noindent where $\epsilon$ is such that $\pi^{\B}_k \epsilon= 
\vcenter{\xymatrix@C=-0pc@R=1pc{& \pi_k\tB \deq &&& \biLim{k<j}{\p_k} \dl & \dc{\nu_j} & \dr \fa_{\E,\Y}^j && \fb \deq \\
& \pi_k\tB &&& \fb\sB^j && \pi^{\E}_0 \deq && \fb \deq \\
&& \pi_k\tB \clunodos{\beta'_k} \opunodosb{(\pi\kj\tB)\inv \; \; } &&&& \pi^{\E}_0 \deq && \fb \deq \\
& \B\kj \op{\eps_k\inv} &&& \pi_j \deq && \pi^{\E}_0 \deq && \fb \deq \\
\pi_k\tB && \fa\sB^j && \pi_j\tB && \pi^{\E}_0 && \fb}}
\ \forall \ k<j$.

Consider also $\ff{q}_j$ as in \ref{lema1.2}.

%

Now consider $(\ff{h}',\lambda',\rho')$ a filler for the following diagram

 $$\xymatrix@R=3pc@C=1.5pc{\ff{A} \ar[rr]^{\h_0^j\ff{a}} \ar[d]_{\ff{i}} &\ar@{}[d]|{\cong \; \Downarrow \; \gamma'}& \ff{Y}_j \ar[d]^{\ff{q}_j} \\
            \ff{X} \ar[rr]_{\ff{h}} && \ff{P}_j}$$

\noindent where $\gamma'$ is such that $\pi^{\Y}_k \pi^j_0 \gamma= 
\vcenter{\xymatrix@C=-0pc@R=1pc{ \pi_k\tY \deq && \pi_0^j && \q_j && \h_0^j \deq &&& \fa \deq \\
\pi_k\tY &&& \fa\sY^j \cl{\beta_0^j} &&& \h_0^j \deq &&& \fa \deq \\
&& \Y\kj \cldosuno{\eta_k} \dl && \dc{\beta_0^{k,j}{}\inv} && \dr \h_0^j &&& \fa \deq \\
&& \h_0^k \deq &&&& \E\pkj \opdosuno{\beta_k\inv} &&& \fa \deq \\
&& \h_0^k \dl & \dc{\theta_k\inv} & \dr \pi_k &&& \e_j \deq && \fa \deq \\
&& \pi_k \deq && \fa_{\E,\Y}^j \deq &&& \e_j \dl & \dc{\gamma} & \dr \fa \\
&& \pi_k \deq && \fa_{\E,\Y}^j \dl & \dcr{\nu_0^j{}\inv} && \dr \fb && \ii \deq \\
&& \pi_k && \pi_0^j &&& \h && \ii}}
$ and \mbox{$\pi^j_1 \gamma= 
\vcenter{\xymatrix@C=-0pc@R=1pc{ \pi_1^j && \q_j && \h_0^j \deq &&& \fa \deq \\
& \p_j \cl{\beta_1^j} \dl & \dcr{\delta_j} && \dr \h_0^j &&& \fa \deq \\
& \pi_j\tB \deq &&& \h_1^j \dcellopbb{(\beta_1^{0,j})\inv} &&& \fa \deq \\  
& \pi_j\tB \deq &&& \E(0<j) \op{\beta_0\inv} &&& \fa \deq \\
& \pi_j\tB \deq && \pi_0 \deq && \e_j \dl & \dc{\gamma} & \dr \fa \\
& \pi_j\tB \ardr && \pi_0 \dc{\nu_1^j{}\inv} && \fb \ardl && \ii \deq \\
&& \pi_1^j && \h &&& \ii}}
$.}          

Finally, consider $\ff{g}$ as in the following diagram

$$\xymatrix@R=3pc@C=3.5pc{\ff{X} \ar@/^4ex/[rrd]^{\ff{h}'} \ar@{-->}[rd]|{\comw{M^M} \ff{g} \comw{M^M} } \ar@/_4ex/[rdd]_{\pi^{\E}_0 \ff{b}}  & \ar@{}[d]|{\cong \; \Uparrow \; \rho_0^{j}} \\
            \ar@{}[r]|{\cong \; \Downarrow \; \rho_1^{j}} & \ff{E}(j) \ar[r]^{\h_0^j} \ar[d]_{\h^j_1} 
            \ar@{}[rd]^{bipb \quad\quad\quad\quad\quad\quad\quad\quad\quad\quad\quad}_{\quad\quad\quad\quad\quad\quad\quad\quad\quad\quad\quad\quad \cong \; \Downarrow \; \delta_j } 
            & \ff{Y}_j \ar[d]^{\ff{p}_j}\\
            & \ff{B} \ar[r]_{\pi^{\ff{B}}_j} & \ff{B}_j }$$
            
It can be checked that $(\ff{g}, \lambda, \rho)$ is a filler for diagram \eqref{diagramaej} where $\lambda$ is such that \mbox{$\h_0^j\lambda = 
\vcenter{\xymatrix@C=-0pc@R=1pc{& \h_0^j \dl & \dcr{\lambda'} && \dr \fa \\  & \h' \opb{\rho_0^j{}\inv} &&& \ii \deq \\  \h_0^j && \g && \ii}}
$,} \mbox{$\h_1^j \lambda= 
\vcenter{\xymatrix@C=-0pc@R=1pc{ & \h_1^j \dcellopbb{(\beta_1^{0,j})\inv} &&& \fa \deq \\
& \E(0<j) \op{\beta_0\inv} &&& \fa \deq \\
\pi_0 \deq && \e_j \dl & \dc{\gamma} & \dr \fa \\
\pi_0 \dl & \dc{\rho_1^j{}\inv} & \dr \fb && \ii \deq \\
\h_1^j && \g && \ii}}
$} and $\rho$ is such that \mbox{$\h_0^k\pi^{\E}_k \rho= 
\vcenter{\xymatrix@C=-0.2pc@R=1pc{ & \h_0^k \deq && \pi_k\tE && \e_j && \g \deq \\
& \h_0^k \dl & \dcr{\beta_0^{k,j}} && \dr \E\pkj \cl{\beta_k} &&&  \g \deq \\
& \Y\kj \deq &&& \h_0^j &&& \g \\
& \Y\kj \opunodos{\eta_k\inv \quad} &&&&& \h' \cldosuno{\rho_0^j \;} \deq \\
\pi_k\tY \deq &&& \fa\sY^j \opb{(\beta_0^j)\inv} &&& \h' \deq \\
\pi_k\tY \deq && \pi_0^j \deq && \q_j && \h' \\
\pi_k\tY \deq && \pi_0^j \dl & \dcr{\nu_0^j} && \dr \h \cl{\rho'} \\
\pi_k\tY \dl & \dc{\theta_k} & \dr \fa_{\E,\Y}^j &&& \fb \deq \\
\h_0^k && \pi_k\tE &&& \fb}}
$} and \mbox{$\h_1^k\pi^{\E}_k \rho= 
\vcenter{\xymatrix@C=-0.2pc@R=1pc{ \h_1^k \deq && \pi_k\tE && \e_j && \g \deq \\
\h_1^k &&& \E\pkj \cl{\beta_k} &&& \g \deq \\
&& \h_1^j \cldosuno{\beta_1^{k,j}} \dl & \dcr{\rho_1^j} &  && \dr \g \\
&& \pi_0\tE \opdosuno{\pi_{0<k}\inv} &&&& \fb \deq \\
\h_1^k &&& \pi_k\tE &&& \fb}}$}
$\ \forall \ k<j$.

Then, by b), $\Y \mr{\ff{p}}\ff{B}\in L$.

We have proved that $L$ contains all strong fibrations. Then, by c), $L$ contains all fibrations as we wanted to prove.

  \item The proof is analogous to the previous one and we leave it to the reader. 
 \end{enumerate}
\end{proof}

\begin{lemma}\label{lema5'.2}
If $\ff{f}$ and $\ff{g}$ are two composable morphisms in $\Prop{C}$ such that $\ff{f}$ and $\ff{g}$ are both fibrations or both cofibrations, then if two out of the three morphisms $\ff{f}$, $\ff{g}$ and $\ff{g}\ff{f}$ are weak equivalences, so is the third one.
\end{lemma}

\begin{proof} \comw{a}
We will do the case where $\ff{f}$ and $\ff{g}$ are both cofibrations. The other case is analogous and we omit it.
\begin{itemize}
 \item[-] Case I: Suppose that $\ff{f}$ and $\ff{g}$ are both weak equivalences. By \ref{lema2'.2}, it is enough to check that $\ff{g}\ff{f}$ is a trivial cofibration and, by \ref{lema3'.2}, this can be checked by proving that $\ff{g}\ff{f}$ has the left lifting property with respect to all fibrations. The proof is dual to the proof of the fact that fibrations are closed under composition from \ref{sobran axiomas}.
%
%
%
%
%
%
%
%

\item[-] Case II: Suppose that $\ff{f}$ and $\ff{g}\ff{f}$ are both weak equivalences. We want to check that $\ff{g}$ has the right lifting property with respect to all fibrations. So suppose that we have a fibration $\ff{p}$ and a diagram of the form 

  $$\xymatrix@R=1.5pc@C=1.5pc{\ff{Y} \ar@{}[rrdd]|{\cong \; \Downarrow \; \gamma } \ar[rr]^{\ff{a}} \ar[dd]_{\ff{g}} & & \ff{E} \ar[dd]^{\ff{p}} \\ 
			     &   &\\
			     \ff{Z} \ar[rr]_{\ff{b}} & & \ff{B} }$$  
			     
Since $\ff{g}\ff{f}$ is a trivial cofibration, there exists a filler $(\ff{h}',\lambda',\rho')$ for the following diagram

  $$\xymatrix@R=1.5pc@C=1.5pc{\ff{X} \ar@{}[rrdd]|{\cong \; \Downarrow \; \gamma \ff{f}} \ar[rr]^{\ff{a}\ff{f}} \ar[d]_{\ff{f}} & & \ff{E} \ar[dd]^{\ff{p}} \\ 
			     \ff{Y} \ar[d]_{\ff{g}} &   &\\
			     \ff{Z} \ar[rr]_{\ff{b}} & & \ff{B} }$$  
			     
Since $\ff{f}$ is a trivial cofibration, there exists $\ff{X}'\mr{\ff{f}'}\ff{Y}'\in \cc{H}om_p(\hat{\ff{J}}^{op},\cc{C})$ for some cofinite filtered poset $\ff{J}$ with a unique initial object such that $\ff{f}$ is a retract of $\ff{f}'$. From the fact that $\cc{H}om_p(\hat{\ff{J}}^{op},\cc{C})$ is a closed 2-bmodel 2-category and \ref{2.3.5 de EH}, we have that $\ff{f}'$ induces a trivial cofibration $\ff{Y}'\bigtriangleup \widetilde{\ff{X}'}\bigtriangledown \ff{Y}' \mr{\ff{k}_{\ff{f}'}} \widetilde{\ff{Y}'} \in \cc{H}om_p(\hat{\ff{J}}^{op},\cc{C})$ and so, by \ref{kretracto} $\ff{f}$ induces a trivial cofibration $\ff{Y}\bigtriangleup \widetilde{\ff{X}}\bigtriangledown \ff{Y} \mr{\ff{k}_{\ff{f}}} \widetilde{\ff{Y}} \in \Prop{C}$. Then there exists a filler $(\ff{h}_0,\lambda_0,\rho_0)$ for the following diagram

$$\xymatrix@R=3pc@C=5pc{\Y \bigtriangleup \widetilde{\X} \bigtriangledown \Y \ar[d]_{\fk_{\f}} \ar[r]^>>>>>>>>>>>>{\fa \ \bigtriangleup \ \fa \f \sigma\tX \ \bigtriangledown \ \h'\g} \ar@{}[rd]|{\cong \; \Downarrow \; \alpha \ \bigtriangleup \ \beta \ \bigtriangledown \ \delta} & \E \ar[d]^{\p} \\
\widetilde{\Y} \ar[r]_{\fb \g \sigma\tY} & \B}$$


\noindent where $\alpha= 
\vcenter{\xymatrix@C=-0pc@R=1pc{ \p \dl && \dc{\gamma} && \dr \fa \\
 \fb \deq &&&& \g \opdostres{=} \\
 \fb \deq && \g \deq &&&&& id\sY \optrestres{\gamma_0\tY} \\
 \fb && \g && \sigma\tY && \fk_{\f} && \fb_{\f} && \lambda_0 }}$,
$\beta= 
\vcenter{\xymatrix@C=-0pc@R=1pc{\p \dl & \dc{\gamma} & \dr \fa &&& \f \dl & \dcr{\beta_{\f}\inv} && \dr \sigma\tX \\
\fb \deq && \g \deq &&& \nabla_{\f} \op{\theta_{\f}} &&& \fa_{\f} \deq \\
\fb && \g && \sigma\tY && \fk_{\f} && \fa_{\f}}}
$ and
\mbox{$\delta= 
\vcenter{\xymatrix@C=-0pc@R=1pc{\p && \h' &&& \g \deq \\
& \fb \cl{\rho'} \deq &&&& \g \opdostres{=} \\
& \fb \deq && \g \deq &&&&& id\sY \optrestres{\gamma_1\tY} \\
& \fb && \g && \sigma\tY && \fk_{\f} && \fb_{\f} && \lambda_1 }}$.}

By similar arguments, $\ff{g}$ induces a trivial cofibration $\ff{Z}\bigtriangleup \widetilde{\ff{Y}} \mr{\ff{k}'_{\g}} \widetilde{\ff{Z}}$. Then there exists a filler $(\ff{h}_1,\lambda_1,\rho_1)$ for the following diagram


$$\xymatrix@R=3pc@C=3.5pc{ \Z \bigtriangleup \widetilde{\Y} \ar[d]_{\fk'_{\g}} \ar@{}[dr]|{\cong \; \Downarrow \; \alpha' \ \bigtriangleup \ \beta'} \ar[r]^{\h' \ \bigtriangleup \ \h_0} & \E \ar[d]^{\p} \\
\widetilde{\Z} \ar[r]_{\fb \sigma\tZ} & \B}$$

  
\noindent where \mbox{$\alpha'= 
\vcenter{\xymatrix@C=-0pc@R=1pc{& \p && \h' \\
&& \fb \cl{\rho} \opdostres{=} \\
\fb \deq &&&&& id_{\Z} \optrestres{\gamma_0\tZ} \\
\fb \deq && \sigma\tZ \deq && \fk_{\g} \deq && \fb_{\g} \dl & \dc{\mu_{\g}\inv} & \dr \lambda_0 \\
\fb \deq && \sigma\tZ \deq && \fk_{\g} && \tilde{\fk}_{\g}  && \fb_{\g}' \deq \\
\fb && \sigma\tZ &&& \fk'_{\g} \cl{=} &&& \fb_{\g}'}}
$} and
\mbox{$\beta'= 
\vcenter{\xymatrix@C=-0pc@R=1pc{ & \p \ardl &&& \dc{\rho_0} && \h_0 \ardr \\
\fb \deq &&& \g \dl && \dc{\beta_{\g}\inv} && \dr \sigma\tY \\
\fb \deq &&& \nabla_{\g} \op{\theta_{\g}} &&&& \fa_{\g} \deq \\
\fb \deq && \sigma\tZ \deq && \fk_{\g} \deq &&& \fa_{\g} \op{\alpha_{\g}\inv} \\
\fb \deq && \sigma\tZ \deq && \fk_{\g} && \tilde{\fk}_{\g}  && \fa'_{\g} \deq \\
\fb && \sigma\tZ &&& \fk'_{\g} \cl{=} &&& \fa'_{\g}}}$.}

\noindent It can be checked that $(\ff{h}_1\circ \ii^{\Z}_1, 
\vcenter{\xymatrix@C=-0.2pc@R=1pc{ &&&&&&&& \fa \ardlllll \ardrrrr \dcl{\lambda_0 \fb_{\f} \lambda_1} \\
&&& \h_0 \optrestres{\lambda_1 \fa_{\g}'} &&&&& \fk_{\f} \deq && \fb_{\f} \deq && \lambda_0 \deq \\
\h_1 \deq &&& \fk'_{\g} \op{=} &&& \fa'_{\g} \deq && \fk_{\f} && \fb_{\f} && \lambda_0 \\
\h_1 \deq && \fk_{\g} \deq && \tilde{\fk}_{\g} \deq && \fa'_{\g} \dl && \dc{\delta'_{\g}} && \dr \ii\tY_0 \cldosdos{=} \\
\h_1 \deq && \fk_{\g} \deq && \tilde{\fk}_{\g} \dl & \dc{\mu_{\g}} & \dr \fb'_{\g} &&&& \g \deq \\
\h_1 \deq && \fk_{\g} && \fb_{\g} && \lambda_0 &&&& \g \deq \\
\h_1 &&&& (\ii\tZ)_1 \cldosdos{=} &&&&&& \g }}
, 
\vcenter{\xymatrix@C=-0pc@R=1pc{\p \dl & \dc{\rho_1} & \dr \h_1 && \ii_1\tZ \deq \\
\fb \deq && \sigma\tZ && \ii_1\tZ \\
\fb &&& id_{\Z} \cl{\gamma_1\tZ} \\
& \fb \clunodos{=} }}
)$ is the filler that we were looking for. 
 
\item[-] Case III: Suppose that $\ff{g}$ and $\ff{g}\ff{f}$ are weak equivalences. By \ref{lema4'.2}, it is enough to check that $\ff{f}$ has the left lifting property with respect to all fibrations between fibrant objects. So suppose that we have a fibration $\ff{E}\mr{\ff{p}}\ff{B}$ between fibrant objects and a diagram of the form

$$\xymatrix@R=1.5pc@C=1.5pc{\ff{X} \ar@{}[rrdd]|{\cong \; \Downarrow \; \gamma } \ar[rr]^{\ff{a}} \ar[dd]_{\ff{f}} & & \ff{E} \ar[dd]^{\ff{p}} \\ 
			     &  &\\
			     \ff{Y} \ar[rr]_{\ff{b}} & & \ff{B}}$$  
			     
Since $\ff{g}$ is a trivial cofibration, there exists a filler $(\ff{h}',\lambda',\rho')$ for the following \mbox{diagram}

$$\xymatrix@R=1.5pc@C=1.5pc{\ff{Y} \ar@{}[rrdd]|{\cong} \ar[rr]^{\ff{b}} \ar[dd]_{\ff{g}} & & \ff{B} \ar[dd]\\ 
			     &  &\\
			     \ff{Z} \ar[rr] & & \ * \ }$$ 

Since $\ff{g}\ff{f}$ is a trivial cofibration there exists a filler $(\ff{h},\lambda,\rho)$ for the following diagram

 $$\vcenter{\xymatrix@R=1.5pc@C=1.5pc{\ff{X} \ar@{}[rrd]|{\cong  \; \Downarrow \; \gamma } \ar[rr]^{\ff{a}} \ar[d]_{\ff{f}} & & \ff{E} \ar[dd]^{\ff{p}} \\ 
			     \ff{Y} \ar@{}[rd]|>>>>{\cong  \; \Downarrow \; \lambda'} \ar[d]_{\ff{g}} \ar@/^1ex/[rrd]^{\fb} && \\
			     \ff{Z} \ar[rr]_{\ff{h}'} & & \ff{B} }}
\vcenter{\xymatrix@R=1.5pc@C=1.5pc{\quad = \quad}}
 \vcenter{\xymatrix@R=1.5pc@C=1.5pc{\ff{X} \ar@{}[rrdd]|{\cong  \; \Downarrow \; \lambda' \ff{f} \circ \gamma } \ar[rr]^{\ff{a}} \ar[d]_{\ff{f}} & & \ff{E} \ar[dd]^{\ff{p}} \\ 
			     \ff{Y} \ar[d]_{\ff{g}} & \\
			     \ff{Z} \ar[rr]_{\ff{h}'} & & \ff{B} }}$$ 
                
\noindent It is straightforward to check that $(\ff{h}\ff{g},\lambda, 
\vcenter{\xymatrix@R=1pc@C=.1pc{ \p && \h && \g \deq \\ & \h' \cl{\rho} &&& \g \\ && \fb \clunodos{\; \; \lambda'{}\inv} }})$ is the filler that we were looking for.

\end{itemize}
 
\end{proof}

\begin{lemma}\label{lema6'.2}
If $\ff{X}\mr{\ff{f}}\ff{Y}$ is a weak equivalence and $\ff{f}\Mr{\alpha \cong} \ff{p}\ff{i}$ with $\ff{i}$ a trivial cofibration and $\ff{p}$ a fibration, then $\ff{p}$ is a trivial fibration.  
\end{lemma}

\begin{proof}
 Since $\ff{f}$ is a weak equivalence, we can factorize $\ff{f}\Mr{\beta \cong} \ff{p}'\ff{i}'$ where $\ff{p}'$ is a trivial fibration and $\ff{i}'$ is a trivial cofibration. Then there exist fillers $(\ff{g},\lambda,\rho)$, $(\ff{g}',\lambda',\rho')$ for the following diagrams

$$\vcenter{\xymatrix@C=1.5pc@R=1.5pc{\X \ar[rrdd]|{\comw{M^M} \ff{f} \comw{M^M} } \ar[rr]^{\ii'} \ar[dd]_{\ii}  & \ar@{}[dr]|{\cong \; \Uparrow \; \beta} & \Z' \ar[dd]^{\p'}\\
                  \ar@{}[dr]|{ \cong \; \Downarrow \; \alpha} &&  \\ 
                  \Z \ar[rr]_{\p} && \Y}}
\vcenter{\xymatrix@C=1.5pc@R=1.5pc{=}}
\vcenter{\xymatrix@R=1.5pc@C=1.5pc{\ff{X} \ar@{}[rrdd]|{\cong \; \Downarrow \; \alpha \circ \beta^{-1} } \ar[rr]^{\ff{i}'} \ar[dd]_{\ff{i}} & & \ff{Z}' \ar[dd]^{\ff{p}'}\\ 
			     &  &\\
			     \ff{Z} \ar[rr]_{\ff{p}} & & \ff{Y}}}
\vcenter{\xymatrix@C=1.5pc@R=1.5pc{\quad\quad\quad}}
\vcenter{\xymatrix@R=1.5pc@C=1.5pc{\ff{X} \ar@{}[rrdd]|{\cong \; \Downarrow  \; \beta^{-1} \circ \alpha } \ar[rr]^{\ff{i}} \ar[dd]_{\ff{i}'} & & \ff{Z} \ar[dd]^{\ff{p}}\\ 
			     &  &\\
			     \ff{Z}' \ar[rr]_{\ff{p}'} & & \ff{Y}}}
\vcenter{\xymatrix@C=1.5pc@R=1.5pc{=}}
\vcenter{\xymatrix@C=1.5pc@R=1.5pc{\X \ar[rrdd]|{\comw{M^M} \ff{f} \comw{M^M} } \ar[rr]^{\ii} \ar[dd]_{\ii'}  & \ar@{}[dr]|{\cong \; \Uparrow \; \alpha} & \Z \ar[dd]^{\p}\\
                  \ar@{}[dr]|{ \cong \; \Downarrow \; \beta} &&  \\ 
                  \Z' \ar[rr]_{\p'} && \Y}}$$   

By \ref{2.3.5 de EH}, there exist fillers $(\ff{h}_0,\lambda_0, \rho_0)$, $(\ff{h}_1,\lambda_1, \rho_1)$ for the following diagrams
 
$$\vcenter{\xymatrix@R=3pc@C=5pc{\Z \bigtriangleup \widetilde{\X} \bigtriangledown \Z \ar[d]_{\fk_{\ii}} \ar[r]^>>>>>>>>>>>>{id_{\Z} \ \bigtriangleup \ \ii \sigma\tX \ \bigtriangledown \ \g'\g} \ar@{}[rd]|{\cong \; \Downarrow \; id_{\p} \ \bigtriangleup \ id_{\p \ii} \ \bigtriangledown \ \rho \circ \rho' \g} & \Z \ar[d]^{\p} \\
\widetilde{\Z} \ar[r]_{\p \sigma\tZ} & \Y}}
\vcenter{\xymatrix@R=3pc@C=5pc{ \quad \quad }}
\vcenter{\xymatrix@R=3pc@C=5pc{\Z' \bigtriangleup \widetilde{\X} \bigtriangledown \Z' \ar[d]_{\fk_{\ii'}} \ar[r]^>>>>>>>>>>>>{id_{\Z'} \ \bigtriangleup \ \ii' \sigma\tX \ \bigtriangledown \ \g\g'} \ar@{}[rd]|{\cong \; \Downarrow \; id_{\p'} \ \bigtriangleup \ id_{\p' \ii'} \ \bigtriangledown \ \rho' \circ \rho \g'} & \Z' \ar[d]^{\p'} \\
\widetilde{\Z}' \ar[r]_{\p' \sigma^{\Z'}} & \Y}}$$



By working as in the proof of \ref{lema5'.2} Case II, one can check that $\ff{p}$ and $\ff{p}'$ have similar lifting properties and so, by \ref{lema3'.2}, $\ff{p}$ is a trivial fibration.
%
%
%
%
\end{proof}

\begin{lemma}\label{lema7'.2}
If $\ff{X}\mr{\ff{f}}\ff{Y}$ is a weak equivalence in $\Prop{C}$ and $\ff{f}\Mr{\alpha\cong} \ff{p}\ff{i}$ with $\ff{i}$ a cofibration and $\ff{p}$ a trivial fibration, then $\ff{i}$ is a trivial cofibration.  
\end{lemma}

\begin{proof}
 The proof is analogous to the proof of \ref{lema6'.2} and is omitted.
\end{proof}

The following lemmas assume that 2-N3a) holds in $\cc{C}$ but they would be analogous in case 2-N3b) holds instead of 2-N3a).

\begin{lemma}\label{lema8'.2}
 If $\ff{E}\mr{\ff{p}}\ff{B}$ is a trivial fibration in $\Prop{C}$, then there exists a trivial cofibration $\ff{B}\mr{\ff{s}}\ff{E}$ such that $\ff{p}\ff{s}\cong id$. Furthermore, every quasi-section of $\ff{p}$ is a trivial cofibration.
\end{lemma}

\begin{proof}
By axiom 2-N3, $\ff{B}$ is cofibrant. Then, by \ref{lema1'.2}, we have a filler $(\ff{s},\lambda,\rho)$ for the following diagram: 

$$\xymatrix@R=1.5pc@C=1.5pc{\emptyset \ar@{}[rrdd]|{\cong }\ar[rr] \ar[dd] & & \ff{E} \ar[dd]^{\ff{p}} \\ 
			     &  &\\
			     \ff{B} \ar[rr]_{id_{\ff{B}}} & & \ff{B}}$$ 
			     
Then $\ff{p}\ff{s} \cong id$.

Let $\ff{s}'$ be a quasi-section of $\ff{p}$. One can proceed as in the proof of \ref{lema5'.2} Case II to prove that $\ff{s}'$ has the left lifting property with respect to all fibrations and so is a trivial cofibration.
\end{proof}

\begin{lemma}\label{lema9'.2}
 If $\ff{X}\mr{\ff{f}}\ff{Y}$ is a trivial fibration, $\ff{Y}\mr{\ff{g}}\ff{Z}$ is a trivial cofibration in $\Prop{C}$ and $\ff{h}\cong \ff{g}\ff{f}$, then $\ff{h}$ is a weak equivalence.
\end{lemma}

\begin{proof}
First observe that since we have already proved axioms 2-M6a), 2-M6b) and \mbox{2-M6c),} we can assume that axiom 2-M7 holds in $\Prop{C}$. Then it is enough to prove the lemma only for the case where $\ff{h}=\ff{g}\ff{f}$. 

By axiom 2-M2, $\ff{g}\ff{f}$ can be factored as $\ff{g}\ff{f}\Mr{\cong \alpha} \ff{p}\ff{i}$ where $\ff{p}$ is a trivial fibration and $\ff{i}$ is a cofibration. Also, by \ref{lema8'.2}, there exists a trivial cofibration $\ff{Y}\mr{\ff{s}}\ff{X}$ such that $\ff{f}\ff{s}\Mr{\cong \beta}1$.
 
 Since $\ff{p}$ is a fibration and $\ff{g}$ is a trivial cofibration, there exists a filler $(\ff{h},\lambda,\rho)$ for the following diagram
 
 $$\xymatrix@R=3pc@C=1.5pc{\ff{Y} \ar[rr]^{\ff{i}\ff{s}} \ar[d]_{\ff{g}} 
             & \ar@{}[d]|{\cong \; \Downarrow \; \ff{g}\beta \circ \alpha \ff{s}}
             & \ff{W} \ar[d]^{\ff{p}}
             \\
             \ff{Z} \ar[rr]_{id_{\ff{Z}}} 
             && \ff{Z}}$$

Observe that, by \ref{lema8'.2}, $\ff{h}$ is a trivial cofibration. Then, since $\ff{g}$ is also a trivial cofibration, by axiom 2-M3 (this axiom is satisfied by \ref{sobran axiomas} and the fact that we have already proved axioms 2-M6a),b) and c)) plus \ref{lema5'.2}, $\ff{h}\ff{g}$ is a trivial cofibration. Then, since $\ff{i}\ff{s}\cong \ff{h}\ff{g}$, $\ff{i}\ff{s}$ is also a trivial cofibration and so, by \ref{lema5'.2}, $\ff{i}$ is a trivial cofibration which concludes the proof.
\end{proof}

{\bfseries Axiom 2-M5w}: Case I: Suppose that $\ff{f}$ and $\ff{g}$ are weak equivalences: In this case, we can factorize $\ff{f}\cong \ff{p}\ff{i}$ and $\ff{g}\cong \ff{q}\ff{j}$ where $\ff{p}$, $\ff{q}$ are trivial fibrations and $\ff{i}$, $\ff{j}$ are trivial cofibrations. By \ref{lema9'.2}, $\ff{j}\ff{p}$ is a weak equivalence and so it can be factorized  as $\ff{j}\ff{p}\cong \ff{r}\ff{k}$ where $\ff{r}$ is a trivial fibration and $\ff{k}$ is a trivial cofibration. 

\noindent By \ref{lema5'.2} plus axiom 2-M3, $\ff{q}\ff{r}$ is a trivial fibration and $\ff{k}\ff{i}$ is a trivial cofibration. Then $\ff{g}\ff{f} \cong \ff{q}\ff{j}\ff{p}\ff{i}\cong \ff{q}\ff{r}\ff{k}\ff{i}$ and so is a weak equivalence as we wanted to prove.

Case II: Suppose that $\ff{f}$ and $\ff{g}\ff{f}$ are weak equivalences: In this case, we can factorize $\ff{f}\cong \ff{p}\ff{i}$ where $\ff{p}$ is a trivial fibration and $\ff{i}$ is a trivial cofibration. By axiom 2-M2, we can also factorize $\ff{g}\cong \ff{q}\ff{j}$ where $\ff{q}$ is a fibration and $\ff{j}$ is a trivial cofibration. Then $\ff{g}\ff{f} \cong \ff{q}\ff{j}\ff{p}\ff{i}$ where $\ff{j}\ff{p}$ is a weak equivalence by \ref{lema9'.2}. Then we can factorize $\ff{j}\ff{p}\cong \ff{r}\ff{k}$ where $\ff{r}$ is a trivial fibration and $\ff{k}$ is a trivial cofibration and so $\ff{g}\ff{f}\cong \ff{q}\ff{r}\ff{k}\ff{i}$ where $\ff{q}\ff{r}$ is a fibration and $\ff{k}\ff{i}$ is a trivial cofibration by \ref{lema5'.2}. Therefore, by \ref{lema6'.2}, $\ff{q}\ff{r}$ is a trivial fibration and so, by \ref{lema5'.2}, $\ff{q}$ is a trivial fibration which concludes the proof that $\ff{g}$ is a weak equivalence.

Case III: Suppose that $\ff{g}$ and $\ff{g}\ff{f}$ are weak equivalences: The proof is analogous to the proof of Case II but using \ref{lema7'.2} instead of \ref{lema6'.2}.

To conclude the proof, suppose that $\ff{f}$ is an isomorphism. Then it can be easily checked that $\ff{f}$ has the left lifting property with respect to all fibrations and so, by \ref{lema3'.2}, $\ff{f}$ is a trivial cofibration. Besides, $\ff{f}$ can be factored as $\ff{f}\cong id \ff{f}$ and so is a weak equivalence.

This was the only remaining axiom to conclude that $\Prop{C}$ is a closed 2-bmodel 2-category as we wanted to prove.
\cqd

\pagebreak
\begin{center}
{\bf Resumen en castellano de la secci\'on \ref{2-modelos}}
\end{center}

En esta secci\'on dotamos a $\Pro{C}$ de una estructura de ``closed 2-bmodel \mbox{2-category}'' cuando $\cc{C}$ la posee. Esta secci\'on est\'a inspirada en la prueba dada en \cite{EH} del hecho de que $\ff{Pro}(\ff{C})$ es una ``closed model category'' en el caso 1-dimensional. La demostraci\'on en nuestro contexto result\'o ser mucho m\'as complicada debido a que los diagramas no conmutan estrictamente sino que conmutan salvo isomorfismo. Por esta raz\'on, nos vimos obligados a trabajar con pseudo-funtores y transformaciones pseudo-naturales si bien los objetos y los morfismos en $\Pro{C}$ son 2-funtores y transformaciones 2-naturales. Se podr\'ia pensar (y nosotros lo hicimos por un tiempo) que trabajar con 2-funtores y transformaciones pseudo-naturales ser\'ia suficiente pero no lo es. La raz\'on para tomar pseudo-funtores queda evidenciada en la prueba del axioma 2-M2 para $\pCJ$ donde $\ff{Z}$ resulta un pseudo-funtor que no es necesariamente un 2-funtor a\'un cuando todos los dem\'as lo son. 

La demostraci\'on del teorema principal que establece que $\Pro{C}$ es una ``closed 2-bmodel 2-category'' tiene tres pasos. El primero, \ref{closed 2-model structure for CJ}, consiste en definir una estructura de ``closed 2-bmodel 2-category'' para la 2-categor\'ia $\pCJ$ (\ref{ccHom}) a partir de una estructura para $\cc{C}$, donde $\ff{J}$ es un poset cofinito y filtrante con un \'unico objeto inicial. Cabe comentar que en \cite{EH} no se pide que $\ff{J}$ tenga un \'unico objeto inicial, lo cual para nosotros es un requisito esencial, incluso en el caso 1-dimensional. El segundo paso, \ref{closed 2-bmodel structure in Prop}, consiste en usar la estructura en $\pCJ$ para definir una en la 2-categor\'ia 
$\Prop{C}$. Para esto se prueban primero los aspectos de completitud y co-completitud finita que resultar\'an en la demostraci\'on del axioma 2-M0b y luego se demuestran de manera encadenada el resto de los axiomas que requieren la demostraci\'on de varias propiedades que destacamos como lemas pues tienen inter\'es en s\'i mismas. Finalmente, el tercer paso consiste en transferir esta estructura a $\Pro{C}$ usando que esta 2-categor\'ia es ``retract pseudo-equivalente'' a $\Prop{C}$ ( \ref{proppseudoeqapro}) mediante el resultado probado en \ref{pseudo-equivalencerespetaaxiomas}.   

\pagebreak

\bibliographystyle{unsrt}

\end{document}